\documentclass[a4paper,11pt,english,reqno]{amsart}
\usepackage{amsmath,amssymb,amsfonts,dsfont,amsthm,upgreek,bm,mathtools}    
\usepackage[utf8]{inputenc}
\usepackage[english]{babel}
\usepackage[T1]{fontenc} 
\usepackage{graphicx,xcolor,pgfkeys}
\usepackage{float}
\usepackage[font=small,labelfont=bf]{caption}
\usepackage[a4paper, margin=2.4cm]{geometry}
\usepackage{tikz,pgfplots}
\usepackage{tkz-fct}
\usepgfplotslibrary{polar}
\usetikzlibrary{patterns, arrows}

\usepackage{enumerate}
\usepackage{enumitem}
\usepackage[normalem]{ulem}
\newcommand\redsout{\bgroup\markoverwith{\textcolor{red}{\rule[0.5ex]{2pt}{1pt}}}\ULon}
\newcommand{\stkout}[1]{\ifmmode\text{\redsout{\ensuremath{#1}}}\else\redsout{#1}\fi}
\usepackage{scalerel,stackengine}
\stackMath
\newcommand\reallywidecheck[1]{%
\savestack{\tmpbox}{\stretchto{%
  \scaleto{%
    \scalerel*[\widthof{\ensuremath{#1}}]{\kern-.6pt\bigwedge\kern-.6pt}%
    {\rule[-\textheight/2]{1ex}{\textheight}}%WIDTH-LIMITED BIG WEDGE
  }{\textheight}% 
}{0.6ex}}%
\stackon[1pt]{#1}{\scalebox{-0.8}{\tmpbox}}%
}

\newcommand{\prob}{\mathds{P}}
\newcommand{\erw}{\mathds{E}}
\newcommand{\e}{\mathrm{e}}

\newcommand{\loc}{\mathrm{loc}}
\newcommand{\comp}{\mathrm{comp}}

\newcommand{\tr}{\mathrm{tr\,}}

\newcommand{\supp}{\mathrm{supp\,}}

\newcommand{\dist}{\mathrm{dist\,}}

\newcommand{\Op}{\mathrm{Op}}
\newcommand{\mO}{\mathcal{O}}
\newcommand{\C}{\mathds{C}}
\newcommand{\R}{\mathds{R}}

\newcommand{\N}{\mathds{N}}

\newcommand{\wt}{\widetilde}

\newcommand{\Vol}{\mathrm{Vol}}

\newcommand{\wh}{\widehat}

\newcommand{\gI}{\mathfrak{I}}
\newcommand{\gd}{\mathfrak{d}}

\newcommand{\by}{\boldsymbol{y}}
\newcommand{\wit}{\widetilde}
\newcommand{\wih}{\widehat}

\newcommand{\bw}{\boldsymbol{w}}

\newcommand{\E}{\mathbb{E}}

\newcommand{\fK}{\mathfrak{K}}

\newcommand{\cO}{\mathcal{O}}

\newcommand{\cK}{\mathcal{K}}

\newcommand{\cU}{\mathcal{U}}

\renewcommand{\geq}{\geqslant}

\renewcommand{\leq}{\leqslant}

\newcommand{\bx}{\boldsymbol{x}}

\newcommand{\bH}{\boldsymbol{H}}

\newcommand{\dd}{\mathrm{d}}
\newtheorem{thm}{Theorem}
\newtheorem{corollary}[thm]{Corollary}

\newtheorem{prop}[thm]{Proposition}
\newtheorem{lem}[thm]{Lemma}

\newtheorem{definition}[thm]{Definition}
\newtheorem{rem}[thm]{Remark}

\newtheorem{ex}[thm]{Example}
\newtheorem{hypo}[thm]{Hypothesis}

\numberwithin{equation}{section}
\numberwithin{thm}{section}

\usepackage{hyperref}
\hypersetup{pdfborder=0 0 0, 
	    colorlinks=true,
	    citecolor=blue,
	    linkcolor=blue,
	    urlcolor=blue,
	    pdfauthor={Maxime Ingremeau, Martin Vogel}
	   }
\def\corM{\textcolor{black}}
 \def\red{\textcolor{black}}

 \def\MI{\textcolor{black}}
\tikzset{
xmin/.store in=\xmin, xmin/.default=-3, xmin=-3,
xmax/.store in=\xmax, xmax/.default=3, xmax=3,
ymin/.store in=\ymin, ymin/.default=-3, ymin=-3,
ymax/.store in=\ymax, ymax/.default=3, ymax=3,
}
\newcommand {\grille}
{\draw[help lines] (\xmin,\ymin) grid (\xmax,\ymax);}
\newcommand {\axes} {
\draw[->] (\xmin,0) -- (\xmax,0);
\draw[->] (0,\ymin) -- (0,\ymax);
}
\newcommand {\fenetre}
{\clip (\xmin,\ymin) rectangle (\xmax,\ymax);}

\title{Emergence of Gaussian fields in noisy quantum chaotic dynamics}
\author{Maxime Ingremeau}
\address[Maxime Ingremeau]{\MI{Institut Fourier UMR 5582, Laboratoire de Mathématiques\\Université Grenoble Alpes, CS 40700, 38058 Grenoble cedex 9, France}}
\email{Maxime.Ingremeau@univ-grenoble-alpes.fr}
\author{Martin Vogel}
\address[Martin Vogel]{Institut de Recherche Math{\'e}matique Avanc{\'e}e - UMR 7501, 
Universit{\'e} de Strasbourg et CNRS, 7 rue René-Descartes, 67084 Strasbourg Cedex, France.}
\email{vogel@math.unistra.fr}
\date{\today}
\usepackage[most]{tcolorbox}
\begin{document}
\begin{abstract}
	We study the long time \red{semiclassical} Schr\"odinger evolution of Lagrangian states $f_h$ 
	on a compact Riemannian manifold $(X,g)$ of negative sectional curvature. We 
	consider two models of random Schr\"odinger operators 
	$P_h^\alpha=-h^2\Delta_g +h^\alpha Q_\omega$, $0<\alpha\leq 1$, where  
	the Laplace-Beltrami operator $-h^2\Delta_g$ on $X$ 
	is subject to a small random perturbation $h^\alpha Q_\omega$ given by 
	either a random potential or a random pseudo-differential operator. 
	Here, the potential or the symbol of $Q_\omega$ is bounded, but oscillates 
	and decorrelates at scale $h^{\beta}$, $0< \beta < \frac{1}{2}$.
	We prove a quantitative result that, under appropriate conditions on 
	$\alpha,\beta$, in probability with respect to $\omega$ the long 
	time propagation 
	$$
		\e^{\frac{i}{h}t_h P_h^\alpha } f_h, 
		\quad o(|\log h|)=t_h\to\infty, ~~h\to 0,
	$$
	rescaled to the local scale of $h$ around a uniformly at random 
	chosen point $x_0$ on $X$, converges in law to an isotropic stationary 
	monochromatic Gaussian field -- the Berry \red{random wave model}. We also provide 
	an $\omega$-almost sure version of this convergence along sufficiently fast 
	decaying subsequences $h_j\to 0$. 
\end{abstract}
\maketitle
\tableofcontents

\section{Introduction}
\textbf{Background.} 
The theory of quantum chaos aims at understanding the nature  
of a quantum system when its associated classical Hamiltonian 
system is chaotic. A guiding example is the Laplace-Beltrami 
operator on a negatively curved smooth Riemannian manifold $X$. 
There the geodesic flow has the Anosov property \cite{Ebe} which, 
in a sense, is the ideal chaotic behavior. The corresponding quantum dynamics 
is given, in the high-energy or semiclassical limit, by the 
unitary group generated by the Laplace-Beltrami operator $\Delta_g$ on 
$L^2(X)$. The chaotic nature of the geodesic flow is conjectured 
(and indeed proven in some cases) to have \red{ a deep}  
influence on the spectral properties of the Laplacian. For instance, the  
random matrix conjecture by Bohigas-Giannoni-Schmit \cite{BGS84a,BGS84b,Bo91} 
states that the fluctuations of the high-lying eigenvalues 
should resemble those of large Wigner random matrices. 
The corresponding eigenfunctions are conjectured by Rudnick-Sarnack  
to be \emph{uniquely quantum ergodic} \cite{RS94} 
(see also \cite{Vo79}). More precisely, it is conjectured that 
the family of eigenfunctions $\{\psi_\lambda\}_{\lambda}$ of $\Delta_g$ 
indexed by their corresponding eigenvalue satisfies
\begin{equation}\label{eq:QE}
	\langle \Op_{\lambda^{-1/2}}(a) \psi_\lambda | \psi_\lambda\rangle 
	\to \int_{S^*X}a d\rho, \quad \lambda \to \infty,
\end{equation}
for any $a\in C^\infty(S^*X)$.  This conjecture is motivated by the quantum 
ergodicity theorem of \v{S}nirel'man \cite{Shni}, Zelditch \cite{Zel} and 
Colin de Verdière \cite{CdV},  claiming that (\ref{eq:QE}) holds for a density 
one sequence of eigenvalues $\lambda$.  We refer the reader to \cite{Dya} for 
an account of recent advances regarding the quantum unique ergodicity conjecture.
\\
\par 
Another way of understanding the delocalization properties of the Laplacian's 
eigenfunctions is covered by \emph{Berry's random wave conjecture} \cite{Ber}. 
It claims that, in the high energy limit, quantum chaotic eigenfunctions 
should resemble at a local scale a random superposition of plane waves.  \red{This conjecture has been confirmed by numerical and experimental evidence (see for instance (\cite{hejhal1992topography}, \cite{aurich1993statistical}, \cite{backer1998rate}, \cite{barnett2006asymptotic}) or experimentally (\cite{savytskyy2004experimental}, \cite{hul2005investigation} \cite{kuhl2007nodal})).  However,  the precise mathematical formulation of this conjecture has been debated for a long time,  since it compares  a 
sequence of deterministic objects with a random object.
%t, was considered as a 
%heuristic rather than a precise mathematical statement. 
Motivated }
by the Benjamini-Schramm convergence in the theory of large random graphs, 
and by work by Bourgain \cite{Bou} on the torus, 
it was recently suggested in \cite{ABLM, LWL} \MI{(see also \cite{Alba})} that to make sense of the 
randomness in Berry's heuristic one should look at the eigenfunctions near 
a random point. More precisely, when $\psi_\lambda$ -- an eigenfunction of 
the Laplacian at energy $\lambda$ -- is rescaled to the scale of the wavelength 
$\lambda^{-\frac{1}{2}}$ around a randomly chosen point on the manifold, it defines 
a family of random functions, whose law 
should converge weakly to that of an isotropic stationary monochromatic 
Gaussian random field. See section \ref{ss:local_limits_def} 
for definitions and for a more precise statement. 
Note that this interpretation of Berry's conjecture implies the quantum 
unique ergodicity conjecture, as is proven in \cite{LWL}. 
\\
\par
\textbf{The setting: Quantum chaotic propagation.}
In the present work we will adopt a semiclassical point of view. Rescaling the 
eigenvalue equation $(-\Delta_g - \lambda)\psi_\lambda = 0$ by the eigenvalue 
$\lambda=h^{-2}$ we get the semiclassical equation $(-h^2\Delta_g -1)\psi_h=0$ 
where $h>0$ denotes the semiclassical parameter. Moreover,  in this paper, we will 
not be concerned with genuine eigenfunctions 
of the Laplacian but rather with another important question in quantum chaos: understanding 
the long-time behaviour of the Schrödinger \red{evolution of some} highly oscillating 
initial data: \red{Lagrangian (or WKB) states}.  \red{More precisely,} we wish to study,  \red{for $t_h \gg 1$,} % the long-time evolution of highly oscillatory 
%initial data under the Schr\"odinger evolution semigroup 
 $\e^{\red{i h t_h \Delta_g}} f_h$ which is the solution to 
\begin{equation}\label{intro_eq1}
	\begin{cases}
	ih\partial_t u = -h^2\Delta_g u, \\
	u|_{t=0} = f_h = a\e^{\frac{i}{h}\phi}
	\end{cases}
\end{equation} 
\red{where $a, \phi$ are suitable smooth $h$-independent functions.}
\par
For such long-time propagated quantum objects, one can sometimes prove 
properties analogous to those of genuine eigenfunctions.
For instance, in \cite{Schu},  Schubert considered Lagrangian states 
$f_h$ \red{as in (\ref{intro_eq1}),} associated with Lagrangian manifolds that are transverse to the 
stable directions of the dynamics (see section \ref{subsec:Lag}), on 
a manifold of negative sectional curvature. He could show that the 
analogue of (\ref{eq:QE}) holds, with $\psi_\lambda$ replaced with 
$\e^{-i h t_h \Delta_g} f_h$, where $t_h$ goes to infinity as $h\to 0$, 
while remaining smaller than some constant times $|\log h|$.
Hence, the large-time evolution of Lagrangian states under the semiclassical 
Schrödinger equation \eqref{intro_eq1} satisfies \red{an analogue of} quantum unique ergodicity. 
It is thus natural to wonder if such functions do also satisfy an analogue 
of Berry's conjecture.
\par 
This questions was first raised in \cite{IngRiv}, and was given a partial 
positive answer. Namely, in \cite{IngRiv},  the authors considered Lagrangian 
states  \red{as in (\ref{intro_eq1}),} with a \emph{generic} phase \red{$\phi$}.  They first took the limit $h\to 0$, and 
then $t\to \infty$ to obtain convergence to a Gaussian field.
This result is probably not optimal, and it seems natural to conjecture that 
the family of functions $\e^{-i h t_h \Delta} f_h$ satisfy Berry's conjecture
as soon as $f_h$ is a Lagrangian state associated with a Lagrangian manifold that 
is transverse to the stable directions of the dynamics \red{(without any additional genericity assumption)}, and as soon as 
$t_h\to \infty$ with $t_h \leq c |\log h|$,  for some small enough $c$.
\\
\par 
\textbf{The result: noisy quantum chaotic propagation.}
The aim of this paper is to prove a result of this kind, not for the genuine 
semiclassical Laplacian $-\frac{h^2}{2}\Delta_g$, but for generic small 
perturbations %of it 
of the form 
\begin{equation*}
	P_{h,\omega}= -\frac{h^2}{2}\Delta_g + h^\alpha Q_\omega, \quad 
	\alpha>0,
\end{equation*}
where $Q_\omega$ is either a bounded random potential or a bounded 
semiclassical random 
pseudo-differential operator obtained from the quantization of a 
random symbol oscillating at scale $h^\beta$, $\beta \in ]0,1/2[$. 
The presence of a small noise term can be motivated by the fact that 
in genuine physical situations an ``ideal'' evolution operator 
can be perturbed by many different sources, many of which are uncontrolled 
by the experimentalist. It therefore seems relevant on its own to 
study the propagation of initial data under the Schr\"odinger evolution 
semi-group induced by $P_{h,\omega}$. 
\\
\par
The aim of this paper is to study the family of functions 
\begin{equation*}
	\e^{\red{-}it_h P_{h,\omega}} f_h
\end{equation*}
where $f_h$ is a Lagrangian state associated with a Lagrangian manifold that is 
close enough to the unstable directions of the dynamics. Following our 
interpretation of Berry's conjecture, we rescale the propagated Lagrangian 
state to the microscale $h$ around \red{a point $\textsc{x}$ chosen uniformly at random}
on $X$. This rescaling around $\textsc{x}$ makes $\e^{\red{-}it_h P_{h,\omega}} f_h$ a 
random smooth function that depends on the \emph{additional} random parameter 
$\omega$.
\\
\par 
\emph{
\red{Under 
appropriate conditions on $\alpha$ and $\beta$, and supposing that $t_h\to \infty$ with $t_h \ll |\log h|$,  our main results can be summed up as follows. }}

\red{\emph{Theorem \ref{th:MartinEtMaximeSontDesSuperBeauxGosses}  says that  the function $(\e^{\red{-} it_h P_{h,\omega}} f_h)$, rescaled around a \emph{fixed} (and generic) point,  converges in law (with respect to $\omega$) to an isotropic stationary 
monochromatic Gaussian field -- the Berry Gaussian field, or Berry random wave model. 
%In this result, randomness only comes from the parameter $\omega$.
}}

\red{\emph{Theorem \ref{th:MartinEtMaximeSontDesBeauxGosses} 
states that  the law 
of the function $(\e^{\red{-} it_h P_{h,\omega}} f_h)$, rescaled around a \emph{random} point, 
converges in probability (with respect to $\omega$), to the Berry Gaussian field.}}
\par 
\emph{
The quantitative nature of our result shows (Corollary \ref{cor:LWL}) that, for 
any sufficiently fast decaying subsequence $h_j\to 0$, the randomly rescaled 
smooth function $(\e^{\red{-} it_h P_{h,\omega}} f_h)$ satisfies Berry's conjecture 
$\omega$-almost surely. 
}
\\
\par 
The idea of adding a small generic perturbation to the semiclassical 
Laplace-Beltrami operator to obtain additional properties on the propagator 
is not new. 
For instance, in \cite{EswTot, CaJaTo},  the authors propagate eigenfunctions by the 
Schrödinger equation perturbed by a random perturbation, and obtain improved 
$L^p$ bounds by averaging over the perturbation; however, these results do not 
give information about the eigenfunctions (or propagated eigenfunctions) of a 
genuine Schrödinger operator. 
In a similar spirit, in \cite{EswRiv}, the authors perturbed the 
Laplace-Beltrami operator on a manifold of negative sectional curvature, by adding to 
it a small random potential of size $\gg h^{1/2}$. They show that, for any 
initial data which is microlocalized near the energy layer there is a high 
probability that its propagation up to time $O(|\log h|)$ by the perturbed 
Schrödinger equation satisfies some form of quantum ergodicity.
\par 
Note that the kind of perturbations we consider is somehow different from 
those of \cite{EswTot, CaJaTo, EswRiv}. The perturbations imposed in these 
papers are always large enough to modify the underlying classical dynamics: 
a wave packet microlocalized around $(x_0,\xi_0)$, when propagated by the 
perturbed Schrödinger equation in the time scales under consideration in 
these papers, is not microlocalized around the image of $(x_0,\xi_0)$ by 
the corresponding geodesic flow.
In contrast, we permit much smaller perturbations, which do not affect 
the classical dynamics \red{on macroscopic scales}, but which will only modify the phases of wave packets.
\par
This allows for some delicate phase cancellations between wave packets, which 
we believe to be a good toy model for quantum chaos. It should thus be easier 
to prove quantum chaotic properties for eigenfunctions of $P_{h,\omega}$ with 
a generic $\omega$ than for the genuine Laplacian; such considerations will be 
pursued elsewhere.
\\
\par
\textbf{Acknowledgements.} This article was influenced by discussions 
with Alejandro Rivera, and we would like to thank him for that. 
The authors would also like to thank Stéphane Nonnenmacher for suggesting 
Remark \ref{RemStephane}, as well as Ofer Zeitouni for a helpful discussion. 
Both authors were partially funded by the Agence Nationale de la Recherche, 
through the project ADYCT (ANR-20-CE40-0017).
\section{Main results}
\subsection{Lagrangian states}\label{subsec:Lag}
Let $(X,g)$ be a smooth compact connected Riemannian manifold without 
boundary and of negative sectional curvature. A \emph{Lagrangian state} 
on $X$ is a family of functions $f_h\in C^\infty(X)$ indexed by $h\in]0,1]$, 
defined by
\begin{equation}\label{eq:LagState}
f_h(x)=a(x)\e^{i\phi(x)/h},
\end{equation}
where $\phi\in C^\infty(\cO)$ for some simply connected open subset 
$\cO\subset X$ and $a\in C^\infty_c(\cO)$. To a Lagrangian state we can 
associate a \emph{Lagrangian manifold}
\begin{equation*}
	\Lambda_\phi:= \{ (x, d_x \phi) ; x\in \cO\}\ \subset T^*X.
\end{equation*}
A Lagrangian state is called \emph{monochromatic} if 
$\Lambda_\phi\subset S^*X:= \{(x,\xi)\in T^*X; |\xi|_x=1\}$, i.e. if 
\begin{equation}\label{eq:LagState_monochrom}
	|d_x \phi|_x =1 \quad  \text{for all } x\in \cO.
\end{equation}
As we will explain in Section \ref{sec:Hyperbolicity}, the dynamics of the 
geodesic flow is hyperbolic on $S^*X$, so for any $\rho\in S^*X$ we may 
decompose the tangent spaces $T_\rho S^*X$ into \emph{unstable}, \emph{neutral} 
and \emph{stable directions}
\begin{equation*}
	T_\rho S^*X=E_\rho^+\oplus E_\rho^0 \oplus E_\rho^-.
\end{equation*} 
\begin{definition}\label{def:LagInstable}
For every $\eta>0$, we say that a Lagrangian manifold 
$\Lambda_{\phi}\subset S^*X$ is $\eta$-unstable if, 
for every $\rho\in \Lambda_\phi$ and for every $v\in T_\rho \Lambda$, 
writing $v= (v_+, v_-, v_0)\in E_\rho^+\oplus E_\rho^-\oplus E_\rho^0$, 
we have 
\begin{equation*}
	|(0,v_-,0)|_{\rho} \leq \eta |v|_\rho.
\end{equation*} 
\end{definition}
Recall that the intrinsic distance $\mathrm{dist}_{\Lambda}(\rho_1, \rho_2)$ 
between two points $\rho_1, \rho_2\in \Lambda$ is the minimal length of 
curves in $\Lambda$ joining $\rho_1$ and $\rho_2$, the length being computed 
using an arbitrary metric on $T^*X$ (see section \ref{sec:notation}).  We define the 
\emph{distortion} of $\Lambda$ as
\begin{equation}\label{eq:DefDistorsion}
\mathrm{distortion}(\Lambda) 
:= \sup\limits_{\rho_1, \rho_2\in \Lambda} 
	\frac{\mathrm{dist}_{\Lambda}(\rho_1, \rho_2)}{\mathrm{dist}_{T^*X}(\rho_1, \rho_2)}.
\end{equation}
\subsection{Noisy propagation of Lagrangian states}
\label{subsec:PropLagState1}
Let $h>0$ and let $0\leq \delta = \delta(h) \ll 1$.  Consider the 
Schrödinger-type operator
\begin{equation}\label{eq:SchroedingerOp}
	P_h^\delta := -\frac{h^2}{2} \Delta_g + \delta Q_\omega, 
\end{equation}
where $\Delta_g$ denotes the Laplace-Beltrami operator on $(X,g)$ and 
where $Q_\omega$ is a random perturbation described in detail below. 
The aim of this paper is to study the large-time evolution of monochromatic 
Lagrangian states $f_h$ on $X$ under the Schrödinger equation 
\begin{equation*}
	\begin{cases}
	ih\partial_t u = P_h^\delta u, \\
	u|_{t=0} = f_h.
	\end{cases}
\end{equation*} 
In other word\red{s} we are interested in the propagated Lagrangian state 
\begin{equation}\label{intro_eq2}
	u=\e^{\red{-}i\frac{t}{h}P_h^\delta}f_h, \quad \text{for } t\gg1.
\end{equation}
We will consider the two types of random perturbations $Q_\omega$: 
\begin{enumerate}[label=\arabic*.]
	\item The case where $Q_\omega$ is the 
	operator of multiplication by a random real-valued smooth function 
	\begin{equation}\tag{Random Potential case}\label{eq:Pot}
		q_\omega : X \longrightarrow \R.
	\end{equation}
	\item The case where $Q_\omega$ is a pseudo-differential operator 
	given by the quantization 
	\begin{equation}\tag{Random $\Psi$DO case}\label{eq:Pseudo}
		Q_\omega := \Op_h(q_\omega)
	\end{equation}
	of a random smooth real-valued function
	\begin{equation*}
		q_\omega : T^*X \longrightarrow \R.
	\end{equation*}
	\red{We refer to Definition \ref{def:Quantization} below for the choice of quantization 
	used in the \ref{eq:Pseudo}.}
\end{enumerate}
Let us now describe what models of random functions $q_\omega$ we consider. 
Fix a parameter $\beta \in ]0, 1/2[$, let $J_h\subset \N$ be a set of indices 
of cardinality $|J_h| =O(h^{-M})$, for some $M>0$, and let 
$\{q_j\}_{j\in J_h}$ be a family of \red{$h$-dependent} smooth compactly 
supported functions on $X$, when we are in the \ref{eq:Pot}, or on $T^*X$ 
when we are in the \ref{eq:Pseudo}. To construct a random function on 
$X$ (resp. on the phase space $T^*X$) from the single-site potentials $q_j$,  
we let $\omega=\{\omega_j\}_{j\in J_h}$ be a sequence of independent and 
identically distributed (iid) random variables (the precise assumptions we make 
on the $\{\omega_j\}$ will be described in Hypothesis \ref{Hyp:Prob} below) 
and we set
\begin{equation}\label{eq:randomSymbol}
	\begin{split}
	&q_\omega(\rho) 
	:= 
	\sum_{j\in J_h} \omega_j \,q_j(\rho), \quad \text{in the \ref{eq:Pseudo},} 
	\\
	&q_\omega(x)
	:= 
	\sum_{j\in J_h} \omega_j \,q_j(x), \quad \text{in the \ref{eq:Pot}.} 
	\end{split}
\end{equation}
We make the following additional assumptions.
\begin{hypo}[Hypotheses on the single-site potential]\label{HypPot}
	$\phantom{.}$
\begin{enumerate}[label=\roman*.]
	\item Each $q_j$ is compactly supported, with a support of diameter 
		$O(h^\beta)$ uniformly in $j\in J_h$. 
	\item There exists $C>0$, independent of $h$, such that for all $\rho \in X$ (resp. $\rho \in T^*X$), $\rho$ belongs 
	to the support of at most $C$ functions $q_j$.
	\item For any $k\in \N$, there exists $C_k>0$ such that
	\begin{equation}\label{eq:PotentialDer}
		\|q_j\|_{C^k} \leq C_k h^{-\beta k} \quad \forall j\in J_h.
	\end{equation}
	\item \MI{There exists $c>0$ such that for all $h\in (0,1]$ and all $j\in J_h$, $\max_{\rho\in X} |q_j(\rho)|\geq c$ (resp. $\max_{\rho\in T^*X} |q_j(\rho)|\geq c$). }
	\item There exists $c_0>0$ such that, for any $T>0$ and any $\rho \in S^*X$, 
	we have \red{for all $h\in (0,1]$}
	\begin{equation}\label{eq:PotentialEverywhere}
	 \sum_{j\in J_h}  \int_0^T  q_j \left(\Phi^t (\rho)\right) dt  
	 \geq c_0 T,
	\end{equation}
	in the \ref{eq:Pseudo}. Here, $\Phi^t : T^*X \longrightarrow T^*X$ 
	denotes the geodesic flow. In the \ref{eq:Pot}, we work with the 
	same assumption but with $q_j \left(\Phi^t (\rho)\right)$ replaced 
	by $q_j \left(\pi_X\circ \Phi^t (\rho)\right)$, where 
	$\pi_X:T^*X\to X$ denotes the projection on the base manifold $X$. 
\end{enumerate}
\end{hypo}

\red{Note that, in the case (\ref{eq:Pseudo}),  the first three points in Hypothesis \ref{HypPot} imply that $q_\omega$ is a random element of the symbol class $S^{-\infty}_\beta(T^*X)$ (cf. Section 
	\ref{sec:SemClass} for definition of this notion).}
	
%\red{The last point in Hypothesis \ref{HypPot} does not necessarily imply that the supports of the $q_j$ overlap, but we will only give an example where they do overlap.}

%
\begin{ex}
To build such a family of single-site potentials, one may for instance cover 
$X$ (resp. $S^*X)$ by geodesic balls $B(\rho_j, h^\beta)$ of radius $h^\beta$ 
and centred at $\rho_j$, such that each point belongs to at most $C$ 
balls\footnote{\red{\corM{Notice} that this implies that each point of $X$ belongs 
to at most $C'= \frac{\sup_{x\in M} \mathrm{Vol}(B(x,  3 h^\beta))}{\inf_{x\in M} \mathrm{Vol}(B(x,  h^\beta))}$ 
balls of radius $2h^\beta$, so that $C'$ is bounded independently of $h$.}}.
We may then take 
\begin{equation*}
	q_j 
	=
	\chi \left(
		h^{-\beta} \mathrm{dist}(\rho_{j,h},\rho)
		\right),
\end{equation*}
where $\chi\in C_c^\infty([0, \infty);[0,1])$ takes value $1$ on $[0,1]$, and 
where $\mathrm{dist}$ means either $\mathrm{dist}_X$ or $\mathrm{dist}_{T^*X}$.
\end{ex}

\red{In the example above, the supports of the $q_j$ overlap; however, note that this is not necessary for \eqref{eq:PotentialEverywhere} to hold.}

\begin{hypo}\label{Hyp:BetaDelta}
We suppose that there exists $0<\varepsilon_0< \frac{1}{4}$ and $h_0>0$ such that 
for all $h\leq h_0$, we have
\begin{equation}\label{eq:CondBeta}
\delta h^{-2\beta - \varepsilon_0} \leq 1,
\end{equation}
\begin{equation}\label{eq:CondBetaDelta}
\delta^2  h^{\beta-2} \geq h^{-\varepsilon_0}.
\end{equation}
In the  \ref{eq:Pot}, we will also need to assume that 
\begin{equation}\label{eq:CondBetaDelta2}
\delta  h^{\beta-1} \leq h^{\varepsilon_0}.
\end{equation}
\end{hypo}
\begin{rem}\label{rem:specialdelta}
It is natural to consider the case $\delta = h^\alpha$ 
	with $2\beta \leq \alpha$. However, we will stick with a 
	coupling constant $\delta$ for the sake of generality.
	\par
	Note that when $\delta = h^\alpha$, conditions (\ref{eq:CondBeta}) 
	and (\ref{eq:CondBetaDelta}) rewrite
	$$0<\beta < \min \left( \frac{\alpha}{2}, 2-2\alpha \right),$$
	while (\ref{eq:CondBetaDelta2}) rewrites
	$$\beta > 1-\alpha.$$
	These conditions are plotted on Figure \ref{Fig:Parameters}. 
\end{rem}

%\begin{tikzpicture} [xmin=-2,xmax=2,ymin=0,ymax=5]
%\grille \axes \fenetre
%\draw plot[smooth] (\x,\x^2);
%\end{tikzpicture} \draw [very thin, gray] (0,0) grid (3,2);

\begin{figure}
\begin{center}
\begin{tikzpicture}[xmin=-0.8,xmax=4.9,ymin=-0.8,ymax=4.9]
\grille \axes \fenetre
 %\draw  [red] (2.42,0) node{$\bullet$};
%\draw [red, thick] [->] (2.42,0) -- (2.42,0.4) ;
%\draw [red] (2.6, 0.2) node{$\rho_0$};
\draw (-0.3,4.7) node{$\beta$};
\draw (4.7, -0.3) node{$\alpha$};
\draw plot[smooth] (\x,\x/2);
\draw plot[smooth] (\x,8-2*\x);
\draw plot[smooth] (\x, 4 -\x);
\filldraw[draw=black,fill=gray!20]
plot (0,0)--(4,0)-- (8/3,4/3) --cycle;
\filldraw[draw=black,fill=gray!80]
plot (4,0)--(16/5,8/5)-- (8/3,4/3) --cycle;
\draw [thick](-0.1,2) -- (0.1,2);
\draw (-0.3,1.7) node{$\frac{1}{2}$};
\draw [thick](-0.1,4) -- (0.1,4);
\draw (-0.3,3.7) node{$1$};
\draw [thick](2,-0.1) -- (2,0.1);
\draw (1.7,-0.35) node{$\frac{1}{2}$};
\draw [thick](4,-0.1) -- (4,0.1);
\draw (3.7,-0.3) node{$1$};
%\filldraw[draw=black,fill=gray!20]
%plot [smooth,domain=0:1] (\x,{sqrt(\x)})
%-- plot [smooth,domain=1:0] (\x,\x^2)
%-- cycle;
%\draw [domain=-pi:pi] plot (\x,{sin(\x r)});.
\end{tikzpicture}
\end{center}
\caption{Admissible parameters $\alpha$ and $\beta$, see Hypothesis 
\ref{Hyp:BetaDelta}. The dark grey region is admissible for the 
\ref{eq:Pot} and the \ref{eq:Pseudo}, while the light grey region is 
only admissible in the \ref{eq:Pseudo}.}\label{Fig:Parameters}

\end{figure}

Finally, we need some assumption on the probability distributions $\omega_j$.

\begin{hypo}\label{Hyp:Prob} We suppose that the iid random variables 
$(\omega_j)_{j\in J_h}$ 
\red{	  have a common distribution with a compactly supported density 
	  $m\in C^2_c(\R;[0,+\infty[)$ with respect to the Lebesgue measure.}

\end{hypo}

\subsection{Randomization, local weak limits and the Berry Gaussian field}
\label{ss:local_limits_def}
The aim of this paper is to compare a noisily propagated Lagrangian state 
$u=\e^{i\frac{t}{h}P_h^\delta}f_h$ \eqref{intro_eq2} locally near a 
\emph{randomly chosen point} on $X$ with a stationary isotropic smooth 
monochromatic Gaussian stochastic process. To do this 
we will -- roughly speaking -- pick a point $x_0$ of $X$ uniformly at random, 
rescale $u$ near $x_0$ to the microscopic scale $h$ in local geodesic 
coordinates, and then compare this now probabilistic rescaled version 
of $u$ with a Gaussian stochastic process. To make this precise we will recall 
notions introduced in \cite{IngRiv}. 
\subsubsection{Random smooth functions \red{on $\R^d$}.} 
In what follows we equip the space $C^\infty(\R^d)$ with the 
topology of the convergence of all derivatives over all compact 
sets, i.e. for the topology induced by the family of seminorms
\begin{equation*}
	\|f\|_n := \max_{x\in \overline{B(0,n)}} 
	\max_{|\alpha|\leq n}|\partial^\alpha f(x)|, \quad 
	n\in \N. 
\end{equation*}
Notice that $C^\infty(\R^d)$ is a separable Fr\'echet, 
and therefore a Polish, space. The above topology can \corM{be} 
metrized with the distance
\begin{equation}\label{eq:DistanceFunctions}
\mathrm{d}(f,g) 
= \sum_{n=1}^{\infty} 2^{-n} \min \left( \|f-g\|_n, 1 \right).
\end{equation}
We equip $C^\infty(\R^d)$ with the Borel $\sigma$-algebra 
$\mathcal{B}(C^\infty(\R^d))$. 
A \emph{random smooth function on $\R^d$} is a random variable with 
values in $C^\infty(\R^d)$. We refer the reader to the review 
\cite[Appendix A]{NS} for more details on this notion. 
\par 
We highlight 
the notion of \emph{convergence in law}. A sequence of random smooth functions
$\{\mathfrak{f}_n\}_{n\in\N}$ on $\R^d$ is said to converge in law to a random 
smooth function $\mathfrak{f}$ on $\R^d$, i.e. 
\begin{equation*}
	\mathfrak{f}_n \overset{d}{\longrightarrow} \mathfrak{f}, 
	\quad n\to \infty,
\end{equation*}
if the laws of the random functions converge weakly. More explicitly,  
this means that for all bounded continuous functions $F\in C_b(C^\infty(\R^d))$ 
we have that 
\begin{equation}\label{eq:BrezelPousseToi}
	\erw [F(\mathfrak{f}_n)] \longrightarrow \erw[F(\mathfrak{f})], 
	\quad n\to \infty.
\end{equation}
\begin{ex}
\red{Examples of functionals that could be used in (\ref{eq:BrezelPousseToi}) would be $F_1(\mathfrak{f}) = \chi_1(\mathfrak{f}(0))$, or $F_2(\mathfrak{f}) =\int_{|x|<R}  \chi_2(x,  \mathfrak{f}(x),  \nabla \mathfrak{f}(x)) \dd x$, where $R>0$ and $\chi_1 : \C \longrightarrow \C , \chi_2 : \R^d \times \C \times \C^d \longrightarrow \C$ are continuous bounded functions.}
\end{ex}
\subsubsection{Randomization and local weak limits.} Our aim is to study the convergence  
of a sequence of deterministic smooth functions on $X$, \red{rescaled to scale $h>0$,  around  a randomly chosen point.}
%near a randomly 
%chosen point at the scale $h>0$. 
To avoid any topological difficulties,  \red{we will choose random points in small open sets
%we define this convergence locally,  
though all of our results will hold 
regardless of the choice of open sets.}
\par
Let $\mathcal{U}\subset X$ be a small enough open set so that we can 
define an orthonormal frame $V=(V_1,\dots,V_d)$ on it, that is to say 
a family of smooth sections 
$(V_i)_{i=1,\dots,d}: \mathcal{U}\longrightarrow TX$ such that, for each 
$x\in \mathcal{U}$, $(V_1(x),\dots,V_d(x))$ is an orthonormal basis of $T_xX$. 
If $x\in \mathcal{U}$ and $\by\in \R^d$, we will write 
$\by V(x) := \by_1 V_1(x)+\dots+\by_d V_d(x) \in T_xX$, and
\begin{equation}\label{eq:JamesBrownIsTheBest}
	\exp_{x}(\by) := \exp_x (\by V(x))\, .
\end{equation}
Here $\exp_{x}$ denotes the exponential map restricted to $T_x X$. 
Note that the map \eqref{eq:JamesBrownIsTheBest} is well defined 
for all $\by\in \R^d$ since the underlying Riemannian manifold is 
complete. All the constructions in this section will depend on the 
choice of the local frame $V$, and will hence not be intrinsic.
\par
With the above quantities and definitions in mind, we can define our 
notion of \emph{local weak limit}. 
\begin{definition}\label{def:LWL}(Local weak limit)
Let $(X,g)$ be a compact smooth Riemannian manifold. Let $\mathcal{U}\subset X$ 
be an open set and $V$ an orthonormal frame on $\mathcal{U}$ as in 
\eqref{eq:JamesBrownIsTheBest}. Let $\{f_h\}_{h>0}$ be a family of functions 
in $C^\infty(X)$, and let $\mathfrak{f}$ be a smooth random function in 
$C^{\infty}(\R^d)$. Let $\textsc{x}$ be a random variable with values in 
$\cU$ uniformly distributed with respect to the Riemannian volume measure 
on $\cU$. 
\par
Then, we say that $\mathfrak{f}$ is the local weak limit of $\{f_h\}_h$ in the 
frame $V$ if the random smooth function 
$\mathfrak{f}_{\textsc{x},h}(\by):=f_h(\exp_{\textsc{x}}(h\by))$ on $\R^d$ 
converges in law to $\mathfrak{f}$ as $h\to 0$, i.e. if 
\begin{equation*}
	\mathfrak{f}_{\textsc{x},h} \overset{d}{\longrightarrow} \mathfrak{f}, 
	\quad h\to 0.
\end{equation*}
\end{definition}
Let us give some remarks on that definition: $\mathfrak{f}_{\textsc{x},h}$ is a 
well defined \red{random function in $C^\infty(\R^d)$}
%. We refer the 
%reader to the review \cite[Appendix A]{NS} for more details on this 
%notion. 
%
and, by definition,
%Furthermore, we recall that, by definition, 
saying that $\mathfrak{f}$ 
is the local weak limit of $\{\mathfrak{f}_{\textsc{x},h} \}_h$ in 
the frame $V$ means that, for any bounded continuous functional 
$F: C^\infty(\R^d) \longrightarrow \R$, we have 
\begin{equation}\label{eq:ExpectUnif}
	\mathbb{E}_{\mathrm{x}} [F (\mathfrak{f}_{\textsc{x},h})]:=
	\frac{1}{\Vol(\mathcal{U})} \int_\mathcal{U} 
		F (\mathfrak{f}_{\textsc{x},h})dv_g(\textsc{x})
		\longrightarrow
		\mathbb{E}[F(\mathfrak{f})] 
		\quad \text{as }
		h\to 0,
\end{equation}
where $dv_g$ denotes the Riemannian volume measure on $X$. 
\subsubsection{The Berry Gaussian field.}
An almost surely (or a.s.) $C^\infty$ (centered) Gaussian field 
on $\R^d$ is a random variable $\mathfrak{f}$ taking, up to a 
set of probability $0$, values in $C^\infty(\R^d)$ such that for any 
finite collection of points $\bx_1,\dots,\bx_k\in \R^d$, the random vector 
$(\mathfrak{f}(\bx_1),\dots,\mathfrak{f}(\bx_k))\in \C^d$ is 
(centered) Gaussian. We say that two fields $\mathfrak{f}_1$ and 
$\mathfrak{f}_2$ are \emph{equivalent} if they have 
the same law. In the sequel, unless otherwise stated, we will always 
identify fields which are equivalent. 
\par
Let $\mathfrak{f}$ be an a.s. $C^\infty$, centered Gaussian field on $\R^d$. 
Then, its \emph{covariance kernel} $K:(\bx,\by)\mapsto \E[f(\bx)\overline{f(\by)}]$ 
defined on $\R^d\times \R^d$ is \emph{positive definite}, meaning that for 
each $k$-tuple $(\bx_1,\dots,\bx_k)\in (\R^d)^k$, the matrix 
$K(\bx_i,\bx_j)_{i,j}$ is positive. 
As explained for instance in Appendix A.11 of \cite{NS}, the function 
$K$ belongs to $C^\infty(\R^d\times \R^d)$ and there is actually a 
one-to-one correspondence (up to equivalence) between smooth covariance 
kernels and a.s. $C^\infty$ centred Gaussian fields on $\R^d$.
\begin{definition}\label{DefBerry}
The Berry Gaussian field with normalization constant $\lambda\in\R$, 
denoted by $\mathrm{BGF}_\lambda$, is the unique (up to equivalence) 
a.s. $C^\infty$ stationary Gaussian field on $\R^d$ whose covariance 
kernel is $\lambda \int_{\mathbb{S}^{d-1}}\e^{i(x-y)\cdot\xi} d\sigma (\xi)$, 
where $\sigma$ is the uniform probability measure on $\mathbb{S}^{d-1}$.
\par
If $F : C^\infty(\R^d) \longrightarrow \R$ is a bounded 
continuous functional, its expectation with respect to the 
$\mathrm{BGF}_\lambda$ will be denoted by $\E_{\mathrm{BGF}_\lambda}[F]$.
\end{definition}
We remark that the Berry Gaussian field $\mathrm{BGF}_\lambda$ is 
the unique (up to equivalence) \emph{monochromatic stationary 
isotropic Gaussian field} with normalization $\E[|\mathrm{BGF}_\lambda(0)|^2] 
= \lambda$. Indeed, stationary means that its covariance 
kernel depends only on the difference $(x-y)$. Isotropic means that the 
covariance kernel is invariant under (the same) rotation of $x$ and $y$, 
so it only depends on $|x-y|$. Monochromatic means that the covariance 
kernel satisfies $-\Delta K = K$ with respect to both variables $\bx$ and $\by$, 
corresponding to the fact that a realization $f$ of the Berry 
Gaussian field satisfies $-\Delta f = f$ a.s. 
\subsection{The main results}
Let $f_h$ be a monochromatic Lagrangian state whose associated 
Lagrangian manifold is $\eta$-unstable for $\eta$ small enough. 
We study the local weak limit of the propagated Lagrangian state 
$u_h=\e^{\red{-}i\frac{t}{h}P_h^\delta}f_h$ \eqref{intro_eq2}. Following 
Definition \ref{def:LWL} we are interested in studying the limiting 
law of the random function
\begin{equation}\label{eq_LocalPropState}
	\left(\e^{\red{-} i\frac{t}{h}P_h^\delta}f_h\right)(\exp_{\textsc{x}}(h\by))
\end{equation}
where $\textsc{x}$ is a uniformly distributed random variable in 
$\mathcal{U}\subset X$. Notice that in this expression we have two 
different sources of randomness: one coming from the perturbation 
$Q_\omega$ and one coming from $\textsc{x}$. To make this distinction 
clear we will denote the expectation with respect to $\mathrm{x}$ by 
$\erw_{\mathrm{x}}$, see \eqref{eq:ExpectUnif}, and the probability 
with respect to the law of $\omega$ by $\prob_\omega$. Accordingly 
we will denote the associated expectation by $\erw_\omega$. 
\\
\par
Our first result shows that when fixing $\textsc{x}$, away from a 
set of asymptotically negligible measure, the random function 
\eqref{eq_LocalPropState} converges in law to a $\mathrm{BGF}$. 
\begin{thm}\label{th:MartinEtMaximeSontDesSuperBeauxGosses}
Let $X$ be a compact connected Riemannian manifold with 
negative sectional curvature and without boundary. Let 
$P_h^\delta$ be as in \eqref{eq:SchroedingerOp} and 
suppose that Hypotheses \ref{HypPot}, \ref{Hyp:BetaDelta} and \ref{Hyp:Prob}
are satisfied. Let $D>0$. There exists $\eta=\eta(D)>0$ such that the following holds.
\par
Let $f_h = a \e^{\frac{i}{h}\phi}$ be a monochromatic Lagrangian state 
associated with a Lagrangian manifold which is $\eta$-unstable, has distortion
$\leq D$ and satisfies $\|\phi\|_{C^{3}}\leq D$. There exists 
$X_h^0 \subset X$, with $ \mathrm{Vol}(X_h^0) \to 0$ as $h\to 0$,  
such that the following holds: 
Let $\mathcal{U}\subset X$ be an open set, and $V$ be an orthonormal 
frame on $\mathcal{U}$. Let $(t_h)_{h>0}$ be such that 
$t_h \to +\infty$, as $h\to 0$, and $|t_h|=o_{h\to 0} (|\log h|)$. 
Then, for every $x\in \mathcal{U} \setminus X_h^0$
\begin{equation*}
	\left(\e^{\red{-}\frac{i}{h}t_hP_h^\delta}f_h\right)(\exp_{x}(h\cdot))
	\overset{d}{\longrightarrow}
	\mathrm{BGF}_{\lambda_a}
\end{equation*}
with $\lambda_a= \frac{\|a\|_{\MI{L^2}}^2}{\mathrm{Vol}(X)}$.
\end{thm}
\red{In the previous theorem,  the quantity $\|\phi\|_{C^3}$ is not 
intrinsic, and is introduced by fixing an atlas on the manifold as in 
\eqref{eq:sa8.1} below.  The quantity $\eta(D)$ thus depends implicitly 
on this construction.}
The assumption on the distortion of $\Lambda$ (as defined in (\ref{eq:DefDistorsion})) 
is purely technical, and it automatically follows from bounds on $\|\phi\|_{C^2}$ 
if $\phi$ is defined on a convex set.
We insist here on the fact that the convergence in law stated in  
Theorem \ref{th:MartinEtMaximeSontDesSuperBeauxGosses} is 
with respect to the $\omega$ random variables coming from the 
random perturbation and for a \emph{fixed} $x$.
\\
\par
Our second main result concerns the random smooth function 
\begin{equation}\label{eq_LocalPropState2}
	\left(\e^{\red{-}i\frac{t}{h}P_h^\delta}f_h\right)(\exp_{\textsc{x}}(h\by))
\end{equation}
where $\textsc{x}$ is a uniformly distributed random variable in 
$\mathcal{U}\subset X$. The random variable $\textsc{x}$ generates 
the law of \eqref{eq_LocalPropState2} which depends on $\omega$. 
Indeed, this law is, with respect to $\omega$ a random probability 
measure. The result below states that this random law converges 
weakly in probability (with respect to $\omega$) to the law of 
the BGF.

\begin{thm}\label{th:MartinEtMaximeSontDesBeauxGosses}
Let $X$ be a compact connected Riemannian manifold with 
negative sectional curvature and without boundary. Let 
$P_h^\delta$ be as in \eqref{eq:SchroedingerOp} and 
suppose that Hypotheses \ref{HypPot},  \ref{Hyp:BetaDelta} and \ref{Hyp:Prob}
are satisfied. Let $D>0$. There exists $\eta=\eta(D)>0$ such that the following 
holds:
\par
Let $\mathcal{U}\subset X$ be an open set, and $V$ be an orthonormal 
frame on $\mathcal{U}$. Let $f_h = a \e^{\frac{i}{h}\phi}$ be a 
monochromatic Lagrangian state associated with a Lagrangian manifold 
which is $\eta$-unstable, has distortion $\leq D$ and 
satisfies $\|\phi\|_{C^{3}}\leq D$. Let $(t_h)_{h>0}$ be such that 
$t_h \underset{h\to 0}{\longrightarrow} +\infty$ and 
$|t_h|=o_{h\to 0} (|\log h|)$. Then, for any $\varepsilon>0$ and 
every $F\in C_b(C^\infty(\R^d))$ we have that for $h>0$ small enough 
\begin{equation}\label{eq:LocalRandomX}
	\prob_\omega
	\left[
	\left| \erw_{\mathrm{x}}[
		F(\e^{\red{-}\frac{i}{h}t_hP_h^\delta}f_h(\exp_{\textsc{x}}(h\cdot)))] 
		- \E_{\mathrm{BGF}_{\lambda_a}}[F] 
	\right| \geq \varepsilon
	\right] = O(h^\infty)
\end{equation}
where $\lambda_a= \frac{\|a\|_{\MI{L^2}}^2}{\mathrm{Vol}(X)}$.
\end{thm}

\begin{rem}
The assumption $|t_h|=o_{h\to 0} (|\log h|)$ can be slightly weakened. Indeed, 
the proof of Theorem \ref{th:MartinEtMaximeSontDesBeauxGosses} actually shows 
that for every $L\in \N$, there exists $c_L>0$ such that, if 
$t_h \leq c_L |\log h|$, we have 
that for any $\varepsilon>0$ and 
every $F\in C_b(C^L(\R^d))$ we have that for $h>0$ small enough 
\begin{equation*}
	\prob_\omega
	\left[
	\left| \erw_{\mathrm{x}}[
		F(\e^{\red{-}\frac{i}{h}t_hP_h^\delta}f_h(\exp_{\textsc{x}}(h\cdot)))] 
		- \E_{\mathrm{BGF}_{\lambda_a}}[F] 
	\right| \geq \varepsilon
	\right] = O(h^\infty).
\end{equation*}
Here $F: C^L(\R^d)\longrightarrow \R$ is a bounded functional which 
is continuous for the topology of convergence of derivatives over 
compact sets.
\par
The possibility of extending our results to longer time scales 
will be further discussed in Remark \ref{RemStephane}.
\end{rem}
\begin{corollary}\label{cor:LWL}
Under the assumptions of 
Theorem \ref{th:MartinEtMaximeSontDesBeauxGosses} we have that 
for any sequence $h_j\to 0$, $j\to\infty$, such that there exists 
an $M>0$ such that $(h_j^M)_{j\in\N} \in \ell^1(\N)$, we have that 
for every $F\in C_b(C^\infty(\R^d))$,  
\begin{equation}\label{eq:LocalRandomX_2}
\erw_{\mathrm{x}}[
	F(\e^{\red{-}\frac{i}{h_j}t_{h_j}P_{h_j}^\delta}f_{h_j}(\exp_{\textsc{x}}(h_j\cdot)))] 
	\to  \E_{\mathrm{BGF}_{\lambda_a}}[F], \quad j\to \infty,
\end{equation}
$\omega$-almost surely.  \red{Furthermore}, $\omega$-almost surely, 
$\mathrm{BGF}_{\lambda_a}$ is the local weak limit of 
$\{\e^{\red{-}\frac{i}{h_j}t_{h_j}P_{h_j}^\delta}f_{h_j}\}_{h_j}$ in the 
frame $V$.
\end{corollary}
\red{Note that, in the first part of Corollary \ref{cor:LWL}, the set of probability $1$ depends on $F$; this}
 first part follows readily from 
Theorem \ref{th:MartinEtMaximeSontDesBeauxGosses} and the Borel-Cantelli 
lemma.  The second part (about the local weak limit) is slightly more 
involved, and will be proved at the end of section \ref{sec:ProofMainTheorems}.
\subsection{Ideas of the proof and organization of the paper}
The central tool to obtain the results of the previous paragraph is 
the WKB method, which gives a precise description of the evolution 
of a Lagrangian state by the semiclassical Schrödinger equation. 
Namely,  when working on the universal cover $\wit{X}$ of a manifold 
of negative curvature,  it is standard that the function 
$\e^{i\frac{t}{h}\wit{P}_h^\delta}\wit{f}_h$ can be well approximated 
by another Lagrangian state:
\begin{equation}\label{eq:IntroWKB}
\wit{a}( \wit{x} ; t,h, \delta) e^{\frac{i}{h} \wit{\phi}( \wit{x} ; t,h,\delta)}.
\end{equation}

We will show that for the perturbations described in Subsection \ref{subsec:PropLagState1}, we may actually write
\begin{equation}\label{eq:IntroWKB2}
\e^{\red{-}i\frac{t}{h}\wit{P}_h^\delta}\wit{f}_h(\wit{x}) \approx a(\wit{x} ; t) e^{\frac{i}{h} \phi(\wit{x} ; t)} e^{i \frac{\delta}{h} \wit{\Theta}(\wit{x};t,h,\delta)},
\end{equation}
so that the randomness of $P_h^\delta$ appears only through the random phase $\wit{\Theta}$. Actually,  $\wit{\Theta}$ can be written as the integral of $q_\omega$ over a \red{perturbed trajectory close to }a geodesic going from $\wit{\Lambda}_\phi$ to $\wit{x}$.

When working on the initial manifold $X$, we need to sum contributions coming from different sheets in the universal cover, so that we get
\begin{equation}\label{eq:IntroWKB3}
\e^{\red{-}i\frac{t}{h}P_h^\delta}f_h(x) \approx \sum_k a_k( x ; t) e^{\frac{i}{h} \phi_k( x ; t)} e^{i \frac{\delta}{h} \Theta_k( x;t,h,\delta)},
\end{equation}
where the number of terms grows exponentially with $t$.  When performing a rescaling at scale $h$, we get
\begin{equation}\label{eq:IntroWKB4}
\left(\e^{\red{-}i\frac{t}{h}P_h^\delta} f_h\right)(\exp_x(h\by)) \approx \sum_k a_k( x ; t) e^{\frac{i}{h} \phi_k( x ; t)} e^{i \frac{\delta}{h} \Theta_k( x;t,h,\delta)} e^{i \by \cdot \nabla \phi_k(x;t)},
\end{equation}
\red{where we write $\nabla \phi_k(x; t)$ for $(V_1 \phi_k, ... V_d \phi_k)_x$. The expression in \eqref{eq:IntroWKB4} is} thus the sum of a large number of plane waves (in the $\by$ variable) with random phases. We will show that the phases can be made independent by excluding a set of points $x$ of small measure, so that  Theorem \ref{th:MartinEtMaximeSontDesSuperBeauxGosses} will  follow from  the Central Limit Theorem. 

To obtain Theorem \ref{th:MartinEtMaximeSontDesBeauxGosses}, we will show that,  if $x$ and $x'$ are at a distance $h^{\beta - \varepsilon}$ from each other, then the phases $\Theta_k( x;t,h,\delta)$ and $\Theta_k( x';t,h,\delta)$ are independent from each other,  for most choices of $x$ and $x'$. This will allow us to transfer the randomness coming from $P_h^\delta$ to spatial randomness, obtained by picking the point $x$ at random.

In section \ref{sec:Semi}, we will recall some facts about semiclassical analysis, and about the functional spaces we will use.  Section \ref{Sec:Class} will be devoted to the description of the classical dynamics of the geodesic flow, and of the Hamiltonian flow induced by the perturbation $q_\omega$.  Section \ref{sec:WKB} will present the WKB method on the universal cover, describing the functions $\wit{a}$ and $\wit{\phi}$ appearing in (\ref{eq:IntroWKB}). In section \ref{sec:RegularityOfWKB}, we will prove precise regularity estimates on $\wit{a}$ and $\wit{\phi}$, so as to be able to reach the simpler expression \eqref{eq:IntroWKB2}.  In section \ref{sec:Proj}, we will project the WKB state obtained in the previous sections on the base manifold $X$ and perform local rescalings, in order to derive expressions like \eqref{eq:IntroWKB3} and \eqref{eq:IntroWKB4}. Section \ref{Sec:Indep} is devoted to the delicate issues of independence between the phases $\Theta_k(x)$. Finally,  in section  \ref{sec:ProofMainTheorems}, we will prove Theorems \ref{th:MartinEtMaximeSontDesSuperBeauxGosses} and \ref{th:MartinEtMaximeSontDesBeauxGosses}.

\subsection{Notations and conventions}
\label{sec:notation}
In the sequel, $(X,g)$ will be a smooth connected Riemannian 
manifold of negative sectional curvature without boundary. 
We denote by $r_I$ its injectivity radius, which is a finite positive 
number as soon as $X$ is compact.
\par
We write $\chi_1\succ \chi_2$ if $\chi_1,\chi_2\in C_c^\infty$ 
take values in $[0,1]$ and $\supp \chi_2 \subset \complement\, 
\supp (1-\chi_1)$. Similarly, we write for an open relatively compact set 
$K$ that $\chi \succ \mathbf{1}_K$ and $\mathbf{1}_K \succ \chi$, 
if $\overline{K} \subset \complement\, \supp (1-\chi)$ and 
$\supp \chi \subset K$, respectively. 
\par
If $M$ is a matrix, its transpose will be denoted by $M^\dagger$. 
If $A$ is a measurable subset of $\R^d$ or of a Riemannian manifold, 
its volume will be denoted either by $\mathrm{Vol}(A)$ or by $|A|$. 
If $A$ is a finite set, we will denote its cardinality by $\mathrm{Card}(A)$ or $|A|$. 
Writing $a \asymp b$ means that there exists a constant $C>1$ such 
that $C^{-1} a \leq b \leq Ca$.
\\
\paragraph{\textbf{Cotangent space}}
We denote by $\mathrm{dist}_X$ the geodesic distance on $X$. We denote by 
$T^*X$ the cotangent bundle of $X$, and by $\pi_X : T^*X \longrightarrow X$ 
the canonical projection. We recall that the cotangent space $T^*X$ can be 
equipped in a canonical way with a symplectic form $\sigma$. 
\par  
By $|\cdot|_x$ and by $\langle \cdot , \cdot\rangle_x$ we denote the norm 
and scalar product on $T^*_xX$ (respectively on $T_x X$ whenever convenient) 
induced by the metric $g$. Furthermore, we equip the cotangent bundle 
$T^*X$ with an arbitrary metric $g_0$ such that the induced geodesic 
distance $\mathrm{dist}_{T^*X}$ on $T^*X$ is so that 
$\mathrm{dist}_{T^*X}(\rho_1,\rho_2)\geq c\dist_X(\pi_X (\rho_1),\pi_X (\rho_2))$ 
for some fixed constant $c>0$. This is for instance the case when we 
take $g_0$ to be the Sasaki metric on $T^*X$ induced by $g$.
\par
We will denote by $S^*X\subset T^*X$ the unit cotangent bundle, and by 
$\Phi^t : T^*X\longrightarrow T^*X$ the geodesic flow. We will denote by 
the same letter its restriction $\Phi^t : S^*X \longrightarrow S^*X$. 
\\
\\
\paragraph{\textbf{Universal cover}}
We will denote the universal cover of $X$ by $\widetilde{\pi}:\widetilde{X}\to X$. 
Since $X$ is a connected Riemannian manifold of negative sectional curvature, 
$\widetilde{X}$ is a simply-connected manifold of negative sectional 
curvature.  We equip $\wt{X}$ and $T^*\wt{X}$ with the lifted Riemannian metrics 
$\wt{g}$ and $\wt{g}_0$, respectively.
We denote by $\widehat{\pi}: T^*\widetilde{X} \longrightarrow T^*X$ 
the local diffeomorphism given by $\widehat{\pi}(\widetilde{x},\widetilde{\xi}) = 
(\widetilde{\pi}(\widetilde{x}),(d_{\widetilde{x}}\widetilde{\pi})^{-T}\widetilde{\xi})$.
\section{Semiclassical analysis on smooth manifolds}\label{sec:Semi}
We present a brief review of the calculus of semiclassical 
pseudo-differential operators on a smooth $d$-dimensional \red{Riemannian} manifold 
$Y$ (not necessarily compact).  For the material reviewed here as well as further details 
we refer to the standard literature \cite{Ho84,GrSj94,Ma02,Zw12,DZ19}. 
\subsection{Semiclassical pseudo-differential operators}
\label{sec:SemClass}
We begin by defining suitable symbol classes. \corM{Let $m\in\R$ and let $\eta\in[0,1/2[$. 
For an open set $V\subset \R^d$ the local symbol class $S^m_\eta(T^*V)$, $T^*V\simeq V\times \R^d$, 
is the set of all functions $a(\cdot\,; h) \in C^{\infty}(T^*V)$ such that, for all $\alpha, \beta \in \N^d$, 
and any compact set $\mathfrak{K} \subset V$,
\begin{equation}\label{eq:sa1}
\sup \limits_{0 < h \leq 1} \sup_{(x,\xi)\in \mathfrak{K}\times \R^d }
	h^{\eta(|\alpha| +|\beta|)} \langle \xi\rangle^{-m+ |\beta|}
	|\partial_x^\alpha\partial_\xi^\beta a(x,\xi;h)| < +\infty,
\end{equation}
where $\langle \xi\rangle = (1+|\xi|^2)^{1/2}$ denotes the Japanese bracket. 
}
\corM{The class of symbols $S_\eta^{m}(T^*Y)$ is the set of all functions 
$a(\cdot\,; h) \in C^{\infty}(T^*Y)$ such that for any coordinate patch $U\subset Y$ 
and diffeomorphism $\kappa: T^*U \to T^*V\simeq V\times \R^d$, $V \subset \R^d$ we have 
that 
\begin{equation}\label{eq:sa1}
	(\kappa^{-1})^*a \in S^m_\eta(T^*V).
\end{equation}
Note that the symbol class $S_\eta^{m}(T^*Y)$ does not depend on the choice of an
atlas.
% \red{To this end,  we consider a countable locally finite covering of $Y$ by open sets $U_k$:
% \begin{equation}\label{eq:PeyresqCestCool}
% Y= \bigcup_{k\in K} U_k,
% \end{equation}
% where, for every $k\in K$, there is a diffeomorphism 
% $\kappa_k: U_k \to V_k$ 
% between $U_k\subset Y$ and $V_k \subset \R^d$. }
}

%\begin{equation}\label{eq:sa1}
%	S_\eta^{m}(T^*X) = 
%	\left\{
%	a(\cdot\,; h) \in C^{\infty}(T^*X); h\in]0,1],~
%	|\partial_x^\alpha\partial_\xi^\beta a(x,\xi;h)| \leq C_{\alpha,\beta}h^{-\eta(|\alpha| +|\beta|)} \langle \xi\rangle^{m-|\beta|}
%	\right\},
%\end{equation}
%
%
\par
We will define the symbol space of order $-\infty$ by 
$S_\eta^{-\infty}(T^*Y):= \bigcap_{m}S_\eta^{m}(T^*Y)$. A linear continuous map 
$R=R_h: \mathcal{E}'(X) \to C^{\infty}(Y)$ is called \emph{negligible} 
if its distribution kernel $K_R$ is smooth and each of its 
$C^\infty(Y\times Y)$ seminorms is $O(h^\infty)$, i.e. it satisfies 
\begin{equation}\label{eq:sa2}
	\partial_x^\alpha \partial^\beta_y K_R(x,y)= O(h^\infty), 
\end{equation}
for all $\alpha,\beta\in \N^d$, when expressed in local coordinates.
\par
A linear continuous map $P_h: C_c^{\infty}(Y)\to \mathcal{D}'(Y)$ is called a 
\emph{semiclassical pseudo-differential operator} belonging to the space 
$\Psi_{h,\eta}^{m}$ if and only if we can express $P_h$ as 
\begin{equation}\label{eq:sa2a}
	P_h =  \sum_{k\in K} \chi_k \kappa_k^* 
	\Op_h(p_k)(\kappa_k^{-1})^* \chi_k + K_h, 
\end{equation}
\corM{where the $\kappa_k: U_k \to V_k$ are a collection of diffeomorphisms 
between open sets $U_k\subset X$ and $V_k \subset \R^d$ with the 
collection of coordinates patches $U_k$ being locally finite,
$p_k \in S^m_\eta(T^*\R^d)$, $K_h$ is negligible, and 
$\chi_k\in C_c^\infty(U_k)$.} We will refer to the induced family 
$(\kappa_k,\chi_k)_{k}$ as \emph{cut-off charts}. In \eqref{eq:sa2a} 
we use the standard semiclassical quantization of the symbols $p_k$
\begin{equation}\label{eq:sa4}
	\Op_h(p_k) u(x) = \frac{1}{(2h\pi)^d} \iint_{\R^{2d}} 
	\e^{\frac{i}{h} \cdot(x-y)\cdot \xi} 
	p_k(x,\xi;h)u(y) dy d\xi, 
	\quad u\in C_c^\infty(V_k),
\end{equation}
seen as an oscillatory integral. Here $\langle \cdot, \cdot\rangle$ denotes 
the Euclidean scalar product on $\R^d$. 
\\
\par 
Equivalently, a linear continuous map $P_h: C_c^{\infty}(Y)\to \mathcal{D}'(Y)$ 
is in $\Psi_{h,\eta}^{m}$ if and only if the following two conditions hold: 
\begin{enumerate}
	\item $\phi P_h \psi$ is negligible for all 
	$\phi,\psi \in C^\infty_c(Y)$ with $\supp \phi \cap \supp \psi = \emptyset$
	(\emph{pseudolocality});
	\item for every cut-off chart $(\kappa,\chi)$ there exists a symbol 
	$p_{\kappa,\chi}\in S^m_\eta(T^*\R^d)$ such that
	\begin{equation}\label{eq:sa3}
		\chi P_h\chi 
		= \chi\kappa^*  \Op_h(p_{\kappa,\chi})(\kappa^{-1})^*\chi.
	\end{equation}
\end{enumerate}
The property of pseudolocality can be extended to $h$-dependent cut-off functions 
$\phi,\psi\in C_c^\infty$ with support contained in some $h$-independent compact 
set, with $|\partial^\alpha \phi(x)|,|\partial^\alpha \psi(x)| \leq 
O_\alpha(h^{-\varepsilon|\alpha|})$, for some $0\leq \varepsilon<1/2$ and 
$\dist (\supp\phi,\supp\psi)\geq h^{\varepsilon_0}/C$, $0\leq\varepsilon_0<1/2$, 
$C>0$.
\\
\par 
Given a symbol $p\in S_\eta^{m}(T^*Y)$ one can obtain an operator 
$P_h \in \Psi_{h,\eta}^{m}$, for instance, in the following way: 
Take a partition of unity $\{\psi_k\}_{k \in K}$ subordinate to 
a locally finite covering of $Y$ by coordinate charts  
$\{\kappa_k:Y\supset U_k\to V_k\subset \R^d\}_{k\in K}$ 
such that \corM{$\sum \psi_k =1$}. \corM{Moreover, take $C^\infty_c(U_k)\ni
\psi_k'\succ \psi_k$.} Then
\corM{
\begin{equation}\label{eq:sa1.1}
	P_h = \sum_{k \in K} \psi_k'\kappa_k^* 
	\Op_h(p_k) 
	(\kappa_k^{-1})^* \psi_k' ~ \in \Psi_{h,\eta}^{m},
\end{equation}
where $p_k:=(\psi_k p)\circ\widehat{\kappa}_k^{-1} \in S^m_\eta(T^*\R^d)$ 
and the symplectomorphism $\widehat{\kappa}_k^{-1}:T^*V_k \to T^*U_k$ is defined 
by 
\begin{equation}\label{eq:sa1.1.0}
\widehat{\kappa}_k^{-1}(x,\xi)=(\kappa_k^{-1}(x),
(d_x\kappa^{-1}_k)^{-T}\xi).
\end{equation}
}
The correspondence $P_h \mapsto p$ is not globally 
well-defined, but it gives rise to a bijection 
\begin{equation}\label{eq:sa5}
	\Psi_{h,\eta}^{m}/ h^{1-2\eta}\Psi_{h,\eta}^{m-1} \longrightarrow 
	S_\eta^{m}(T^*Y) / h^{1-2\eta}S_\eta^{m-1}(T^*Y),
\end{equation}
see for instance \cite[Appendix E.1.7]{DZ19} for details. 
The image $\sigma_P$ of $P$ under the map \eqref{eq:sa5} is called 
\emph{principal symbol} of $P$. Note that for $P_h$ in \eqref{eq:sa1.1} 
the principal symbol is $\sigma(P_h)=p$. \corM{When $a\in S^m_\eta(T^*\R^d)$ 
admits an asymptotic expansion $a \sim a_0 + h^{k_1}a_1 + h^{k_2}a_2 + 
\dots \in S^m_\eta(T^*\R^d)$, $k_1< k_2 < \dots $, we call $a_0$ the 
principal part of $a$. When \eqref{eq:sa3} holds, then 
\begin{equation}\label{eq:principalSymbrel}
	\sigma(P_h)\circ \wh{\kappa}^{-1} = p_{\kappa,\chi}^0 \text{ on } T^*O,
\end{equation}
where $O$ is a small open neighbourhood of the support of 
$\chi\circ \kappa^{-1}$ and $p_{\kappa,\chi}^0$ is the principal part of 
$p_{\kappa,\chi}$. 
}
\subsection{Semiclassical Sobolev spaces}
We now recall the definition of \emph{semiclassical 
Sobolev spaces} on $Y$. First of all,  when $Y=\R^d$, we
define for $s\in\R$ the \emph{semiclassical Sobolev space} 
$H_{h}^s(\R^d)\subset \mathcal{S}'(\R^d)$ as the space of 
all tempered distributions $u\in\mathcal{S}'(\R^d)$ such that 
\begin{equation*}
	\| u \|_{H_{h}^s(\R^d)} 
	= \| \Op_h ( \langle \xi \rangle^s) u \|_{L^2(\R^d)} <\infty.
\end{equation*}
When $Y$ is a smooth manifold, we define for $s\in\R$ the \emph{local 
semiclassical Sobolev space} $H_{h,\loc}^s(Y)\subset \mathcal{D}'(Y)$ 
as the set of all distributions $u\in\mathcal{D}'(Y)$ such that 
\begin{equation*}
		(\kappa^{-1})^*\chi u \in H_{h}^s(\R^d)
\end{equation*}
for all local coordinate charts $\kappa: Y\supset U \to V\subset \R^d$ 
and cut-off functions $\chi\in C^\infty_c(U)$. We can turn 
$H_{h,\loc}^s(Y)$ into a Fréchet space by equipping it with the 
countable family of seminorms 
$\{\| (\kappa_k^{-1})^*\chi_k u \|_{H_{h}^s(\R^d)}\}_{k \in K}$,  %$0< K\leq +\infty$,
for any fixed, open and locally finite countable covering of $Y$ with 
coordinate charts $\{\kappa_k:X\supset U_k\to V_k\subset \R^d\}_{k\in K}$ 
and subordinate partition of unity $\{\chi_k\}_{k\in K}$, 
$\chi_k \in C^\infty_c(U_k)$. The topology on $H_{h,\loc}^s(Y)$ 
induced by such a family of seminorms is independent 
of the choice of open locally finite covering, coordinate charts and 
partition of unity. 
\\
\par 
We denote by $H^s_{h,\comp}(Y)$ the space of all elements of 
$H_{h,\loc}^s(Y)$ which are supported inside some $h$-independent 
compact subset of $Y$. When we are dealing with the case $h=1$ in 
the local Sobolev norms, we will simply write $H^s_{\comp}(Y)$ and
 $H^s_{\loc}(Y)$ for the corresponding spaces. We note the following 
 regularity result:  each $A\in \Psi_{h,\eta}^{m}(X)$ is bounded uniformly 
 in $h$ on compact sets as an operator  
 \begin{equation}\label{eq:sa6}
 A=	\Op_h(a):H_{h,\comp}^{s}(Y) \longrightarrow H_{h,\loc}^{s-m}(Y).
 \end{equation}
\\
\textbf{When $Y$ is compact} then 
\begin{equation*}
	H^s_{h,\comp}(Y) = H_{h,\loc}^s(Y) =: H_{h}^s(Y)
\end{equation*}
and we can equip it with the norm
\begin{equation}\label{eq:sa5.1}
	\| u \|_{H_{h}^s(Y)}^2 = \sum_{k\in K}
	\|(\kappa_k^{-1})^*\chi_k u \|_{H_{h}^s(\R^d)}^2, \quad 
	\mathrm{Card}(K)< +\infty, 
\end{equation}
where $\{\kappa_k\}$ is a finite collection of coordinate charts 
with a subordinate partition of unity $1=\sum_k\chi_k$ as above. 
This norm is not intrinsically defined, but taking different coordinate 
patches and cut-off functions in \eqref{eq:sa5.1} yields an equivalent norm (\red{uniformly in $h$}). 
\\
\par 
Similarly we define, for every $L\in \N$,  the $C^L$ norms on $Y$, by 
\begin{equation}\label{eq:sa5.15}
	\| u \|_{C^L(Y)} = \max_{k\in K}
	\|(\kappa_k^{-1})^*\chi_k u \|_{C^L(\R^d)}.
\end{equation}
Since $Y$ is compact it follows that taking different coordinate 
patches and cut-off functions in \eqref{eq:sa5.15} yields an equivalent norm. 
By standard arguments one then gets the Sobolev inequalities 
\begin{equation}\label{eq:sa8.2}
	\| u \|_{C^L(Y)} = O_{d,s,L}(1)h^{-d/2-L}
	\| u \|_{H_{h}^s(Y)} \quad \text{for } s>L +d/2. 
\end{equation}
\\
\textbf{If $Y$ is a non-compact manifold}, it is more delicate to obtain well-defined 
Sobolev norms. For the following discussion we refer the reader to 
\cite[Appendix A]{Sh92}. 
\par 
A smooth Riemannian manifold $Y$ is called a \emph{manifold of bounded 
geometry} if it has a strictly positive injectivity radius 
$r_{\mathrm{inj}}>0$ and every covariant derivative of the Riemann curvature 
tensor $R$ is bounded, i.e. for very $m=0,1,\dots$ there exists 
a $C_m>0$ such that $|\nabla^m R| \leq C_m$. 
\\
\par
Let $Y$ be a manifold of bounded geometry and denote by $dv_g$ the 
Riemannian density on $X$. We define the semiclassical 
Sobolev norm $\| \cdot \|_{H_{h}^s(Y)}$, $s\in\N$, on $C_c^\infty(Y)$ by 
\begin{equation}\label{eq:sa7}
	\| u \|_{H_{h}^s(Y)}^2 = \sum_{m=0}^s \int_Y|(h\nabla)^m u|^2dv_g, 
\end{equation}
where $|\cdot|$ is understood as the norm on tensors induced by the 
Riemannian metric $g$. 
\par 
We then define the Sobolev space $H_h^s(Y)$ to be the completion  
of $C_c^\infty(Y)$ with respect to the norm \eqref{eq:sa7}. The space 
$H_h^s(Y)$ has a natural Hilbert space structure, and it is 
naturally included in the space of distributions $\mathcal{D}'(Y)$. 
In particular $H^0_h(Y)=L^2(Y,dv_g)$, where the latter is defined via the 
$L^2$ norm with respect to the integration measure $dv_g$. 
The usual embedding theorems hold, i.e. $H^s_h(Y) \subset C^k_b(Y)$ 
if $s>k+d/2$. 
\par 
Since by assumption $Y$ has a strictly positive injectivity radius, 
we have the following result essentially due to M.~Gromov \cite{Gr81}, 
see also \cite[Lemma A.1.2]{Sh92} for a proof.
\begin{lem}\label{lem:geodesicCoord}
Let $Y$ be a smooth Riemannian manifold of bounded geometry 
and put $\varepsilon_0=r_{\mathrm{inj}}/3$. Then, for every 
$\varepsilon\in (0,\varepsilon_0)$ there exists a countable 
covering of $Y$ by balls of radius $\varepsilon$ such that 
$Y=\bigcup_{k\in K} B(x_k,\varepsilon)$ and such that the covering 
of $Y$ by balls $B(x_k,2\varepsilon)$ with double radius and 
the same centres satisfies that the maximal number of the balls 
with non-empty intersection is finite. 
\end{lem}
This result implies the existence of a ``uniform'' partition 
of unity of $Y$ subordinate to the covering by balls from the 
above Lemma. Indeed, for every $\varepsilon\in (0,\varepsilon_0)$, 
with $\varepsilon_0>0$ as in the Lemma above, there exists a partition 
of unity $1=\sum_k \chi_k$ on $Y$ such that 
$\chi_{k} \in C^\infty_c(Y;[0,\infty[)$ with 
$\supp\chi_k\subset B(x_k,2\varepsilon)$ (the points $x_k$ as 
in the above Lemma) and such that for every $\alpha\in\N^d$ there 
exists a constant $C_\alpha>0$ such that 
$|\partial^\alpha_y \chi_k(y)|\leq C_\alpha$ in geodesic normal
 coordinates \red{(see \cite[Definition 1.4.4]{Jost})} and uniformly with respect to $k$.
\par 
Using this partition of unity we can give an alternative definition of the 
semiclassical Sobolev norm. Indeed, let 
$\kappa_k:B(x_k,2\varepsilon)=:U_k\to V_k\subset\R^d$ be the local 
geodesic coordinate chart in $B(x_k,2\varepsilon)$, then we can define 
for $s\in\R$
\begin{equation}\label{eq:sa8}
	\| u \|_{H_{h}^s(Y)}^2 = \sum_{k\in K}
	\|(\kappa_k^{-1})^*\chi_k u \|_{H_{h}^s(\R^d)}^2.
\end{equation}
The norms \eqref{eq:sa7} and \eqref{eq:sa8} are equivalent for $s\in\N$. 
\\
\par 
Similarly we define the $C^L$ norms on an open set $\cU\subset Y$, 
i.e. for $L\in\N$ 
\begin{equation}\label{eq:sa8.1}
	\| u \|_{C^L(\cU)} := \sup_{k\in K}
	\|(\kappa_k^{-1})^*\chi_k u \|_{C^L(\R^d\cap \kappa_k(\cU\cap U_k))}.
\end{equation}
Furthermore, we have the Sobolev inequalities 
\begin{equation*}
	\| u \|_{C^L(\mathcal{U})} = O_{d,s,L}(1)h^{-d/2-L}
	\| u \|_{H_{h}^s(\mathcal{U})}, \quad \text{for } s>L +d/2. 
\end{equation*}
\red{
When we only want to control the derivatives of a $C^L$ function $f$ 
on some open set $U\subset \R^d$ we 
use the $\dot{C}^L(U)$ norm defined by 
\begin{equation*}
	\| f \|_{\dot{C}^L(U)} = \max_{1\leq |\alpha|\leq L}
	\|\partial^\alpha f \|_{L^\infty(U)}.
\end{equation*}
Similarly we define the $\dot{C}^L$, $L\in\N$, norms on an open set 
$\cU\subset Y$, i.e. 
\begin{equation}\label{eq:sa8.1b}
	\| u \|_{\dot{C}^L(\cU)} := \sup_{k\in K}
	\|(\kappa_k^{-1})^*\chi_k u \|_{\dot{C}^L(\R^d\cap \kappa_k(\cU\cap U_k))}.
\end{equation}
}
\red{In this paper, we will consider the case when $Y= \wit{X}$ is the universal cover of a compact smooth Riemannian manifold $X$ of negative sectional curvature.}
%of negative sectional curvature and with its universal cover $\wt{X}$. 
Let $\kappa_k:U_k\to V_k\subset\R^d$, $k=1,\dots, M$, 
be the local geodesic coordinates on $X$ with $U_k =B(x_k,2\varepsilon)$ 
as in the discussion above \eqref{eq:sa8}, such that the coordinate patches 
$U_k$, $k=1,\dots,M$, form a finite open covering of $X$. 
Furthermore, let $\chi_k\in C^\infty_c(X;[0,1])$, $k=1,\dots,M$, be a 
finite partition of unity of $X$ subordinate to this open covering. 
\red{ In this case we will use these charts and cut-off functions 
to define the Sobolev norms on $X$ as in \eqref{eq:sa5.1}. Next, recall that the 
covering map $\widetilde{\pi}:\widetilde{X}\to X$ was defined in Section 
\ref{sec:notation}. By the covering property we have that if $U_k\subset X$ 
is a sufficiently small open set (in other words if $\varepsilon>0$ is small enough), 
then $\widetilde{\pi}^{-1}(U_k)=
\bigsqcup_{\iota \in I}\widetilde{U}_{k,\iota}$ is a countable union 
of disjoint open sets. Furthermore, the restriction 
$\widetilde{\pi}_{k,\iota}:=\widetilde{\pi}|_{U_{k,\iota}}: \widetilde{U}_{k,\iota}
\to U_k$ is a diffeomorphism.} Using 
that the covering map is a local isometry, we find that the lifted 
coordinate charts $\widetilde{\kappa}_{k,\iota}=\wt{\pi}_{k,\iota}^*\kappa_k$, 
are local geodesic coordinates on $\widetilde{X}$, the universal covering of $X$,   
with coordinate patches $\widetilde{U}_{k,\iota}=
B(\wt{\pi}^{-1}_{k,\iota}(x_k),2\varepsilon)$. Furthermore, the coordinate 
patches form a locally finite open covering of $\widetilde{X}$. The 
lifted cut-off functions $\wt{\chi}_{k,\iota}=\wt{\pi}_{k,\iota}^*\chi_{k}$ form a locally finite 
partition of unity on $\wt{X}$, i.e. $\sum_{k,\iota}\wt{\chi}_{k,\iota}=1$ 
on $\wt{X}$. We will be working mostly with the following Sobolev and $C^L$ 
norms: let $\wt{\cU}\subset \red{\wit{X}}$ be an open set, then 

\begin{equation}\label{eq:sa8.3}
	\| u \|_{H_{h}^s(\wt{X})}^2 = \sum_{k,\iota}
	\|(\kappa_{k,\iota}^{-1})^*\wt{\chi}_{k,\iota} u \|_{H_{h}^s(\R^d)}^2, 
	\quad 
	\| u \|_{C^L(\wt{\cU})} := \max_{k,\iota}
	\|(\kappa_{k,\iota}^{-1})^*\wt{\chi}_{k,\iota} u 
		\|_{C^L(\R^d\cap\kappa_{k,\iota}(\wt{\cU}\cap \wt{U}_{k,\iota}))}.
\end{equation}
\red{
The $\dot{C}^L$ norms on $\wt{\cU}$ are defined similarly.} 
Furthermore, the Sobolev inequalities \eqref{eq:sa8.2} remain valid. 

\subsection{Choice of quantization and lifts to the universal cover}
\label{subsec:LiftPseudo}
\MI{As mentioned before,} we will be working not only on a compact connected smooth Riemannian manifold 
$X$ of negative sectional curvature, but 
also on its universal cover $\widetilde{X}$. It will be useful to 
discuss lifting a pseudo-differential operator $P$ on $X$ to a 
pseudo-differential operator $\widetilde{P}$ on $\widetilde{X}$. 
\par 
Recall the covering map $\widetilde{\pi}:\widetilde{X}\to X$ and 
the map $\widehat{\pi}:T^*\widetilde{X}\to T^*X$ defined in Section 
\ref{sec:notation}. By the covering property we have that if $U\subset X$ 
is a sufficiently small open set, then $\widetilde{\pi}^{-1}(U)=
\bigsqcup_{\iota \in I}\widetilde{U}_{\iota}$ is a countable union 
of disjoint open sets. Furthermore, the restriction 
$\widetilde{\pi}_{\iota}:=\widetilde{\pi}|_{U_{\iota}}: \widetilde{U}_{\iota}
\to U$ is a diffeomorphism. Hence, we may lift a chart $\phi:U\to V\subset \R^d$ 
to a chart $\widetilde{\phi}_\iota:=(\widetilde{\pi}_{\iota}^{-1})^*\phi: 
\widetilde{U}_{\iota} \to V$ on $\widetilde{X}$. Similarly, we can lift a 
function $\chi\in C_c^\infty(U)$ to $\widetilde{\chi}_\iota:=
(\widetilde{\pi}_{\iota}^{-1})^*\chi \in C_c^\infty(\widetilde{U}_{\iota})$. 
Slightly abusing notation, we will also denote by $\widetilde{\chi}_\iota$ 
its extension by $0$ outside its support to a smooth compactly supported function 
on $\wt{X}$. 
\subsubsection{\textbf{Choice of quantization on a compact manifold.}} 
\MI{When working on a compact manifold, we will fix a (noncanonical) choice of quantization as in \eqref{eq:sa1.1}.}
\red{
\begin{definition}[Choice of quantization on a compact manifold]\label{def:Quantization} 
Fix a finite partition of unity $\{\vartheta_k\}_{k \in K}$, such that \corM{$\sum \vartheta_k =1$}, 
subordinate to a finite covering of $X$ by coordinate charts 
$\{\kappa_k:X\supset U_k\to V_k\subset \R^d\}_{k\in K}$ 
with small enough chart domains such that $\widetilde{\pi}^{-1}(U_k)$ is a disjoint union 
of open sets for every $k$ (as in the paragraph after \eqref{eq:sa8.1b}). 
\corM{Moreover, take $C^\infty_c(U_k)\ni\vartheta_k'\succ \vartheta_k$.}
Given a symbol $p \in S_{\eta}^{m}(T^*X)$, we will call 
\corM{
\begin{equation}\label{eq:sa1.15}
	\Op_h(p):=P_h = \sum_{k \in K} \vartheta'_k\kappa_k^* 
	\Op_h( p_k) 
	(\kappa_k^{-1})^* \vartheta'_k,
\end{equation}
the quantization of $p$. Here, where $p_k:=(\vartheta_k p)\circ\widehat{\kappa}_k^{-1} \in S^m_\eta(T^*\R^d)$, 
c.f. \eqref{eq:sa1.1.0}. }
\end{definition}
One can easily check that $\Op_h(p)\in \Psi_{h,\eta}^{m}$ with $\sigma(\Op_h(p)) = p$. 
}
\subsubsection{\textbf{Lifting a quantization to the universal cover}.} 
\red{
We will lift the operator defined by the quantization choice in 
Definition \ref{def:Quantization} to a pseudo-differential operator on 
$\wt{X}$ in the following way.
}
\corM{
\begin{definition}[Lift of quantization]\label{def:QuantizationLift}
Let $\{\vartheta_k\}_{k \in K}$, $\{\vartheta_k'\}_{k \in K}$ and 
$\{\kappa_k\}_K$ be as in Definition 
\ref{def:Quantization} and recall the notation of lifted charts and 
functions introduced at the beginning of Section \ref{subsec:LiftPseudo}. 
Given a symbol $p \in S_{\eta}^{m}(T^*X)$, will call   
\begin{equation}\label{eq:liftePseudo4}
	\wt{P}_h := \sum_{k \in K,\iota \in I } 
	\wt{\vartheta}'_{k,\iota}\wt{\kappa}_{k,\iota}^* 
	\Op_h(p_k) 
	(\wt{\kappa}_{k,\iota}^{-1})^* \wt{\vartheta}'_{k,\iota}~
\end{equation}
a lift of $\Op_h(p)$ to the universal cover $\wt{X}$. Here, 
$\{p_k\}_{k\in K}$ is as in Definition \ref{def:Quantization}.
\end{definition}
}
\red{From the discussion in Section \ref{sec:SemClass} we see that 
$\wt{P}_h\in \Psi_{h,\eta}^{m}(\wt{X})$ and notice that this lifted 
quantization produces properly supported operators. Moreover, it is easy to 
check that the following three properties hold: 
}
% Given a semiclassical pseudo-differential operator 
% $P\in \Psi_{h,\eta}^{m}(X)$ with principal symbol 
% $\sigma(P)=p\in S_{\eta}^{m}(T^*X)$, we say that 
% $\widetilde{P}\in \Psi_{h,\eta}^{m}(\widetilde{X})$ 
% is a lift of $P$ to the universal cover $\widetilde{X}$ if 
% the following three conditions hold:
%
\begin{enumerate}
	\item the principal symbol of $\widetilde{P}_h$ is 
	given by lifted principal symbol of $P_h$, i.e. 
	\begin{equation}\label{eq:liftePseudo1}
	\sigma(\widetilde{P}_h) = \widehat{\pi}^*\sigma(P_h)\in 
		S^{m}_{\red{\eta}}(T^*\widetilde{X});
	\end{equation}
	\item for every cut-off chart $(\kappa,\chi)$ on $X$ such 
	that \eqref{eq:sa3} holds with $p_{\kappa,\chi}\in S^m_\eta(T^*\R^d)$, 
	we have that 
	\begin{equation}\label{eq:liftePseudo2}
		(\widetilde{\kappa}_\iota^{-1})^*
		\widetilde{\chi}_\iota \widetilde{P}_h\widetilde{\chi}_\iota
		\widetilde{\kappa}_\iota^* 
		\corM{=(\kappa^{-1})^*\chi P_h\chi\kappa^* }
		=  \chi\circ\kappa^{-1} \Op_h(p_{\kappa,\chi})\chi\circ\kappa^{-1}.
	\end{equation}
	\item for every $\phi,\psi \in C^\infty_c(X)$ with 
	$\supp \phi \cap \supp \psi = \emptyset$ which may depend on $h$ 
	as in the paragraph after \eqref{eq:sa3}, we have that 
	$\widetilde{\phi}_{\iota'} \widetilde{P}_h\widetilde{\psi}_\iota$ 
	is negligible and that for every $N\in\N$ 
	\begin{equation}\label{eq:liftePseudo3}
		\widetilde{\phi}_{\iota'}\widetilde{P}_h\widetilde{\psi}_\iota =
		O_{\phi,\psi,N}(h^\infty)
		:H^{-N}_{h}(\wt{X}) \to H^N_{h}(\wt{X}).
	\end{equation}
	In other words $\widetilde{\phi}_{\iota'} \widetilde{P}_h\widetilde{\psi}_\iota$ 
	is a bounded operator $H^{-N}_{h}(\wt{X}) \to H^N_{h}(\wt{X})$ 
	with operator norm $=O_{\phi,\psi,N}(h^\infty)$ which is independent of $\iota,\iota'$. 
\end{enumerate}
%
% \par
% Given a $P\in \Psi_{h,\eta}^{m}(X)$, such a lift $\widetilde{P}$ always 
% exists. For instance, we may construct $\wt{P}$ from the non-canonical 
% quantization \eqref{eq:sa1.1} by lifting the cut-off charts. More precisely 
% %
% \begin{equation}\label{eq:liftePseudo4}
% 	\wt{P}_h = \sum_{k \in K,\iota \in I } 
% 	\psi_{k,\iota}\kappa_{k,\iota}^* 
% 	\Op_h(p_{\kappa_k}) 
% 	(\kappa_{k,\iota}^{-1})^* \psi_{k,\iota}~ \in \Psi_{h,\eta}^{m}(\wt{X})
% \end{equation}
% %
% satisfies (\ref{eq:liftePseudo1}--\ref{eq:liftePseudo3}). 
%
%
\subsubsection{\textbf{Lifting a differential operator}} 
Lifting differential operators to $\widetilde{X}$ is more straightforward 
since they are local operators: the lift of the Laplace-Beltrami operator 
$\Delta_g$ on $X$ to the cover $\widetilde{X}$ is given by 
$\Delta_{\widetilde{g}}$. Indeed, one can easily check that for 
$f\in C^\infty(X)$
\begin{equation}\label{eq:liftedLaplacian}
	\Delta_{\widetilde{g}}\widetilde{\pi}^*f = \widetilde{\pi}^*\Delta_{g}f.
\end{equation}
We will drop the metric in the subscript whenever it is clear which 
Laplacian we are using. 
\par 
We then define the lifted Schr\"odinger operator \eqref{eq:SchroedingerOp} by
\begin{equation}\label{eq:liftedSchroedinger}
	\widetilde{P}_h^\delta := -\frac{h^2}{2} \Delta_{\widetilde{g}} 
	+ \delta \widetilde{Q}_\omega,
\end{equation}
\red{
where $\widetilde{Q}_\omega$ is the multiplication operator by 
$\wt{q}_\omega=\wt{\pi}^*q_\omega$ in the \eqref{eq:Pot}, or 
$\widetilde{Q}_\omega$ is the lift of $Q_\omega=\Op_h(q_\omega)$ 
to the universal cover, as Definition \ref{def:QuantizationLift}, 
in the \eqref{eq:Pseudo}.
}
\subsubsection{\textbf{Mapping properties of the pull-back} $\wt{\pi}^*$}
The pull-back action $\wt{\pi}^*$ via the covering map $\wt{\pi}$ is a 
continuous linear map $C^n(X)\to C^n(\wt{X})$, $n\in\N$, and 
$L^2(X,dv_g)\to L^2_{\loc}(\wt{X},dv_{\wt{g}})$, and 
$H^k_h(X)\to H^k_{h,\loc}(\wt{X})$, $k\in\N$. Indeed, take the 
countable family of lifted cut-off chart $(\wt{\kappa}_{k,\iota},\wt{\chi}_{k,\iota})$ 
defined in the paragraph above \eqref{eq:sa8.3}, then
\begin{equation}\label{eq:PullBackMapping1}
	\| (\wt{\kappa}_{k,\iota}^{-1})^*\wt{\chi}_{k,\iota} \wt{\pi}^*u \|_{H_{h}^s(\R^d)}
	=
	\| (\kappa_{k}^{-1})^*\chi_{k} u \|_{H_{h}^s(\R^d)}.
	% \| \wt{\pi}^* u \|_{H_{h}^k(X)}
\end{equation}
The formal adjoint $^t\wt{\pi}^*$ of $\wt{\pi}^*$ is given by 
\begin{equation}\label{eq:PullBackMapping2}
	(^t\wt{\pi}^* u)(x) = \sum_{\wt{\pi}(\wt{x})=x} u(\wt{x}), 
	\quad u \in C^\infty_c(\wt{X}),
\end{equation}
mapping $C^\infty_c(\wt{X})\to C^\infty(X)$ linearly. Let 
$u\in H^k_{h,\comp}(\wt{X})$, then 
\begin{equation}\label{eq:PullBackMapping3}
\begin{split}
	\|^t\wt{\pi}^* u\|_{H^k_{h}(X)} 
	&\leq \sum_{k,\iota} \|\chi_k (\pi_{k,\iota}^{-1})^*u\|_{H^k_{h}(X)} \\
	&= \sum_{k,\iota} \|(\kappa_k^{-1})^*\chi_k (\pi_{k,\iota}^{-1})^*u\|_{H^k_{h}(\R^d)} \\
	&= \sum_{k,\iota} \|(\wt{\kappa}_{k,\iota}^{-1})^*\wt{\chi}_{k,\iota}u\|_{H^k_{h}(\R^d)} \\
	&\leq C\|u\|_{H^k_{h}(\wt{X})},
\end{split}
\end{equation}
where $C>0$ only depends on the support of $u$. Hence, $^t\wt{\pi}^*$ 
is bounded uniformly in $h>0$ on compact sets as an operator 
\begin{equation}\label{eq:PullBackMapping4}
	^t\wt{\pi}^*: H^k_{h,\comp}(\wt{X}) \longrightarrow 
	H^k_{h}(X).
\end{equation}
\par Let $I\subset \R$ be an open interval and let 
$u\in C^{\infty}(I\times \wt{X})$ be such that $\supp u(\tau, \cdot)$ 
is compact for every $\tau \in I$. A straightforward computation 
in local coordinates shows that 
\begin{equation}\label{eq:PullBackMapping5}
	(i\partial_\tau + \Delta_g)(^t\wt{\pi}^* u)= 
	{^t\wt{\pi}^*}(i\partial_\tau + \Delta_{\wt{g}}) u 
	~~ \in C^{\infty}_c(I\times \wt{X}).
\end{equation}
\section{Classical dynamics}\label{Sec:Class}

\red{In this section, we will show several properties related to the Hamiltonian dynamics induced by the potential $q_\omega$ from  (\ref{eq:randomSymbol}). Many of these properties will be very similar to those of the usual geodesic flow; however, the Hamiltonian dynamics considered here depends on the parameters $h$ and $\omega$, and we will have to obtain estimates that are uniform with respect to these parameters.}

\subsection{Hamiltonian flows}\label{Sec:HamiltonFlows}

Let $(X,g)$ be a compact connected manifold of negative sectional curvature 
and let $q_\omega$ be as in \eqref{eq:randomSymbol}. Let $h\in ]0,h_0]$ with 
$h_0>0$ sufficiently small.  For 
\red{$\delta= \delta(h)>0$}, we put 
\begin{equation}\label{eq:Hamiltonian}
	p(x,\xi;\delta) := \frac{1}{2} |\xi|_x^2+ \delta q_\omega(x, \xi).
\end{equation}
This ($\delta$-dependent and therefore possibly $h$-dependent) Hamiltonian 
induces a Hamiltonian vector field $H_{p}$ which may be defined by the 
pointwise relation $H_p \lrcorner\, \sigma= -dp$, or in canonical symplectic 
coordinates by
\begin{equation}\label{eq:HamiltonianVF} 
	H_{p} 
	=\sum_{k=1}^d \frac{\partial p}{\partial \xi_k}
	\frac{\partial }{\partial x_k}
	-
	\frac{\partial p}{\partial x_k}
	\frac{\partial }{\partial \xi_k}.
\end{equation}
Hence, $H_p$ is a smooth section of $TT^*X$. For 
$\lambda, \lambda_1, \lambda_2 \geq 0$ we define the energy layers
\begin{equation}\label{eq:EnergyLayers}
\begin{split}
\mathcal{E}_{\delta,\lambda} 
	&:= \{(x,\xi) \in T^*X ; p(x,\xi;\delta) =\lambda\} 
	\subset T^*X\\
\mathcal{E}_{\delta,(\lambda_1, \lambda_2)} 
	&:= \{(x,\xi) \in T^*X; \lambda_1 < p(x,\xi;\delta) <\lambda_2\}  
		= \bigcup_{\lambda \in (\lambda_1, \lambda_2)} \mathcal{E}_{\delta,\lambda}.
\end{split}
\end{equation}
Similarly we define $\mathcal{E}_{\delta,[\lambda_1, \lambda_2]}$ 
by replacing the strict inequalities with non strict ones.
When $\delta>0$, these sets depend on the random parameter $\omega$, though we do 
not write it explicitly. 
\par
For every $\delta>0$, we denote by 
$$ \Phi^{t}_\delta:=\exp(tH_{p}), \quad \corM{t\in\R,} $$
the Hamiltonian flow on $T^*X$ generated by the vector field $H_p$. 
Furthermore, for every $\lambda>0$, we denote by 
$\Phi^{t,\lambda}_\delta=\Phi^{t}_\delta|_{\mathcal{E}_{\delta,\lambda}}$ its 
restriction to $\mathcal{E}_{\delta,\lambda}$. 
We will often write $\Phi^t$ for the flow $\Phi_0^{t,1}$, which 
is simply the  geodesic flow acting on $\mathcal{E}_{0,\frac{1}{2}} = S^*X$. 
\par 
Notice that there exists $\delta_0>0$  \red{independent of $h$} such that, for all 
$0\leq \delta<\delta_0$, we have
\begin{equation}\label{eq:NiveauxEnergieComparables}
\mathcal{E}_{0, \frac{1}{2}} 
\subset \mathcal{E}_{\delta, (\frac{1}{4},1)}
\subset \mathcal{E}_{0, (\frac{1}{8},2)}.
\end{equation}
Until further notice, we will always assume that $\delta_0$ is small enough 
so that (\ref{eq:NiveauxEnergieComparables}) holds.

\red{\begin{rem}\label{Rem:LeavingBall}
In the paper, we will use several times the following elementary fact. 
Let $\gamma, c, c'>0$. There exists $c'', h_0>0$ such that the following 
holds for all $0 < h < h_0$ and all $x_0\in X$. If $\rho\in S^*X$ with 
$\pi_X(\rho) \in B(x_0, c h^\gamma)$, then, for all 
$t\in \left[c'' h^\gamma,  r_I\right]$, we have $\pi_X \Phi^t_0(\rho) 
\notin B(x_0, c' h^\gamma)$. Here, $r_I$ denotes the injectivity radius of 
$X$ \corM{and $B(x_0, r)$ denotes the geodesic ball of radius $r>0$ centered at 
$x_0$.} 
\end{rem}}
\subsection{Hyperbolicity}\label{sec:Hyperbolicity}
\par
Since $X$ has negative curvature, we have that for every $\lambda>0$ the 
Hamiltonian flow $\Phi_0^{t,\lambda}$ is Anosov \cite{Ebe}. This implies that 
for each $\rho\in \mathcal{E}_{0,\lambda}$, there exist subspaces 
$E_\rho^+$, $E_\rho^-$, $E_\rho^0$ of $T_\rho  \mathcal{E}_{0,\lambda}$ 
-- respectively called the \emph{unstable}, \emph{stable} and \emph{neutral} 
direction at $\rho$ -- such that:
\begin{itemize}
\item $T_\rho  \mathcal{E}_{0,\lambda}=E_\rho^+\oplus E_\rho^-\oplus E_\rho^0$ 
	for every $\rho\in \mathcal{E}_{0,\lambda}$.
\item The distributions $E_\rho^+$, $E_\rho^-$ and $E_\rho^0$ depend 
	H\"older-continuously on $\rho$. 
\item The distribution $E_\rho^0$ is one dimensional and generated 
	by $\frac{\mathrm{d}}{\mathrm{d}t}|_{t=0}\Phi_0^{t,\lambda}(\rho)$. 
\item $E_\rho^+$ and $E_\rho^-$ are both $(d-1)$-dimensional, and for 
each $t\in\R$, we have 
\begin{equation}\label{eq:DirInv}
	d_\rho \Phi_0^{t,\lambda}(E^\pm_\rho)
	=E^\pm_{\textcolor{black}{\Phi_0^{t,\lambda}}(\rho)}.
\end{equation}
\item There exists $C_0>0$ and $A_0>1$ such that for each 
$\rho\in \mathcal{E}_{0,\lambda}$, $t>0$, $\xi^+\in E^+_\rho$ 
and $\xi^-\in E^-_\rho$,
\begin{equation}\label{eq:DefHyp}
\begin{aligned}
 |d_\rho\Phi_0^{-t,\lambda}(\xi^+)|_{\Phi_0^{-t,\lambda}(\rho)}&\leq C_0 A_0^{-t} |\xi^+|_{\rho}\\
 |d_\rho\Phi_0^{t,\lambda}(\xi^-)|_{\Phi_0^{t,\lambda}(\rho)}&\leq C_0 A_0^{-t}|\xi^-|_{\rho}\, .
\end{aligned}
\end{equation}
\end{itemize}
For $\rho = (x,\xi) \in \mathcal{E}_{0,\lambda}$, for some $\lambda>0$, write 
\begin{equation}\label{eq:DirNeutre2}
\hat{E}^0_\rho :=
%\{(x,s\xi);\, s\in\R\}.
\red{T_{\rho}} \{(x,s\xi);\, s\in\R\}
\end{equation}

Then
\begin{equation}\label{eq:DirNeutre2a}
	T_\rho T^*X = E_\rho^+\oplus E_\rho^0\oplus E_\rho^- \oplus \hat{E}_\rho^0.
\end{equation}
The structural stability lemma\footnote{\red{Recall that,  to apply this lemma, we need the flows $\Phi_\delta^{t,\lambda}$ and $\Phi_0^{t,\lambda}$ to be close in the $C^1$ topology, which is ensured by condition \eqref{eq:CondBeta}.}} (\cite[Theorem 18.2.3]{Kat}) \red{along with \eqref{eq:CondBeta}} implies that \red{ there exists $h_0>0$ such that for all $h\in ]0, h_0[$} and 
all $\omega$, the Hamiltonian flow $\Phi^{t,\lambda}_\delta$ is also an Anosov 
flow on each energy level $\mathcal{E}_{\delta, \lambda}$ for $\lambda\in (\frac{1}{4},1)$, satisfying properties analogous to those stated above. 
In particular, if $\lambda \in (\frac{1}{4},1)$ and 
$\rho\in \mathcal{E}_{\delta, \lambda} \subset \mathcal{E}_{0, 
(\frac{1}{8},2)}$, 
the Hamiltonian flow $\Phi^{t,\lambda}_\delta$ has stable, unstable and 
neutral directions at $\rho$ which are denoted by $E_{\delta,\rho}^\pm$ and 
$E_{\delta,\rho}^0$ respectively. Furthermore, the map
\begin{equation}\label{eq:ContinuityDIrections}
[0, \delta_0] \times  
	\mathcal{E}_{0, (\frac{1}{4},1)}
\ni (\delta, \rho) 
\mapsto 
(E_{\delta,\rho}^-, E_{\delta,\rho}^+,E_{\delta,\rho}^0) 
\text{ is continuous.}
\end{equation}

Thus, for \red{$h_0$} small enough, we find that for any $\lambda \in (\lambda_1, \lambda_2)$, 
any \red{$h\in [0, h_0]$} and any $\rho\in \mathcal{E}_{\delta, \lambda}$,
\begin{equation}\label{eq:DecompoTangentSpace}
	T_\rho T^*X = E_{\delta,\rho}^+\oplus 
	E_{\delta,\rho}^0\oplus E_{\delta,\rho}^- \oplus \hat{E}_{\rho}^0.
\end{equation}

By compactness and (\ref{eq:ContinuityDIrections}) we may find $c_1, c_2>0$ 
such that if $v = (v_+, v_0, v_-, \hat{v}_0) \in T_\rho T^*X 
= E_{\delta,\rho}^+\oplus E_{\delta,\rho}^0\oplus 
E_{\delta,\rho}^- \oplus \hat{E}_{\rho}^0$, we have
\begin{equation}\label{eq:NormeEquiv}
c_1 |v|_{\rho} \leq  |v_+|_{\rho} + |v_0|_{\rho} + |v_-|_{\rho} + |\hat{v}_0|_{\rho} \leq c_2 |v|_{\rho}.
\end{equation}

Here, we identify $v_+$ with $(v_+,0,0,0)\in T_\rho T^*X$ and similarly 
for $v_0,v_-$ and $\hat{v}_0$. 

\red{For every $\rho\in \mathcal{E}_{\delta, (\lambda_1, \lambda_2)}$, 
we may choose a basis of \corM{unit} vectors of $T_\rho T^*X$, 
$$\boldsymbol{e}(\rho)=(e_{1,+}(\rho), ..., e_{d-1, +}(\rho), e_{0}(\rho), e_{1,-}(\rho), ..., e_{d-1, -}(\rho), \hat{e}_0(\rho))$$ adapted to the decomposition \eqref{eq:DecompoTangentSpace}, such that $|\det(\boldsymbol{e}(\rho))|$ is bounded from below independently of $h\in (0, h_0]$, $\omega$ and of $\rho\in  \mathcal{E}_{\delta, (\lambda_1, \lambda_2)}$. In such a basis,} we have
\begin{equation}\label{eq:Matrix}
d_\rho \Phi_\delta^{t} 
	= 
	\begin{pmatrix}
	M_{\rho,\delta,t} & 0 & 0 & r_{1;\delta,t}(\rho)\\
	0 & 1 & 0 & \lambda t +  r_{2;\delta,t}(\rho)\\
	0 & 0 & (M_{\rho,\delta,t}^{-1})^\dagger &  r_{3;\delta,t}(\rho) \\
	0 & 0 & 0 & 1 +  r_{4;\delta,t}(\rho)
	\end{pmatrix},
\end{equation}
where $M_{\rho,\delta,t}$ is a $(d-1) \times (d-1)$ matrix such that 
$\|M_{\rho,\delta,t} v\| \geq C A_{\red{0}}^t |v|_\rho$ for any $v \in \R^{d-1}$. 
Here, the constants $C>0$ and $A_{\red{0}}>1$ can be chosen independent 
of $\delta\in [0, \delta_0)$ and $\rho \in \mathcal{E}_{\delta, (\frac{1}{4},1)}$.
Furthermore, it follows from \eqref{eq:ContinuityDIrections} that 
for every $j\in \{1,...,4\}$ and every $T>0$, we have
\begin{equation*}
	\sup_{t\in[0,T]}
	\sup_{\rho\in  \mathcal{E}_{0, (\frac{1}{8}, 2)}} 
 	\|r_{j;\delta,t}(\rho)\| \underset{\delta \to 0}{\longrightarrow} 0.
\end{equation*}
Continuing, we define the weak stable and unstable directions by
\begin{equation*}
	E_{\delta,\rho}^{\pm, 0} 
	:= 
	E_{\delta,\rho}^\pm \oplus E_{\delta,\rho}^0.
\end{equation*}

\subsection{Dynamics on the universal cover}

Let $X$ be as in the previous section. Recall that its universal 
cover is denoted $\widetilde{\pi}:\widetilde{X}\to X$. 
We may lift the Hamiltonian $p$ \eqref{eq:Hamiltonian} to a Hamiltonian 
\begin{equation}\label{eq:liftedHamiltonian}
	\widetilde{p}(x,\xi;\delta) := \widetilde{p} = \widehat{\pi}^*p : T^*\widetilde{X} \longrightarrow \R, 
\end{equation}
and define corresponding lifted energy layers 
\begin{equation}\label{eq:EnergyLayerLift}
\widetilde{\mathcal{E}}_{\delta,\lambda} :=  \widehat{\pi} ^{-1}\left( \mathcal{E}_{\delta,\lambda} \right), 
~
\widetilde{\mathcal{E}}_{\delta, (\lambda_1, \lambda_2)} 
:= \widehat{\pi} ^{-1}\left(\mathcal{E}_{\delta, (\lambda_1, \lambda_2)} \right)
~
\subset T^*\widetilde{X}.
\end{equation}
Notice that this direct definition is equivalent to the definition of these energy layers similar 
to \eqref{eq:EnergyLayers} \emph{mutatis mutandis}.
The lifted Hamiltonian $\widetilde{p}$ induces a Hamiltonian vector field $H_{\widetilde{p}}$,  
defined by the pointwise relation 
\begin{equation}\label{eq:liftedHV}
	H_{\widetilde{p}}\,\lrcorner\, \widetilde{\sigma} = -d{\widetilde{p}},
\end{equation}
where $\widetilde{\sigma}$ is the canonical symplectic form on $T^*\widetilde{X}$. We denote 
by $\widetilde{\Phi}_\delta^t$ the Hamiltonian flow generated by $H_{\widetilde{p}}$. Notice that 
\begin{equation}\label{eq:ProjDyn}
	\wih{\pi} \circ \widetilde{\Phi}_\delta^t = \Phi_\delta^t \circ \widehat{\pi}.
\end{equation}
In particular, these dynamics enjoy all the local properties described in 
Section \ref{sec:Hyperbolicity}. 

\subsection{Bounds on the derivatives of the flow}\label{ssec:BoundsFlow}

\red{In the rest of this section,  we will prove properties of the dynamics up to times $\gd |\log h|$ for some $0 < \gd \ll 1$.  We will always consider such time scales (which are much smaller than the usual Ehrenfest time in Quantum Chaos) in the rest of the paper,  corresponding to the result of Theorems \ref{th:MartinEtMaximeSontDesSuperBeauxGosses} and \ref{th:MartinEtMaximeSontDesBeauxGosses}. This will ensure us that any quantity growing exponentially with $t$ can be made smaller than $h^{-\varepsilon}$ for any $\varepsilon>0$, up to taking $\gd$ small enough.}

%\red{In the sequel, we will always consider times $0\leq t \leq \gd |\log h|$ for some $\gd \ll 1$.  In particular, }
%

\red{Let %$\frac{1}{4}< \lambda< 1$, and let $\rho \in \mathcal{E}_{\delta,\lambda}$.
$\rho \in \mathcal{E}_{0, ( \frac{1}{4},1 )}$.
We would like to control the derivatives of $\Phi_\delta^t(\rho)$ 
with respect to $t$. 
To this end, we first recall that, as in 
\eqref{eq:NiveauxEnergieComparables},  for $\delta_0>0$ 
sufficiently small, we have that 
\begin{equation}\label{eq:GronN00}
	\mathcal{E}_{0,(\frac{1}{4},1)}\subset 
	\mathcal{E}_{\delta,( \frac{1}{8}, 2)}\subset 
	\mathcal{E}_{0,( \frac{1}{16},  4)}.
\end{equation}
The Hamilton vector field $H_p$ \eqref{eq:HamiltonianVF} is tangent to 
$\mathcal{E}_{\delta,\lambda}$, $\lambda>0$, so the induced flow 
$\Phi^t_\delta$ preserves $\mathcal{E}_{\delta,\lambda}$, and 
we deduce that 
\begin{equation}\label{eq:GronN11}
	\Phi^t_\delta(\mathcal{E}_{0,(\frac{1}{4},1 )}) \subset 
\mathcal{E}_{0,( \frac{1}{16},  4)}, \text{ for all } t\in\R. 
\end{equation}}

\red{We will thus
%To this end,  recall that $\mathcal{E}_{\delta, (\frac{1}{4},1)}\subset \mathcal{E}_{0 ,(\frac{1}{8},2)}$ provided $\delta$ is small enough, and 
consider a covering of $\mathcal{E}_{0, (\frac{1}{16},4)}$ by finitely 
many open sets $\mathcal{U}_j$, each of them endowed with a 
local chart $\red{\hat{\kappa}}_j : \mathcal{U}_j \longrightarrow \R^{2d}$, 
which is a diffeomorphism onto its image. }
\par 
\red{%Let $\rho\in \mathcal{E}_\lambda$. 
For every 
$t\in \R$, let us denote by $J_{\red{\delta}}(t)$ the set of indices $j$ 
such that $\Phi_\delta^t(\rho) \in \mathcal{U}_j$. Hence, 
for 
%every  $t'$ in a neighbourhood of $t$ and 
every $j_0, j_1\in J_{\red{\delta}}(t)$, the map 
$\boldsymbol{\Phi}_{\delta,j}^t :=  \red{\hat{\kappa}}_{j_1} \circ \Phi_\delta^t\circ 
\red{\hat{\kappa}}_{j_0}^{-1}$ is a smooth map from 
$\R^{2d} \supset \kappa_{j_0} (\mathcal{U}_{j_0})$ to $ 
\red{\hat{\kappa}}_{j_1} (\mathcal{U}_{j_1})\subset \R^{2d}$. We then define,  for every $L\geq 1$}
\begin{equation*}
\red{	g_{\rho,L,\delta}(t) 
	:= 
	\max_{j_0, j_1 \in J_{\delta}(t)} \max_{\gamma \in \N^{2d} ; 1\leq |\gamma|\leq L}  
	\left| 
		\frac{\partial^{|\gamma|}}{\partial \boldsymbol{\rho}^\gamma} 
		\boldsymbol{\Phi}_{\delta,j}^{t}(\boldsymbol{\rho}) 
	\right|, ~~~~ \boldsymbol{\rho} = \hat{\kappa}_{j_0}(\mathcal{U}_j).}
\end{equation*}
computed using the Euclidean structure on  
$\R^{2d}\supset \red{\hat{\kappa}}_{j} (\mathcal{U}_{j})$. 

\red{The fact that the derivatives of the flow can grow at most exponentially with time is a common fact in dynamical systems. However, here, we must describe precisely how these estimates blow up as $h\to 0$, since the Hamilton vector field generating the dynamics has derivatives which blow up as $h\to 0$: this is the content of the following lemma.}

\begin{lem}\label{lem:DerivFlot}
For any $\varepsilon>0$, any $\lambda>0$ and any $L\in \N$, 
there exists $\gd>0$ \MI{and} $C_L>0$ such that the following holds: For all 
$h\in (0,1]$, all $t\in \R$ with $|t|\leq \gd |\log h|$ and all 
$\red{\rho \in \mathcal{E}_{0 ,(\frac{1}{4},1)}}$,
%$\rho\in \mathcal{E}_{0, \lambda}$,
%
\begin{equation*}
g_{\rho, L,\red{\delta}}(t)
\leq 
C_Lh^{-\varepsilon} \left( 1+ h^{-(\beta+ \varepsilon) (L+1)} \delta \right).
\end{equation*}
\end{lem}
\red{Note that, thanks to (\ref{eq:CondBeta}),  $g_{\rho, L,\red{\delta}}(t)$ will be bounded \MI{by $h^{-\varepsilon}$} for $L=1$, but it may blow up \MI{more rapidly} as $h\to 0$ for $L\geq 2$.}

\begin{proof}
Let us denote by $\bH_{h,j}$ the Hamiltonian vector field generating 
the dynamics, written in coordinate chart $\kappa_j$. By (\ref{eq:Hamiltonian}), 
we have $\|\bH_{h,j}\|_{C^L} \leq C(L) (1+ \delta h^{-\beta(L+1)})$.

For any multi-index $\gamma\in \N^d$ with $|\gamma|\leq L$, we 
may write
\begin{equation}\label{eq:RecG}
\begin{aligned}
\frac{\partial}{\partial t} 
	\left| 
		\frac{\partial^{|\gamma|}}{\partial \boldsymbol{\rho}^\gamma} 
		\boldsymbol{\Phi}_{\delta,j}^{t}(\boldsymbol{\rho})
	\right|^2 
&= 2 \left\langle \frac{\partial}{\partial t} 
	\frac{\partial^{|\gamma|}}{\partial \boldsymbol{\rho}^\gamma} 
	\boldsymbol{\Phi}_{\delta,j}^{t}(\boldsymbol{\rho}), 
	\frac{\partial^{|\gamma|}}{\partial \boldsymbol{\rho}^\gamma} 
	\boldsymbol{\Phi}_{\delta,j}^{t}(\boldsymbol{\rho})  
	\right\rangle\\
&=2 \left| \left\langle \frac{\partial^{|\gamma|}}{\partial \boldsymbol{\rho}^\gamma} 
	\left[ 
		\boldsymbol{H}_{h,j} 
		\left( 
			\boldsymbol{\Phi}_{\delta,j}^{t}(\boldsymbol{\rho})  
		\right) 
	\right], 
	\frac{\partial^{|\gamma|}}{\partial \boldsymbol{\rho}^\gamma} 
	\boldsymbol{\Phi}_{\delta,j}^{t}(\boldsymbol{\rho})  \right \rangle
	\right|\\
	&\leq 2 g_{\rho,L,\red{\delta}}(t)  \left| \frac{\partial^{|\gamma|}}{\partial \boldsymbol{\rho}^\gamma} \left[
		\boldsymbol{H}_{h,j}
		\left( 
			\boldsymbol{\Phi}_{\delta,j}^{t}(\boldsymbol{\rho})  
		\right)  \right]
\right|. 
\end{aligned}
\end{equation}
Note that, in the second equality, $\boldsymbol{H}_{h,j} 
		\left( 
			\boldsymbol{\Phi}_{\delta,j}^{t}(\boldsymbol{\rho})  
		\right)$ 
is a vector field evaluated at $\boldsymbol{\Phi}_{\delta,j}^{t}(\boldsymbol{\rho})$; 
it is thus a vector, whose scalar product we take with the vector 
$\frac{\partial^{|\gamma|}}{\partial \boldsymbol{\rho}^\gamma} 
\boldsymbol{\Phi}_{\delta,j}^{t}(\boldsymbol{\rho})$. It should 
not be though of as a derivative acting on 
$\frac{\partial^{|\gamma|}}{\partial \boldsymbol{\rho}^\gamma} 
\boldsymbol{\Phi}_{\delta,j}^{t}(\boldsymbol{\rho})$.
\par
First of all, we apply (\ref{eq:RecG}) with $L=1$, integrate 
over time and take the maximum over $j$ to obtain
\begin{equation*}
	 g^{\red{2}}_{\rho,1,\red{\delta}}(t)
\leq C+  C(1+ \delta h^{-2\beta})  \int_{0}^t g^{\red{2}}_{\rho,1,\red{\delta}}(s) ds.
\end{equation*}
In particular, thanks to (\ref{eq:CondBeta}), there exists $h_0>0$ and 
$c_0>0$ such that, for all $h\in (0, h_0]$,
\begin{equation*}
	 g^{\red{2}}_{\rho,1,\red{\delta}}(t)
\leq C+ c_0  \int_{0}^t g^{\red{2}}_{\rho,1,\red{\delta}}(s) ds.
\end{equation*}
and hence, thanks to Grönwall's lemma
\begin{equation*}
	g^{\red{2}}_{\rho,1,\red{\delta}}(t)\leq C  \e^{c_0 t}.
\end{equation*}
for some $C>0$, provided (\ref{eq:CondBeta}) is satisfied. In particular, there exists 
$\gd>0$, $C_1>0$ such that for all $t\in [0, \gd |\log h|]$, we have 
$|g_{\rho,1,\red{\delta}}(t)|\leq C_1 h^{-\varepsilon}$.
\\
\par 
Next,  we use equation (\ref{eq:RecG}) to estimate recursively 
$g_{\rho,L,\red{\delta}}(t)$. More precisely, we will prove inductively 
that for all $L\in \N$, and all $\varepsilon>0$, there exists $\gd_L>0$ such that 
for all $t\in [0, \gd_{L} |\log h|]$, we have

\begin{equation}\label{eq:CondRecG2}
g_{\rho,L,\red{\delta}}(t)
\leq 
 C_L h^{-\varepsilon} \left( 1+ h^{-(\beta+ \varepsilon)(L +1)} \delta \right),
\end{equation}
and,  \red{for all $\gamma\in \N^{2d}$ with $|\gamma| \leq L$,}
\begin{equation}\label{eq:CondRecX}
 \left| \frac{\partial^{|\gamma|}}{\partial \boldsymbol{\rho}^\gamma} 
	\left[ 
		\boldsymbol{H}_{h,j} 
		\left( 
			\boldsymbol{\Phi}_{\delta,j}^{t}(\boldsymbol{\rho})  
		\right) 
	\right]\right| \leq C'_L g_{\rho,L,\red{\delta}}(t) 
		 + O(\MI{1+}\delta h^{-(\beta+ \varepsilon) (L+1)}).
\end{equation}
We have already proved (\ref{eq:CondRecG2}) for $L=1$. 

Now,  let us suppose 
that (\ref{eq:CondRecX}) holds at rank $L$. Using (\ref{eq:RecG}), 
we get, for any $\gamma \in \N^d$ with $|\gamma|\leq L$
\begin{align*}
\left| 
		\frac{\partial^{|\gamma|}}{\partial \boldsymbol{\rho}^\gamma} 
		\boldsymbol{\Phi}_{\delta,j}^{t}(\boldsymbol{\rho})
	\right|^2 &\leq C + 2\int_0^t g_{\rho,L,\red{\delta}}(s)  \left| \frac{\partial^{|\gamma|}}{\partial \boldsymbol{\rho}^\gamma} \left[
		\boldsymbol{H}_{h,j}
		\left( 
			\boldsymbol{\Phi}_{\delta,j}^{t}(\boldsymbol{\rho})  
		\right)  \right]
\right|  ds\\
&\leq C + 2  \int_0^t  \left[C'_L g_{\rho,L,\red{\delta}}^2(s) +  g_{\rho,L,\red{\delta}}(s)(1+ O(\delta h^{-(\beta+ \varepsilon) 
(L+1)}))\right] ds\\
&\leq C +2 \left(C'_L + \frac{1}{2}\right)\int_0^t g_{\rho,L,\red{\delta}}^2(s) ds 
+ t (1+ O(\delta h^{-(\beta+ \varepsilon) (L+1)}))^2.
\end{align*}

Taking the maximum over $j$ and $\gamma$ and using Grönwall's lemma, we get
\begin{equation}\label{eq:CondRecG}
g^{\red{2}}_{\rho,L,\red{\delta}}(t)
\leq 
C_L \e^{\left(C'_L+ \frac{1}{2}\right) t} \left(1+ \MI{t}\delta h^{-(\beta+ \varepsilon) (L\MI{+1})}\right)^{\MI{2}} ,
\end{equation}
from which \eqref{eq:CondRecG2} follows.

Hence, (\ref{eq:CondRecG2}) holds at rank $L$ provided 
(\ref{eq:CondRecX}) holds at rank $L$. All we have to show is that 
(\ref{eq:CondRecG2}) up to rank $L$ implies (\ref{eq:CondRecX}) at rank 
$L+1$.  To do this,  \red{we take $\gamma\in \N^{2d}$ with $|\gamma| = L+1$.}

\red{By using the multidimensionnal Faà di Bruno formula (see \cite{Faa}),  or by an induction using the chain rule $L+1$ times, we see that the derivative $\frac{\partial^{L+1}}{\partial \boldsymbol{\rho}^\gamma} 
	\left[ 
		\boldsymbol{H}_{h,j} 
		\left( 
			\boldsymbol{\Phi}_{\delta,j}^{t}(\boldsymbol{\rho})  
		\right) 
	\right]$
	can be written as a sum of two kinds of terms:
	\begin{itemize}
	\item[$\bullet$] One term is $d_{\boldsymbol{\Phi}_{\delta,j}^{t}(\boldsymbol{\rho})}
\boldsymbol{H}_{h,j}  \left( \frac{\partial^{|\gamma|}}{\partial 
\boldsymbol{\rho}^\gamma} \boldsymbol{\Phi}_{\delta,j}^{t}(\boldsymbol{\rho}) 
\right)$, and thus involves derivatives of $\boldsymbol{\Phi}_{\delta,j}^{t}(\boldsymbol{\rho})$ of order $L+1$.  Since $d_{\boldsymbol{\Phi}_{\delta,j}^{t}(\boldsymbol{\rho})}
\boldsymbol{H}_{h,j}$ is bounded by $c_0$,  this gives the 
$c_0 g_{\rho,L+1,\delta}(t)$ contribution in (\ref{eq:CondRecX})
\item[$\bullet$] All the other terms contain products of a smaller number of derivatives of 
$\boldsymbol{\Phi}_{\delta,j}^{t}(\boldsymbol{\rho})$ at a power at most $L$, 
times derivatives of $\boldsymbol{H}_{h,j}$ of order at most $L+1$.  Such terms may thus be written as
\begin{equation}\label{eq:Faa}
C \left(\frac{\partial^{|\gamma'|} }{\partial \boldsymbol{\rho}^{\gamma'}}\MI{\Big{|}_{\Phi^t_{\delta,j}(\boldsymbol{\rho})}} \boldsymbol{H}_{h,j,\ell}\right)  \times  \prod_{k=1}^{2d}  \prod_{\vartheta\in \N^{2d} ; |\vartheta| \leq L} \left(\frac{\partial^{|\vartheta|}}{\partial 
\boldsymbol{\rho}^{\vartheta}} \boldsymbol{\Phi}_{\delta,j, k}^{t}(\boldsymbol{\rho})\right)^{n_{k,\vartheta}} ,
\end{equation}
\corM{for some $n_{k,\theta}\in \N$,} where 
$\boldsymbol{\Phi}_{\delta,j, k}^{t}$ denotes the $k$-th component of 
$\boldsymbol{\Phi}_{\delta,j}^{t}$,  $\boldsymbol{H}_{h,j,\ell}$ denotes 
the $\ell$-th component of $\boldsymbol{H}_{h,j}$. 
\end{itemize}}

\red{Considering a factor of the form (\ref{eq:Faa}), we recall \corM{from \cite{Faa}} that we have}
\MI{\begin{equation*}
\begin{aligned}
L+1 &= \sum_{k=1}^{2d} \sum_{\vartheta\in \N^{2d} ; 1\leq |\vartheta| \leq L} |\vartheta| n_{k,\vartheta}\\
|\gamma'| &= \sum_{k=1}^{2d} \sum_{\vartheta\in \N^{2d} ; 1\leq |\vartheta| \leq L} n_{k,\vartheta}
\end{aligned}
\end{equation*}}
% set $N:= \sum_{k=1}^{2d} \sum_{\vartheta\in \N^{2d} ; 1\leq |\vartheta| \leq L} (|\vartheta|\MI{-1}) n_{k,\vartheta}$, and we note that we have
%\begin{equation}\label{eq:NumberDerivatives}
%|\gamma'| + N = L+1.
%\end{equation}}

%\red{Equation (\ref{eq:NumberDerivatives}) can be seen using the Faà di Bruno formula,  or directly by induction: when taking a derivative of (\ref{eq:Faa}), if the derivative acts on $\frac{\partial^{|\gamma'|} \boldsymbol{H}_{h,j,\ell}}{\partial \boldsymbol{\rho}^{\gamma'}}$, it will increase $|\gamma'|$ by one, while if the derivative acts on the factor $ \prod_{k=1}^{2d}  \prod_{\vartheta\in \N^{2d} ; |\vartheta| \leq L} \left(\frac{\partial^{|\vartheta|}}{\partial 
%\boldsymbol{\rho}^{\vartheta}} \boldsymbol{\Phi}_{\delta,j, k}^{t}(\boldsymbol{\rho})\right)^{n_{k,\vartheta}}$, it will increase $N$ by one.}

\MI{Now, since we assumed that (\ref{eq:CondRecG2}) holds up to rank $L$, we see that a factor of the form (\ref{eq:Faa}) is bounded as
\begin{equation*}
\begin{aligned}
&C \MI{\left( 1+ \delta h^{-\beta |\gamma'|} \right)} \times \prod_{k=1}^{2d} \prod_{\vartheta\in \N^{2d} ; |\vartheta| \leq L}   \left(g_{\rho,  |\vartheta|, \delta}(t)\right)^{n_{k,\vartheta}}\\ 
&\leq C' \MI{\left( 1+ \delta h^{-\beta |\gamma'|} \right)} 
  \prod_{k=1}^{2d} \prod_{\vartheta\in \N^{2d} ; |\vartheta| \leq L} \left( h^{-\varepsilon} \left(1+ \delta h^{-(\beta+ \varepsilon)( |\vartheta|+1)} \right)\right)^{n_{k,\vartheta}}\\
&\leq C''   h^{- \varepsilon |\gamma'|}   \left( 1+ \delta h^{-\beta |\gamma'|} \right) \left(1+ \delta (\delta h^{-\beta -\varepsilon})^{|\gamma'|-1}  \times  h^{-(\beta + \varepsilon) (L+2)} \right)  \\
&\leq C''   h^{- \varepsilon |\gamma'|}   \left( 1+ \delta h^{-\beta |\gamma'|} \right) \left(1+ \delta h^{\beta(|\gamma'|-1)}  \times  h^{-(\beta + \varepsilon) (L+2)} \right) ~~~~\text{ thanks to (\ref{eq:CondBetaDelta})}\\
&= C''   h^{- \varepsilon |\gamma'|} \left[1+ \delta h^{-\beta |\gamma'|} + \delta h^{\beta (|\gamma'|-1)} h^{-\beta-\varepsilon (L+2)} + \delta^2 h^{-\beta} h^{-\beta-\varepsilon (L+2)} \right]\\
&\leq  C'''   h^{- \varepsilon |\gamma'|} \left[ 1+ \delta h^{-\beta-\varepsilon (L+2)}\right],
\end{aligned}
\end{equation*}
using (\ref{eq:CondBetaDelta}) again. The result follows.}
\end{proof}

\begin{rem}\label{rem:DerivCover}
\red{The sets $\mathcal{U}_j$ in Lemma \ref{lem:DerivFlot} may be lifted to open sets $\wit{\mathcal{U}}_{j, \iota}$ covering $\wit{\mathcal{E}}_{0, \left(\frac{1}{16}, 4 \right)}$, }
equipped 
with charts $\kappa_{j,\iota}$. We may then define 
$\boldsymbol{\Phi}_{\delta,j, \iota}^t$, and $g_{\rho, L,h}(t)$ 
in a similar fashion. Using \eqref{eq:ProjDyn}, we see that 
we also have
\begin{equation*}
g_{\rho, L,h}(t)
\leq 
C_Lh^{-\varepsilon} \left( 1+ h^{-(\beta+ \varepsilon) (L+1)} \delta \right).
\end{equation*}
\end{rem}

%In Lemma \ref{lem:DerivFlot}, we may suppose that each 
%$\mathcal{U}_j$ is a geodesic ball equipped with geodesic 
%coordinates. The open sets $\mathcal{U}_j$ may be lifted to 
%open sets $\wit{\mathcal{U}}_{j, \iota}\subset 
%\wit{\mathcal{E}}_{0, \left(\frac{1}{4}, 4 \right)}$, equipped 
%with charts $\kappa_{j,\iota}$. We may then define 
%$\boldsymbol{\Phi}_{\delta,j, \iota}^t$, and $g_{\rho, L,h}(t)$ 
%in a similar fashion. Using \eqref{eq:ProjDyn}, we see that 
%we also have
%
%\begin{equation*}
%g_{\rho, L,h}(t)
%\leq 
%C_Lh^{-\varepsilon} \left( 1+ h^{-(\beta+ \varepsilon) (L+1)} \delta \right).
%\end{equation*}
%\end{rem}
%
\subsection{Comparing the dynamics}
The following lemma allows to compare perturbed and 
unperturbed trajectories. It also allows to consider 
trajectories with different values of the random 
parameter $\omega$.  If $\omega^1 = (\omega^1_j)_{j\in J_h}$ 
and $\omega^2 = (\omega^2_j)_{j\in J_h}$, let us write $\| \omega^1 - \omega^{\red{2}}\| 
:= \max_{j\in J_h} |\omega^1_j  - \omega^2_j|$.

\begin{lem}\label{Lem:Gron}
%Let $0<\lambda_1< \lambda_2 <\infty$.
Let $\delta_0, h_0>0$ be as in the beginning of Section 
\ref{Sec:HamiltonFlows} and sufficiently small. 
There exists $C_0$ such that for all $h\in ]0,h_0]$, all 
$0 \leq \delta'\leq \delta \leq \delta_0$, all 
$\rho, \rho'\in \mathcal{E}_{0, \red{(\frac{1}{4}, 1)}}$ and all $t\in \R$, 
we have
\begin{equation}\label{eq:Gronwall2}
\begin{split}
 \mathrm{dist}_{T^*X}(\Phi_{\delta, \omega = \omega^1}^t(\rho),\Phi_{\delta', \omega = \omega^2}^t(\rho'))
 	\leq &C_0 \e^{C_0|t|}  \mathrm{dist}_{T^*X}(\rho,\rho') \\
	& +  C_0 \left(\delta \|\omega^1 - \omega^2\| + (\delta - \delta')\right)   h^{-\beta} \left( \e^{C_0|t|} -1\right).
\end{split}
\end{equation}

In particular, 
\begin{equation}\label{eq:Gronwall}
 \mathrm{dist}_{T^*X}(\Phi_{\delta, \omega = \omega^1}^t(\rho),\Phi_{\delta, \omega=\omega^2}^t(\rho'))
 	\leq C_0 \e^{C_0|t|}  \mathrm{dist}_{T^*X}(\rho,\rho') 
	 +  C_0 \delta h^{-\beta} \left( \e^{C_0|t|} -1\right).
\end{equation}
The same statement holds on the universal cover $T^*\wit{X}$.
\end{lem}
\begin{proof}
\red{We shall give the proof on the base manifold $X$, but the proof on the universal cover $\wit{X}$ is exactly the same.
Let us write $\rho_t := \Phi_{\delta, \omega = \omega^1}^t(\rho)$,  $\rho'_t:=\Phi_{\delta', \omega = \omega^2}^t(\rho')$, and let us introduce the point $\rho''_t := \Phi_{\delta, \omega = \omega^1}^t(\rho')$.  By the triangle inequality, we have
\begin{equation}\label{eq:Triangle}
\mathrm{dist}_{T^*X}(\rho_t, \rho'_t) \leq \mathrm{dist}_{T^*X}(\rho_t, \rho''_t)+ \mathrm{dist}_{T^*X}(\rho_t'', \rho'_t),
\end{equation}
and we will now bound each term in the right-hand side of (\ref{eq:Triangle}).  }

\red{We shall denote by $H_\delta$ and $H_{\delta'}$ the Hamilton vector fields respectively generating the dynamics $\Phi_{\delta, \omega = \omega^1}$ and $\Phi_{\delta', \omega = \omega^2}^t$.  }

\red{\textbf{Bound on $\mathrm{dist}_{T^*X}(\rho_t, \rho''_t)$.}
We let $\gamma : [0,1] \longrightarrow T^*X$ be a smooth curve minimizing the distance between $\rho$ and $\rho'$, parametrized with constant speed. We set $ \gamma_t(s) := \Phi^t_{\delta, \omega = \omega^1}$.  \MI{Recalling that the cotangent bundle $T^*X$ is equipped with an (arbitrary) metric $g_0$}, we have
\begin{align*}
\mathrm{dist}_{T^*X} (\rho_t, \rho''_t) \leq \int_0^1 \left\| \frac{\partial \gamma_t(s)}{\partial s} \right\| ds
\leq \sqrt{\int_0^1 \left\| \frac{\partial \gamma_t(s)}{\partial s} \right\|^2 ds} =: \sqrt{E(t)},
\end{align*}
with equality when $t=0$.
Now, we have
\begin{align*}
|E'(t)| &= 2\left|\int_0^1 \left\langle \MI{\nabla_{\partial_t \gamma_t(s)}}  \frac{\partial \gamma_t(s)}{\partial s},  \frac{\partial \gamma_t(s)}{\partial s}\right\rangle\right| ds\\
&=2 \left| \int_0^1 \left\langle\MI{\nabla_{\partial_s \gamma_t(s)}} \frac{\partial \gamma_t(s)}{\partial t},  \frac{\partial \gamma_t(s)}{\partial s}\right\rangle ds \right|\\
&\leq 2 \int_0^1 \left| \left\langle  \nabla_{\partial_s\gamma_t(s)} H_\delta,   \frac{\partial \gamma_t(s)}{\partial s}\right\rangle \right| ds 
\leq C E(t),
\end{align*}
since $\|\nabla H_\delta\| = O(1+ \delta h^{-2\beta}) = O(1)$.  By Grönwall's inequality, we thus have 
\begin{equation*}
E(t) \leq C e^{C t} E(0),
%\mathrm{dist}_{T^*X}(\rho_t'', \rho'_t) \leq C_0 \e^{C_0|t|}  \mathrm{dist}_{T^*X}(\rho,\rho')
\end{equation*}
from which we obtain the first term in (\ref{eq:Gronwall2}) by taking square roots.}
\\
\\
\red{\textbf{Bound on $\mathrm{dist}_{T^*X}(\rho_t'', \rho'_t)$.}}
\red{For every $s\in [0,1]$, we define a vector field $H^s := (1-s) H_\delta + s H_{\delta'}$,  and we denote by $\Phi^{t; s}$ the associated flow.  Note that we have
\begin{equation}\label{eq:BorneDiffVectorFields}
\MI{\sup_{\rho\in \mathcal{E}_{0,(\frac{1}{16}, 4))}}} \|H_\delta(\rho) - H_{\MI{\delta'}(\rho)}\| \leq  C  \left(\delta \|\omega^1 - \omega^2\| + (\delta - \delta')\right)   \|q_\omega\|_{C^1}
	\leq C \left(\delta \|\omega^1 - \omega^2\| + (\delta - \delta')\right)  h^{-\beta}.
		\end{equation}}
\red{We then set $\Gamma_t(s) := \Phi^{t; s} (\rho')$, which forms a smooth curve joining $\rho''_t$ to $\rho'_t$, so that 
\begin{align*}
\mathrm{dist}_{T^*X}(\rho_t'', \rho'_t) \leq \int_0^1 \left\| \frac{\partial \Gamma_t(s)}{\partial s} \right\| ds
\leq \sqrt{\int_0^1 \left\| \frac{\partial \Gamma_t(s)}{\partial s} \right\|^2 ds} =: \sqrt{F(t)},
\end{align*}
with equality when $t=0$.  We have 
\begin{align*}
|F'(t)| &= 2\left|\int_0^1 \left\langle \MI{\nabla_{\partial_t \Gamma_t(s)}}\frac{\partial \Gamma_t(s)}{\partial s},  \frac{\partial \Gamma_t(s)}{\partial s}\right\rangle ds \right|\\
&= 2 \left| \int_0^1 \left\langle \MI{\nabla_{\partial_s \Gamma_t(s)}} H^s(\Gamma_t(s)),  \frac{\partial \Gamma_t(s)}{\partial s}\right\rangle ds \right|\\
&\leq2  \int_0^1 \left| \left\langle  (\nabla_{\partial s \Gamma_t(s)} H^s)\MI{(\Gamma_t(s))},   \frac{\partial \Gamma_t(s)}{\partial s}\right\rangle \right| ds  +2  \int_0^1 \left| \left\langle  (H_{\delta'} - H_\delta)(\Gamma_t(s)),   \frac{\partial \Gamma_t(s)}{\partial s}\right\rangle \right| ds
\end{align*}}

\red{To bound the second term,  we use (\ref{eq:BorneDiffVectorFields}), along with (\ref{eq:GronN11}), to bound the left-hand term of the scalar product.  By the Cauchy-Schwarz inequality, we deduce that}
\MI{\begin{align*}
|(\sqrt{F})'(t)| = (2F(t))^{-1/2} F'(t)
&\leq (2F(t))^{-1/2} \left( C F(t) + C \left(\delta \|\omega^1 - \omega^2\| + (\delta - \delta')\right)  h^{-\beta} \sqrt{F(t)}\right),%~~~~\text{ by (\ref{eq:BorneDiffVectorFields}) and Cauchy-Schwarz}
\\
&\leq C \sqrt{F(t)} + C\left[ \left(\delta \|\omega^1 - \omega^2\| + (\delta - \delta')\right)  h^{-\beta}\right].
\end{align*}}

\red{We may then use Grönwall's inequality to deduce that
\begin{equation*}
\sqrt{F(t)} \leq C'' \left[ \left(\delta \|\omega^1 - \omega^2\| + (\delta - \delta')\right)  h^{-\beta} \right] \left( e^{C''t} - 1\right),
\end{equation*}
and the result follows.}
\end{proof}
We conclude this paragraph by noting that Lemma \ref{Lem:Gron} implies the classical fact that there exists 
$C>0$ such that
\begin{equation}\label{eq:ExpRate}
\forall \rho, \rho' \in \mathcal{E}_{0, (\frac{1}{8}, 2)},  
\forall t\in \R, \mathrm{dist}_{T^*X} 
\left( 
	\Phi^t(\rho), 
	\Phi^t(\rho') 
\right) 
\leq C \e^{C |t|}
\mathrm{dist}_{T^*X}(\rho, \rho').
\end{equation}

\subsection{Dynamics of Lagrangian manifolds}\label{SecDynLag}

\red{The main aim of the rest of this section will be to show that, if $\Lambda\subset S^*X$ is an $\eta$-unstable Lagrangian manifold (for $\eta$ small enough, in the sense of Definition \ref{def:LagInstable}), then its evolution in the future projects smoothly on $\wit{X}$. This is essentially the content of Proposition \ref{Prop:OnTheCover2} below. Such a statement is standard when considering the evolution by the geodesic flow. However, here, we consider the evolution by the perturbed flow $\Phi^t_\delta$ , which makes the statement and its proof more subtle, and forces us to consider only times $t\leq \gd |\log \delta|$ for some small $\gd>0$.}

\subsubsection{\textnormal{\textbf{Lagrangian manifolds and instability}}}
\begin{definition}
Let $\eta>0$,  $\delta\geq 0$, and let 
$\Lambda\subset \mathcal{E}_{\delta, (\frac{1}{4}, 1)}$ be a submanifold. 
We shall say that $\Lambda$ is \emph{$(\delta; \eta, \kappa )$-unstable} if, 
for any $\rho\in \Lambda$ and any $v = (v_+, v_0, v_-, \hat{v}_0)
\in T_\rho \Lambda \subset T_\rho T^*X = E_{\delta,\rho}^+\oplus 
E_{\delta,\rho}^0\oplus E_{\delta,\rho}^- \oplus \hat{E}_{\delta,\rho}^0$, we 
have that 
\begin{equation*}
|v_-|_\rho \leq \eta |v|_\rho.
\end{equation*}
\begin{equation*}
|\hat{v}_0|_\rho \leq \kappa |v|_\rho.
\end{equation*}
\end{definition}

\begin{lem}\label{Lem:EtaInstable}
There exists $\gamma\red{>1}$,  $T,\kappa_0,  \eta_0, \delta_0>0$ such that the 
following holds. 
%For every $T\geq 0$,  there exists $\kappa(T)>0$, $\delta(T)$ such that, 
For all $0\leq \delta < \delta_0$,  $0\leq \kappa \red{\leq} \kappa_0$ and all 
$0<\eta\leq \eta_0$, if $\Lambda\subset \mathcal{E}_{\delta, (\frac{1}{4},1)}$ 
is $(\delta; \eta, \kappa)$-unstable, then for all $t\in [0, T]$, 
$\Phi^t_{\delta}(\Lambda)$ is $(\delta; \gamma \eta, \gamma\kappa)$ 
unstable. Furthermore, $\Phi^T_{\delta}(\Lambda)$ is $(\delta; \eta, 
\gamma\kappa)$ unstable. 
\end{lem}
We postpone the proof of this lemma until after the following Corollary.
\begin{corollary}\label{Cor:EtaInstable}
There exists  $\gamma\red{>1}$,  $T,\kappa_0, \eta_{0}, \delta_0>0$ such that the 
following holds for all $\eta< \eta_0$,  $\kappa_1< \kappa_0$, and $\delta \in [0, \delta_0)$. 
If $\Lambda \subset \mathcal{E}_{\delta, (\frac{1}{4},1)}$ is 
$(\delta, \eta, \kappa_1)$-unstable, then $\Phi_\delta^t(\Lambda)$ is 
$(\delta, \gamma \eta, \kappa_0)$-unstable for all $0 \leq t \leq 
T \left(1+ \ln \left( \frac{\kappa_0}{\kappa_1}\right) \left(\ln \gamma\right)^{\red{-1}}\right)$.
\par 
Furthermore, there exists $C, c>0$ such that, for all $\rho_0\in \Lambda$ 
and all $w\in T_{\rho_0} \Lambda$, if we write 
$d_{\rho_0} \Phi^t_\delta(w) = (\Phi_\delta^t(\rho_0), v_t)$, we have
\begin{equation}\label{eq:NonContract2}
	|v_t|_{\Phi_\delta^t(\rho_0)} \geq C c^t |w|_{\rho_0}.
\end{equation}

\end{corollary}
\begin{proof}[Proof of Corollary \ref{Cor:EtaInstable}]
\red{Let us write $n_{max} := \left\lfloor \ln \left( \frac{\kappa_0}{\kappa_1}\right) \left(\ln \gamma\right)^{\red{-1}} \right\rfloor  \in \N\cup\{0\}$, so that
$$\gamma^{n_{max}} \kappa_1 \leq \kappa_0.$$
We may therefore apply Lemma \ref{Lem:EtaInstable} $n_{max}+1$ times to obtain the first 
part of the result.}

%Let $n\in \N$ be such that $t\in (nT, (n+1)T)$. In particular, the assumption 
%on $t$ implies that
%%
%$$\gamma^n \kappa_1 < \kappa_0.$$
%%
%We may thus apply Lemma \ref{Lem:EtaInstable} $n$ times to obtain the first 
%part of the result.
%
\par 
%
%The first part of the result follows. 
Equation \eqref{eq:NonContract2} 
follows from iterations of equation \eqref{eq:noncontract} below.
\end{proof}
\begin{proof}[Proof of Lemma \ref{Lem:EtaInstable}]
Let $t\geq 0$, let $\rho\in \Phi_{\delta}^t(\Lambda)$, and let $v\in T_\rho \Phi_{\delta}^t(\Lambda)$. By definition, 
there exists $\rho_0\in \Lambda$ such that $\rho = \Phi_{\delta}^t(\rho_0)$, and $w\in T_{\rho_0}\Lambda$ 
such that $d_{\rho_0} \Phi^t_{\delta}(w)=v$. Let us write 
$\lambda = p(\rho, \delta) \in (\frac{1}{2}, 2) $. 

Let us write $w= (w_+, w_0, w_-, \hat{w}_0)\in E_{\delta,\rho_0}^+\oplus E_{\delta,\rho_0}^0\oplus E_{\delta,\rho_0}^- \oplus \hat{E}_{\delta,\rho_0}^0$. By assumption, we have $|w_-|_\rho \leq \eta |w|_\rho$ and $| \hat{w}_0|_\rho \leq \kappa |w|_\rho$. Now, thanks to (\ref{eq:Matrix}), we have $v= (v_+, v_0, v_-, \hat{v}_0)\in E_{\delta,\rho}^+\oplus E_{\delta,\rho}^0\oplus E_{\delta,\rho}^- \oplus \hat{E}_{\delta,\rho}^0$, with
$$\begin{pmatrix}
v_+\\
 v_0 \\
 v_-\\
  \hat{v}_0
  \end{pmatrix} = \begin{pmatrix}
  M_{\rho, \delta,t} w_++ r_{1;\delta,t}(\rho) \hat{w}_0\\
   w_0 + \lambda t \hat{w}_0 + r_{2;\delta,t}(\rho) \hat{w}_0\\
     (M_{\rho,\delta,t}^{-1})^\dagger w_- + r_{3;\delta,t}(\rho) \hat{w}_0\\
      \hat{w}_0 + r_{4;\delta,t}(\rho) \hat{w}_0\end{pmatrix}.$$

Recall that, for any $T>0$, we may find $\delta(T)>0$ such that, for all $0\leq \delta < \delta(T)$, all $t\in [0,T]$ and all $j\in\{1,...,4\}$, we have
\begin{equation}\label{eq:CondHT}
\sup_{\rho\in  \mathcal{E}_{0, (\frac{1}{2}, 2)}} |r_{j;\delta,t}(\rho)|_\rho \leq \frac{1}{4}.
\end{equation}
      
In particular, 
\begin{equation}\label{eq:Sup2}
|\hat{v}_0|_\rho   \leq \frac{5}{4} |\hat{w}_0|_{\rho_0} \leq \frac{5 \kappa}{4} |w|_{\rho_0} .
\end{equation}
     
Furthermore, 
%if (\ref{eq:CondHT}) holds, we have

\begin{equation}\label{eq:Inf}
\begin{aligned}
|v|_\rho &\geq c_2^{-1} \left[ |v_+|_\rho + |v_0|_\rho \right]\\
% &= c_2^{-1}  \|(M_{\rho_0, t} w_+ + r_{1;h}(\rho) \hat{w}_0, w_0 + \lambda t \hat{w}_0 +  r_{2;h}(\rho) \hat{w}_0, 0, 0)\|\\ 
&\geq c_2^{-1} \left[ C A^t |w_+|_{\rho_0} +  |w_0|_{\rho_0} - (\lambda t+ \frac{1}{2}) |\hat{w}_0|_{\rho_0} \right]\\
&\geq c_2^{-1} \left[ C A^t |w_+|_{\rho_0} +  |w_0|_{\rho_0} - \kappa (\lambda t+ \frac{1}{2}) |w|_{\rho_0} \right].
\end{aligned}
\end{equation}

Recall that by assumption, we have $|w_-|_{\rho_0} + |\hat{w}_0|_{\rho_0} \leq \red{( \kappa + \eta )} |w|_{\rho_0}$, so that 
$|w_+|_{\rho_0} + |w_0|_{\rho_0} \geq  (c_1 - \red{(\kappa+  \eta)}) |w|_{\rho_0}$. 
Therefore, there exists $c_3 = c_3(C, c_1, c_2)>0$ such that, for all $T>0$,  there exists $\kappa(T)>0$ and $\eta_0>0$ such that, if $\kappa < \kappa(T)$, $\eta< \eta_0$ and $\delta< \delta(T)$ with $\delta(T)$ as in (\ref{eq:CondHT}), we have
\begin{equation}\label{eq:noncontract}
|v|_\rho \geq c_3 |w|_{\rho_0}.
\end{equation}

In particular, combined with (\ref{eq:Sup2}), this implies that
\begin{equation*}
|\hat{v}_0|_\rho   \leq \frac{5  \kappa}{4 \red{c_3}} |v|_{\rho}.
\end{equation*}

On the other hand, we have
\begin{align*}
| v_-|_{\rho} &\leq \left|(M_{\rho,\delta,t}^{-1})^\dagger w_- \right|_{\rho_0} \red{+ \left|r_{3;\delta,t}(\rho) \hat{w}_0 \right|_{\rho_0}} \\
&\leq  C  c_1^{-1} A^{-t} |w_-|_{\rho_0} +\red{ \frac{\kappa+\eta}{4} |w|_{\rho_0}}
\\& \leq    c_3 \left(C c_1^{-1} A^{-t} (\MI{\kappa+}\eta) + \red{\frac{\kappa+\eta}{4}} \right)  |v|_{\rho}.
\end{align*}

We thus take $T$ such that $C c_3  c_1^{-1} A^{-\red{T}}<\MI{1}$,  \red{$\kappa(T)$ and $\eta_0$ smaller to ensure that $c_3 \frac{\kappa + \eta}{4} < \frac{1}{2}$}, and
 $\delta_0 := \delta(T)$ such that (\ref{eq:CondHT}) holds, and $\kappa_0 =\kappa(T)$ such that (\ref{eq:noncontract}) holds. This gives us the result.
\end{proof}
\par 
Let us now rephrase the results of this section in terms of the Lagrangian 
states appearing in Theorem \ref{th:MartinEtMaximeSontDesBeauxGosses}.
\begin{prop}\label{Prop:LagInitiale}

\red{Let $\eta_1>0$.}
There exists $c_0>0$,  $\eta_0>0$ and $\delta_0>0$ such that the following 
holds for all $0< \eta < \eta_0$.
\par 
Let $\Lambda \subset S^*X$ be a monochromatic Lagrangian manifold which 
is $\eta$-unstable, as in section \ref{subsec:Lag}. Then for every 
$\delta < \delta_0$, for all $0\leq t < c_0 \red{|}\ln \delta\red{|}$, $\Phi_\delta^t(\Lambda)$ 
is $(\delta,  \red{\eta_1, \eta_1})$-unstable.
\end{prop}
\begin{proof}
Let $\gamma>1$ be as in Corollary \ref{Cor:EtaInstable}, and let $\eta_0 = \frac{\eta_1}{2 \gamma}$.

We claim that there exists $c_{1}$, $\delta_1>0$ such 
that, for all $0 \leq \delta < \delta_1$, $\Lambda$ is 
$(\delta,  \frac{\eta_1}{\gamma}, c_{1}\delta)$-unstable.  
This follows from the continuity 
relation (\ref{eq:ContinuityDIrections}), and from the fact that the map
$\delta \mapsto T_\rho \mathcal{E}_{\delta, \lambda}$ is smooth\red{, while  
$\wit{E}_\rho^0$ does not depend on $\delta$}. 
We may thus apply Corollary 
\ref{Cor:EtaInstable} to deduce the result. 
\end{proof}  
\subsubsection{\textnormal{\textbf{Unstability and projectability}}}
Let $Y$ be a $d$-dimensional Riemannian manifold (not necessarily compact). 
Recall that a Lagrangian submanifold is a submanifold $\Lambda\subset T^*Y$ 
of dimension $d$, such that the canonical symplectic form of $T^*Y$ vanishes 
on $T_\rho \Lambda$ for any $\rho \in \Lambda$, see for instance 
\cite[Chapter 1]{DiSj}. In what follows, we will focus on a special family 
of Lagrangian submanifolds, which can be written as graphs:
\begin{definition}
Let $\Lambda\subset T^*Y$ be a submanifold with $\dim \Lambda = \dim Y$ 
and let $\rho \in \Lambda$. We call $\Lambda$ \emph{projectable at a point} 
$\rho\in\Lambda$ if ${\pi_Y}|_\Lambda:\Lambda\to Y$ is a local diffeomorphism 
near $\rho$, in the sense that each $\rho$ has a neighbourhood in $\Lambda$ 
which is mapped diffeomorphically by $\pi$ onto a neighbourhood of $\pi(\rho)$.
\par
We call $\Lambda$ \emph{locally projectable} if ${\pi_Y}|_\Lambda:\Lambda\to Y$ 
is a local diffeomorphism in the sense that each point $\rho\in\Lambda$ has a 
neighbourhood in $\Lambda$ which is mapped diffeomorphically by $\pi$ onto a 
neighbourhood of $\pi(\rho)$.
\par
Similarly, we call $\Lambda$ \emph{projectable} if 
${\pi_Y }|_\Lambda:\Lambda\to \pi_Y(\Lambda)$ is a diffeomorphism. 
\end{definition}
%a
\begin{rem}\label{Rem:ProjTrans}
The Poincar\'e Lemma shows that a locally projectable submanifold 
$\Lambda \subset T^*Y$ with $\dim \Lambda = \dim Y$ is Lagrangian if and 
only if it is locally given by the graph of a gradient. In other words, 
$\Lambda$ is Lagrangian if and only if for each 
point $\rho \in \Lambda$, we can find a real-valued $C^\infty$ function $\phi(x)$ 
defined near $\pi(\rho)$, such that $\Lambda$ coincides near $\rho$ with the 
manifold $\{(x,d_x\phi); x\in \text{some neighbourhood of } \pi(\rho)\}$.
\par
Also note that by the inverse function theorem, $\Lambda$ is projectable at 
$\rho=(x,\xi) \in \Lambda$ if and only if the manifolds $\Lambda$ and 
$T^*_x Y$ are transverse at $\rho$. 
\end{rem}
\begin{rem}\label{rem:4.7}
Suppose that $\Lambda$ is a simply connected relatively compact locally projectable 
Lagrangian manifold. If ${\pi_Y}|_\Lambda$ is injective then it is a homeomorphism 
since $\Lambda$ is relatively compact. So $\pi_Y(\Lambda)$ is simply connected, and 
therefore $\Lambda$ is projectable since $H^1_{\mathrm{de Rham}}(\pi_Y(\Lambda))=0$. 
\end{rem}
Next, we show that unstable Lagrangian submanifolds are \red{locally} projectable. 

\begin{lem}\label{Lem:InstTrans}
Let $(X,g)$ be a compact Riemannian manifold of negative sectional curvature.   
There exists $\eta_0,\delta_0>0$ such that, if $\delta\in [0, \delta_0]$ and $\Lambda \subset \mathcal{E}_{\delta, (\frac{1}{4},1)}$ is a $(\delta, \eta_0, \eta_0)$-unstable Lagrangian submanifold, then $\Lambda$ is locally projectable.

More precisely, there exists $\alpha_0>0$ such that, if 
$\Lambda \subset \mathcal{E}_{\delta, (\frac{1}{4},1)}$ is a 
$(\delta, \eta_0, \eta_0)$-unstable Lagrangian submanifold, 
then for any $\rho=(x,\xi) \in \Lambda$, the manifolds 
$\Lambda$ and $T^*_xY$ make an angle $\geq \alpha_0$ at $\rho$ \red{(with respect to the measure $g_0$)}.
\end{lem}

\begin{proof}
It is well known that for any $\rho= (x,\xi)\in T^*X$, $\xi\neq 0$, any one 
distribution $E^{\bullet}_\rho$, $\bullet\in\{+,-,0\}$, and $T^*_x X$ are 
transverse submanifolds of $T^*X$, see for instance \cite[Lemma 4.6]{Ing}. 
It follows, that for any $v\in E^{\bullet}_\rho$ and any $w\in T_\rho T^*_x X$, 
we have $|\langle v, w \rangle_{\rho}| \neq  |v|_{\rho} |w|_{\rho}$. 
By compactness we may thus find a constant $c_0>0$ such that, for all 
$\rho \in \mathcal{E}_{\delta,(\frac{1}{4},1)}$, and for any 
$v\in E^{\bullet}_\rho$, $\bullet\in\{+,-,0\}$, and any $w\in T_\rho T^*_x X$, 
we have 
\begin{equation*}
	|\langle v, w \rangle_{\rho} | \leq (1-c_0) |v|_{\rho} |w|_{\rho}.
\end{equation*}
From the continuity \eqref{eq:ContinuityDIrections}, we deduce that, if $\delta_0$ is small enough, 
for all $\rho \in \mathcal{E}_{\delta,(\frac{1}{4},1)}$, and for any $v\in E^{+0}_{\delta,\rho}$ 
and any $w\in T_\rho T^*_x X$, we have 
\begin{equation*}
	|\langle v, w \rangle_{\rho} | \leq \left(1-\frac{c_0}{2}\right) |v|_{\rho} |w|_{\rho}. 
\end{equation*}
Let $\Lambda$ be $(\delta, \eta_0)$-unstable for some $\delta\in[0, \delta_0]$, let $\rho = (x,\xi) \in \Lambda$, 
let $v\in T_\rho \Lambda$ and $w\in T_\rho T^*_x X$ with $|v|_{\rho} = |w|_{\rho}=1$. By the 
instability assumption, if $\eta$ is small enough, there exists $v'\in E_{\delta,\rho}^{+0}$ with 
$|v'|_{\rho}=1$ such that $|v-v'|_{\rho} \leq \frac{c_0}{4}$. We thus have
\begin{equation}\label{eq:Transvers}
| \langle v, w \rangle_{\rho}| \leq | \langle v', w \rangle_{\rho} | + \frac{c_0}{4} \leq 1- \frac{c_0}{4}.
\end{equation}
In particular, we have $v\neq w$, so that $T_\rho \Lambda\cap T_\rho T^*_x X=\{0\}$. 
The second observation in Remark \ref{Rem:ProjTrans} then implies that $\Lambda$ 
is projectable at $\rho$. Since the above argument holds for any $\rho\in \Lambda$ 
we conclude the first part of the statement. The statement about the angle follows 
directly from (\ref{eq:Transvers}).
\end{proof}

\subsubsection{\textnormal{\textbf{\MI{Expansion} of Lagrangian manifolds in negative curvature}}}
Let us recall a standard result concerning manifolds of negative curvature 
(see \cite[Theorem 4.8.2]{Jost}).
\begin{lem}\label{lem:convex}
Let $(X,g)$ be a simply connected manifold of nonpositive sectional curvature, 
and let $\rho, \rho'\in T^*X$. Then the map 
$\R \ni t \mapsto \mathrm{dist}_X^2 (\Phi^t(\rho), \Phi^t(\rho'))$ is convex.
\end{lem}

In particular, we have the following result. \red{Recall that $r_I$ denotes the injectivity radius of the manifold $X$.}

\begin{corollary}\label{cor:Convex}
Let $(X,g)$ be a connected manifold of nonpositive sectional curvature, 
and let $\rho, \rho'\in T^*X$. Suppose that $I\subset \R$ is 
an interval such that for all $t\in I$, $\mathrm{dist}_X
(\Phi^t(\rho), \Phi^t(\rho')) < r_I$. Then the map $I \ni t \mapsto 
\mathrm{dist}_X^2 (\Phi^t(\rho), \Phi^t(\rho'))$ is convex.
\end{corollary}
%
%If $f$ is a  smooth positive convex function on an interval $I$ containing $0$, then we have $f(t) \geq t f'(0)$ for all $t\in I$.  Hence,  in the sequel, Corollary \ref{cor:Convex} will often be combined with the following lemma.
%
Thus,  as long as $\rho$ and $\rho'$ remain close to each other, 
the square of their distance (on the basis $X$) is a convex function. 
This convex function could be constant: this is the case when 
$\rho$ and $\rho'$ belong to the same geodesic. However, the following 
lemma guaranties that, if $\rho$, $\rho'$ belong to an unstable 
Lagrangian manifold and do not belong to the same (short) 
geodesic segment, then the distance between them is ultimately increasing.

\begin{lem}\label{Lem:DernierLemmeAvantLesVacances}
Let $(X,g)$ be a compact manifold of negative curvature, \red{and let $D>0$}.
There exist $\varepsilon >0$, $\eta_0>0$ $t_0>0$ and $C>0$ such that the 
following holds. Let $\Lambda\subset S^*X$ be an $\eta_0$-unstable Lagrangian manifold \red{(in the sense of Definition \ref{def:LagInstable})} \red{that has distortion $\leq D$ (in the sense of (\ref{eq:DefDistorsion}))}, and let 
$\rho_1, \rho_2\in \Lambda$ be such that $\mathrm{dist}_{T^*X}(\rho_1, \rho_2) < \varepsilon$. 
Then there exist $\tau \in (-C\varepsilon,  C\varepsilon)$ such that for 
all $\tau'\in (-\red{C} \varepsilon,  \red{C}\varepsilon)$, we have
\begin{equation*}
	\mathrm{dist}_X(\Phi^{t_0}(\rho_1), \Phi^{t_0+\tau'}(\rho_2)) 
	\red{\geq} 
	2 \mathrm{dist}_X\left(\rho_1, \Phi^{\tau}(\rho_2)\right)
\end{equation*}
and $\mathrm{dist}_X(\Phi^t(\rho_1), \Phi^{t+\tau}(\rho_2)) \leq C \varepsilon$ for 
all $t\in [0, t_0]$ \red{and for all $\tau'\in (- \varepsilon, \varepsilon)$}.
\end{lem}
\red{The statement of the lemma and its proof are illustrated on Figure \ref{Fig2}.}
\begin{proof}
\red{Before proving the statement, let us start by stating a few inequalities we will need.}
First of all, since the distributions $E_\rho^{-0}$ are transverse to the 
vertical fibres of $T^*X$,  since $X$ is compact \red{and since $\Lambda$ has bounded distortion}, there exists 
$C_1, C_2, \varepsilon, \eta >0$ such that, if $\Lambda$ is a Lagrangian 
manifold which is $\eta$-unstable, and if $\rho, \rho'\in \Lambda$ with 
$\mathrm{dist}_{T^*X}(\rho, \rho') < \varepsilon$, then we have
\begin{equation}\label{eq:ComprDistUnstab}
C_1 \mathrm{dist}_{T^*X} (\rho, \rho') \leq \mathrm{dist}_X(\rho, \rho') 
\leq 
C_2 \mathrm{dist}_{T^*X}(\rho, \rho').
\end{equation}
Next, for every $\rho\in S^*X$, we introduce coordinates 
$(u_\rho, s_\rho, n_\rho)$ \MI{which send diffeomorphically a small ball $B(\rho, r)\subset S^*X$ to a neighbourhood of the origin in $\R^{2d-1}$, and} that 
are \emph{adapted},  in the sense that
\begin{itemize}
\item For every $\rho$,  the map $\rho' \mapsto (u_\rho(\rho', s_\rho(\rho'), n_\rho(\rho'))$ 
	is continuous in a neighbourhood of $\rho$.
\item $(u_\rho (\rho'),  n_\rho(\rho'))= (0,0) \Longleftrightarrow \rho'\in W_\rho^+$. 
	We then have $|s_\rho(\rho')| = \mathrm{dist}_{T^*X}(\rho, \rho') $.
\item $(s_\rho (\rho'),n_\rho(\rho')) = (0,0) \Longleftrightarrow  \rho'\in W_\rho^-$. 
	We then have  $|u_\rho(\rho')| = \mathrm{dist}_{T^*X}(\rho, \rho')$.
\item $(u_\rho(\rho'), s_\rho(\rho'))= (0,0) \Longleftrightarrow \Phi^{n_\rho(\rho')}(\rho) = \rho'$.
	%$(u_\rho, s_\rho, n_\rho)$ is a diffeomorphism from a neighbourhood of 
	%$\rho$ in $S^*X$ to a neighbourhood of the origin in $\R^{2d-1}$.
\end{itemize}
\red{By compactness, the radius $r$ of the ball may be chosen independently of $\rho$.} 

Since the vector spaces $E^{\bullet}_\rho$ depend continuously on $\rho$ and 
are transverse, we deduce by compactness that there exist $c_1,c_2>0$ and 
$\varepsilon>0$ such that for all $\rho, \rho'\in S^*X$ with 
$\mathrm{dist}_{T^*X}(\rho, \rho')< \varepsilon$,  we have
\begin{equation}\label{eq:ComparDist}
c_1 \left( |u_{\rho}(\rho')| +  |s_{\rho}(\rho')|  +  |n_{\rho}(\rho')|\right) 
\leq 
\mathrm{dist}_{T^*X} (\rho, \rho') 
\leq 
c_2 \left( |u_{\rho}(\rho')| +  |s_{\rho}(\rho')|  +  |n_{\rho}(\rho')| \right).
\end{equation}

\red{We may now proceed with the proof of the lemma.}  \red{The fact that, \MI{for every $\rho\in S^*X$, the map} $(u_\rho,  s_\rho, n_\rho)$ is a diffeomorphism implies that  there exists $C>0$ such that, for all $\varepsilon>0$ small enough and all $\MI{\rho, \rho'}\in \Lambda$, with $\mathrm{dist}_{T^*X}(\MI{\rho, \rho'})< \varepsilon$, there exists $\tau\MI{(\rho,\rho')} \in (-C \varepsilon, C \varepsilon)$ such that 
\begin{equation}\label{eq:TempsAnnulation}
n_{\MI{\rho}}(\Phi^{\tau\MI{(\rho,\rho')}} (\MI{\rho'}))=0.
\end{equation} 
Here, again, by compactness, the constant $C$ may be taken independent of $\rho, \rho'$.}

Let $\rho_1, \rho_2\in \red{\Lambda}$, with $\mathrm{dist}_{T^*X}(\rho_1, \rho_2)< \varepsilon$, \red{and let $\tau\MI{= \tau(\rho_1,\rho_2)}$ be as in (\ref{eq:TempsAnnulation}).}  We 
then write $\rho_3:=\Phi^\tau(\rho_2)$. Since $\rho_3\in \red{\Lambda}$, \MI{we deduce from \eqref{eq:ComparDist} and \eqref{eq:TempsAnnulation} that}
we have
\begin{equation}\label{eq:SPetit}
\MI{C_3 |u_{\rho_1}(\rho_3)|\leq }	\mathrm{dist}_{T^*X}(\rho_3, \rho_2) \leq C_4 |u_{\rho_1}(\rho_3)|
\end{equation}
%\begin{equation}\label{eq:SPetit}
% |s_\rho(\rho_3)| \leq  \frac{1}{2} |u_\rho(\rho_3)|.
% \end{equation}
provided $\varepsilon$ and $\eta_0$ are small enough\MI{, for some $C_3, C_4>0$ independent of $\rho_1, \rho_2$}.
%
%In particular, we then have 
%$$d_{T^*X}(\rho_3, \rho_1) \leq C_3 |u_\rho(\rho_3)|.$$
%
Now, thanks to (\ref{eq:Matrix}),  for all $\varepsilon>0$ small enough,  there exists $C\red{'}>0$ such that, for 
all $\tau'\in (-\red{C}\varepsilon,  \red{C}\varepsilon)$ and all $t>0$ such that $\mathrm{dist}_{T^*X}\left( \Phi^t(\rho_1), \Phi^{t+\tau'}(\rho_2) \right) < r$,
we have

\begin{equation}\label{eq:Stab}
%|s_{\Phi^t(\rho_1)}(\Phi^t(\rho_3))| &\leq C \e^{-t} |s_{\rho_1}(\rho_3)|\\
|u_{\Phi^t(\rho_1)}(\Phi^{t+\tau'}(\rho_2))| 
\geq 
C' \red{A_0^t} |u_{\rho_1}(\rho_3)|.
\end{equation}
We may then deduce the result by taking $t$ large enough, using 
(\ref{eq:ComprDistUnstab}), (\ref{eq:ComparDist}) and (\ref{eq:SPetit}).
\end{proof}

\begin{figure}
\definecolor{ffqqqq}{rgb}{1,0,0}
\definecolor{qqqqff}{rgb}{0,0,1}
\definecolor{zzttqq}{rgb}{0.6,0.2,0}
\begin{tikzpicture}[line cap=round,line join=round,>=triangle 45,x=1cm,y=1cm]
\clip(-4.5736226905451,-4.23363800695946) rectangle (10.25344205119935,5.439860072412265);
\fill[line width=2pt,color=zzttqq,fill=zzttqq,fill opacity=0.10000000149011612] (-1,4) -- (0,3) -- (0,-3) -- (-1,-2) -- cycle;
\fill[line width=2pt,color=zzttqq,fill=zzttqq,fill opacity=0.10000000149011612] (5,-2) -- (6,-3) -- (6,3) -- (5,4) -- cycle;
\draw [line width=1pt,color=zzttqq] (-1,4)-- (0,3);
\draw [line width=1pt,color=zzttqq] (0,3)-- (0,-3);
\draw [line width=1pt,color=zzttqq] (0,-3)-- (-1,-2);
\draw [line width=1pt,color=zzttqq] (-1,-2)-- (-1,4);
\draw [shift={(39.37360226612631,-1.1724496190809304)},line width=1pt]  plot[domain=3.0296599294300477:3.1745271352157602,variable=\t]({1*39.9086612830548*cos(\t r)+0*39.9086612830548*sin(\t r)},{0*39.9086612830548*cos(\t r)+1*39.9086612830548*sin(\t r)});
\draw [shift={(-0.5023124104641838,-0.0023246038908861645)},line width=1pt]  plot[domain=-0.008350988846211216:3.00553473319539,variable=\t]({1*0.5023299263594502*cos(\t r)+0*0.5023299263594502*sin(\t r)},{0*0.5023299263594502*cos(\t r)+1*0.5023299263594502*sin(\t r)});
\draw [line width=1pt,color=zzttqq] (5,-2)-- (6,-3);
\draw [line width=1pt,color=zzttqq] (6,-3)-- (6,3);
\draw [line width=1pt,color=zzttqq] (6,3)-- (5,4);
\draw [line width=1pt,color=zzttqq] (5,4)-- (5,-2);
\draw [line width=1.5pt,color=qqqqff] (-0.5,0.5)-- (5.502536466025894,0.48083569901689954);
\draw [shift={(5.490827704234865,-0.16065020813020686)},line width=1pt]  plot[domain=0.6540822031584338:2.44186003410664,variable=\t]({1*0.6415927557033555*cos(\t r)+0*0.6415927557033555*sin(\t r)},{0*0.6415927557033555*cos(\t r)+1*0.6415927557033555*sin(\t r)});
\draw [shift={(48.09071881974634,-0.9609838934571975)},line width=1pt]  plot[domain=3.0412539987001868:3.1778616631023615,variable=\t]({1*42.612581709290836*cos(\t r)+0*42.612581709290836*sin(\t r)},{0*42.612581709290836*cos(\t r)+1*42.612581709290836*sin(\t r)});
\draw [shift={(0.33701702345772416,20.390165708829326)},line width=1pt,color=ffqqqq]  plot[domain=4.664019736140885:5.024360270955364,variable=\t]({1*19.630149056652233*cos(\t r)+0*19.630149056652233*sin(\t r)},{0*19.630149056652233*cos(\t r)+1*19.630149056652233*sin(\t r)});
\draw [->,line width=0.5pt] (-0.44796371349093533,1.4627217864612918) -- (-0.41319780183820143,1.9439241963359382);
\draw [->,line width=0.5pt] (-0.5286111067462969,-0.45508405582641476) -- (-0.5346492438720887,-0.9915993486706324);
\draw [->,line width=0.5pt] (-0.9246865613616986,0.26958610124941185) -- (-0.8247983170684472,0.3828223744925836);
\draw [->,line width=0.5pt] (-0.0908107674415991,0.2857756437764907) -- (-0.20055110890367367,0.39926657105796337);
\draw [->,line width=0.5pt] (5.522676781485515,0.9868140837075698) -- (5.549694838229236,1.5074971242892492);
\draw [->,line width=0.5pt] (5.483871647309684,-0.26191740476367775) -- (5.47819199872675,-0.8925890188420881);
\draw [->,line width=0.5pt] (5.870292961392436,0.35669627420186095) -- (5.67640515679825,0.45351774387568417);
\draw [->,line width=0.5pt] (5.096511599032447,0.34546842751663537) -- (5.277874212678469,0.44457047443522363);
\draw [color=qqqqff](5.41604185047848,0.2291869207661375) node[anchor=north west] {$\Phi^t(\rho_1)$};
\draw [color=ffqqqq](0.349564531628937,1.5113215510953628) node[anchor=north west] {$\rho_2$};
\draw [color=qqqqff](-0.5562959176860923,0.2503080956119273) node[anchor=north west] {$\rho_1$};
\draw [color=ffqqqq](6.3354641173053405,2.4406532443101137) node[anchor=north west] {$\Phi^{t+\tau'}(\rho_2)$};
\draw [color=zzttqq](-0.5922812321137196,4.183150169087772) node[anchor=north west] {$\{n_{\rho_1}=0\}$};
\draw [color=zzttqq](5.6701471096629605,4.088104882281717) node[anchor=north west] {$\{n_{\Phi^t(\rho_1)}=0\}$};
\draw [color=ffqqqq](-2.8595517228611085,1.490200376249573) node[anchor=north west] {$\rho_3=\Phi^{\tau}(\rho_2)$};
\begin{scriptsize}
\draw [fill=qqqqff] (-0.5,0.5) circle (2.5pt);
\draw [fill=ffqqqq] (0.5799439718276151,0.7615198440698252) circle (2.5pt);
\draw [fill=ffqqqq] (-0.6121082567435837,0.7829753647523005) circle (2.5pt);
\draw [fill=qqqqff] (5.502536466025894,0.48083569901689954) circle (2.5pt);
\draw [fill=ffqqqq] (6.362204061548653,1.707556866849959) circle (2.5pt);
\end{scriptsize}
\end{tikzpicture}\caption{\red{An illustration of Lemma \ref{Lem:DernierLemmeAvantLesVacances} and of its proof.}}\label{Fig2}
\end{figure}

\subsubsection{\textnormal{\textbf{Projectability of Lagrangian manifolds on the universal cover}}}
\label{sec:DynOnUnivCover}

We shall say that a Lagrangian manifold $\Lambda\subset T^*X$ is $D$-controlled if it can be put in the form
$\Lambda = \{ (x, d_x \phi) , x\in U\}$ with $\|\phi\|_{C^{3}}\leq D$.  \red{Here, as previously (see \eqref{eq:sa8.1}), the $C^3$ norm is defined using some coordinate frame.}

Our aim in this section is to prove the following proposition.

\begin{prop}\label{Prop:OnTheCover2}
There exists $\gd_0>0$ such that, for all $D\geq \red{1}$, we may find 
$\eta(D), \delta_0(D)>0$ such that the following holds for all 
$\delta \in [0, \delta_0)$.

Let $\Lambda \subset S^*X$ be a simply connected monochromatic 
Lagrangian manifold which is $\eta$-unstable \red{(in the sense of Definition \ref{def:LagInstable})}, $D$-controlled,  \red{has diameter $\leq D$}, and
has distortion $\leq D$,  as in section \ref{subsec:Lag}. 

Consider 
a lift $\wit{\Lambda} \subset S^* \wit{X}$. Then, for all 
$0 \leq t \leq \gd_0 |\log \delta|$, $\widetilde{\Phi}_{\delta}^t(\widetilde{\Lambda})\subset 
\widetilde{\mathcal{E}}_{\delta, (\frac{1}{4}, 1)}$ 
is a simply connected Lagrangian submanifold which is projectable.
\end{prop}

\red{Remember that, thanks to Hypothesis \ref{Hyp:BetaDelta}, $\delta$ will be bounded by above and below by powers of $h$, so that Proposition \ref{Prop:OnTheCover2} applies to time scales $0\leq t \leq \gd_0' |\log h|$ for some $\gd_0'>0$.}

Before proving this proposition, we first need to prove results 
about the evolution of $\Lambda$ by $\wit{\Phi}^t$.
\begin{lem}\label{lem411}
Let %$U\subset \wit{X}$ a bounded open set, and let 
$\varepsilon>0$, $D\geq 1$.  There exists $T_0, T_1\geq 0$ such that, 
for all $T\geq T_0$, there exists $\eta(T)>0$ such that, if 
$\Lambda \subset S^* \wit{X}$ is an $\eta$-unstable Lagrangian 
manifold that is $D$-controlled \red{and has diameter $\leq D$},  the following holds. 
If $\rho_1, \rho_2\in \Lambda$,  we may find $T_2 \in [T- T_1, T+ T_1]$ 
such that $\mathrm{dist}_{\wit{X}}(\wit{\Phi}^{-T}(\rho_1), 
\wit{\Phi}^{-T_2}(\rho_2)) \leq \varepsilon 
\mathrm{dist}_{\wit{X}}(\rho_1, \rho_2) $.
\end{lem}
\begin{proof}
If $\rho\in S^*\wit{X}$, we define the unstable manifold at $\rho$ as
\begin{align*}
W_\rho^+&:= \{\rho'\in S^*\wit{X} ; \mathrm{dist}_{T^*\wit{X}} 
(\wit{\Phi}^t(\rho), \wit{\Phi}^t(\rho')) \underset{t\to -\infty}{\longrightarrow} 0\}.
\end{align*}
At $\rho$, the manifold $W_\rho^+$ is tangent to the space 
$E^+_\rho$ introduced in section \ref{sec:Hyperbolicity} 
(see \cite[\S 17.4]{Kat}), while the weak unstable manifold 
$W_\rho^{+ 0}:= \bigcup_{t\in \R} \wit{\Phi}^t(W_\rho^+)$ 
has tangent space at $\rho$ $E^+_\rho \oplus E^0_\rho$.

Recall that $\Lambda = \{ (x, d_x \phi) , x\in U\}$ with $\|\phi\|_{C^3}\leq D$. \red{Let $\rho \in \Lambda$.}
Since $W^{+0}_\rho$ is locally projectable,  we may write $W^{+0}_\rho$ as $\{ (x, d_x \psi) \}$ in a 
neighbourhood of $\rho$, for some smooth function $\psi$.  By definition,  \red{if $x_0= \pi_{\wit{X}}(\rho)$},  we have
we have 
$d_{x_0}\psi = d_{x_0}\phi$, while,  thanks to the $\eta$-instability hypothesis, we have 
$\red{\nabla (d\psi)_{x_0} -\nabla (d\phi)_{x_0}} = O(\eta)$. Finally,  the $D$-control hypothesis implies 
that third derivatives are bounded.  Therefore, 
we may find $\varepsilon_1 (\eta)>0$ such that,  if $\rho, \rho'\in \Lambda$ with $\mathrm{dist}_{T^*\wit{X}}(\rho,\rho')  < \varepsilon_1$, 
\red{then, writing $x' = \pi_{\wit{X}}(\rho')$, we have $\|d_{x'} \psi - d_{x'} \phi\| \leq 2 \eta \mathrm{dist}_{\wit{X}}(x', x_0)$, so that}
 $\mathrm{dist}_{T^*\wit{X}}(\rho',  W^{+0}_\rho) < 2 \eta \mathrm{dist}_{T^*\wit{X}}(\rho,\rho')$. 
In particular,  there exists $\rho'' \in W^{+}_\rho$ and $\tau \in (-c\varepsilon_1, c \varepsilon_1)$ such that $\mathrm{dist}_{T^*\wit{X}}(\red{\Phi^{\tau}}(\rho'),  \rho'') < 2 \eta \mathrm{dist}_{T^*\wit{X}}(\rho,\rho')$  and $\mathrm{dist}_{T^*\wit{X}}(\rho, \rho'') < c ~\mathrm{dist}_{T^*\wit{X}}(\rho, \rho')$.    Here, $c$ depends only on the manifold $(X,g)$.

Now, by definition of unstable manifolds,  for any $\varepsilon>0$,  there exists $T_0\red{= T_0(X,g\MI{,\varepsilon})}\geq 0$ such that,  for all $s\geq T_0$,
%if $t\geq T_0$, 
we have $\mathrm{dist}_{T^*\wit{X}}(\wit{\Phi}^{-s}(\rho), \wit{\Phi}^{-s}(\rho'')) < \varepsilon \mathrm{dist}_{T^*\wit{X}}(\rho, \rho'')$. 
Hence, we have \red{for all $T\geq T_0$}
\red{\begin{align*}
\mathrm{dist}_{T^*\wit{X}}(\wit{\Phi}^{-T}(\rho), \wit{\Phi}^{-T+\tau}(\rho'))
 &\leq \mathrm{dist}_{T^*\wit{X}}(\wit{\Phi}^{-T}(\rho), \wit{\Phi}^{-T}(\rho''))+ \mathrm{dist}_{T^*\wit{X}}(\wit{\Phi}^{-T}(\rho''), \wit{\Phi}^{-T+\tau}(\rho'))\\
 &\red{\leq \varepsilon \mathrm{dist}_{T^*\wit{X}}(\rho, \rho'') + C e^{CT}  \mathrm{dist}_{T^*\wit{X}}(\rho'',  \Phi^{\tau}(\rho'))}\\
&\leq c \varepsilon \mathrm{dist}_{T^*\wit{X}}(\rho, \rho') + C e^{CT} \eta \mathrm{dist}_{T^*\wit{X}}(\rho, \rho'),
\end{align*}}
thanks to (\ref{eq:ExpRate}). This quantity may thus be made smaller than $(c+1) \varepsilon \mathrm{dist}_{T^*\wit{X}}(\rho, \rho'') $ if $\eta$ is chosen smaller than some $\eta_0$ depending on $\varepsilon$ and $T$, but not on $\rho,\rho'$. 

From now on, we fix an $\varepsilon>0$,  a $T\geq T_0$, an $\eta< \eta_0(\varepsilon, T)$, and a corresponding $\varepsilon_1(\eta)$. 
The assumption on the distortion implies that 
 if $\rho_1, \rho_2\in \Lambda$, there exists a curve in $\Lambda$ of length $\leq D \mathrm{dist}_{T^*\wit{X}}(\rho_1, \rho_2)$ joining $\rho_1$ and $\rho_2$. Taking points at a distance $\approx \varepsilon_1$ on this curve, we see that there exists $N\leq  C(D) \varepsilon_1^{-1} \mathrm{dist}_{T^*\wit{X}}(\rho_1, \rho_2)  $ and $\rho_1= \rho'_1,  \rho'_2, ..., \rho'_N= \rho_2$ with $\mathrm{dist}_{T^*\wit{X}}(\rho'_j,\rho'_{j+1})  < \varepsilon_1$ and $\sum_{j=1}^N \mathrm{dist}_{T^*X} (\rho'_j, \rho'_{j+1}) \leq C'(D) \mathrm{dist}_{T^*X} (\rho_1, \rho_2)$.

Using the previous construction,  we build a sequence of times $t_1,..., t_N$ with $t_1= T$, $|t_j-t_{j+1}| \leq c\varepsilon_1$ with $\mathrm{dist}_{T^*\wit{X}}(\wit{\Phi}^{-t_j}(\rho'_j), \wit{\Phi}^{-t_{j+1}}(\rho'_{j+1})) < (c+1) \varepsilon  \mathrm{dist}_{T^*\wit{X}}(\rho'_j,  \rho'_{j+1}) $, so that the triangular inequality yields $\mathrm{dist}_{T^*\wit{X}}(\wit{\Phi}^{-t_1}(\rho'_1), \wit{\Phi}^{-t_N}(\rho'_N)) < C''(D) \varepsilon  \mathrm{dist}_{T^*\wit{X}}(\rho_1, \rho_2)$. 

Note that $|t_N-t_1| \leq C(D) \times c \red{ \mathrm{dist}_{T^*\wit{X}}(\rho_1, \rho_2)} \red{\leq Dc C(D)}$ and recall that $\mathrm{dist}_{\wit{X}}(\rho_1, \rho_2) $ and $\mathrm{dist}_{T^*\wit{X}}(\rho_1, \rho_2) $ are comparable thanks to the hypothesis of $D$-control.  Taking $\varepsilon$ possibly smaller, \red{we deduce the result by setting $T_2= t_N$ and taking $T_1 = Dc C(D)$.}
\end{proof}
\begin{corollary}\label{cor:412}
Let $D>0$. There exists $\eta>0$ such that, if $\Lambda \subset S^* \wit{X}$ 
is an $\eta$-unstable Lagrangian manifold that is $D$-controlled,  
the following holds.  For all $\rho_1, \rho_2\in \Lambda$
and for all $t\geq 0$, we have $\mathrm{dist}_{\wit{X}}
\left( \wit{\Phi}^t (\rho_1), \wit{\Phi}^t(\rho_2)\right) 
\geq  \frac{1}{2}  \mathrm{dist}_{\wit{X}}(\rho_1, \rho_2)$.
%\textcolor{blue}{M: Notation correct ? Dist on $T^*X$ or $X$ ?}
\end{corollary}
\begin{proof}
Suppose for contradiction that there exist $\rho_1, \rho_2\in \Lambda$ and $t>0$ such that 
\begin{equation}\label{eq:ContradictionContraction}
\mathrm{dist}_{\wit{X}}\left( \wit{\Phi}^t (\rho_1), \wit{\Phi}^t(\rho_2)\right) <  \frac{1}{2}  \mathrm{dist}_{\wit{X}}(\rho_1, \rho_2).
\end{equation} 

Thanks to Lemma \ref{lem411}, we know (provided $\eta$ is small enough) that there exists $T$ arbitrarily large such that,  for all $\rho_1, \rho_2\in \Lambda$,   we may find $T'$ at a bounded distance from $T$ such that $\mathrm{dist}_{\wit{X}}\left( \wit{\Phi}^{-T} (\rho_1), \wit{\Phi}^{-T'}(\rho_2)\right)\leq \frac{1}{4}  \mathrm{dist}_{\wit{X}}(\rho_1, \rho_2) $.

Suppose that $T\geq T'$ (the case $T\leq T'$ is treated similarly). We have, using the fact that $\wit{X}$ is a simply connected manifold of negative curvature (so that geodesics minimize the distance between their endpoints):
\begin{align*}
t+ T &= \mathrm{dist}_{\wit{X}} \left(\wit{\Phi}^{t}(\rho_1), \wit{\Phi}^{-T}(\rho_1) \right)\\
&\leq \mathrm{dist}_{\wit{X}} \left(\wit{\Phi}^{t}(\rho_1), \wit{\Phi}^{t}(\rho_2) \right)+ \mathrm{dist}_{\wit{X}} \left(\wit{\Phi}^{t}(\rho_2), \wit{\Phi}^{-T'}(\rho_2) \right)+ \mathrm{dist}_{\wit{X}} \left(\wit{\Phi}^{-T'}(\rho_2), \wit{\Phi}^{-T}(\rho_1) \right)\\
&< t+T' + \frac{3}{4} \mathrm{dist}_{\wit{X}}(\rho_1, \rho_2),
\end{align*}
so that $|T'-T| <  \frac{3}{4} \mathrm{dist}_{\wit{X}}(\rho_1, \rho_2)$.

In particular, we have
\begin{align*}
\mathrm{dist}_{\wit{X}}\left( \wit{\Phi}^{-T} (\rho_1), \wit{\Phi}^{-T}(\rho_2)\right)&\leq \mathrm{dist}_{\wit{X}}\left( \wit{\Phi}^{-T} (\rho_1), \wit{\Phi}^{-T'}(\rho_2)\right) + \mathrm{dist}_{\wit{X}}\left( \wit{\Phi}^{-T} (\rho_2), \wit{\Phi}^{-T'}(\rho_2)\right)\\
&\leq \frac{1}{4}  \mathrm{dist}_{\wit{X}}(\rho_1, \rho_2) + |T'-T|\\
&<  \mathrm{dist}_{\wit{X}}(\rho_1, \rho_2) .
\end{align*}

This inequality, along with (\ref{eq:ContradictionContraction}), contradicts Lemma \ref{lem:convex}.
\end{proof}

We may now proceed with the proof of Proposition \ref{Prop:OnTheCover2}.
\begin{proof}[Proof of Proposition \ref{Prop:OnTheCover2}]

Since simple connectedness is preserved by homeomorphism, $\widetilde{\Phi}_{\delta}^t(\widetilde{\Lambda})$ is simply connected.  \red{It is Lagrangian, because $\widetilde{\Phi}_{\delta}^t$ is a symplectomorphism. The fact that $\widetilde{\Phi}_{\delta}^t(\widetilde{\Lambda})\subset 
\widetilde{\mathcal{E}}_{\delta, (\frac{1}{4}, 1)}$  is proven just as (\ref{eq:GronN11}) using (\ref{eq:GronN00}).}

We thus need to show that
\begin{equation}\label{eq:DifferentProjections}
\forall \rho_1, \rho_2\in \widetilde{\Phi}_{\delta}^t(\widetilde{\Lambda}),  \left(\rho_1 \neq \rho_2\right) \Longrightarrow \left(\pi_{\wit{X}}(\rho_1) \neq \pi_{\wit{X}}(\rho_2)\right),
\end{equation}
which implies that $\pi_{\wit{X}}(\widetilde{\Phi}_{\delta}^t(\widetilde{\Lambda}))$ is simply connected.  The result then follows from Proposition \ref{Prop:LagInitiale}, 
Lemma \ref{Lem:InstTrans} and Remarks \ref{Rem:ProjTrans}, \ref{rem:4.7}.

To prove (\ref{eq:DifferentProjections}), we suppose that there exists 
$\rho_1,\rho_2\in \widetilde{\Phi}_{\delta}^t(\widetilde{\Lambda})$, such that 
$\pi_{\wit{X}}(\rho_1) = \pi_{\wit{X}}(\rho_2)$, and we write 
$\rho'_j = \widetilde{\Phi}_{\delta}^{-t}(\rho_j) \in \widetilde{\Lambda}$ for $j=1,2$. 
We separate two cases:
\begin{itemize}
\item Suppose first that $\mathrm{dist}_{T^*\wit{X}}(\rho'_1, \rho'_2) > h^{ \beta/2}$.  
%Then $\mathrm{dist}_{T^*X} (\rho'_1, \rho'_2) > h^{2\beta/3}$  provided $c_0$ is small enough, thanks to \eqref{eq:Gronwall}.  
Since $\rho'_1, \rho'_2\in \widetilde{\Lambda}$, we thus have $\mathrm{dist}_{\wit{X}} (\rho'_1, \rho'_2) > C h^{\beta/2}$.  \red{This inequality follows from (\ref{eq:ComprDistUnstab}) and from the fact that $\wit{\Lambda}$ has bounded diameter.}

We now use Lemma \ref{Lem:Gron}. Remember that, for any $\varepsilon'>0$,  the quantities $e^{C_0t}$ can be made $O(h^{-\varepsilon'})$ if we choose $\gd_0$ small enough, since $|\log(\delta)|$ is comparable to $|\log h|$.  We thus get
\begin{align*}
\mathrm{dist}_{\wit{X}} \left(\wit{\Phi}^t_0(\rho'_1),  \wit{\Phi}^t_0(\rho'_2)\right) &\leq  \mathrm{dist}_{\wit{X}} \left(\wit{\Phi}^t_0(\rho'_1),  \wit{\Phi}^t_\delta (\rho'_1)\right) + \mathrm{dist}_{\wit{X}} \left(\wit{\Phi}^t_\delta(\rho'_1),  \wit{\Phi}^t_\delta(\rho'_2)\right) +  \mathrm{dist}_{\wit{X}} \left(\wit{\Phi}^t_\delta(\rho'_2),  \wit{\Phi}^t_0(\rho'_2)\right) \\
&\leq C \delta h^{-\beta - \varepsilon'} %~~~\text{for any $\varepsilon'>0$,  if $\gd_0$ is chosen small enough}
\\
& < \frac{1}{2} h^{\beta/2}
\end{align*}
if $h$ is small enough,
thanks to (\ref{eq:CondBeta}).
This contradicts Corollary \ref{cor:412}. 

\item Suppose next that $\mathrm{dist}_{T^*\wit{X}}(\rho'_1, \rho'_2) < h^{ \beta/2}$. 
Then, there exists $s_0\leq C(D) h^{ \beta/2}$ and a curve $\varrho(s)$ in $\widetilde{\Lambda}$, 
with $\varrho(0) =\rho'_1$, $\varrho(s_0) = \rho'_2$ and $\|\frac{d}{ds}\varrho(s)\|=1$ for all $s$. 
Using coordinates on $T^*X$ near $\rho_1$, the bounds on derivatives of $\wit{\Phi}^t_\delta$ 
derived in Remark \ref{rem:DerivCover} imply that, in coordinates, $\rho_2 = \rho_1 + s_0 d_{\rho'_1} 
\wit{\Phi}^t_\delta \MI{(\varrho'(0))} + O(s_0^2 h^{- \varepsilon})$. Now,  thanks to 
Proposition \ref{Prop:LagInitiale} and Lemma \ref{Lem:InstTrans}, %,  which imply that 
%$\widetilde{\Phi}_{\delta}^t(\widetilde{\Lambda})$ is locally projectable, 
the angle 
between $d_{\rho'_1} \wit{\Phi}^t_\delta \MI{(\varrho'(0))}$ and the vertical leaves is 
bounded from below: hence, (\ref{eq:DifferentProjections}) follows.
%\corM{needs some uniformity on that angle...}
\end{itemize}
\end{proof}

\begin{rem}\label{RemMixedFlow}
In the sequel, we will also consider mixed flows $\Phi^t_\delta \circ \Phi^{t'}_0$. 
We can follow the exact same steps of the proof of Proposition \ref{Prop:OnTheCover2} 
to deduce that, for the same $\gd_0$, 
$\Lambda$ and $\delta_0$, the manifold  $\Phi^t_\delta \left( \Phi^{t'}_0(\Lambda)\right)$ 
is simply connected and projectable for all $t,t'\geq 0$ such that 
$t+t' \leq  \gd_0 |\log \delta|$.
\end{rem}
\section{WKB Ansatz on the universal cover $\wt{X}$}\label{sec:WKB}
We still consider a compact connected smooth Riemannian manifold $(X,g)$ 
and $(\widetilde{X},\widetilde{g})$ its universal cover equipped with the 
lifted Riemannian metric. In this section, we shall work on the universal cover 
$\widetilde{X}$ and with the lifted quantities $\widetilde{p}$ \eqref{eq:liftedHamiltonian}, 
$\wt{\Phi}_\delta^t$ \eqref{eq:ProjDyn}, $\widetilde{q}_\omega$, $\widetilde{P}_h^\delta$ and 
$\widetilde{Q}_\omega$, cf. Section \ref{subsec:LiftPseudo}. 
% To ease the notation we 
% will suppress the $\sim$ notation until further notice. 
\\
\par 
Let $h\in ]0,h_0]$. Let $\wt{\mathcal{O}}\subset \wt{X}$ be a simply connected open 
relatively compact set and let $\wt{\phi}_0:\wt{\mathcal{O}}\to \R$ be a smooth map 
as in \eqref{eq:LagState_monochrom}, with $|d_x \wt{\phi}_0| = 1$ for 
all $x\in \wt{\mathcal{O}}$. We consider an initial Lagrangian state 
\begin{equation}\label{eq:trans1}
	f_h(x) =  a(x ;h)\e^{\frac{i}{h} \wt{\phi}_0(x)},
\end{equation}
where $a(\cdot\,;h)$ is a smooth compactly supported map on $\wt{\mathcal{O}}$ 
with $\|a(\cdot;h)\|_{C^L} = O_L(1)$ and such that $a(x;h)\sim a_0(x) +ha_1(x) + \dots $ 
where $a_n$ are smooth 
compactly supported maps on $\wt{\mathcal{O}}$. We assume that there exists 
a compact $h$-independent set $K\Subset\wt{\mathcal{O}}$ such that 
\begin{equation}\label{hyp:compactSupport}
	\supp a(\cdot\,;h) \subset K
\end{equation}
for all $h\in ]0,h_0]$. As discussed in Section 
\ref{subsec:Lag},  \MI{$f_h$ is a Lagrangian state associated with}
%this represents a Lagrangian state supported on
the Lagrangian manifold
\begin{equation*}
\Lambda_0 := \{ (y, d_y \wt{\phi}_0) ; y\in \wt{\mathcal{O}}\} \subset S^*\wt{X}
\end{equation*}
which is assumed to be a monochromatic, $D$-controlled Lagrangian 
manifold of distortion $\leq D$, that is $\eta$-unstable for some $\eta \leq \eta_0(D)$,
 as in Section \ref{subsec:Lag}.
\par
The main aim of this section is to approximate the propagated state 
\begin{equation}\label{eq:trans1.0}
	\red{\widetilde{u}(t,x;h,\delta) = \e^{-i\frac{t}{h}\wt{P}_h^\delta}f(x)}
\end{equation}
by a WKB state, cf. \eqref{eq:WKB_Ansatz} below. 
\\
\par 
Before doing so, let us introduce auxiliary cut-off functions which 
will be useful in this section. First of all, 
it will be useful to consider two intermediate open relatively compact sets 
$\wt{\mathcal{O}}',\wt{\mathcal{O}}''\subset\wt{\mathcal{O}}$, independent of $h$, with 
\begin{equation}\label{eq:EnsemblesEmboites}
	K\subset \wt{\mathcal{O}}'\subset \overline{\wt{\mathcal{O}}'} \subset 
	\wt{\mathcal{O}}''\subset \overline{\wt{\mathcal{O}}''}\subset\wt{\mathcal{O}}.
\end{equation}
We can construct a sequence of 
cut-off functions $\psi$, $\psi_{i} \in C^\infty_c(X;[0,1])$, 
$i=0,\dots,N\in \N$, such that 
\begin{equation}\label{eq:CutOfffun1}
	\mathbf{1}_{\wt{\mathcal{O}}''} \succ 
	\psi \succ 
	\psi_{N} \succ \dots \succ \psi_{0} \succ 
	\mathbf{1}_{\wt{\mathcal{O}}'},
\end{equation}
and 
\begin{equation*}
\begin{split}
	&|\partial_x^\alpha \psi(x)|\leq O_{\alpha}(1), \\
	&|\partial_x^\alpha \psi_{i}(x)| \leq O_{\alpha,N}(1) 
	\quad \alpha \in \N^d, i=0,\dots,N,
\end{split}
\end{equation*}
uniformly in local coordinates. We write 
\begin{equation*}
	\Lambda^\bullet_0= \{ (y, d_y \wt{\phi}_0(y)) ; y\in  \wt{\mathcal{O}}^\bullet\}, 
	\quad \bullet\in \{',''\}
\end{equation*}
and 
\begin{equation}\label{eq:eik0}
\wt{\Phi}_\delta^t \left(\wt{\Phi}^{t'}_0(y, d_y \wt{\phi}_0) \right)
	= (x_\delta^{t,t'}(y), \xi_\delta^{t,t'}(y))
	=:\rho_\delta^{t,t'}(y).
\end{equation}
\red{
\begin{rem}
	It will be crucial further below to consider the mixed 
	flow \eqref{eq:eik0} where we first apply the unperturbed 
	Hamliton flow $\wt{\Phi}^{t'}_0$ until time $t'$, and then the perturbed 
	Hamliton flow $\wt{\Phi}^{t}_\delta$ until time $t$.
\end{rem}
}
We know from Proposition \ref{Prop:OnTheCover2} and Remark \ref{RemMixedFlow} that there exists 
$\mathfrak{d}>0$ such that for any $t,t'\geq 0$ such that $t+t' \leq \gd |\log h|$,
the Lagrangian submanifold $\wt{\Phi}^t_\delta \left( \wt{\Phi}^{t'}_0 (\Lambda_0)\right)$ is projectable and 
simply connected. Hence, for any $t,t'\geq 0$ such that $t+t' \leq \gd |\log h|$, there exists a 
smooth map $\wt{\phi}_{t,t',\delta} : \wt{\mathcal{O}}_{t, t', \delta} \to \R$ such that 
\begin{equation}\label{eq:eik0.0}
\begin{aligned}
\Lambda_{t,t'} 
	&:=\wt{\Phi}^t_\delta \left(\wt{\Phi}^{t'}_0 (\Lambda_0) \right)
	= \{ (x, d_x \wt{\phi}_{t,t',\delta}) ; x\in  \wt{\mathcal{O}}_{t,t',\delta} \}\\
\Lambda^{\bullet}_{t,t'} 
	&:=\wt{\Phi}^t_\delta  \left(\wt{\Phi}^{t'}_0(\Lambda_0^\bullet) \right)
	= \{ (x, d_x \wt{\phi}_{t,t',\delta}) ; x\in  \wt{\mathcal{O}}^\bullet_{t,t',\delta} \}, 
	\quad \bullet\in \{',''\}
\end{aligned}
\end{equation}
where $\wt{\mathcal{O}}_{t,t',\delta} = \pi_{X}\left(\wt{\Phi}^t_\delta \left(\wt{\Phi}^{t'}_0 (\Lambda_0) \right)\right),
\wt{\mathcal{O}}^{\bullet}_{t,t',\delta} = \pi_{X}\left( \wt{\Phi}^t_\delta  \left(\wt{\Phi}^{t'}_0(\Lambda_0^\bullet) \right) \right)
\subset X$ are open sets. 
\red{
\begin{rem}
	Notice that the mixed flow \eqref{eq:eik0} as well as the resulting 
	evolved Lagrangian manifolds \eqref{eq:eik0.0} and the sets  
	$\wt{\mathcal{O}}_{t,t',\delta}$, $\wt{\mathcal{O}}^{\bullet}_{t,t',\delta}$ 
	depend on the random perturbation. Therefore they depend on $\delta$ 
	and on  the random parameter $\omega$. 
\end{rem}
}

For $t=t'=0$ we have not evolved the initial 
Lagrangian manifold $\Lambda_0$, so $\wt{\phi}_{0,\delta} = \wt{\phi}_0$ and 
$\wt{\mathcal{O}}^{\bullet}_{0,0,\delta}=\wt{\mathcal{O}}^{\bullet}$, 
$\wt{\mathcal{O}}_{0,\delta}=\wt{\mathcal{O}}$. \corM{When we consider 
evolution under either solely the perturbed or unpertubed flow we will 
write $\Lambda_t=\Lambda_t(\delta)=\wt{\Phi}^t_\delta(\Lambda_0)$ and 
similarly we shall
write $\wt{\phi}_{t,\delta}$, $\wt{\mathcal{O}}_{t,\delta}^\bullet$, 
$\bullet\in \{',''\}$, and $\rho_\delta^{t}(y)=(x_\delta^{t}(y), \xi_\delta^{t}(y))$ instead of 
$\wt{\phi}_{t,0,\delta}$, $\wt{\mathcal{O}}_{t,0,\delta}^\bullet$ and 
$\rho_\delta^{t,0}(y) =(x_\delta^{t,0}(y), \xi_\delta^{t,0}(y))$.  
}
\begin{lem}
If $t,t'\geq 0$ with $t+t'\leq \gd |\log h|$ for $\gd$ 
and $h$ small enough, we have
\begin{equation}\label{eq:SetInclusion}
	\wt{\mathcal{O}}'_{t,t',\delta}\subset \overline{\wt{\mathcal{O}}'}_{t,t',\delta} \subset 
	\wt{\mathcal{O}}''_{t,t',\delta}\subset \overline{\wt{\mathcal{O}}''}_{t,t',\delta}
	\subset \wt{\mathcal{O}}_{0, t+t',0}
\end{equation}
and 
\begin{equation}\label{eq:SetInclusion2}
\overline{\wt{\mathcal{O}}'}_{0,t+t',0}
	\subset \wt{\mathcal{O}}_{t, t',\delta}''.
\end{equation}
\end{lem}
%using Lemma \ref{Lem:Gron}, we see that,

%
% \red{Indeed,  \MI{the first three inclusion follow from \eqref{eq:EnsemblesEmboites}.} 
% To obtain the last inclusion in (\ref{eq:SetInclusion}), 
% we note that there exists $\varepsilon>0$ such that $\wt{\mathcal{O}}$ 
% contains an $\varepsilon$-neighbourhood of  $\cO''$. By Corollary 
% \ref{cor:412}, this implies that $\cO_{0,  t+t', 0}$ contains an 
% $\frac{\varepsilon}{2}$-neighbourhood of $\cO''_{0, t+t',0}$. Now, 
% by Lemma \ref{Lem:Gron},  for all $h$ small enough, a point in 
% $\cO''_{t,t',\delta}$ is contained in an $\frac{\varepsilon}{2}$-neighbourhood 
% of $\cO''_{0, t+t',0}$, and hence in $\cO_{0,  t+t', 0}$.}
%
%%% new
%
\begin{proof}
\MI{
1. The first three inclusion follow from 
\eqref{eq:EnsemblesEmboites}.To prove the last inclusion in 
\eqref{eq:SetInclusion}, we use the following elementary 
fact\footnote{\MI{Indeed, if there existed a point $\wit{x}$ in 
an $r$-neighbourhood of $\wit{\cK}$ that were not in $\wit{\cU}$, 
there would exist a geodesic segment of length $<r$ from $\wit{x}$ 
to $\wit{\cK}$: but, since this geodesic segment would have to 
cross $\partial \wit{\cU}$, we would reach a contradiction).}} 
if $\widetilde{\cU}\subset \widetilde{X}$ is open, 
$\wit{\cK}\subset \wit{\cU}$ is compact, and if 
$\dist(\partial \wit{\cU}, \wit{\cK}) >r$, then $\wit{\cU}$ contains 
an $r$-neighbourhood of $\wit{\cK}$.}

\MI{Now, there exists a small $r>0$ such that, if we denote by 
$\wit{\cO}''(r)$ an $r$-neighbourhood of $\wit{\cO}''$, we have 
$\overline{\wit{\cO}''(r)}\subset \wit{\cO}$. Since $x^{t+t'}_0$ is a diffeomorphism, 
we have $\partial \left(x_0^{t+t'}\left(\wit{\cO}''(r)\right)\right)= 
x_0^{t+t'}\left(\partial \wit{\cO}''(r)\right)$. But, if 
$y\in \partial \wit{\cO}''(r)$, i.e., if $y$ is at a distance $r$ from 
$\wit{\cO}''$, then we know from Corollary \ref{cor:412} that 
$x_0^{t+t'}(y)$ is at a distance $\geq \frac{r}{2}$ from 
$\corM{x_0^{t+t'}}(\wit{\cO}'')$.}
\par 
\MI{ 
We note that 
\corM{$x_0^{t+t'}\left(\wit{\cO}''(r)\right)\subset \wit{\cO}_{0, t+t',0}$ 
and} applying the previous fact, we deduce that 
$x_0^{t+t'}(\wit{\cO}''(r))$ contains an $\frac{r}{2}$-neighbourhood of $\wit{\cO}''_{0, t+t',0}$. 
Now, by Lemma \ref{Lem:Gron}, for all $h$ small enough, a point in 
$\wt{\cO}''_{t,t',\delta}$ is contained in an \corM{$\frac{r}{2}$-neighbourhood} 
of $\cO''_{0, t+t',0}$, and hence in $\wt{\cO}_{0,  t+t', 0}$. This gives us 
the last inclusion in (\ref{eq:SetInclusion}).}

\MI{2. Note that, in the previous argument, we could replace the 
diffeomorphism $x_0^{t+t'}$ with the diffeomorphism $x_\delta^{t+t'}$. 
\corM{Indeed,} using Lemma \ref{Lem:Gron} along with Corollary \ref{cor:412}, 
we see that any point in $x_\delta^{t+t'}\left(\partial \wit{\cO}'(r)\right)$ 
must be at a distance $\frac{r}{3}$ from $\wt{\mathcal{O}}'_{t,t',\delta}$ 
for $h>0$ small enough. \corM{Then, similarly as in step 1 we deduce} that 
$\wt{\mathcal{O}}_{t,t',\delta}''$ contains an 
$\frac{r}{3}$-neighbourhood of $\wt{\mathcal{O}}'_{t,t',\delta}$. 
\corM{The inclusion \eqref{eq:SetInclusion2} then follows from another 
application of Lemma \ref{Lem:Gron}.}
}
\end{proof}
%%%%
% Indeed consider the last inclusion in \eqref{eq:SetInclusion}. It is enough to 
% show that $\wt{\Phi}^{-(t+t')}_0\circ\wt{\Phi}^t_\delta \circ \wt{\Phi}^{t'}_0 (\overline{\Lambda_0''})
% \subset \Lambda_0$. By Lemma \ref{Lem:Gron} we see that, for $\gd$ small enough, 
% there exists $\gamma>0$ such that  
% Indeed, thanks to Lemma \ref{Lem:Gron}, we see that, for $\gd$ small enough, 
% there exists $\gamma>0$ such that all points of $\overline{\wt{\mathcal{O}}''}_{t,t',\delta}$ 
% are at a distance at most $O(h^\gamma)$ from points of $\overline{\wt{\mathcal{O}}''}_{0,t+t',0}$. 
% Now,  using Lemma \ref{Lem:Gron} again, we deduce that 
% $\pi_X \wt{\Phi}_0^{-t}(\overline{\wt{\mathcal{O}}''}_{t,t',\delta})$ is in a neighborhood of size 
% $O(h^\gamma)$ of $\wt{\mathcal{O}}''$, which is thus included in $\wt{\mathcal{O}}$ for $h$ 
% small enough. Other inclusions are proved analogously.
% %
%
%We will also write $\wt{\phi}_{t,0}$, $\wt{\mathcal{O}}_{t,0}$ and  $\rho_0^{t}(y)=(x_0^{t}(y), \xi_0^{t}(y))$ instead of $\wt{\phi}_{0,t,0}$,  $\wt{\mathcal{O}}_{0,t,0}$ and $\rho_0^{0,t}(y) =(x_0^{0,t}(y), \xi_0^{0,t}(y))$.
%
%
Continuing, we note by \eqref{eq:SetInclusion} that the map \red{$\wt{\phi}_{t,0}$} is well-defined 
on $\wt{\mathcal{O}}''_{t,\delta}$. Moreover, since $\wt{\Phi}^{t,t'}_\delta(\Lambda_0)$ is projectable, 
the map $y \mapsto x_\delta^{t,t'}(y)$ is a diffeomorphism onto its image and we shall 
denote its inverse by $y_\delta^{-t, -t'}(x)$, or, more simply,  $y_\delta^{-t}(x)$ if $t'=0$. Then, the functions 
\begin{equation}\label{eq:CutOfffun2}
	\psi_t := \psi\circ y_\delta^{-t}, \quad \psi_{i,t}:=\psi_i\circ y_\delta^{-t}, 
	~~ i =0,\dots, N,
\end{equation}
are smooth functions in $x$ and $t$ and satisfy 
\begin{equation}\label{eq:CutOfffun31}
	\mathbf{1}_{\wt{\mathcal{O}}''_{\delta,t}} \succ 
	\psi_t \succ 
	\psi_{N,t} \succ \dots \succ \psi_{0,t} \succ 
	\mathbf{1}_{\wt{\mathcal{O}}'_{\delta,t}}. 
\end{equation}
\red{Note that transported cut-off functions are also $\delta$ and $\omega$ dependent 
although we will not denote this explicitly to ease the notation.}
\MI{Since we have $\overline{\wt{\mathcal{O}}''}_{t,\delta}\subset \wt{\mathcal{O}}_{t,\delta}$ by (\ref{eq:EnsemblesEmboites}),  the map}
%In particular 
$\wt{\phi}_{t,\delta}$ is well defined on $\supp\psi_t$ and on 
$\supp \psi_{i,t}$, $i=0,\dots,N$. We will give estimates on the derivatives \MI{of these functions}
further below in the text. 
\\
\par
The aim of this section is to prove the following statement, which describes 
a WKB approximation to the propagated state $\wit{u}(t,x;h,\delta)$, 
cf. \eqref{eq:trans1.0}.
\begin{prop}\label{prop:WKB}
Let $\widetilde{X}$ be the universal cover of a compact connected smooth 
Riemannian manifold $X$ of negative sectional curvature. Let $q_\omega$ 
be as in \eqref{eq:randomSymbol} and let $\wit{Q}_\omega$ be either the lift 
of $\Op_h(q_\omega)$ (in the \ref{eq:Pseudo}) or the multiplication operator with the function $\wt{q}_\omega
=\wt{\pi}^*q_\omega$ (in the \ref{eq:Pot}). Let $\widetilde{\mathcal{O}}\subset 
\widetilde{X}$ be a simply connected relatively compact open set, and let 
$\wt{\phi}_0 : \widetilde{\mathcal{O}}\longrightarrow \R$ be a smooth map such that 
the Lagrangian manifold
\begin{equation*}
\widetilde{\Lambda}_0 := \{ (y, d_y \wt{\phi}_0(y)) ; y\in \widetilde{\mathcal{O}}\}
\end{equation*}
is $D$-controlled,  has distortion $<D$ and is  $\eta$-unstable for some 
$\eta< \eta_0(D)$, as in Proposition \ref{Prop:OnTheCover2}. 
Let $a(x;h)$ be a smooth compactly supported function on $\widetilde{\mathcal{O}}$ 
with $\|a(\cdot;h)\|_{C^L} = O_L(1)$ and such that 
$a(x;h)\sim a_0(x) +ha_1(x) + \dots$ where the $a_n$ are smooth compactly supported 
maps on $\wit{\mathcal{O}}$ such that \eqref{hyp:compactSupport} holds. Let \red{$\psi_t,\psi_{n,t}$}
be as in \eqref{eq:CutOfffun2}, \eqref{eq:CutOfffun31}. 
\par
Then, the function 
\begin{equation}\label{eq:WKB_Ansatz}
u_\delta(t,x;h)
=\e^{\frac{i}{h} \wt{\phi}_{t,\delta}(x)} b(t,x;\delta,h)
=\e^{\frac{i}{h} \wt{\phi}_{t,\delta}(x)} 
\sum_{n=0}^{N-1}h^n b_n(t,x;\delta), 
\quad \red{x \in \wt{X}},
\end{equation}
with $\supp b(t,\cdot;h,\delta)\Subset \wt{\mathcal{O}}''_{t,\delta}$,  
$b_n$ given in \eqref{eq:transEq4}, 
\eqref{eq:transEq5}, 
and phase $\wt{\phi}_{t,\delta}$ given in \eqref{eq:eik0.0}, 
satisfies 
\begin{equation}\label{eq:WKB_Ansatz.0}
\begin{dcases}
	&ih\partial_t u_\delta 
	- 
	\widetilde{P}_{h}^\delta u_\delta   
	=
	-\frac{ih^{N+1}}{2}\e^{\frac{i}{h} \wt{\phi}_{t,\delta}(x)}
	\Delta_{\widetilde{g}} b_{N-1}
	+\delta\sum_{n=1}^{N-1}h^n (1-\psi_{n,t})\wt{Q}_\omega 
	\e^{\frac{i}{h} \wt{\phi}_{t,\delta}(x)}b_{n-1}
	\\
	&\phantom{ih\partial_t u - 	\widetilde{P}_h^\delta u  =}
	+\delta h^{N-1} \e^{\frac{i}{h} \wt{\phi}_{t,\delta}(x)} 
	\mathcal{R}_{\delta,t} b_{N-1}, \quad \red{\text{ on }
	\wit{\mathcal{O}}_{t,\delta}''}\\
	& u_\delta (0,x; h) = \e^{\frac{i}{h} \wt{\phi}_{0}(x)} 
	\sum_{n=0}^{N-1}h^n a_n(x), 
\end{dcases}
\end{equation}
where $\widetilde{P}_{h}^\delta:= -h^2\Delta_{\wt{g}}+\delta\wit{Q}_\omega$ and 
\begin{equation}\label{eq:WKB2.0}
	\mathcal{R}_{\delta,t}
	:=%\psi_t\left(
		\e^{-\frac{i}{h} \wt{\phi}_{t,\delta}(x)} \wit{Q}_\omega 
		 \e^{\frac{i}{h} \wt{\phi}_{t,\delta}(x)}
		-\wit{q}_\omega(x,d_x\wt{\phi}_{t,\delta}(x))
		-\frac{h}{i}H^1_{\wt{\phi}_{t,\delta},\wit{q}_\omega}
		-\frac{h}{2i}\mathrm{div}_{\wit{g}} (H^1_{\wt{\phi}_{t,\delta},\wit{q}_\omega}),
	%\right)
\end{equation}
and the vector field $H^1_{\wit{q}_\omega}$ is defined 
in local canonical coordinates by 
\begin{equation*}
	H^1_{\wt{\phi}_{t,\delta},\wit{q}_\omega}
	:=\sum_{k=1}^d \frac{\partial \wit{q}_\omega(x,d_x\wt{\phi}_{t,\delta})}{\partial \xi_k}
	\frac{\partial }{\partial x_k}. 
\end{equation*}
\red{In \eqref{eq:WKB2.0} the first and third term on the right hand side act as 
(pseudo-)differential operators whereas the second and fourth term act by multiplication 
on the functions $b_{n-1}$ in \eqref{eq:WKB_Ansatz.0}. Moreover, in the \ref{eq:Pot} $\mathcal{R}_{\delta,t}=0$.} 
\end{prop}
%
% \corM{ 
% \begin{rem}
% The definition of $\mathcal{R}_{\delta,t}$ may look at first sight puzzling, 
% however, it is a direct consequence of the transport equations \eqref{eq:WKB2}.
% \end{rem}}
% \red{
% \begin{rem}
% 	% Before continuing let us give two comments on Proposition \ref{prop:WKB}. 
% 	% First, the introduction of the cut-off function $\psi_t$ in on the right 
% 	% hand side of the first line of \eqref{eq:WKB_Ansatz.0} serves the purpose 
% 	% that it allows us to conjugate the restricted operator $\psi_t\wt{Q}_\omega$ 
% 	% with $\e^{\frac{i}{h} \wt{\phi}_{t,\delta}(x)}$ as the latter is well-defined on 
% 	% the support of $\psi_t$. Since the functions $b_{n-1}$ are supported where 
% 	% $\psi_t=1$, the pseudo-locality of $\wt{Q}_\omega$ make this modification 
% 	% negligible for our purposes. This will be made precise in 
% 	% (\ref{eq:ProjectedLagrangianState4.1}-\ref{eq:ProjectedLagrangianState4.2b}) 
% 	% below. Secondly, 
% 	The definition of $\mathcal{R}_{\delta,t}$ may look at 
% 	first sight puzzeling, however, it is a direct consequence of the 
% 	transport equations \eqref{eq:WKB2}. 
% \end{rem}
% }
%
We will present detailed results concerning the regularity of this 
WKB solution and a precise remainder estimate in Section 
\ref{sec:RegularityOfWKB} below. For now, we 
turn to the proof of the above proposition which will take up the 
rest of the current section.  
\subsection{The WKB Ansatz}
Let us now explain how, for any $N\in\N$, the state \eqref{eq:trans1.0}
can be approximated by the WKB state $u_\delta$ defined in 
\eqref{eq:WKB_Ansatz}. We note that 
each $b_n$ can also depend on $h>0$ when $\delta>0$, but we will 
not denote this explicitly. We want $u_\delta$ to solve 
$ih\partial_t u_\delta = \wt{P}_{h}^\delta u_\delta$, i.e. 
\begin{equation}\label{eq:WKB}
\begin{split}
	&(- b \partial_t \wt{\phi}_{t,\delta} + ih \partial_t b)\\
	&= 
	\frac{1}{2}
	\left[
		b |d_x \wt{\phi}_{t,\delta}|_x^2 - ih b \Delta_g \wt{\phi}_{t,\delta}
 		- 2 i h \langle d_x \wt{\phi}_{t,\delta}, d_x b \rangle_x 
		 -h^2 \Delta_g b
	 \right] 
	+ \delta \e^{-\frac{i}{h} \wt{\phi}_{t,\delta}} \wt{Q}_\omega \e^{\frac{i}{h} \wt{\phi}_{t,\delta}} b, 
\end{split}
\end{equation}
\red{on $\mathcal{O}''_{t,\delta}$} up to a small remainder. 
\\
\par
\textbf{Case 1: $Q_{\omega}$ is a random potential}. Consider first the case 
where $\wt{Q}_\omega$ is the operator of multiplication by the function $\wt{q}_\omega$. 
Hence, $P_{h}$ is a differential operator and $H^1_{\wt{\phi}_{t,\delta},\wit{q}_\omega}\equiv 0$ 
and $\mathcal{R}_{\delta,t}\equiv 0$. We can study equation \eqref{eq:WKB} 
on $\wt{\mathcal{O}}''_{\delta,t}$. Furthermore,
\begin{equation*}
	\left(
		\e^{-\frac{i}{h} \wt{\phi}_{t,\delta}} \wt{Q}_\omega  \e^{\frac{i}{h} \wt{\phi}_{t,\delta}}b
		\right)(x) 
	= \wt{q}_\omega(x) b(x).
\end{equation*}
Expanding $b$ as in \eqref{eq:WKB_Ansatz} and regrouping the terms in 
\eqref{eq:WKB} according to matching powers of $h$, we see that the 
functions $\wt{\phi}_t$ and $b_n$ must satisfy the following differential equations
\begin{equation}\label{eq:WKB2a}
\begin{cases}
 \partial_t \wt{\phi}_{t,\delta} 
 	&= -\wt{p}(x,d_x\wt{\phi}_{t,\delta}(x);\delta),\\
 	\partial_tb_0 
	&=  -\frac{1}{2} b_0 \Delta_g \wt{\phi}_{t,\delta} 
	- \langle d_x  \wt{\phi}_{t,\delta}, d_x b_0 \rangle_x, 
 \\
  \partial_t b_n 
  &= -\frac{1}{2} b_n \Delta_g \wt{\phi}_{t,\delta} 
  - \langle d_x  \wt{\phi}_{t,\delta}, d_x b_n \rangle_x
  + \frac{i}{2}\Delta_g b_{n-1},
  ~~n\geq 1, \quad \red{\text{on } \wt{\mathcal{O}}''_{\delta,t}.}
\end{cases}
\end{equation}
The initial condition $\widetilde{u}|_{t=0} = f$ implies that we must 
require that $\wt{\phi}_{0,\delta}=\wt{\phi}_0$ and that $b_n|_{t=0} = a_n$. This 
gives us initial conditions for each of equation in \eqref{eq:WKB2a}. 
\red{Suppose that we have solved these differential equations up to order $n=N-1$, 
such that the support of $b_n(t,\cdot;\delta)$ is contained in the support 
of $\psi_{0,t}$ for every $n=0,\dots,N-1$, then the function $u_\delta (t,x;h)$ of \eqref{eq:WKB_Ansatz} satisfies 
\eqref{eq:WKB_Ansatz.0}. Notice that by \eqref{eq:CutOfffun31} and 
since $\wt{Q}_{\omega}$ is a multiplication operator, the second term 
on the right hand side in the first line of \eqref{eq:WKB_Ansatz.0} 
is equal to $0$.}
\\
\par
\textbf{Case 2: $Q_{\omega}$ is a random pseudo-differential operator.}
When $\wt{Q}_\omega$ is a pseudo-differential operator we follow formally the 
same strategy as above and solve the pseudo-differential equations 
\begin{equation}\label{eq:WKB2}
\begin{cases}
 \partial_t \wt{\phi}_{t,\delta} 
 	&= -\wt{p}(x,d_x\wt{\phi}_{t,\delta}(x);\delta),\\
 	\partial_tb_0 
	&=  \mathcal{L}b_0
 \\
  \partial_t b_n 
  &= \mathcal{L}b_n
  + \frac{i}{2}\Delta_g b_{n-1}
  -ih^{-2}\delta \psi_{n,t}\mathcal{R}_{\delta,t} b_{n-1},
  ~~n\geq 1,  \quad \red{\text{on } \wt{\mathcal{O}}''_{\delta,t}.}
\end{cases}
\end{equation}
where $\mathcal{L}$ is the differential operator defined by
\begin{equation}
	\mathcal{L}u := 
	-\frac{1}{2} u \Delta_g \wt{\phi}_{t,\delta} 
	- \langle d_x  \wt{\phi}_{t,\delta}, d_x u \rangle_x
	-\delta	(H^1_{\wt{\phi}_{t,\delta},\wt{q}_\omega}u)
	- \frac{\delta}{2} \mathrm{div}_g (H^1_{\wt{\phi}_{t,\delta},\wt{q}_\omega})
	u.
\end{equation}
\red{
\begin{rem}
	Notice that the last two terms in the definition of $\mathcal{L}$ are 
	somewhat artificial as we have added them to $\mathcal{L}$ and substracted 
	them again in the error term $\mathcal{R}_{\delta,t}$. This regrouping 
	allows us to show that the solution $b_0$ is given by the (unitary) evolution of 
	$a_0$ under the flow $\pi_X\wt{\Phi}^t_\delta$, see Proposition \ref{lem:TransportEqSol} 
	below for details.
\end{rem}
}
The initial condition $\widetilde{u}|_{t=0} = f$ implies that we  
require that $\wt{\phi}_{0,\delta}=\wt{\phi}_0$ and that $b_n|_{t=0} = a_n$. This 
gives us initial conditions for each of equation in \eqref{eq:WKB2}. 
\red{Suppose that we have solved these differential equations up to order 
$n=N-1$, such that the support of $b_n(t,\cdot;\delta)$ is contained in the support 
of $\psi_{n,t}$ for every $n=0,\dots,N-1$, then the function $u_\delta (t,x;h)$ of 
\eqref{eq:WKB_Ansatz} satisfies \eqref{eq:WKB_Ansatz.0}.}
\\
\par
\red{
\paragraph{\textbf{Overview over the proof of Proposition \ref{prop:WKB}.}} 
The first equation in \eqref{eq:WKB2a} and \eqref{eq:WKB2} is called 
the \emph{Hamilton-Jacobi equation}, whereas the second and the third 
equations are called the $0^{th}$ order and $n^{th}$ order \emph{transport equations}, 
respectively. We will present the proof of Proposition \ref{prop:WKB} 
only in Case 2, since the proof in Case 1 follows the same steps. 
In Section \ref{sec:EikonaEquation} below, we will solve the Hamilton-Jacobi 
equation, in Section \ref{sec:phaseDer} we will discuss estimates on the 
derivatives of the solution to the Hamilton-Jacobi equation, and in Section 
\ref{sec:SolvTransport} below, we will solve the transport equations iteratively 
which then completes the proof of Proposition \ref{prop:WKB} together with 
the above discussion. 
}
\begin{rem}\label{RemStephane}
When writing a function as a Lagrangian state $a(x) \e^{\frac{i}{h} \wt{\phi}(x)}$ 
with $a$ and $\wt{\phi}$ depending on $h$, there is some freedom in the choice of 
$a$ and $\wt{\phi}$: the only essential requirement is that the amplitude $a(x)$ 
should oscillate at a scale \red{$\gg \sqrt{h}$}, to remain in a nice symbol class.
\\
In particular,  equation (\ref{eq:WKB}) will also be satisfied up to a remainder 
$O(h^N)$ if we have
\begin{equation}\label{eq:WKBAlternativ}
\begin{cases}
 \partial_t \wt{\phi}_{t,0} 
 	&= -\wt{p}(x,d_x\wt{\phi}_{t,0}(x);0),\\
 	\partial_t b_0 
	&=  
 	-\frac{1}{2} b_0 \Delta \wt{\phi}_{t,0} 
	 	- \langle d_x \wt{\phi}_{t,0}, d_x b_0 \rangle_x 
		-i h^{-1} \delta \e^{-\frac{i}{h} \wt{\phi}_{t,0}}
		\wt{Q}_\omega \e^{\frac{i}{h} \wt{\phi}_{t,0}} b_0,
 \\
  \partial_t b_n 
  &=  -\frac{1}{2} b_n \Delta \wt{\phi}_{t,0}
  - \langle d_x \wt{\phi}_{t,0}, d_x b_n \rangle_x + \frac{i}{2}\Delta b_{n-1}
%   -ih^{-2}\delta \mathcal{Q}_{\omega,t} b_{n-1} 
	-i h^{-1} \delta  \e^{-\frac{i}{h} \wt{\phi}_{t,0}}
	\wt{Q}_\omega \e^{\frac{i}{h} \wt{\phi}_{t,0}} b_n,
  ~~n\geq 1.
% - 2 \langle d_x \wt{\phi}_{t, \delta}, d_x a \rangle_x - a \Delta \wt{\phi} + ih \Delta a .
\end{cases}
\end{equation}
With such a formulation,  the phase $\wt{\phi}_{t,0}$ does not depend on the 
perturbation, while the amplitudes \red{depend on the perturbation}.
\\
Using techniques analogue to those of \cite{NZ}, it should be possible to 
use the formulation (\ref{eq:WKBAlternativ}) to describe the propagation 
of Lagrangian states by $U_h^\delta(t)$ up to times $M |\log h|$ for an 
arbitrary $M>0$, while in the sequel, we will only be able to consider 
times that are $\gd |\log h|$ for some small $\gd$. 
Such an approach could thus be a starting point to extend Theorem 
\ref{th:MartinEtMaximeSontDesBeauxGosses} to larger times; this will 
be pursued elsewhere.
\\
However,  to have a good control over the functions $b_n$ starting 
from (\ref{eq:WKBAlternativ}), we need the phase depending on the 
perturbation to oscillate less rapidly than $\sqrt{h}$, for which 
we need to have $\alpha> \frac{1}{2}$. This is thus more restrictive 
than the situation covered by Hypothesis \ref{Hyp:BetaDelta}. This is 
why we focus on the formulation (\ref{eq:WKB2}) in this paper.
\par
Furthermore, the independence considerations of Section \ref{Sec:Indep} 
give strong constraints on the time scales on which we are working, 
making it complicated to extend Theorems \ref{th:MartinEtMaximeSontDesSuperBeauxGosses} 
and \ref{th:MartinEtMaximeSontDesBeauxGosses} up to arbitrary logarithmic times.
\end{rem}
\subsection{Solving the Hamilton-Jacobi equation \red{for $\wt{\phi}_{t,\delta}$}}\label{sec:EikonaEquation}
\subsubsection{Applying the method of characteristics}
\par
\corM{Solving the Hamilton-Jacobi equations in \eqref{eq:WKB2a}, \eqref{eq:WKB2} is 
standard but we will discuss here one approach for the readers' convenience.} 
Recall that, since $\wt{\Phi}^t_\delta(\Lambda_0)$ is projectable, the map $y \mapsto x_\delta^t(y)$ 
is a diffeomorphism onto its image and that we denote its inverse by 
$y_\delta^{-t}(x)$. 
%\pu{More generally, if $t,t'\geq 0$ with $t+t'\leq \gd |\delta|$, the map $y\mapsto x^{t,t'}_\delta(y) := ...$ A FINIR}
Hence, 
\begin{equation}\label{eq:eik1.1}
	d_x \wt{\phi}_{t,\delta} = \xi_\delta^t(y_\delta^{-t}(x)),
\end{equation}
or, equivalently,  
\begin{equation}\label{eq:eik1}
d_{x_\delta^t(y)} \wt{\phi}_{t,\delta} = \xi_\delta^t(y).
\end{equation}
\par
Let $y_0\in \wt{\mathcal{O}}_0$ and let $t_0 >0$. Fix a system of local coordinates 
$(y^1,..., y^d)$ for $y$ near $y_0$, and fix another system of local coordinates 
$(x^1,..., x^d)$ for $x$ in a small neighbourhood of $x_\delta^{t_0}(y_0)\in 
\wt{\mathcal{O}}_{t_0, \delta}$. Notice that by taking a small enough neighbourhood around 
$x_\delta^{t_0}(y_0)$ we can arrange so that this system of coordinates 
$(x^1,..., x^d)$ is a system of coordinates near $x_\delta^{t}(y_0)$ for all 
$t\in]t_0-\varepsilon,t_0+\varepsilon[$ for $\varepsilon >0$ sufficiently small. 
We also take the canonical coordinates $(\xi_1,\dots, \xi_d)$ in 
$T^*_x\wt{\mathcal{O}}_{t,\delta}$ induced by the coordinate system $(x^1,..., x^d)$. 
Notice that both those coordinate systems are independent of 
$t\in]t_0-\varepsilon,t_0+\varepsilon[$. We shall write 
$(x_{\delta,1}^t(y),..., x_{\delta,d}^t(y))$ for the coordinates of 
$x_\delta^t(y)$, and $(\xi_{\delta,1}^t(y),..., \xi_{\delta,d}^t(y))$ for the 
coordinates of $\xi_{\delta}^t(y)$, in this system.
\par
In local coordinates, equation \eqref{eq:eik1} takes the form
\begin{equation*}
\left(\frac{\partial}{\partial x^i} \wt{\phi}_{t,\delta}\right)(x_\delta^t(y)) 
	=  \xi_{\delta,i}^t(y),
\end{equation*}
and the Hamilton equations $\frac{d}{dt} \wt{\Phi}^t_\delta(\rho_\delta^t) 
= H_p|_{\rho_\delta^t}$ take the form  
\begin{equation}\label{eq:HamiltonEq}
	\begin{dcases}
	\frac{d }{d t}x_{\delta,i}^t = \partial_{\xi_{i}} \wt{p}(\rho_\delta^t;\delta),\\
	\frac{d}{d t}  \xi_{\delta,i}^t= -\partial_{x_{i}} \wt{p}(\rho_\delta^t;\delta)
	\end{dcases}
\end{equation}
Thus, using the Einstein summation convention, we obtain in local coordinates
\begin{equation*}
\begin{split}
\frac{\partial}{\partial {y^j}} 
	\bigg[ \left(\frac{\partial}{\partial t}  
	\wt{\phi}_{t,\delta}\right)& (x_\delta^t(y))\bigg]
=  \frac{\partial x_{\delta,i}^t(y)}{\partial y^j} 
	\left(\frac{\partial^2 }{\partial t \partial {x^i}} 
	\wt{\phi}_{t,\delta}\right) (x_\delta^t(y))\\
&=  \frac{\partial x_{\delta,i}^t(y)}{\partial y^j} 
	\left[ \frac{\partial}{\partial t} 
	\left( \left( \frac{\partial}{\partial x^i} \wt{\phi}_{t,\delta} \right)
	\!(x_\delta^t(y))
	\right) 
	- \left(\frac{\partial}{\partial t} x_{\delta,k}^t(y)\right)  
	\left( \frac{\partial^2}{\partial x^i \partial x^k} \wt{\phi}_{t,\delta}\right)
	\!(x_\delta^t(y)) \right]\\
&=  \frac{\partial x_{\delta,i}^t(y)}{\partial y^j} 
	\left[ \frac{\partial}{\partial t} \xi_{\delta,i}^t(y)
	 -\left(\frac{\partial}{\partial t} x_{\delta,k}^t(y)\right) 
	 \!\left( \frac{\partial^2}{\partial x^i \partial x^k} 
 	\wt{\phi}_{t,\delta}\right)\!(x_\delta^t(y)) \right]\\
&=  \frac{\partial x_{\delta,i}^t(y)}{\partial y^j}
	\left[ -\left(\frac{\partial}{\partial x^i} p\right)
		\!(\rho_\delta^t(y);\delta) 
	-\left( \frac{\partial}{\partial \xi^k} p \right)\!(\rho_\delta^t(y);\delta)
	\times 
	\left(\frac{\partial^2}{\partial x^i \partial x^k} \wt{\phi}_{t,\delta}\right)
	\!(x_\delta^t(y)) \right]\\
&= -\frac{\partial}{\partial y^j} p (\rho_\delta^t(y);\delta).
\end{split}
\end{equation*}
Hence, 
\begin{equation}\label{eq:eik2}
d_y (\partial_t  \wt{\phi}_{t,\delta}\circ x_\delta^t(y)) 
	=- d_y p (\rho_\delta^t(y);\delta), 
\end{equation}
i.e. the maps $y\mapsto (\partial_t\wt{\phi}_{t,\delta})(x^t_\delta(y)))$ and 
$p (\rho_\delta^t(y);\delta))$ have the same gradient, so they must coincide 
up to a constant $c_\delta(t)$, which we can take equal to zero, since 
$\wt{\phi}_{t,\delta}$ is defined up to an additive constant. Therefore, 
\begin{equation}%\label{eq:PhaseConstantFlot}
 \left(\partial_t \wt{\phi}_{t,\delta}\right) (x_\delta^t(y)) 
 	= - p (\rho_\delta^t(y);\delta) 
 	%= - p\left(x_h^t(y), d_{x_h^t(y)}\wt{\phi}_t (x_h^t(y));h\right) 
	=  -p (y, d_y \wt{\phi}_0(y); \delta),
 \end{equation}
since $p$ is invariant under the action of the flow $\wt{\Phi}^t_\delta$. 
\par
If $x\in \wt{\mathcal{O}}_{t,\delta}$, we may thus take 
$y = y_{\delta}^{-t}(x)$ to obtain that
\begin{equation*}
	 \partial_t \wt{\phi}_{t,\delta}(x) 
	 = 
	 - p\!\left(x, d_x \wt{\phi}_{t,\delta};\delta\right). 
\end{equation*}
Therefore, the map $ \wt{\phi}_{t,\delta}(x)$ is a  solution of the system
\begin{equation}\label{eq:Eikonal}
 \begin{cases}
	\frac{\partial}{\partial t} \wt{\phi}_{t,\delta}(x) 
	= 
	- \wt{p}(x,d_x \wt{\phi}_{t,\delta};\delta)
	\\
	\psi_h \big|_{t=0} = \wt{\phi}_0.
\end{cases}
\end{equation}
\subsubsection{An approximate solution for the Hamilton-Jacobi equation}
\red{Recall first the notions introduced in the paragraph before \eqref{eq:CutOfffun2}. 
We want to compare the solution of the perturbed and unperturbed Hamilton-Jacobi 
equations. In view of \eqref{eq:SetInclusion}, we will work therefore on $\wt{\mathcal{O}}''_{t,\delta}$. 
Let $x\in \wt{\mathcal{O}}''_{t,\delta}$}, and let $0\leq s\leq t$. We define the point 
(see Figure \ref{fig:Zeta})
\begin{equation*}
\zeta^{s,t}_\delta(x)
:= \rho_0^{t-s}\left( y_\delta^{-s, -(t-s)}(x)\right)
 = \wt{\Phi}^{-s}_\delta \left( x, d_x \wt{\phi}_{s,t-s,\delta} \right).
\end{equation*}
We claim that $\zeta^{s,t}_\delta(x)$ does not depend on $t$ \red{in the 
following sense}
\begin{equation}\label{eq:PIndepT}
\text{ for all } 0\leq t,t' \text{ with $t+t' \leq \gd \log \delta$,} ~~ 
\left(x\in \wt{\mathcal{O}}''_{t,\delta}\cap  \wt{\mathcal{O}}''_{t',\delta}\right) 
\Longrightarrow \left(\zeta^{s,t}_\delta(x) = \zeta^{s,t'}_\delta(x)\right).
\end{equation}
Indeed, by definition,  $y_\delta^{-s, -(t-s)}(x)$ is the unique point 
$y$ such that $\pi_X\left(\wt{\Phi}_\delta^{s} \left( \wt{\Phi}_0^{t-s}  (y, d_y \wt{\phi}_0) 
\right)\right)= x$. Similarly, $y_\delta^{-s, -(t'-s)}(x)$ is the only point 
$y'$ such that $\pi_X\left(\wt{\Phi}_\delta^{s} \left( \wt{\Phi}_0^{t'-s}  (\red{y'}, d_{\red{y'}} \wt{\phi}_0) 
\right)\right)= x$. Hence, since by Remark \ref{RemMixedFlow} the Lagrangian manifold 
$\wt{\Phi}^s_\delta\circ \wt{\Phi}^{t-s}_0(\Lambda_0)$ is projectable, it follows that 
$y' = \pi_X\left( \rho_0^{t-t'}(y)\right)$, from which (\ref{eq:PIndepT}) follows. 
\definecolor{ffqqqq}{rgb}{1,0,0}
\definecolor{qqzzff}{rgb}{0,0.6,1}
\definecolor{qqqqff}{rgb}{0,0,1}

\begin{figure}
\begin{center}
\begin{tikzpicture}[scale=1.2, line cap=round,line join=round,>=triangle 45,x=1cm,y=1cm]
\clip(-5.174962280502898,-1.4355416830535686) rectangle (6.75076285489357,6.946015855130226);
\draw [rotate around={0:(0,0)},line width=1pt,color=qqqqff,fill=qqqqff,fill opacity=0.08] (0,0) ellipse (2.2360679774997907cm and 1cm);
\draw [rotate around={0:(0,4)},line width=1pt,color=qqzzff,fill=qqzzff,fill opacity=0.09] (0,4) ellipse (3.1622776601683817cm and 1cm);
\draw [->,line width=1pt,color=qqqqff] (0,0) -- (0,1);
\draw [->,line width=1pt,color=qqqqff] (2,0) -- (2,1);
\draw [->,line width=1pt,color=qqqqff] (-2,0) -- (-2,1);
\draw [->,line width=1pt,color=qqqqff] (1,0) -- (1,1);
\draw [->,line width=1pt,color=qqqqff] (-1,0) -- (-1,1);
\draw [->,line width=1pt,color=qqzzff] (0,4) -- (0,5);
\draw [->,line width=1pt,color=qqzzff] (1,4) -- (1.5,5);
\draw [->,line width=1pt,color=qqzzff] (2,4) -- (3,5);
\draw [->,line width=1pt,color=qqzzff] (-1,4) -- (-1.5,5);
\draw [->,line width=1pt,color=qqzzff] (-2,4) -- (-3,5);
\draw [line width=1pt,color=ffqqqq] (0,0)-- (0,4);
\draw [shift={(0.5330781525065621,4.019748055717795)},line width=1pt,color=ffqqqq]  plot[domain=1.8509267705166443:3.1786210503940184,variable=\t]({1*0.5334438137090377*cos(\t r)+0*0.5334438137090377*sin(\t r)},{0*0.5334438137090377*cos(\t r)+1*0.5334438137090377*sin(\t r)});
\draw [shift={(0.19740715575977225,4.9795520314820765)},line width=1pt,color=ffqqqq]  plot[domain=-1.1724477287520685:0.2668907851049957,variable=\t]({1*0.4851391580088101*cos(\t r)+0*0.4851391580088101*sin(\t r)},{0*0.4851391580088101*cos(\t r)+1*0.4851391580088101*sin(\t r)});
\draw [shift={(1.1398878169067652,5.538513143504153)},line width=1pt,color=ffqqqq]  plot[domain=2.5855198261779226:3.878985255403154,variable=\t]({1*0.6410457709464107*cos(\t r)+0*0.6410457709464107*sin(\t r)},{0*0.6410457709464107*cos(\t r)+1*0.6410457709464107*sin(\t r)});
\draw [->,line width=1pt,color=ffqqqq] (0.5954254429055894,5.876892226405684) -- (1.0729773820653716,6.565913061970178);
\draw [color=qqqqff](2.5458757055605528,0.6980761925619469) node[anchor=north west] {$\Lambda_0$};
\draw [color=qqzzff](2.771561738999331,3.513212504403548) node[anchor=north west] {$\wt{\Phi}_0^{t-s}(\Lambda_0)$};
\draw [color=ffqqqq](0.8948041977715996,6.363983453103904) node[anchor=north west] {$(x, d_x \wt{\phi}_{s,t-s,\delta})$};
\draw [color=ffqqqq](-1.0769790417461431,4.095244906429871) node[anchor=north west] {$\zeta_{\delta}^{s,t}(x)$};
\draw [color=ffqqqq](-0.7800237345898566,-0.2922873659631213) node[anchor=north west] {$y_\delta^{-s,-(t-s)}(x)$};
\begin{scriptsize}
\draw [fill=ffqqqq] (0,0) circle (2pt);
\draw [fill=ffqqqq] (0,4) circle (2pt);
\draw [fill=ffqqqq] (0.5954254429055894,5.876892226405684) circle (2pt);
\end{scriptsize}
\end{tikzpicture}
\end{center}
\caption{Construction of the point $\zeta_\delta^{s,t}(x)$}\label{fig:Zeta}
\end{figure}
Note that, if $x\in \wt{\mathcal{O}}''_{s,\delta}$, we have the simpler expression
$$\zeta^{s,t}_\delta(x) = \wt{\Phi}_\delta^{-s} (x, d_x \wt{\phi}_{s,\delta}).$$
Next, we claim that 
\begin{equation}\label{eq:PhiIntegrale2}
	\wt{\phi}_{t,\delta}(x) 
	= %\wt{\phi}_0(y_0^{-t}(x))
	\wt{\phi}_{t,0}(x) - \delta
		\int_0^t \wt{q}_\omega\left(\zeta_\delta^{s,t}(x)\right) d s, 
	\quad \red{x\in \wt{\mathcal{O}}''_{t,\delta}.}
\end{equation}
Indeed, both expressions coincide when $t=0$, and, thanks to (\ref{eq:PIndepT}), 
we see that the integrand does not depend on $t$, so that the derivative with respect 
to $t$ of the right-hand side is equal to
\begin{align*}
 -\frac{1}{2} - \delta 
\wt{q}_\omega(\wt{\Phi}_\delta^{-t} (x, d_x \wt{\phi}_{t,\delta}))
&= -  \frac{1}{2}|d_{y_\delta^{-t}(x)} \wt{\phi}_0 |_x^2 
- \delta \wt{q}_\omega(\wt{\Phi}_\delta^{-t} (x,d_x \wt{\phi}_{t,\delta}))
~~\text{ (since $\wt{\phi}_0$ is monochromatic)}\\
&=  - \wt{p}( y_\delta^{-t}(x), d_{y_\delta^{-t}(x)} \wt{\phi}_0 ;\delta)\\
&= -\wt{p}(\wt{\Phi}_\delta^{-t}(x, d_x \wt{\phi}_{t,\delta}));\delta)\\
&= - \wt{p}(x, d_x \wt{\phi}_{t,\delta};\delta).
\end{align*}
In the last line we used that $\wt{p}(\cdot, \cdot; \delta)$ is invariant under  
$\wt{\Phi}^t_\delta$, the flow generated by the Hamilton vector field 
\eqref{eq:liftedHV}.
\par
Note that, when \red{$x\in \wt{\mathcal{O}}''_{s,\delta}$} for all $s\in [0,t]$, 
equation (\ref{eq:PhiIntegrale2}) reduces to
\begin{equation*}
	\wt{\phi}_{t,\delta}(x) 
	= %\wt{\phi}_0(y_0^{-t}(x))
	\wt{\phi}_{t,0}(x) - \delta
		\int_0^t \wt{q}_\omega(\wt{\Phi}_\delta^{-s} (x, d_x \wt{\phi}_{s,\delta}))d s.
\end{equation*}
In the sequel, we will often want to have an approximate expression 
for $\wt{\phi}_{t,\delta}$. Such approximations will rely on the following lemma.

\begin{lem}\label{lem:BoundDiffFlows}
Suppose that $\beta, \delta, \varepsilon_0$ are as in (\ref{eq:CondBeta}). 
Then for any $0<\varepsilon < \varepsilon_0$, there exists $\gd>0$ and $C>0$ 
such that 
for all  $t\in \R$ with $0\leq t \leq \mathfrak{d} |\log h|$, we have 
for all \red{$x\in \wt{\mathcal{O}}''_{t,\delta}$} and all $0\leq s \leq t$
\begin{equation*}
	\mathrm{dist}_{T^*X}
	\left(
		(x, d_x \wt{\phi}_{t,0}), (x, d_x \wt{\phi}_{s,t-s,\delta}) 
	\right) 
	\leq 
	C\delta h^{-\beta- \varepsilon},
\end{equation*}
and, in particular,
\begin{equation*}
	\mathrm{dist}_{T^*X}
	\left(
		(x, d_x \wt{\phi}_{t,0}), (x, d_x \wt{\phi}_{t,\delta}) 
	\right) 
	\leq 
	C\delta h^{-\beta- \varepsilon}.
\end{equation*}
\end{lem}

\begin{proof}
Let us write $y' = y_\delta^{-s, -(t-s)}(x)$ and $\rho_0 :=(x, d_x \wt{\phi}_{t,0})$, 
$\rho_\delta := (x, d_x \wt{\phi}_{s,t-s,\delta})$. By definition, we have 
$\rho_\delta = \wt{\Phi}^s_\delta \left( \wt{\Phi}^{t-s}_0 \left(y', d_{y'}
\wt{\phi}_0 \right)\right)$. Let us also consider $\rho'_0 :=  \wt{\Phi}_0^t (y',  d_{y'}\wt{\phi}_0)$. 
Thanks to Lemma \ref{Lem:Gron}, we know that for any $\varepsilon>0$, 
provided $\gd$ is small enough, we have
\begin{equation}\label{eq:LePtiGronwallDeLaFin}
	\mathrm{dist}_{T^*X} ( \rho_\delta, \rho'_0) \leq C_0\delta 
	h^{-\beta - \varepsilon}.
\end{equation}
If we write $x'_0:= \pi_X (\rho'_0)$, we thus have
\begin{equation}\label{eq:DistSansPerturb}
\mathrm{dist}_X(x, x'_0) \leq c C_0\delta h^{-\beta - \varepsilon},
\end{equation}
since, for any $\rho_1, \rho_2\in T^*X$,  $\mathrm{dist}_{T^*X}
(\rho_1,\rho_2)\geq c\dist_X(\pi_X (\rho_1),\pi_X (\rho_2))$,
see Section \ref{sec:notation}.
Since $\rho'_0\in \wt{\Phi}^t_0(\Lambda_0)$, it can be written as 
$\rho'_0 = (x'_0, d_{x'_0} \wt{\phi}_{t,0})$, and equation (\ref{eq:BorneCL}) 
below with $L=2$ along with (\ref{eq:DistSansPerturb}) implies that, 
provided that $\gd$ is small enough,
  \begin{equation*}
  \mathrm{dist}_{T^*X}(\rho_0, \rho'_0) \leq \red{C h^{-\varepsilon} }
  	\mathrm{dist}_X (x,x'_0) 
	\leq C\delta h^{-\beta- 2\varepsilon }.
  \end{equation*}
 Combining thus with (\ref{eq:LePtiGronwallDeLaFin}), we deduce 
\begin{equation*}
  \mathrm{dist}_{T^*X}(\rho_0, \rho_\delta) \leq C' \delta 
  	h^{-\beta - 2 \varepsilon} 
\end{equation*} 
The result follows by taking $\varepsilon>0$ and thus 
$\gd>0$ smaller.
\end{proof}
Combining Lemmas \ref{lem:BoundDiffFlows} and \ref{Lem:Gron}, we 
obtain that there exists a $C>0$ such that for any $t\in [0, \gd |\log h|]$, 
any $s \in [0,t]$ and all \red{$x\in \wt{\mathcal{O}}''_{t,\delta}$}, we have
\begin{equation}\label{eq:PetitCoupDeStressDeLaFin}
	\mathrm{dist}_{T^*X}
	\left( \zeta_\delta^{s,t}(x),  \wt{\Phi}_0^{-s}
		(x, d_x \wt{\phi}_{t,0})
	\right) 
	\leq 
	C\delta h^{-\beta- 2 \varepsilon}.
\end{equation}
provided $\gd$ is small enough. Therefore,  (up to taking 
$\varepsilon$ and $\gd$ smaller), we obtain from (\ref{eq:PhiIntegrale2}) that
\begin{equation}\label{eq:PhaseApprochee}
\wt{\phi}_{t,\delta}(x) 
=  \wt{\phi}_{t,0}(x)
- \delta \int_0^t  \wt{q}_\omega 
	\left( \wt{\Phi}_{0}^{-s} (x, d_x \wt{\phi}_{t,0}) \right)
	 d s + O \left( \delta^2 h^{-2\beta-2\varepsilon} \right).
\end{equation}

\subsection{Controlling the derivatives of the phase}\label{sec:phaseDer}

In this section we provide estimates on the $\dot{C}^L$ norms of the time evolved 
phase $\wt{\phi}_{t,\delta}$.  \MI{Recall that these norms were introduced in \eqref{eq:sa8.1b}.}

\begin{prop}\label{lem:ControlDeriv}
Suppose that $\beta, \delta, \varepsilon_0$ are as in (\ref{eq:CondBeta}). 
For any $0<\varepsilon<\varepsilon_0$, there exist $\mathfrak{d}>0$,  
and $h_0>0$ such that the following holds: 
for all $L\in \N$,  there exists a constant $C_L>0$ such that 
for all $h\in (0,h_0]$ and all $t \in [0,\mathfrak{d}|\log h|]$
we have that
\begin{equation}\label{eq:BorneC1}
	\|\wt{\phi}_{t,\delta}\|_{\dot{C}^1(\wt{\mathcal{O}}_{t,\delta})} \leq C_1 .
\end{equation}
and, if $L\geq 2$
\begin{equation}\label{eq:BorneCL}
	\|\wt{\phi}_{t,\delta}\|_{\dot{C}^L(\wt{\mathcal{O}}_{t,\delta})} 
	\leq C_L h^{- (L-2)( \beta + \varepsilon) - \varepsilon}.
\end{equation}

\end{prop}
Before proving the proposition, we shall prove two \red{results} concerning 
estimates on compositions of functions, which we will use several 
times in the sequel. 

\red{
\begin{lem}\label{lem:CLCompo}
Let $f_h: \R^d\supset \Omega \longrightarrow \R^{d'}$ and 
$g_h : \R^{d''} \supset \Omega' \longrightarrow \Omega$ be smooth 
functions.  Let $k\in \N$, and suppose that there exists $\sigma_1, \sigma_1', \sigma_2, \sigma_3 \geq 0$ 
with $\sigma_1+\sigma_2\leq \sigma_1'$ such that for all $1\leq k'\leq k$, we have
\begin{align*}
\|f_h\|_{\dot{C}^{k'}} &= O(h^{-\sigma_3 -(k'-1) \sigma_1})\\
\|g_h\|_{\dot{C}^{k'}}&= O(h^{-\sigma_2 - (k'-1) \sigma'_1}).
\end{align*}
We then have
\begin{equation*}
	\|f_h\circ g_h\|_{\dot{C}^k} = O(h^{- \sigma_3  -\sigma_2 -(k-1)\sigma'_1}).
\end{equation*} 
\end{lem}
}
\red{
\begin{proof}
We argue as at the end of the proof of Lemma \ref{lem:DerivFlot}.
If $\gamma$ is a multi-index of length $k\geq 1$, the multidimensional
F\`aa di Bruno formula \cite{Faa} implies that $\partial_{x}^\gamma (f_h\circ g_h)$ 
may be written as a finite sum of terms of the form
\begin{equation}\label{eq:Faa3}
C \left[\left(\partial_{x}^{\gamma'} f_h\right) \circ g_h \right] 
\times  \prod_{j=1}^{d}  \prod_{\vartheta\in \N^{d''} ; |\vartheta| 
\leq k}\left( \partial_{x}^{\vartheta} g_h^j )\right)^{n_{j,\vartheta}},
\end{equation}
\corM{for some $n_{j,\vartheta}\in\N$. Here, $g_h^j$, $j=1,\dots,d$, is the $j$-th 
component of the vector-valued function $g_h$.} Moreover, we recall that
\begin{align*}
k&=\sum_{j=1}^{d} \sum_{\vartheta\in \N^{d''} ;  |\vartheta| \leq k} |\vartheta| n_{j,\vartheta} \\
|\gamma'| &= \sum_{j=1}^{d} \sum_{\vartheta\in \N^{d''} ;  |\vartheta| \leq k}  n_{j,\vartheta}.
\end{align*}
In the end, each term gives a contribution that is bounded by 
$C h^{- \mathfrak{L}}$, with 
\begin{align*}
\mathfrak{L} &= \sigma_3 + (|\gamma'|-1) \sigma_1 + \sum_{j=1}^d  \sum_{\vartheta\in \N^{d''} ; |\vartheta| \leq k} n_{j,\vartheta} \left(\sigma_2 + (|\vartheta| -1) \sigma_1'\right)\\
&= \sigma_3 + (|\gamma'|-1) \sigma_1 + k \sigma_1'+ (\sigma_2 - \sigma'_1) |\gamma'|\\
&=  \sigma_3 +\sigma_2 + (|\gamma'|-1) \left(\sigma_1 + \sigma_2 - \sigma'_1 \right) + (k-1) \sigma_1'\\
&\leq  \sigma_3 +\sigma_2 + (k-1) \sigma_1',
\end{align*}
and the result follows.
\end{proof}
}
\red{
\begin{corollary}\label{cor:CLInverse}
Let $\psi_h : \R^d \supset \Omega \longrightarrow \Omega' \subset \R^d$ be 
a smooth diffeomorphism.  Let $\partial \psi_h$ denote the Jacobian matrix 
of $\psi_h$. Suppose that there exists $\sigma, \sigma', \sigma'' \geq 0$ 
with %$\sigma\geq \sigma'$ 
$\sigma''\geq \sigma'+ \sigma$
such that
\begin{equation}\label{eq:EstimDerivInverse_1}
\|(\partial \psi_h)^{-1}\|_{C^0} = O (h^{-\sigma'}),
\end{equation}
and for all $k\geq 1$, 
\begin{equation}\label{eq:EstimDerivInverse_0}
	\|\partial \psi_h\|_{\dot{C}^k} = O(h^{-\sigma - (k-1) \sigma''}).
\end{equation}
Then, we have, for all $k\geq 1$,
\begin{equation}\label{eq:EstimDerivInverse}
	\|(\psi_h)^{-1}\|_{\dot{C}^k} 
	= O(h^{-(k-1)(\sigma' + \sigma'') -  \sigma ' }).
\end{equation}
\end{corollary}
}
\corM{
\begin{proof}
1. We shall set $f_h= \left(\partial \psi_h\right)^{-1}$ and $g_h = \psi_h^{-1} $. 
We start by recalling that
\begin{equation}\label{eq:DerivInvers}
\partial \psi_h^{-1} = f_h \circ g_h,
\end{equation}
so 
\begin{equation}\label{eq:DerivInvers1}
\|g_h\|_{\dot{C}^k} \leq O(1) \|f_h \circ g_h\|_{C^{k-1}}.
\end{equation}
2. We begin by with estimating $\|f_h\|_{C^{k}}$. The case $k=0$ is 
given by \eqref{eq:EstimDerivInverse_1}. For the case $k\geq 1$, 
we first note by an induction argument that for any $\gamma\in \N^d$ 
with $|\gamma|\geq 1$ the derivative 
$\partial^\gamma f_h$ is given by a finite linear combination of terms 
of the form 
\begin{equation*}
	f_h (\partial^{\beta_1} f_h^{-1} ) f_h (\partial^{\beta_2} f_h^{-1} ) f_h \dots 
	f_h (\partial^{\beta_n} f_h^{-1} ) f_h =: T_{n,\beta},
\end{equation*}
where $\beta_j\in \N^d$, $j=1,\dots,n$, with $\sum_1^n \beta_j = \gamma$ and 
$1\leq n \leq |\gamma|$. By \eqref{eq:EstimDerivInverse_0}, \eqref{eq:EstimDerivInverse}, 
\begin{equation*}
	\sup_{\Omega'}\| T_{n,\beta}\| 
	\leq
	O(1)\|f_h\|^{n+1}_{C^0} \prod_1^n \|f_h^{-1}\|_{\dot{C}^{|\beta_j|}}
	\leq O(1)h^{-\sigma'(n+1)-\sigma n -\sum_1^n(|\beta_j|-1)\sigma''}.
\end{equation*}
Using that $\sigma'' \geq \sigma'+\sigma$ we get 
\begin{equation*}
	\sigma'(n+1)+\sigma n +\sum_1^n(|\beta_j|-1)\sigma'' 
	=\sigma' + n(\sigma'+\sigma - \sigma'') + |\gamma|\sigma''
	\leq \sigma' + |\gamma|\sigma''.
\end{equation*}
So 
\begin{equation}\label{eq:DerivInvers2}
	\|f_h\|_{C^{k}} = O(h^{-\sigma' - k\sigma''}).
\end{equation}
3. We deduce from \eqref{eq:DerivInvers}, \eqref{eq:DerivInvers1} that 
$\|g_h\|_{\dot{C}^1} = O(h^{-\sigma'})$. For $k\geq 2$, we use the splitting 
\begin{equation*}
	\|f_h \circ g_h\|_{C^{k-1}} \leq
	\|f_h \circ g_h\|_{C^{0}} + \|f_h \circ g_h\|_{\cdot{C}^{k-1}}.
\end{equation*}
We may then apply Lemma \ref{lem:CLCompo} 
recursively, with $\sigma_1 = \sigma''$, $\sigma_1'= \sigma' + \sigma''$, 
$\sigma_2 = \sigma'$ and $\sigma_3 = \sigma+\sigma'$. Thus,
\begin{align*}
	\|g_h\|_{\dot{C}^k} \leq O(1)\|f_h \circ g_h\|_{C^{k-1}}
	&= O(h^{-\sigma'}) + O(h^{-(\sigma'+\sigma'') - \sigma' - (k-2)(\sigma'+\sigma'')})\\
	&= O(h^{-\sigma'}) + O(h^{ - \sigma' - (k-1)(\sigma'+\sigma'')})
\end{align*}
and the result follows.
\end{proof}
}
We may now proceed with the proof of Proposition \ref{lem:ControlDeriv}.

\begin{proof}[Proof of Proposition \ref{lem:ControlDeriv}] 
Recall that equation \eqref{eq:eik1} implies that
\begin{equation}\label{eq:ExprPhiG}
(x,  d_x \wt{\phi}_{t,\delta}) 
= \wt{\Phi}_\delta^t 
	\left( 
		y^{-t}_{\delta, \Lambda_t} (x), 
		d_{y^{-t}_{\delta, \Lambda_t} (x)} \wt{\phi}_0
	\right),
\end{equation}
and \eqref{eq:BorneC1} directly follows from the fact that $\wt{\Phi}^t_\delta$ preserves the  energy 
layers $\wt{\mathcal{E}}_{\delta, \lambda}\subset T^* \wt{X}$ which are lifts of bounded 
energy layers in $T^*X$. Next, we modified our notation a bit: 
by adding $\Lambda_0$ and $\Lambda_t$ to the subscript we insist on the fact that 
$x^t_{\delta, \Lambda_0}:\pi_X(\Lambda_0)\to \pi_X(\Lambda_t)$ and 
$y^{-t}_{\delta, \Lambda_t}:\pi_X(\Lambda_t)\to\pi_X(\Lambda_0)$, respectively.
\par 
To estimate $\|\wt{\phi}_{t,\delta}\|_{\dot{C}^L(\wt{\mathcal{O}}_t)}$, we must first estimate the 
derivatives of $y^{-t}_{\delta, \Lambda_t}$. 
More precisely, \corM{for charts $\kappa_{k,\iota}$} as in \eqref{eq:sa8.3}, we define
\begin{align}\label{eq:localization1a}
x^t_{\delta, \Lambda_0, k,\iota, k',\iota'} 
	&= \kappa_{k, \iota} \circ x^t_{\delta, \Lambda_0} \circ 
		\left(\kappa_{k',\iota'}^{-1}\right) : 
		\kappa_{k',\iota'}(\widetilde{U}_{k',\iota'}\cap \wit{\cO}_0) 
		\subset \R^d \longrightarrow \kappa_{k,\iota}(\widetilde{U}_{k,\iota}\cap 
		\wit{\cO}_t )\subset \R^d\\
		\label{eq:localization1b}
y^{-t}_{\delta, \Lambda_t,  k,\iota, k',\iota'} 
	&= \kappa_{k',\iota'} \circ y^{-t}_{\delta, \Lambda_t} \circ 
	\left(\kappa_{k,\iota}^{-1}\right) :  \kappa_{k,\iota}(\widetilde{U}_{k,\iota} \cap 
		\wit{\cO}_t) \longrightarrow 	\kappa_{k',\iota'}(\widetilde{U}_{k',\iota'}\cap \wit{\cO}_0),
\end{align}
and we want to bound $\|y^{-t}_{\delta, \Lambda_t,k,\iota, k',\iota'} \|_{\corM{\dot{C}^L}}$.
\par 
First, using Lemma \ref{lem:DerivFlot} and the smoothness of $\wt{\phi}_0$, 
we see that $\corM{L\geq 1}$ there exists $C_L, \gd_L>0$ such that if 
$t\in [0, \gd_{L} |\log h|]$ we have
\begin{equation}\label{eq:CLg}
\begin{split}
\| x^t_{\delta, \Lambda_0, k,\iota, k',\iota'} \|_{\corM{\dot{C}^L}} 
&\leq 
C_Lh^{-\varepsilon} \left( 1+ h^{-(\beta+ \varepsilon) (L+1)} \delta \right)\\
&\leq 
\corM{
C_Lh^{-\varepsilon} \left( 1+ h^{-\varepsilon-(\beta+ \varepsilon) (L-1)} \right).}
\end{split}
\end{equation}
\corM{
In the last inequality, we used \eqref{eq:CondBeta}. Moreover, 
the constants independent of  $(k,\iota, k',\iota')$. In particular, 
we have $\|x^t_{\delta, \Lambda_0,k,\iota, k',\iota'}\|_{\dot{C}^{1}} 
= O(h^{-\varepsilon})$.}
\\
\par
Next, if $\partial_{\by} x^t_{\delta, \Lambda_0, k,\iota, k',\iota'}$ denotes the Jacobian 
matrix of $x^t_{\delta, \Lambda_0, k,\iota, k',\iota'}$ at $\by$, we claim that there exists 
$C, c>0$, independent of $k,\iota, k',\iota'$,  $t$ and $\by\in \kappa_{k',\iota'}
(\widetilde{U}_{k',\iota'}\cap \wit{\cO}_0)$, such that
\begin{equation}\label{eq:LowerJacobian}
\left\| \left( \partial x^t_{\delta, \Lambda_0,k,\iota, k',\iota'} \right)^{-1}(\by) \right\|
 \leq C c^{t}.
\end{equation}
Indeed, if $y\in \wit{\cO}_0$ and $u\in T_y \wit{\cO}_0$, let us write 
$w:= (u, d_y (d \wt{\phi}_0) u)\in T_{(y,  d_y \wt{\phi}_0)} T^*\wit{X}$. Then we 
know from (\ref{eq:NonContract2}) that $\|d_{(y, d_y\wt{\phi}_0)} 
\wt{\Phi}_\delta^t (w)\| \geq C c_0^t \|w\| \geq C c_0^t \|u\|$ for some $c_0>0$. 
Now, $d_{(y, d_y\wt{\phi}_0)} \wt{\Phi}_\delta^t (w)$ is a vector of 
$T_{\wt{\Phi}^t_\delta(y, d_y\wt{\phi}_0)}\wt{\Phi}_\delta^t (\Lambda_0)$, 
and we know that $\wt{\Phi}_\delta^t (\Lambda_0)$ is projectable, so that we also have 
$\|\pi_{T\wit{X}} d_{(y, d_y\wt{\phi}_0)} \wt{\Phi}_\delta^t (w)\| \geq C' c_0^t \|u\|$. 
Noting that, for any $\bw\in \R^d$, $(\partial_{\by} x^t_{\delta, \Lambda_0,k,\iota, k',\iota'})  
(\bw)$ can be identified with a $\pi_{T\wit{X}} d_{(y, d_y\wt{\phi}_0)} \wt{\Phi}_\delta^t (w)$ 
written in coordinates gives us equation (\ref{eq:LowerJacobian}).
\par 
We may thus apply Corollary \ref{cor:CLInverse} with \corM{
$\sigma = 2\varepsilon$, $\sigma' = \varepsilon$, $\sigma''= \beta + 3\varepsilon$,} 
and we deduce that, \MI{for all $L\geq 1$}
\begin{equation}\label{eq:BorneGInverse}
\begin{aligned}
\| y^{-t}_{\delta, \Lambda_t, k,\iota, k',\iota'} \|_{\MI{\dot{C}}^L} 
&\leq C_L h^{- \varepsilon \corM{- (L-1) (\beta+4\varepsilon)}}.
\end{aligned}
\end{equation}
with constants independent of $k,\iota, k',\iota'$.
\par

Coming back to (\ref{eq:ExprPhiG}), we may apply Lemma \ref{lem:CLCompo} 
(along with Lemma \ref{lem:DerivFlot} \corM{and \eqref{eq:CondBeta}}) with
$\sigma_1= \beta +  \MI{\varepsilon}$,  \corM{$\sigma_1'= \beta+4\varepsilon$}, 
\corM{$\sigma_2 = \varepsilon$ and $\sigma_3= 2\varepsilon$}, to deduce that 
\MI{for all $L\geq 1$ }
$$\corM{\| d \wt{\phi}_{t,\delta} \|_{\dot{C}^L} = O(h^{-3 \varepsilon -  (L-1) (\beta + 4\varepsilon) })}.$$
We deduce the result by taking $\varepsilon$ smaller.
\end{proof}
The same proof applies to the phases $\wt{\phi}_{t,t',\delta}$ 
introduced in \eqref{eq:eik0.0}.  In the sequel, we will only 
need the fact that for any $\varepsilon>0$, there exists $C, \gd>0$ 
such that,  if $0\leq t,t' \leq \gd |\log h|$ we have 
\corM{
\begin{equation}\label{eq:C2Perturb}
\|\wt{\phi}_{t,t',\delta}\|_{C^2(\wt{\cO}_{t,t',\delta})} \leq C h^{-\varepsilon}.
\end{equation}
}
\subsection{Solving the transport equations}\label{sec:SolvTransport}
In this section, we follow a general approach to solving the transport 
equations, see for instance \cite{AN}. From Proposition \ref{Prop:OnTheCover2}
we know that the transported Lagrangian manifold $\wt{\Phi}^t_\delta(\Lambda_0)$ 
\eqref{eq:eik0.0} is projectable for all times $0\leq t \leq \gd |\log h|$ for 
some $\gd >0$. Hence, the flow \eqref{eq:eik0} induces a flow (defined only for 
times $0\leq t\leq \gd  |\log h|$)
\begin{equation}\label{eq:transEq1}
	x^t_{\delta,\Lambda_s} : 
	\wt{\mathcal{O}}_{s,\delta}
	\ni y\mapsto 
	x_{\delta,\Lambda_s}^t(y) \in 
	\wt{\mathcal{O}}_{t+s,\delta}, \quad s\geq 0.
\end{equation}
This flow has the property that 
$x^t_{\delta,\Lambda_{s+\tau}} \circ x^\tau_{\delta,\Lambda_s} 
= x^{t+\tau}_{\delta,\Lambda_s}$. In view of the discussion after 
\eqref{eq:eik0.0}, we write for the inverse of $x^t_{\delta,\Lambda_s}$
\begin{equation}\label{eq:transEq1b}
	y^{-t}_{\delta,\Lambda_{s+t}} : 
	\wt{\mathcal{O}}_{s+t,\delta}
	\ni x\mapsto 
	y_{\delta,\Lambda_{s+t}}^{-t}(x) \in 
	\wt{\mathcal{O}}_{s,\delta}.
\end{equation}
We now estimate the derivatives of  $x^t_{\delta,\Lambda_s}$ and 
$y^{-t}_{\delta,\Lambda_{s+t}}$. To this end, we introduce 
$x^t_{\delta, \Lambda_s,  k,\iota, k',\iota'}$ and 
$y^{-t}_{\delta, \Lambda_{t+s},  k,\iota, k',\iota'}$ just as in 
(\ref{eq:localization1a}, \ref{eq:localization1b}).
\begin{lem}\label{lem:DerivativesOfFlow1}
Suppose that $\beta, \delta, \varepsilon_0$ are as in 
(\ref{eq:CondBeta}). We have that for any $0<\varepsilon<\varepsilon_0$, 
there exist $\mathfrak{d}>0$, and $h_0>0$ such that the following holds: 
for all $L\in \N$,  there exists a constant $C_L>0$ such that 
for all $h\in (0,h_0]$ and all $s,t \in [0,\mathfrak{d}|\log h|]$, with 
$s+t\leq \mathfrak{d}|\log h|$, we have for all $k,\iota, k',\iota'$ and all $L\geq 1$
\begin{equation}\label{eq:EstimateFlow1b}
	\|x^t_{\delta,\Lambda_s,k,\iota, k',\iota'}\|_{C^L}, 
	\|y^{-t}_{\delta,\Lambda_{s+t},k,\iota, k',\iota'} \|_{C^L} 
	\leq C_L(h^{-(L-1)(\beta +\varepsilon) - \varepsilon }),
\end{equation}
with constants independent of $k,\iota, k',\iota'$.

\end{lem}

\begin{proof}
	Observing that $x^t_{\delta,\Lambda_s} = 
	x^{t+s}_{\delta,\Lambda_0}\circ y^{-s}_{\delta,\Lambda_s}$ and 
	$y^{-t}_{\delta,\Lambda_{s+t}} = 
	x^{s}_{\delta,\Lambda_0}\circ y^{-(s+t)}_{\delta,\Lambda_{s+t}}$, 
	the result follows by applying Lemma \ref{lem:CLCompo} 
	to the estimates \eqref{eq:CLg}, and 
	\eqref{eq:BorneGInverse}, \corM{and then taking $\varepsilon>0$ 
	smaller.}
\end{proof}
\begin{rem}\label{rem:Cut-off_est1}
Let $0<\varepsilon<\varepsilon_0$. Applying Lemma \ref{lem:CLCompo} to \eqref{eq:CutOfffun2} with \MI{$\sigma_1=0$},
$\sigma'_1=(\beta+\varepsilon)$, $\sigma_2=\varepsilon$, $\sigma_3=0$, 
yields that there exist $\mathfrak{d}>0$, and $h_0>0$ such that 
for all $L\in \N$, there exists a constant $C_L>0$ such that 
for all $h\in (0,h_0]$ and all $s \in [0,\mathfrak{d}|\log h|]$,
\begin{equation}\label{eq:SetInclusion_Cut-off2}
	\|\psi_t\|_{C^L(\wt{\mathcal{O}}_{t,\delta})},
	\|\psi_{n,t}\|_{C^L(\wt{\mathcal{O}}_{t,\delta})}
	\leq C_L(h^{-(L-1)(\beta+\varepsilon)-\varepsilon)}),
	\quad L\geq 1. 
\end{equation}
When $L=0$ we trivially obtain a uniform bound of order $O(1)$ since 
$|\psi_t|,|\psi_{n,t}|\leq 1$, see the paragraph above \eqref{eq:CutOfffun1}. 
\end{rem}

For $0\leq t\leq c|\log h|$ and $s\geq 0$ (cf. \eqref{eq:transEq1}), 
we introduce the operator $T^t_{\delta,\Lambda_s}$, mapping functions $f$ on 
$\wt{\mathcal{O}}_{s,\delta}=\pi_{\widetilde{X}} \Lambda_s$ into functions on 
$\wt{\mathcal{O}}_{t+s,\delta}=\pi_{\widetilde{X}} \Lambda_{t+s}$, defined by
\begin{equation}\label{eq:homSolTrans}
	(T^t_{\delta,\Lambda_s} f)(x) 
	:= f\circ y^{-t}_{\delta,\Lambda_{s+t}}(x) 
	\left( J^{-t}_{\delta,\Lambda_{s+t}}(x)\right)^{1/2},
\end{equation}
where $f$ is a smooth function on $\wt{\mathcal{O}}_{s, \delta}$ and 
$J^{-t}_{\delta,\Lambda_{s+t}}(x)$ is the inverse of the Jacobian of the 
map $x^t_{\delta,\Lambda_{s+t}}$ at the point $x$ with respect to the 
Riemannian volume on $\widetilde{X}$. In local coordinates we find that 
\begin{equation}\label{eq:homSolTrans.1}
	J^{-t}_{\delta,\Lambda_{s+t}}(x)
	=
	\frac{\sqrt{\det \wh{g}(y^{-t}_{\delta,\Lambda_{s+t}}(x))}
	\bigg| \det \frac{\partial y^{-t}_{\delta,\Lambda_{s+t}}}{\partial x}(x)\bigg|}
	{ \sqrt{\det g(x)}},
\end{equation}
\red{
defined on $\wt{\mathcal{O}}_{s+t, \delta}$. Here, $g(x)$ is the metric 
expressed in local coordinates $x$ and $\wh{g}(y)$ is the metric  
expressed in local coordinates $y$.}
\red{
\begin{rem}
	The operator 
	$T^t_{\delta,\Lambda_s}:L^2(\wt{\mathcal{O}}_{s,\delta})\to L^2(\wt{\mathcal{O}}_{t+s,\delta})$ 
	is an isometry. 
\end{rem}
}
\begin{prop}\label{lem:Jacobian}
The quantity $J^{-t}_{\delta,\Lambda_{s+t}}(x)$ defined in 
\eqref{eq:homSolTrans.1} satisfies 
\begin{equation}\label{eq:transEq3}
	J^{-t}_{\delta,\Lambda_{s+t}}(x) = 
	\exp\left\{
	\int_0^{-t}
	\left(\Delta_g\wt{\phi}_{s+\tau+t,\delta}(y_{\delta,\Lambda_{s+t}}^{\tau}(x))
	+\delta\, \mathrm{div}_g(H^1_{\wt{\phi}_{s+\tau+t,\delta},\wit{q}_\omega})(y_{\delta,\Lambda_{s+t}}^{\tau}(x))
	\right)
	d\tau
\right\}.
\end{equation}
\par 
Moreover, suppose that $\beta, \delta, \varepsilon_0$ are as in 
(\ref{eq:CondBeta}). For any $0<\varepsilon<\varepsilon_0$, there exist 
$\mathfrak{d}>0$, and $h_0>0$ such that the following holds: 
for all $L\in \N$, there exists a constant $C_L>0$ such that 
for all $h\in (0,h_0]$ and all $s,t \in [0,\mathfrak{d}|\log h|]$, 
with $s+t\leq \mathfrak{d}|\log h|$, we have that
\begin{equation}\label{eq:transEq3b}
	\|\log J^{-t}_{\delta,\Lambda_{s+t}}\|_{C^L(\wt{\mathcal{O}}_{s+t,\delta})} \leq 
	C_L  h^{-\varepsilon-L(\beta+\varepsilon)}.
\end{equation}
Here, the $C^L$ norms are to be understood as in \eqref{eq:sa8.1}
\end{prop}
We postpone the proof of this Proposition to the end of this section. Before 
going on we remark that when $\delta=0$ the formula \eqref{eq:transEq3} has 
already been stated in \cite[Equation (3.9)]{AN}. Furthermore, notice that 
when $\wt{Q}_\omega$ is the operator of multiplication by $\wt{q}_\omega$, then 
$H^1_{\wt{\phi}_{t,\delta},\wt{q}_\omega}=0$, so the second integrand in \eqref{eq:transEq3} vanishes. 
Continuing, recall from the discussion after \eqref{eq:trans1} that 
$a(x;h)\sim a_0(x) +ha_1(x) + \dots$ where $a,a_n$ are smooth compactly 
supported maps on $\wt{\mathcal{O}}$. Furthermore, recall that the support of 
$a(\cdot;h)$ is assumed to be contained in a compact subset of $\wt{\mathcal{O}}$ 
independently of $h$. 
\begin{prop}\label{lem:TransportEqSol}
1. The $0^{th}$ order transport equation in \eqref{eq:WKB2} and in 
\eqref{eq:WKB2a} (depending on which case of $\wt{Q}_\omega$ we consider) is 
explicitly solved by
\begin{equation}\label{eq:transEq4}
	b_0(t,x;\delta) = (T^t_{\delta,\Lambda_0} a_0)(x), \quad t\geq 0, 
\end{equation}
with intial condition $b_0(0,x;\delta) = a_0(x)$. Furthermore, 
$\supp b_0(t,\cdot;\delta)\subset \supp \psi_{0,t}$. 
\\
\par
2. The higher order transport equations are solved iteratively for 
$n\geq 1$ by
\begin{equation}\label{eq:transEq5}
	b_n(t,x;\delta) =(T^t_{\delta,\Lambda_0} a_n)(x) 
	+\frac{i}{2}\int_0^t 
	\left(
	T^{t-s}_{\delta,\Lambda_{s}} \Delta b_{n-1}(s,\cdot\,;\delta)
	-2h^{-2}\delta T^{t-s}_{\delta,\Lambda_{s}} \psi_{n,s}\mathcal{R}_{\omega,s} b_{n-1}
	(s,\cdot\,;\delta)
	\right)\!(x) d s,
\end{equation}
with initial condition $b_n(0,x;\delta) = a_n(x)$. Furthermore, 
$\supp b_n(t,\cdot;\delta)\subset \supp \psi_{n,t}$. Additionally 
we have that when $\wt{Q}_\omega$ is a multiplication operator then 
$\mathcal{R}_{\omega,s}=0$ and 
$\supp b_n(t,\cdot;\delta)\subset \supp \psi_{0,t}$.
\end{prop}

\begin{proof}[Proof of Proposition \ref{lem:TransportEqSol}] 
1. \red{The second assertion \eqref{eq:transEq5} follows immediately 
from the first by applying Duhamel's formula to \eqref{eq:WKB2}.}
\\
\par
2. We check that \eqref{eq:transEq4} solves the $0^{th}$ order transport 
equation in \eqref{eq:WKB2} in local coordinates. By continuity 
we can arrange to take the same local coordinates for a small open time interval 
$t\in ]t_0-\varepsilon_0,t_0+\varepsilon_0[\cap [0,+\infty[$ with $t_0\geq 0$ 
and $\varepsilon_0>0$. To ease the notation we will drop the subscripts of 
$x^t_{\delta,\Lambda_0}$, $y^{-t}_{\delta,\Lambda_{t}}$ and 
$J^{-t}_{\delta,\Lambda_{t}}(x)$. In the following computations if 
$f : \R^d\longrightarrow \R$, $\nabla_x f$ is to be understood as a 
$n \times 1$ matrix while $\nabla^\dagger_x f$ is its transpose 
(a $1 \times n$ matrix), and $\cdot$ denotes either the matrix product or the 
usual scalar product in $\R^d$ (which needs not coincide with the Riemannian 
scalar product $\langle \cdot , \cdot \rangle_x$).
\par
A direct computation yields 
\begin{equation}\label{eq:0orderTPE1}
	\partial_t b_0(t,x)
	= (J^{-t}(x))^{1/2}  \times \left(\nabla^\dagger_x a_0(x) \cdot \left(\partial_t y^{-t}(x)\right) \right)
	+ a_0(y^{-t}(x)) \times \partial_t(J^{-t}(x))^{1/2} , 
\end{equation}
and
\begin{equation}\label{eq:0orderTPE2}
	\nabla_xb_0(t,x)
	= (J^{-t}(x))^{1/2}  \times \left( \frac{\partial y^{-t}(x)}{\partial x}\right)^\dagger \cdot \nabla_x a_0(x)
	+ a_0(y^{-t}(x)) \nabla_x (J^{-t}(x))^{1/2} . 
\end{equation}
Applying $\frac{d}{dt}$ to the relation $x^t(y^{-t}(x))=x$ yields the identity
\begin{equation}\label{eq:0orderTPE3}
	(\partial_t x^t)(y^{-t}(x))=-\frac{\partial x^{t}}{\partial y}(y^{-t}(x)) 
		\cdot \partial_t y^{-t}(x).
\end{equation}
Since $\frac{\partial y^{-t}(x)}{\partial x}=(\frac{\partial x^{t}}{\partial y}(y^{-t}(x)) )^{-1}$, 
we get by combining (\ref{eq:0orderTPE1}--\ref{eq:0orderTPE3}) that
\begin{equation}\label{eq:0orderTPE4}
	\partial_t b_0(t,x)
	= -\nabla_x b_0(t,x)\cdot (\partial_t x^{t})(y^{-t}(x)) 
	+ \frac{b_0(t,x)}{2}
	\left(
		\partial_t
		+ (\partial_t x^{t})(y^{-t}(x)) \cdot \nabla_x
	\right)\log (J^{-t}(x)).
\end{equation}
Next, we focus on the second term on the right hand side of 
\eqref{eq:0orderTPE4}. Write $L:=(\partial_t x^{t})(y^{-t}(x)) \cdot \nabla_x$. 
By \eqref{eq:homSolTrans.1} we get that 
\begin{equation}\label{eq:0orderTPE5}
\begin{split}
	(\partial_t +L)\log (J^{-t}(x))
	&= \frac{1}{2}\tr \left[ 
		\corM{\wh{g}}(y^{-t}(x))^{-1} (\partial_t +L)\corM{\wh{g}}(y^{-t}(x))
	\right] \\
	&\phantom{=}
	+\tr \left[ 
		\left(\frac{\partial y^{-t}(x)}{\partial x}\right)^{-1} 
		(\partial_t +L)\left(\frac{\partial y^{-t}(x)}{\partial x}\right)
	\right]\\
	&\phantom{=} 
		- L \log \sqrt{\det g(x)}.
\end{split}
\end{equation}
Here in the traces the differential operator $\partial_t+L$ is applied 
componentwise to the matrices. 
\par
Using \eqref{eq:0orderTPE3}, 
a straightforward computation shows that 
\begin{equation*}
 	(\partial_t +L)\corM{\wh{g}}(y^{-t}(x))=0.
\end{equation*}
\par 
Turning to the second term on the right hand side of \eqref{eq:0orderTPE5}, 
we first note that \eqref{eq:HamiltonEq} implies that
\begin{equation}\label{eq:ReformulationClassicalDynamics}
(\partial_t x^{t})(y^{-t}(x)) = (d_{\xi}\wt{p})(x,d_x\wt{\phi}_{t,\delta};\delta).
\end{equation}
Using this fact along with \eqref{eq:0orderTPE3} and the 
symmetry of the Hessian (with respect to $x$) of $y^{-t}$, we get that 
\begin{equation*}
	\partial_t \left(\frac{\partial y^{-t}(x)}{\partial x}\right) 
	=- L \left(\frac{\partial y^{-t}(x)}{\partial x}\right) 
		- \left(\frac{\partial y^{-t}(x)}{\partial x}\right)  \cdot 
		M,
\end{equation*}
where $M$ is the $d\times d$ matrix whose entries are given by
\begin{equation*}
	M_{m,n}= \partial_{x_m} \big[(\partial_{\xi_n}\wt{p})
			(x,d_x\wt{\phi}_{t,\delta};\delta)\big].
\end{equation*}
Hence,
\begin{equation}\label{eq:0orderTPE6}
	(\partial_t +L)\log (J^{-t}(x))
	= 
	- \mathrm{tr}[M] 
		- L \log \sqrt{\det g(x)}.
\end{equation}
Using \eqref{eq:ReformulationClassicalDynamics}, \eqref{eq:Hamiltonian}, 
and plugging \eqref{eq:0orderTPE6} into \eqref{eq:0orderTPE4}, we get 
\begin{equation}\label{eq:0orderTPE7}
\begin{split}
\partial_t b_0(t,x)
	&=-\nabla_x b_0(t,x)\cdot(d_\xi \wt{p})(x,d_x\wt{\phi}_{t,\delta};\delta)
	-\frac{b_0(t,x)}{2}
		\nabla_x\cdot(d_\xi \wt{p})(x,d_x\wt{\phi}_{t,\delta};\delta) 
	\\
	&\phantom{=}
	-\frac{b_0(t,x)}{2} (d_\xi \wt{p})(x,d_x\wt{\phi}_{t,\delta};\delta)\cdot \left(\nabla_x
	\log\det\sqrt{\det g(x)}\right)	
	\\
	&= - \langle d_x \wt{\phi}_{t,\delta}, d_x b_0 \rangle_x
	-\frac{1}{2} b_0 \Delta_g \wt{\phi}_{t,\delta} 
	-\delta	\nabla_x b_0(t,x)\cdot (\nabla_\xi \wt{q}_\omega) (x,d_x\wt{\phi}_{t,\delta} )
	\\
	&\phantom{=}
	- \frac{\delta b_0(t,x)}{2}
	\left[((\nabla_\xi \wt{q}_\omega)(x,d_x\wt{\phi}_{t,\delta})\cdot\nabla_x)
	\log \sqrt{\det g(x)}
	+
	\nabla_x\cdot(\nabla_\xi \wt{q}_\omega)(x,d_x\wt{\phi}_{t,\delta}) 
	\right]
	\\
	&= - \langle d_x \wt{\phi}_{t,\delta}, d_x b_0 \rangle_x
	-\frac{1}{2} b_0 \Delta_g \wt{\phi}_{t,\delta} 
	-\delta	(H^1_{\wt{\phi}_{t,\delta},\wit{q}_\omega}b_0) (x) 
	- \frac{\delta}{2} \mathrm{div}_g (H^1_{\wt{\phi}_{t,\delta},\wit{q}_\omega})
	b_0.
\end{split}
\end{equation}
Hence, \eqref{eq:transEq4} solves the $0^{th}$ order transport equation in 
\eqref{eq:WKB2}. To end the proof of \eqref{eq:transEq4} we note that $b(0,x)=a_0(x)$. 
Finally,  the fact that $\supp b_0(t,\cdot;\delta)\subset \supp \psi_{0,t}$ 
follows from equation (\ref{eq:transEq4}) and Lemma \ref{Lem:Gron}.
\end{proof}

We end this section with the proof of Proposition \ref{lem:Jacobian}.
\begin{proof}[Proof of Proposition \ref{lem:Jacobian}] 
1. Recall the local coordinate representation \eqref{eq:homSolTrans.1} 
of $J^{-t}_{\delta,\Lambda_{s+t}}(x)$ and write 
$K(t,x)=\log J^{-t}_{\delta,\Lambda_{s+t}}(x)$.
\par 
First note that $K(0,x)=0$. We know from \eqref{eq:0orderTPE6} 
that 
\begin{equation}\label{eq:JacExpress1}
	(\partial_t +L)K(t,x)
	= 
	- \nabla_x \cdot \left(\nabla_\xi \wt{p}(x,d_x\wt{\phi}_{s+t,\delta};\delta)\right)
		- L \log \sqrt{\det g(x)}.
\end{equation}
Strictly speaking \eqref{eq:JacExpress1} has been proven in 
\eqref{eq:0orderTPE6} only for the case $s=0$, however,  
one can easily check that it also holds in the case $s\geq 0$. 
\par 
Recall from the discussion before \eqref{eq:0orderTPE4} 
that $L=(\partial_t x_{\delta,\Lambda_{s}}^{t})
(y_{\delta,\Lambda_{s+t}}^{-t}(x)) \cdot\nabla_x$. Hence, 
the homogeneous form of the equation \eqref{eq:JacExpress1} 
is solved by $(y_{\delta,\Lambda_{s+t}}^{-t})^*f$, $f\in C^1$. 
However, since $y_{\delta,\Lambda_{s}}^0(x)=x$, the initial 
condition $K(0,x)=0$ implies that the 
homogeneous solution is constantly equal to $0$. The unique 
solution to equation \eqref{eq:JacExpress1} with initial 
condition $K(0,x)=0$ can then be obtained by Duhamel's formula:
\begin{equation*}%\label{eq:JacExpress2}
\begin{split}
	K(t,x)
	&= 
	-\int_0^t (y_{\delta,\Lambda_{s+t}}^{\tau-t})^* 
	\left( \nabla_x \cdot (\nabla_\xi \wt{p})(x,d_x\wt{\phi}_{s+\tau,\delta};\delta)
	+ (\nabla_\xi \wt{p})(x,d_x\wt{\phi}_{s+\tau,\delta};\delta) \cdot\nabla_x
	\log \sqrt{\det g(x)} 
	\right)d\tau
	\\
	&= 
	\int_0^{-t}
	\left(\Delta_g\wt{\phi}_{s+\tau'+t,\delta}(y_{\delta,\Lambda_{s+t}}^{\tau'}(x))
	+ \delta\,\mathrm{div}_g(H^1_{\wt{\phi}_{s+\tau'+t,\delta},\wit{q}_\omega})(y_{\delta,\Lambda_{s+t}}^{\tau'}(x))
	\right)d\tau'.
\end{split}
\end{equation*}
In the last line we performed the linear change of variables 
$\tau = t +\tau'$ followed by a direct computation similar to 
the last line of \eqref{eq:0orderTPE7}. Taking the exponential of 
$K(t,x)$ implies \eqref{eq:transEq3}. 
\\
\par
% {eq:CondBeta}

2. Let us now turn to the regularity estimate. Let 
$\beta, \delta, \varepsilon_0$ be as in (\ref{eq:CondBeta}). 
Using Lemma \ref{lem:CLCompo}, Proposition \ref{lem:ControlDeriv}, 
Lemma \ref{lem:DerivativesOfFlow1} and \eqref{eq:PotentialDer}, we 
find that for any $0<\varepsilon<\varepsilon_0$, there exist 
$\mathfrak{d}>0$, and $h_0>0$ such that the following holds: 
for all $L\in \N$, there exists a constant $C_L>0$ such that 
for all $h\in (0,h_0]$ and all $s,t \in [0,\mathfrak{d}|\log h|]$, 
with $(s+t)\leq\mathfrak{d}|\log h|$ and all \red{$\tau'\in[-t,0]$}
\begin{equation}\label{eq:JacExpress2}
\begin{split}
	&\|(\Delta_g\wt{\phi}_{s+\tau'+t,\delta})\circ y_{\delta,\Lambda_{s+t}}^{\tau'} 
	\|_{C^L(\wt{\mathcal{O}}_{s+t,\delta})}
	=O_L(h^{-\varepsilon -L(\beta+\corM{2\varepsilon})}), \\
	&\|\delta (\nabla_\xi \wt{q}_\omega)(x,d_x\wt{\phi}_{s+\tau,\delta};\delta) 
	\circ 
	y_{\delta,\Lambda_{s+t}}^{\tau'}\|_{C^L(\wt{\mathcal{O}}_{s+t,\delta})}
	=O_L(\delta h^{-\beta\corM{-\varepsilon} -L(\beta+\corM{2\varepsilon})}), \\
	&\|\delta [\nabla_x (\nabla_\xi \wt{q}_\omega)(x,d_x\wt{\phi}_{s+\tau,\delta};\delta)]
	\circ 
	y_{\delta,\Lambda_{s+t}}^{\tau'}\|_{C^L(\wt{\mathcal{O}}_{s+t,\delta})}
	=O_L(\delta h^{-2\beta -\varepsilon -L(\beta+\corM{2\varepsilon})}), \\
	&\left\| (\nabla_x \log \sqrt{\det g} )
	\circ
	y_{\delta,\Lambda_{s+t}}^{\tau'}\right\|_{C^L(\wt{\mathcal{O}}_{s+t,\delta})}
	=
	\begin{cases}
	O(1), ~~ L=0,\\
	O_L(h^{\beta - L(\beta+\MI{\varepsilon})}), ~~ L\geq 1,
	\end{cases}
\end{split}
\end{equation}
where in the last line we used as well that the derivatives of 
$\log\det g(x)$ are uniformly bounded on $\wt{\mathcal{O}}_{s+t,\delta}$ 
with respect to $h$ since $g$ is the lifted metric from the underlying 
compact manifold. 
\par
The product rule shows that 
\begin{equation}\label{eq:JacExpress3}
\delta \left\| \left[ 
		(\nabla_\xi \wt{q}_\omega)(x,d_x\wt{\phi}_{s+\tau,\delta};\delta) \cdot
		(\nabla_x \log \sqrt{\det g} )
		\right]\circ
		y_{\delta,\Lambda_{s+t}}^{\tau'}\right\|_{C^L(\wt{\mathcal{O}}_{s+t,\delta})}
		=O_L\!\left(\delta h^{-\beta \corM{-\varepsilon} -L(\beta+\corM{2\varepsilon})}\right).
\end{equation}
After potentially increasing the constant in the error estimate, 
we see that \eqref{eq:JacExpress3} is bounded by the right hand \corM{side} 
of the third line of \eqref{eq:JacExpress2}. By \eqref{eq:CondBeta} 
we see the \corM{third} line in \eqref{eq:JacExpress2} 
is bounded by $O_L(h^{-\varepsilon -L(\beta+\corM{2\varepsilon})})$. Thus, 
\corM{by taking $\varepsilon$ smaller,} 
\begin{equation*}
	\|K(t,\cdot)\|_{C^L(\wt{\mathcal{O}}_{s+t,\delta})} \leq C_L |\log h| 
 	h^{-\varepsilon-L(\beta+\MI{\varepsilon})}.
\end{equation*}
The estimate \eqref{eq:transEq3b} follows.
% by taking $\varepsilon>0$, and hence $\mathfrak{d}>0$, smaller. 
\end{proof}
\section{$C^L$ and Sobolev estimates of the WKB solution on $\wt{X}$}\label{sec:RegularityOfWKB}
In all the sequel,  we will suppose that $\beta, \delta, \varepsilon_0$ are 
as in Hypothesis \ref{Hyp:BetaDelta}. In this section we continue working 
with a smooth connected compact Riemannian manifold $X$ and its universal 
cover $\wit{X}$, see the beginning of Section \ref{sec:WKB}. We will 
largely use the notations from section \ref{subsec:LiftPseudo}.

\subsection{\MI{Estimates on the leading term in the WKB Ansatz}}
\MI{We shall start with estimates involving the term \corM{$b_0$ 
\eqref{eq:transEq4}} in the WKB expansion \corM{\eqref{eq:WKB_Ansatz}}.}

\begin{lem}\label{cor:DecayAmplitude}
\corM{Recall the notation introduced in \eqref{eq:eik0.0} and 
the subsequent paragraph.}
There exists $c,C>0$ such that, if $\gd>0$ and $h_0>0$ 
is small enough, we have for all $h\in ]0,h_0]$ and 
all $0\leq t \leq \gd |\log h|$
\begin{equation*}
	\sup_{x\in \corM{\wit{\mathcal{O}}_{t,0}}} |b_0(t,x;0)| \leq C\e^{-ct}.
\end{equation*}
\end{lem}
\begin{proof}
%Thanks to Proposition \ref{Prop:LeadingAmpl}, it suffices to prove the 
%result for $b_0(t,x;0)$ instead of $b_0(t,x;\delta)$.
%\\
%
By (\ref{eq:transEq1b}), (\ref{eq:homSolTrans}) and \eqref{eq:transEq4} 
we have 
$$b_0(t,x;0) = a_0  \circ y^{-t}_{0,\Lambda_{t}}(x)  
\left( J^{-t}_{0,\Lambda_{t}}(x)\right)^{1/2},
$$
where $J^{-t}_{0,\Lambda_{t}}$ \red{is the Jacobian determinant of 
$y^{-t}_{0,\Lambda_{t}}$, or, in other words, the inverse of the 
Jacobian determinant of $x^t_{0,\Lambda_0}$, as in \eqref{eq:homSolTrans.1}, 
\eqref{eq:transEq3}.}
% is the Jacobian determinant of 
% $y^{-t}_{0,\Lambda_{t}}$, or, in other words, the inverse of the 
% Jacobian determinant of $x^t_{0,\Lambda_0}$.
\\
We will now show that this  Jacobian determinant grows exponentially,  \MI{so that
\begin{equation}\label{eq:DecayJacob}
J^{-t}_{0,\Lambda_{t}}\leq C e^{-c't}
\end{equation}
for some $C, c'>0$.}

If $y\in \wt{\mathcal{O}}$,  let us write $\rho= (y, d_y \wit{\phi}_0)$.
If $w\in T_y \corM{\wt{X}}$,  we have
\begin{equation*}
d_y x^t_{0,\Lambda_0} (w) = d_{\wt{\Phi}_0^t(\rho)} \pi_{\corM{\wt{X}}} \circ  
d_\rho \wt{\Phi}^t_0( w,  \red{\operatorname{Hess}}_{y}(\corM{\wt{\phi}_0})(w)).
\end{equation*}
\\
We will now explain why $d_{\wt{\Phi}_0^t(\rho)} \pi_{\corM{\wt{X}}} \circ  d_\rho \wt{\Phi}^t_0$ 
has eigenvalues growing exponentially with $t$, except one which remains 
bounded from below.
Let $v\in T_{(y, d_y \phi_0)} \Lambda_0$.
\begin{itemize}
\item If $v\in E^0_{0,\rho}$, then $|d_\rho \wt{\Phi}^t_0(v)|_{\wt{\Phi}_0^t(\rho)}$ 
is bounded from above and from below independently of $t$, as follows from 
(\ref{eq:NormeEquiv}) and (\ref{eq:Matrix}). Furthermore, since $E^0_{0,\rho}$ 
is transverse to the vertical fibres of $T^*X$, we obtain that 
$|d_{\wt{\Phi}^t(\rho)} \pi_X \circ  d_\rho \wt{\Phi}^t_0 (v)|$ is bounded from below 
independently of $t$.
\item  If $v\in E^+_{0,\rho} \oplus E^-_{0,\rho}$, using the fact that 
$\Lambda_0$ is transverse to the stable directions, we deduce from 
(\ref{eq:NormeEquiv}) and (\ref{eq:Matrix}) that 
$d_{\rho} \wt{\Phi}^t_0(v)$ is a vector whose norm is $\geq C \e^{ct}$ for some 
$C,c>0$, and that this vector is very close to the unstable direction 
$E^-_{\wt{\Phi}_0^t(\rho)}$. Therefore, using the fact that the unstable 
directions are transverse with the vertical fibres,  we deduce that
$$d_{\wt{\Phi}_0^t(\rho)} \pi_{\corM{\wt{X}}} \circ d_{\rho} \wt{\Phi}^t_0(v)\geq C' \e^{ct}.$$
\end{itemize}
It follows from this that the Jacobian determinant of $x^t_{0,\Lambda_0}$ 
grows exponentially with $t$. The statement follows.
\end{proof}

\MI{Next, we show that the first term in the WKB expansion does not depend much on $\delta$.}

\begin{prop}\label{Prop:LeadingAmpl}
\corM{Recall the notation introduced in \eqref{eq:eik0.0} and the subsequent 
paragraph.} Let $0<\varepsilon< \varepsilon_0$. There exists $\mathfrak{d}>0$, $h_0>0$ 
such that for all $h\in (0, h_0]$ and all $t\in [0, \gd |\log h|]$, 
then $\supp b_0(t,\cdot\,;\delta)\subset \wt{\mathcal{O}}_{t,\delta}''
\subset\wt{\mathcal{O}}_{t,0}$ and 
$\supp b_0(t,\cdot\,;0)\subset\wt{\mathcal{O}}_{t,0}'\subset 
\wt{\mathcal{O}}_{t,0}$. Moreover, 
\begin{equation*}
		| b_0(t,x;\delta) - b_0(t,x;0) | = O(\delta h^{-2\beta -\varepsilon}),
\end{equation*}
%
% and
% %
% \begin{equation*}
% 	|\partial_x^\eta b_0(t,x;\delta)| 
% 	= O(\delta h^{-(n+1)(2\beta -\varepsilon)}), \quad n\in \N, 
% \end{equation*}
%
uniformly in $t\in [0,\mathfrak{d} |\log h|]$ and $x\in \wt{\mathcal{O}}_{t,0}$.
\end{prop}
\begin{proof}
\corM{
1. The inclusion concerning $\supp b_0(t,\cdot\,;\delta)$ follows from  
Proposition \ref{lem:TransportEqSol} and \eqref{eq:SetInclusion}. 
The inclusion concerning $\supp b_0(t,\cdot\,;0)$ is an easy consequence of  
\eqref{hyp:compactSupport}, \eqref{eq:EnsemblesEmboites}, 
\eqref{eq:homSolTrans} and \eqref{eq:transEq4}. By \eqref{eq:SetInclusion2} 
we know that $\overline{\wt{\mathcal{O}}}_{t,0}'\subset \wt{\mathcal{O}}_{t,\delta}''$. 
Thus, we may from now on work with $x\in \wt{\mathcal{O}}_{t,\delta}''$ which 
is a set where the flows $y^{-t}_{\delta,\Lambda_t}$, $y^{-t}_{0,\Lambda_t}$ 
as well as the phases $\wt{\phi}_{t,\delta}$, $\wt{\phi}_{t,0}$ are well-defined. 
}
\\
\par 
2. Recall from \eqref{eq:transEq4} that $b_0(t,x;\delta) = 
(T^t_{\delta,\Lambda_0} a_0)(x)$, with 
\begin{equation*}
	(T^t_{\delta,\Lambda_0} f)(x) 
	= f\circ y^{-t}_{\delta,\Lambda_{t}}(x) 
	\left( J^{-t}_{\delta,\Lambda_{t}}(x)\right)^{1/2},
\end{equation*}
cf. \eqref{eq:homSolTrans}, and where $J^{-t}_{\delta,\Lambda_{t}}(x)$ 
is as in \red{\eqref{eq:homSolTrans.1}}, \eqref{eq:transEq3}. Let 
$\corM{\varepsilon_0>}\varepsilon>0$ be fixed but arbitrary and let $h>0$. 
Let $\mathfrak{d}>0$ and let $0\leq t\leq \mathfrak{d} |\log h|$. 
By Lemma \ref{Lem:Gron} it follows that for $\mathfrak{d}>0$ small 
enough
\begin{equation}\label{eq:transEq12}
	\mathrm{dist}_{\widetilde{X}} (x_\delta^t(y), x_0^t(y)) 
	\leq 
	O( \delta h^{-\beta -\varepsilon}), 
\end{equation}
for all $y\in \wt{\mathcal{O}}$ and all $t\in[0, \mathfrak{d} |\log h|]$. 
\corM{Let $x\in \wt{\mathcal{O}}_{t,\delta}''$}. Thanks to Lemmas 
\ref{Lem:Gron} and \ref{lem:BoundDiffFlows}, we know that if $\gd$ is 
small enough,
\begin{equation}\label{eq:transEq11}
	\mathrm{dist}_{\widetilde{X}} 
	\left(\corM{y^{-t}_{0,\Lambda_t}(x)}, \corM{y^{-t}_{\delta,\Lambda_t}(x)}  \right)
	=
	O( \delta h^{-\beta -2\varepsilon}),
\end{equation}
and, by Taylor expansion,  we get
\begin{equation}\label{eq:transEq13}
	| a_0\circ y^{-t}_{\delta,\Lambda_t}(x) - a_0\circ y^{-t}_{0,\Lambda_t}(x) | 
	= 
	O(\delta h^{-\beta -2\varepsilon}).
\end{equation}
Both estimates \eqref{eq:transEq11}, \eqref{eq:transEq13} are uniform 
in $t\in[0, \mathfrak{d} |\log h|]$ and \corM{$x\in\wt{\mathcal{O}}_{t,\delta}''$}.
\\
\par
3. Next, using Proposition \ref{lem:ControlDeriv}, Remark \ref{rem:DerivCover}, 
Lemma \ref{lem:DerivFlot} and Lemma \ref{lem:CLCompo}, we deduce from 
\eqref{eq:PhiIntegrale2} that, for any $\varepsilon>0$, there exists 
$\gd>0$ such that for all $t\leq \gd |\log h|$,
\begin{equation}\label{eq:CloseDeltaPhase}
	\Delta \wt{\phi}_{t,\delta}(x) 
	= 
	\Delta \wt{\phi}_{t,0}(x) + O(\delta h^{-2\beta-\varepsilon} |\log h| ),
\end{equation}
\corM{uniformly in \corM{$x\in\wt{\mathcal{O}}_{t,\delta}''$}.} 
\MI{Furthermore, using} Proposition \ref{lem:ControlDeriv} and \eqref{eq:transEq11}, we deduce that
\begin{equation}\label{eq:transEq14}
	(\Delta \wt{\phi}_{t,0})( y^{-t}_{\delta,\Lambda_0} (x)) = 
	(\Delta \wt{\phi}_{t,0})( y^{-t}_{0,\Lambda_0} (x)) 
		+ O(\delta h^{-2\beta -\MI{4}\varepsilon}),
\end{equation}
uniformly in $t\in[0, \mathfrak{d} |\log h|]$ and 
\corM{$x\in \wt{\mathcal{O}}_{t,\delta}''$}. 
Noting that $\delta\, \mathrm{div}_g(H^1_{\wt{\phi}_{s+\tau'+t,\delta},\wit{q}_\omega})
(y_{\delta,\Lambda_{s+t}}^{\tau}(x))= O(\delta h^{-2\beta-\varepsilon})$,  we deduce from (\ref{eq:transEq3}) along with \eqref{eq:CloseDeltaPhase} and \eqref{eq:transEq14} that 
%we may perform 
%a Taylor expansion in \eqref{eq:transEq3} to obtain
%
\begin{equation}\label{eq:transEq10}
\begin{aligned}
	(J^{-t}_{\delta,\Lambda_{t}}(x) )^{1/2}&= 
	\left(1 + O(\delta h^{-2\beta- \MI{4}\varepsilon}|\log h|) \right)
	\exp\left\{ \frac{1}{2}
	\int_0^{-t}\Delta_g\wt{\phi}_{\tau+t,0}(y_{\MI{0},\Lambda_{t}}^{\tau}(x)) d\tau 
	\right\}\\
	&=
	\left(1 + O(\delta h^{-2\beta- \MI{4}\varepsilon}|\log h|) \right) (J^{-t}_{0,\Lambda_{t}}(x) )^{1/2}.
	\end{aligned}
\end{equation}
%

%Notice that the exponential in \eqref{eq:transEq10} is bounded by an 
%arbitrarily small negative power of $h$ provided that we 
%choose $\mathfrak{d}>0$ small enough. 
Combining \eqref{eq:transEq10}, \eqref{eq:transEq14}, \corM{\eqref{eq:DecayJacob}} 
and \eqref{eq:transEq13}, 
we conclude that for any $\varepsilon>0$, there exists a $\mathfrak{d}>0$ 
sufficiently small, such that for $h>0$ small enough 
\begin{equation}\label{eq:transEq2}
	\left|a_0\circ y^{-t}_{\delta,\Lambda_0}(x) 
		\left( J^{-t}_{\delta,\Lambda_t}(x)\right)^{1/2}
		-a_0\circ y^{-t}_{0,\Lambda_0}(x) 
		\left( J^{-t}_{0,\Lambda_t}(x)\right)^{1/2}\right| 
	= 
	O(\delta h^{-2\beta -\MI{4}\varepsilon}\corM{|\log h|}), 
\end{equation}
for all $t\in[0,\mathfrak{d} |\log h|]$ and all \corM{$x\in \wt{\mathcal{O}}_{t,\delta}''$}. 
\MI{Taking $\varepsilon$, and hence $\gd$ smaller}, this proves the result.
\end{proof}

\MI{We conclude this subsection with an estimate on the operator $T^t_{\delta,\Lambda_s}$ appearing in \eqref{eq:homSolTrans}.}
\begin{prop}\label{lem:lem_P1}
Suppose that $\beta, \delta, \varepsilon_0$ are as in 
(\ref{eq:CondBeta}) and let $T^t_{\delta,\Lambda_s}$
be as in \eqref{eq:homSolTrans}. 
If $0<\varepsilon< \varepsilon_0$, there exists 
$\mathfrak{d}>0$, $h_0>0$ such that the following holds 
uniformly in $h\in (0, h_0]$ and $t,s \in[0, \mathfrak{d} |\log h|]$, 
with $s+t \leq \mathfrak{d} |\log h|$:
\par
Let $f\in C_c^\infty(\wt{\mO}_{s,\delta})$ such that, for all $L\in \N$ there 
exists a constant $C_L>0$ such that
\begin{equation*}
	\|f\|_{C^L} \leq C_L h^{-L (\beta + \varepsilon)}.
\end{equation*}
Then, \red{for all $L\in\N$} there exists a constant $C_L'>0$ such that 
\begin{equation}\label{eq:lem_P1.11}
	\|T^t_{\delta,\Lambda_s} f\|_{C^L} 
	\leq C'_L h^{-L(\beta+2\varepsilon)}.
\end{equation}
\end{prop}
\begin{proof}
By \eqref{eq:homSolTrans} and the product rule, it is sufficient to show that 
\begin{equation}\label{eq:JeSuisPresqueAStrasbourg}
\begin{aligned}
	\left\|\left( J^{-t}_{\delta,\Lambda_{s+t}}(\cdot)\right)^{1/2}\right\|_{C^L} 
	&\leq C_L h^{-L(\beta+\varepsilon)},\\
	\left\|f\circ y^{-t}_{\delta,\Lambda_{s+t}}(\cdot) \right\|_{C^L} 
	&\leq  C_L h^{-L (\beta+2\varepsilon)}.
\end{aligned}
\end{equation}
\MI{The fact that 
$$\left\| \left( J^{-t}_{\delta,\Lambda_{s+t}}\right)^{1/2}\right\|_{C^0}
\leq C$$ 
follows directly from equations \eqref{eq:DecayJacob} and (\ref{eq:transEq10}). 
Next, let $0<\varepsilon<\varepsilon_0$. }We know from the second 
assertion of Proposition \ref{lem:Jacobian} that there exist 
$\mathfrak{d}>0$ and $h_0>0$ such that for all $L\in \N$, 
there exists a constant $C_L>0$ such that for all $h\in (0,h_0]$ and 
all $s,t \in [0,\mathfrak{d}|\log h|]$, with $s+t\leq \mathfrak{d}|\log h|$, 
we have that \eqref{eq:transEq3b} holds. This together with \corM{the 
multidimensional F\`aa di Bruno formula as in \eqref{eq:Faa3}}, yields 
that for any $\alpha\in \N^d\backslash\{0\}$, $h\in (0,h_0]$ and  
$s,t \in [0,\mathfrak{d}|\log h|]$, with $s+t\leq \mathfrak{d}|\log h|$,
\begin{equation}\label{eq:lem_P1.1_pf1a}
	\partial^\alpha_xJ^{-t}_{\delta,\Lambda_{s+t}}(x)
	= O_{|\alpha|}\!\left(
	h^{-|\alpha|(\beta+\corM{2\varepsilon})}\right)
	J^{-t}_{\delta,\Lambda_{s+t}}(x),
\end{equation}
uniformly on $\wit{\cO}_{s+t,\delta}$. Notice here that the same estimate 
holds also with $J^{-t}$ replaced by $(J^{-t})^{1/2}$.  This gives us the 
first estimate in (\ref{eq:JeSuisPresqueAStrasbourg}).
\par

To obtain the second part of (\ref{eq:JeSuisPresqueAStrasbourg}), we 
use Lemma \ref{lem:CLCompo}, with $\sigma_2 = \varepsilon$, 
and $\sigma_1= \sigma_3  = \beta + \varepsilon$ \MI{and $\sigma_1' = \beta+2\varepsilon$}, obtaining
\begin{equation*}
	\left\|f\circ y^{-t}_{\delta,\Lambda_{s+t}}(\cdot) \right\|_{C^L}  
	\leq C_L h^{\MI{-\beta - 2\varepsilon - (L-1) (\beta+2\varepsilon)}},
\end{equation*}
as announced.
\end{proof}
	
\subsection{\corM{$C^L$ estimates on the amplitudes $b_n$ in the WKB expansion}}
\corM{We aim to estimate the $C^L$ norms of the terms $b_n$ in the WKB expansion 
\eqref{eq:WKB_Ansatz}.} First, we present the following technical result.
\begin{lem}\label{Lem:DistCutOff}
Let $\beta, \delta, \varepsilon_0$ be as in (\ref{eq:CondBeta}) 
and let $\psi_{N+1,t}:=\psi_t, \psi_{N,t},\dots,\psi_{0,t}$ be as in 
\eqref{eq:CutOfffun2}.
For any $0<\varepsilon<\varepsilon_0$, there exist $\mathfrak{d}>0$ 
such that for any $N\in\N$ and $h_0>0$ small enough we have that 
for all $t\in [0,\mathfrak{d}|\log h|]$, all $h\in ]0,h_0]$ and 
all $n=0,\dots,N$ 
\begin{equation}\label{eq:supportdistance}
	\dist_{\wt{X}}(\overline{\wt{\mathcal{O}}''_{t,\delta}}\cap\supp(1-\psi_{n+1,t}),\supp\psi_{n,t}) 
	\geq h^{2\varepsilon}.
\end{equation}
\end{lem}
\begin{proof}
Given $\varepsilon \in ]0,\varepsilon_0[$, thanks to Lemma \ref{Lem:Gron}, we may
take $\mathfrak{d}>0$ small enough such that for all $t\in [0,\mathfrak{d}|\log h|]$ 
and all $\rho_0, \rho_1 \in \mathcal{E}_{\delta, (\frac{1}{2}, 2)}$
\begin{equation}\label{eq:supportdistance1}
		\mathrm{dist}_{T^*\wt{X}}(\wt{\Phi}_\delta^{-t}(\rho_0),\wt{\Phi}_\delta^{-t}(\rho_1))
			\leq h^{-\varepsilon/2} \mathrm{dist}_{T^*X}(\rho_0,\rho_1) 
			+  C \delta h^{-\beta} h^{-\varepsilon/2}.
\end{equation}
Suppose now that \eqref{eq:supportdistance} does not hold. So there exists 
an $N\in\N$ such that for all $h_0>0$ there exists a $t\in [0,\mathfrak{d}|\log h|]$, 
an $h\in]0,h_0]$ and a $n\in \{0,\dots,N\}$ such that 
\begin{equation}\label{eq:supportdistance2}
	\dist_{\wt{X}}(\overline{\wt{\mathcal{O}}''_{t,\delta}}\cap\supp(1-\psi_{n+1,t}),\supp\psi_{n,t})
	 < h^{2\varepsilon}.
\end{equation}
Let $x_0 \in \overline{\wt{\mathcal{O}}''_{t,\delta}}\cap\supp(1-\psi_{n+1,t})$ 
and $x_1\in\supp\psi_{n,t}$. Put $y_k=y^{-t}_\delta(x_k)$, $k=0,1$, and 
$\rho_k = \rho_\delta^t(y^{-t}_\delta(x_k))$, in the notations introduced 
before Proposition \ref{prop:WKB}. Plugging this 
into \eqref{eq:supportdistance1} yields 
\begin{equation*}
\begin{split}
	\dist_{T^*\wt{X}}((y_0,d_{y_0}\corM{\wt{\phi}}_{0,\delta}),
					(y_1,d_{y_1}\corM{\wt{\phi}}_{0,\delta}))
	&\leq h^{-\varepsilon/2} \mathrm{dist}_{T^*\wt{X}}(\rho_0,\rho_1) 
	+  C \delta h^{-\beta} h^{-\varepsilon/2}\\
	&\stackrel{\eqref{eq:BorneCL} }{\leq} 
	C h^{-3\varepsilon/2 }\mathrm{dist}_{\wt{X}}(x_0,x_1) 
	+  C \delta h^{-\beta} h^{-\varepsilon/2}\\
	&\stackrel{\eqref{eq:supportdistance2} }{\leq} 
	C h^{\varepsilon /2}
	+  C \delta h^{-\beta} h^{-\varepsilon/2}\\
	&\stackrel{(\ref{eq:CondBeta}) }{\leq} 
	h^{\varepsilon/2}
	+  C h^{\varepsilon/2+\beta}.
\end{split}
\end{equation*}
Thanks to \eqref{eq:CutOfffun1}, the left hand side of the 
first line is bounded from below by some positive constant. Since we can 
take $h_0>0$, and therefore $h>0$, arbitrarily small, we have a contradiction 
which concludes the proof. 
\end{proof}
Next, we need to understand better the operator $\e^{-\frac{i}{h} 
\wit{\phi}_{t,\delta}(x)} \wt{Q}_\omega \e^{\frac{i}{h} \wit{\phi}_{t,\delta}(\cdot)}$ 
appearing in (\ref{eq:WKB_Ansatz.0}), when working in the \ref{eq:Pseudo}.
\begin{prop}\label{lem:TE1}
Let $\beta, \delta, \varepsilon_0$ be as in \eqref{eq:CondBeta} 
with $\beta+\varepsilon_0<1/2$. 
Let $Q_\omega\in \Psi_{h,\beta}^{-\infty}(X)$ be as in the \ref{eq:Pseudo} 
with principal symbol $q_\omega$, 
and let $\wt{Q}_\omega\in \Psi_{h,\beta}^{-\infty}(\wt{X})$ 
be a lift of $Q_\omega$ with principal symbol $\wt{q}_\omega=\wh{\pi}^*q_\omega$ 
and satisfying \eqref{eq:liftePseudo1}, \eqref{eq:liftePseudo2}, 
\eqref{eq:liftePseudo3}. Let $t \in [0,\mathfrak{d}|\log h|]$, $\mathfrak{d}>0$ 
small enough, and let $\psi_t$, $\psi_{n,t}$, $n=0,\dots,N$ be as 
in \eqref{eq:CutOfffun31}. Let $\corM{\wt{\phi}}_{t,\delta}$ be as in 
\eqref{eq:eik1.1}, \eqref{eq:Eikonal} and let $0<\varepsilon< \varepsilon_0$. 
\par
Then, there exist $\mathfrak{d}>0$ and $h_0>0$ such that for all 
$h\in (0, h_0]$ and all $t \in [0,\mathfrak{d}|\log h|]$
\begin{equation}\label{eq:WKBn2}
	\wt{Q}_{t,\omega}^n:=
	\psi_t \e^{-\frac{i}{h} \corM{\wt{\phi}}_{t,\delta}(x)} \wt{Q}_\omega
		\e^{\frac{i}{h} \corM{\wt{\phi}}_{t,\delta}(\cdot)} \psi_{n,t}
	\in \Psi_{h,\beta+\varepsilon}^{-\infty}(\wt{X})
\end{equation}
with principal symbol $\wt{q}_{\omega,n}^t(x,\xi):=(\wh{\pi}^*q_\omega)(x,\xi + d_x\corM{\wt{\phi}}_{t,\delta})\psi_{n,t}(x) \in 
S^{-\infty}_{\beta+\varepsilon}(T^*\wt{X})$. 
\par 
Moreover,  for each lifted cut-off chart $(\wt{\kappa}_\iota,
\wt{\chi}_\iota)$, $\iota \in I$, on $\wt{X}$, see Section 
\ref{subsec:LiftPseudo} for the definition, we have that 
\begin{equation}\label{eq:LocalSymbola}
	(\wt{\kappa}_\iota^{-1})^*\wt{\chi}_\iota
	\wt{Q}_{t,\omega}^n
	\wt{\chi}_\iota\wt{\kappa}_\iota^*
	= 
	\chi\circ\kappa^{-1}\Op_h(\wt{q}_{\kappa,\iota,\omega,n}^t)
	\chi\circ\kappa^{-1},
\end{equation}
\corM{
with $\wt{q}_{\kappa,\iota,\omega,n}^t = (\wt{q}_{\kappa,\iota,\omega,n}^t)^0
+O(h^{1-2(\beta+ \varepsilon)}) \in S^{-\infty}_{\beta+\varepsilon}(T^*\R^d)$, 
for $h\in]0,h_0]$ and 
with the constant in the symbol estimates uniform in $(x,\xi)\in T^*\R^d$ 
and in $t\in [0,\mathfrak{d}|\log h|]$ and independent of $\iota$. 
Moreover, 
\begin{equation}\label{eq:LocalSymbolb}
	(\wt{q}_{\kappa,\iota,\omega,n}^t)^0 = 
	((\widehat{\kappa}^{-1})^* q_\omega)(x,\xi+d_x(\widetilde{\kappa}_{\iota}^{-1})^* \corM{\wt{\phi}}_{t,\delta})
	(\psi_{n,t}\circ \wt{\kappa}_\iota^{-1})(x), 
	\text{ on } T^*O,
\end{equation}
where $O$ is a small open neighbourhood of $\supp \chi\circ\kappa^{-1}$.}
\par
Furthermore, if $\theta,\varphi \in C^\infty_c(X)$ are such that 
$\supp \theta \cap \supp \varphi = \emptyset$, then for every $N\in\N$ we have
\begin{equation}\label{eq:ConjPseudo0}
	\wt{\theta}_{\iota'}
	\wt{Q}_{t,\omega}
	\wt{\varphi}_{\iota}
	=
	O_{\theta,\varphi,N}(h^\infty)
	:H^{-N}_{h}(\wt{X}) \to H^N_{h}(\wt{X}).
\end{equation}
\end{prop}
\red{
\begin{rem}
	This theorem is standard when the phase $\wit{\phi}_{t,\delta}$ \MI{is}
	a smooth function with compact support and bounded derivatives. 
	The difficulty here is that $\wit{\phi}_{t,\delta}$ is $h$-dependent 
	and that both its support and its derivatives are large as a 
	function of $h$.  \MI{We will thus} apply the standard 
	proof while keeping track of these to peculiarities. 
\end{rem}
}
\begin{proof}[Proof of Proposition \ref{lem:TE1}]
1. To prove the last point of the proposition, we first consider the case 
when $\theta,\varphi \in C^\infty_c(X)$ have sufficiently small supports 
so that their lifts $\wt{\theta}_{\iota'}, \wt{\varphi}_\iota \in  C^\infty_c(\wt{X})$ 
are well defined. It follows from \eqref{eq:DefHyp} that there exists a 
constant $C>0$ such that 
\begin{equation}\label{eq:ConjPseudo1.01}
\mathrm{vol}_{\wt{g}}(\wt{\mathcal{O}}_{t,\delta})= 
	O( h^{-C\mathfrak{d}})
\end{equation}
for all $t\in [0,\mathfrak{d}|\log h|]$. Take a locally finite 
partition of unity of $\wt{X}$ by a lifted partition of unity of $X$, as in 
the paragraph above {\eqref{eq:sa8.3}.} We see that we can cover 
$\wt{\mathcal{O}}_{t,\delta}$ by $O( h^{-C\mathfrak{d}})$ many cut-off functions 
from the lifted partition of unity. Combining this 
observation with \eqref{eq:liftePseudo3}, \eqref{eq:SetInclusion_Cut-off2} 
and Proposition \ref{lem:ControlDeriv} yields \eqref{eq:ConjPseudo0}. 
\\
\par
2. Let $(\wt{\kappa}_{\iota},\wt{\chi}_{\iota})$ be a lifted cut-off chart 
such that $\supp \chi\circ \kappa^{-1}$ is convex and such that 
$\supp \wt{\chi}_{\iota} \cap \supp \psi_{n,t}\neq \emptyset$. 
The case of a general cut-off function can be recovered by a partition 
of unity.  By \eqref{eq:liftePseudo2}, 
\begin{equation}\label{eq:ConjPseudo1}
	(\wt{\kappa}_{\iota}^{-1})^*\wt{\chi}_{\iota}
	\wt{Q}_{t,\omega}^n\wt{\chi}_{\iota}\wt{\kappa}_{\iota}^*
	= \chi_\kappa \psi_{t,\iota}
	\e^{-\frac{i}{h}\wit{\phi}_{t,\delta,\iota}}
	\Op_h((q_\omega)_{\kappa,\corM{\chi}})
	\e^{-\frac{i}{h}\wit{\phi}_{t,\delta,\iota}}
	\psi_{n,t,\iota}
	\chi_\kappa .
\end{equation}
where $\psi_{t,\iota}:=\psi_t\circ \wt{\kappa}_{\iota}^{-1}$, 
$\psi_{n,t,\iota}:=\psi_{n,t}\circ \wt{\kappa}_{\iota}^{-1}$, 
$\chi_\kappa = (\kappa^{-1})^*\chi$, 
$\wit{\phi}_{t,\delta,\iota} := \wit{\phi}_{t,\delta}\circ \wt{\kappa}_{\iota}^{-1}$ 
\corM{and $(q_\omega)_{\kappa,\chi}\in S^{-\infty}_\beta(T^*\R^d)$. 
By \eqref{eq:principalSymbrel} we have that its principal part is 
\begin{equation}\label{eq:WKBlaireau.0}
	(q_\omega)_{\kappa,\chi}^0 = q_\omega \circ \wh{\kappa}^{-1} 
	\text{ on } T^*O,
\end{equation}
where $O$ is a small open neighbourhood of $\supp \chi_\kappa$.} 
\par 
We will recover the symbol $\wt{q}_{\kappa,\iota,\omega,n}^t$ \eqref{eq:LocalSymbola}  
in each coordinate patch $(k,\iota)$ by the method of oscillatory 
testing \cite[Theorem 4.19]{Zw12}. Let $C^\infty_c(X)\ni\chi' 
\succ \chi$ with support in a sufficiently small neighbourhood of 
the support of $\chi$ such that the support of 
$\chi_\kappa' = (\kappa^{-1})^*\chi'$ is convex. Set 
\begin{equation}\label{eq:WKBlaireau}	
\begin{split}
	&r_{\kappa,\iota,t}(x,\xi;h)
	:=
	\e^{-\frac{i}{h} x\cdot \xi}
	\psi_{t,\iota}(x)
	\e^{-\frac{i}{h}\corM{\wt{\phi}}_{t,\delta,\iota}(x)}
	\chi_\kappa'
	\Op_h((q_\omega)_{\kappa,\chi})\left(
	\chi_\kappa'(y)
	\e^{-\frac{i}{h}\corM{\wt{\phi}}_{t,\delta,\iota}(\cdot)}
	\psi_{n,t,\iota}(\cdot)
	\e^{\frac{i}{h} \cdot  \xi }\right)\\
	&= 
	\frac{1}{(2h\pi)^d} \iint_{\R^{2d}} 
	\e^{\frac{i}{h} (\eta -\xi) \cdot (x-y) 
	+\frac{i}{h}( \corM{\wt{\phi}}_{t,\delta,\iota}(y)- \corM{\wt{\phi}}_{t,\delta,\iota}(x))} 
	(q_{\omega})_{\kappa,\chi}(x,\eta;h)
	\psi_{t,\iota}(x)\chi_\kappa'(x)\chi_\kappa'(y)\psi_{n,t,\iota}(y) d y d\eta.
\end{split}
\end{equation}
\corM{Notice here that $r_{\kappa,\iota,t}(x,\xi;h)=0$ when $x\notin \supp \chi'_\kappa$ 
and} that $r_{\kappa, \iota,t}$ depends on $\iota$ only through the phases 
$\corM{\wt{\phi}}_{t,\delta,\iota}$ and the cut-off functions $\psi_{t,\iota}$, 
$\psi_{n,t,\iota}$. 
\par 
Our aim is now to show that 
$r_{\kappa,\iota,t}\in S^{-\infty}_{\beta+\varepsilon}(T^*\R^d)$, and that it 
has the form \eqref{eq:LocalSymbolb}. Once this is shown, \eqref{eq:LocalSymbola} 
follows from \cite[Theorem 4.19]{Zw12} and \eqref{eq:WKBn2} follows 
from using additionally \eqref{eq:ConjPseudo0} and a partition of unity. 
\par 
3. To evaluate (\ref{eq:WKBlaireau}), we use Kuranishi's trick and 
write
\begin{equation*}
	\corM{\wt{\phi}}_{t,\delta,\iota}(y)- \corM{\wt{\phi}}_{t,\delta,\iota}(x)
	= \langle F_{t,\delta}(x,y), (y-x) \rangle,
\end{equation*}
with
\begin{equation*}
	F_{t,\delta}(x,y) 
	= \int_0^1 \nabla \corM{\wt{\phi}}_{t,\delta,\iota}(\tau y + (1-\tau) x) d\tau.
\end{equation*}	
\par
Performing a linear change of variable in (\ref{eq:WKBlaireau}), we get
\begin{equation}\label{eq:WKBaignade}
r_{\kappa,\iota,t}(x,\xi;h)
= 	\frac{1}{(2h\pi)^d} \iint_{\R^{2d}} 
	\e^{\frac{i}{h} (\eta -\xi)\cdot(x-y)} 
	\check{q}_{\omega,\kappa} (x,y,\eta ; h)
	\psi_{t,\iota}(x)\chi_\kappa'(x)\chi_\kappa'(y)\psi_{n,t,\iota}(y) 
	d y d\eta.
\end{equation}
with $\check{q}_{\omega,\kappa} (x,y,\eta ; h) := 
(q_{\omega})_{\kappa,\chi} (x,\eta + F_{t,\delta}(x,y);h)$.
\par 
By Proposition \ref{lem:ControlDeriv} we have that for 
$0<\varepsilon<\varepsilon_0$ there exist $\mathfrak{d}>0$ and $h_0>0$ 
such that for all $h\in]0,h_0]$
\begin{equation*}
\|F_{t,\delta}\|_{\corM{C^L((\supp \chi_\kappa')^2}} = 
\begin{cases}
O(1), \quad L=0, \\
O(h^{-\varepsilon- (L-1) (\beta+ \varepsilon)}), \quad L\geq 1,
\end{cases}
\end{equation*}
uniformly in $t\in[0,\mathfrak{d}|\log h|]$. Notice here that the 
constant in the error estimate can be taken independently of $\iota$. 
\par 
Recall that $q_\omega \in S^{-\infty}_{\beta}(T^*X)$ and notice that 
since $\|F_{t,\delta}\|_{C^0} =O(1)$, we find that for each $M\in\N$ 
there exists a $C_M>0$ (independent of $\iota$) such that for all 
$\eta\in T_x\R^d$, all $x,y\in \supp \chi\circ \kappa^{-1}$ and all 
$t\in[0,\mathfrak{d}|\log h|]$
\begin{equation*}
	\langle \eta + F_{t,\delta}(x,y)\rangle^{-M}
	\leq C_M \langle \eta \rangle^{-M}.
\end{equation*}
The product and chain rule then yield that for all $\alpha,\gamma,\tau\in\N^d$, 
$M\in\N$ there exists a constant $C_{\alpha,\gamma,\tau,M}>0$ independent of 
$\iota$ such that for all $t\in[0,\mathfrak{d}|\log h|]$
\begin{equation}\label{eq:SymbolClass}
	|\partial_x^\alpha\partial_y^\gamma \partial_\eta^\tau 
	\check{q}_{\kappa_k}(x,y,\eta;h)|\leq C_{\alpha,\gamma,\tau,M}  
	h^{- (|\alpha|+|\gamma|)(\beta + \varepsilon) -|\tau|\beta}
	\langle \eta \rangle^{-M},
\end{equation}
\corM{uniformly for $x,y$ in a small open neighbourhood of $\supp\chi_\kappa'$.} 
The method of stationary phase applied to \eqref{eq:WKBaignade} yields
\begin{equation}\label{eq:WKBn6}
	r_{\kappa, \iota, t}(x,\xi;h)  
	= \sum_{n=0}^{N-1} r_{k, \iota, t,n}(x,\xi;h)  
	+ R_{\kappa,\iota,t,N}(x,\xi; h),
\end{equation}
where
\begin{equation}\label{eq:WKBn7}
\begin{split}
r_{\kappa, \iota, t, n}(x,\xi;h) 
&=  \frac{h^{n}i^n}{n!} (\partial_y\cdot\partial_\eta)^n
\check{q}_{\omega,\kappa} (x,y,\eta ; h)
\psi_{t,\iota}(x)\chi_\kappa'(x)\chi_\kappa'(y)\psi_{n,t,\iota}(y) 
\bigg|_{\substack{y=x\\ \eta=\xi}}.
\end{split}
\end{equation}
Moreover, 
\begin{equation}\label{eq:WKBn8}
\begin{split}
&R_{\kappa,\iota,t,N}(x,\xi;h)\\
&= O(h^N)\!\!\sum_{|k + \beta| \leq 2d+1}
	\|\partial_y^k\partial_\eta^\beta(\partial_y\cdot\partial_\eta)^N
		\check{q}_{\omega,k} (x,y,\eta ; h) 
		\psi_{t,\iota}(x)\chi_\kappa'(x)\chi_\kappa'(y)\psi_{n,t,\iota}(y) 
		\|_{L^1(\R_y^{d}\times \R^d_\eta)},
\end{split}
\end{equation}
where the constant in the estimate depends only on the dimension $d$ and 
$N$. \corM{Moreover, notice that $r_{\kappa, \iota, t,n}(x,\xi;h), R_{\kappa,\iota,t,N}(x,\xi;h)=0$ 
when $x\notin\supp \chi_\kappa'$.}
\par 
Using (\ref{eq:SymbolClass}), \corM{Proposition \ref{lem:ControlDeriv}} and Remark \ref{rem:Cut-off_est1} 
along with the rapid decay of $q_\omega$ in the $\xi$ variable, we obtain by the product 
rule that for all $m\in \N$ and all 
$\alpha,\gamma \in \N^d$ there exists a constant $C_{\alpha,\gamma,m,n}>0$ 
such that for all $t\in [0,\mathfrak{d}|\log h|]$ and all $(x,\xi)\in T^*\R^d$, 
\begin{equation}\label{eq:WKBn9}
	|\partial_x^\alpha \partial_\xi^\gamma r_{\kappa, \iota, t, n}(x,\xi;h)| 
	\leq C_{\alpha,\gamma,m,n} h^{n(1-2\beta - 4 \varepsilon)} h^{-( |\alpha| + |\gamma|) (\beta + 2\varepsilon)}
	\langle \xi \rangle^{-m}, 
\end{equation}
and
\begin{equation}\label{eq:WKBn10}
	|\partial_x^\alpha \partial_\xi^\gamma R_{\kappa,\iota,t,N}(x,\xi;h)| 
		\leq C_{\alpha,\beta,m,N} h^N
		h^{-(2d+1 + 2N + |\alpha| + |\gamma|) (\beta + 2\varepsilon)}
		\langle \xi \rangle^{-m}.
\end{equation}
We insist here on the fact that the symbol estimates are independent of 
$\iota$ indexing the sheet in the universal cover to which we have lifted the 
cut-off chart $(\kappa,\chi)$. 
\par
Taking $\varepsilon>0$, and hence $\mathfrak{d}>0$, smaller, the symbol 
estimates \eqref{eq:WKBn9} and \eqref{eq:WKBn10} hold with $2\varepsilon$ 
replaced by $\varepsilon$. We deduce that 
\begin{equation}\label{eq:WKBn11}
	r_{\kappa,\iota, t}(x,\xi;h) \sim \sum_{n=0}^\infty r_{\kappa,\iota,t,n}(x,\xi;h) 
	\quad \text{in } S^{-\infty}_{\beta+\varepsilon}(T^*\R^d)
\end{equation}
with the constant in the symbol estimates uniform in $(x,\xi)\in T^*\R^d$
and in $t\in [0,\mathfrak{d}|\log h|]$ and independent of $\iota$. This 
implies \eqref{eq:LocalSymbola}. 
\par 
\corM{On a small neighbourhood of $\supp \chi_\kappa$ we have 
$\chi_\kappa'= 1$. Moreover, 
$\psi_t=1$ on the support of $\psi_{n-1,t}$ by \eqref{eq:CutOfffun31}. 
Thus, restricting to $x$ in  a small neighbourhood of $\supp \chi_\kappa$, we get}
% to the factor $\chi_\kappa$ in \eqref{eq:ConjPseudo1} we can restrict 
% the symbol $r_{k,\iota, t}(\cdot,\xi;h)$ to the support of $\chi_\kappa$ 
% where $\chi_\kappa'\equiv 1$.  In particular, 
% we have 
%
\begin{equation*}
	r_{\kappa, \iota, t, 0}(x,\xi;h) 
	= (q_{\omega})_{\kappa,\chi} (x,\xi + \nabla \wit{\phi}_{t,\delta,\iota}(x);h)
	\psi_{n,t,\iota}(x) 
\end{equation*}
which by \eqref{eq:WKBlaireau.0} yields the formula for the leading part 
in \eqref{eq:LocalSymbolb} and thus, \corM{using a partition of unity,} 
yields the formula for the principal symbol of \eqref{eq:WKBn2}. 
\end{proof}

Now, we may use the previous proposition to give estimates on the 
operator $\mathcal{R}_{\delta,t}$ defined in \eqref{eq:WKB2.0}.
\begin{prop}\label{lem:lem_P2}
Let $\beta, \delta, \varepsilon_0$ be as in (\ref{eq:CondBeta}) with 
$\beta+\varepsilon_0<1/2$ \red{and let $\red{\psi_t}$ be as in 
\eqref{eq:CutOfffun2}, \eqref{eq:CutOfffun31}}. Let $\mathcal{R}_{\delta,t}$ be as in 
\eqref{eq:WKB2.0}.  For every $0<\varepsilon< \varepsilon_0$, there 
exists $\mathfrak{d}>0$, $h_0>0$ such that the following holds uniformly 
in $h\in ]0,h_0]$ and $t\in [0,\mathfrak{d}|\log h|]$: 
Let $f\in C_c^\infty(\wt{\mO}_{t,\delta})$ such that, for all $L\in \N$ there 
exists a constant $C_L>0$ such that
\begin{equation}\label{eq:WKBnn11.0}
	\|f\|_{C^L} \leq C_L h^{-L (\beta + \varepsilon)}.
\end{equation}
Then, for all $L \in \N$ and all $n=1,\dots,N$ there exists a constant $C_{n,L}>0$
such that 
\begin{equation}\label{eq:WKBnn11}
	\|\red{\psi_t}\mathcal{R}_{\delta,t}\psi_{n-1,t} f\|_{C^L} 
	\leq 
	\corM{
	C_{n,L} h^{1-(L+2)(\beta+\varepsilon)}}.
\end{equation}

\end{prop}
\begin{proof}
1. 
Let $0<\varepsilon< \varepsilon_0$, let $\mathfrak{d}>0$, $h_0>0$ be 
small enough, let $h\in (0, h_0]$ and $t \in [0,\mathfrak{d}|\log h|]$.
Recall the discussion and quantities in the paragraph above \eqref{eq:sa8.3}. 
Let $\chi_{k}'	 \succ \chi_k$ with support in a sufficiently small 
neighbourhood of the support of $\chi_k$. Then we have the relation 
$\wt{\chi}_{k,\iota}' \succ \wt{\chi}_{k,\iota}$, $\iota \in I$, for the 
lifted cut-off functions and 

\begin{equation}\label{eq:WKBnn2.2}
	\|\red{\psi_t\mathcal{R}_{\delta,t}}\psi_{n-1,t}f\|_{C^L} 
	\leq \sum_{\substack{ k\in K \\ \iota\in I} }
	\bigg(
	\|\widetilde{\chi}'_{k,\iota}
	\red{\psi_t\mathcal{R}_{\delta,t}}\psi_{n-1,t}
	\widetilde{\chi}_{k,\iota}f \|_{C^L}
	+\|(1-\widetilde{\chi}'_{k,\iota})
	\red{\psi_t\mathcal{R}_{\delta,t}}\psi_{n-1,t}
	\widetilde{\chi}_{k,\iota} f \|_{C^L}
	\bigg).	
\end{equation}
Here $K$ is finite and $I$ is a countable set. However, it follows 
from \eqref{eq:ConjPseudo1.01} that \corM{due to the} $\psi_t$ prefactor 
the sum in \eqref{eq:WKBnn2.2} 
is over at most $O(h^{-2C\mathfrak{d}})$ many terms. This is bounded 
by $O(h^{-\varepsilon})$ for $\mathfrak{d}$ small enough and both estimates  
are uniform in $t \in [0,\mathfrak{d}|\log h|]$. 
\par
Since the supports of $(1-\widetilde{\chi}_{k,\iota}')$ and 
$\widetilde{\chi}_{k,\iota}$ are disjoint, it follows from 
\eqref{eq:WKBn2}, \eqref{eq:WKB2.0} that 
\begin{equation}\label{eq:WKBnn2.3}
	(1-\widetilde{\chi}_{k,\iota}')
	\red{\psi_t\mathcal{R}_{\delta,t}}\psi_{n-1,t} \widetilde{\chi}_{k,\iota}
		=
	(1-\widetilde{\chi}'_{k,\iota})\wt{Q}_{t,\omega}^{\corM{n-1}}
	\widetilde{\chi}_{k,\iota}.  
\end{equation}
Hence, keeping in mind \eqref{eq:ConjPseudo1.01}, it 
follows from \eqref{eq:ConjPseudo0} that for every 
$L\in\N$  
\begin{equation}\label{eq:WKBnn2.4}
\begin{split}
	\sum_{\substack{ k\in K \\ \iota\in I} }
	\|(1-\widetilde{\chi}'_{k,\iota})
	&\red{\psi_t\mathcal{R}_{\delta,t}}\psi_{n-1,t}
	\widetilde{\chi}_{k,\iota}f \|_{H^L_h(\wt{X})}
\leq 
\sum_{\substack{ k',k\in K \\ \iota',\iota\in I} }
\|(1-\widetilde{\chi}_{k,\iota}')
	\wt{Q}_{t,\omega}^{\corM{n-1}}
	\widetilde{\chi}_{k,\iota}\widetilde{\chi}_{k',\iota'}f 
	\|_{H^L_h(\wt{X})}
=
O_{L}(h^\infty),
\end{split}
\end{equation}
uniformly in $t \in [0,\mathfrak{d}|\log h|]$. By the Sobolev 
inequalities \eqref{eq:sa8.2} and \eqref{eq:WKBnn2.4}, we find that 
\begin{equation}\label{eq:WKBnn2.4.a}
\sum_{\substack{ k\in K \\ \iota\in I} }
	\|(1-\widetilde{\chi}'_{k,\iota})
	\red{\psi_t\mathcal{R}_{\delta,t}}\psi_{n-1,t}
	\widetilde{\chi}_{k,\iota}f \|_{C^L}\\
=
O_{L}(h^\infty).
\end{equation}
%
% Here, we used 
% \eqref{eq:ConjPseudo1.01} which implies that 
% $\| \psi_{n,t} \|_{H^0_h(\wt{X})} = O(h^{-C\mathfrak{d}/2})$ 
% uniformly in $t \in [0,\mathfrak{d}|\log h|]$.
% \\
% \par 
%
2. Next we turn to the first term on the right hand side of 
\eqref{eq:WKBnn2.2}. By \eqref{eq:sa8.1} 
\begin{equation}\label{eq:WKBnn2.5}
	\|\widetilde{\chi}'_{k,\iota}
	\red{\psi_t\mathcal{R}_{\delta,t}}\psi_{n-1,t}
	\widetilde{\chi}_{k,\iota}f \|_{C^L}
	= \max_{k'\in K, \iota' \in I}
	\|(\wt{\kappa}_{k',\iota'}^{-1})^*(\widetilde{\chi}_{k',\iota'}
	\widetilde{\chi}'_{k,\iota})
	\red{\psi_t\mathcal{R}_{\delta,t}}
	\widetilde{\chi}_{k,\iota}\psi_{n-1,t}
	f\|_{C^L(\R^d)}. 
\end{equation}
By Lemma \ref{lem:geodesicCoord} we know that the maximal number 
of the balls $\widetilde{U}_{k,\iota}$ with non-empty intersection is 
bounded by a constant $0<C_0<\infty$. Given an index pair $(k,\iota)$
there are at most $C_0$ many index pairs $(k',\iota')$ such that 
$\supp \widetilde{\chi}_{k',\iota'}\cap \supp \widetilde{\chi}'_{k,\iota}
\neq\emptyset$. We will denote this condition by the relation 
$(k,\iota)\sim (k',\iota')$. Now, if $(k,\iota)\sim (k',\iota')$, we have 
\begin{equation*}
	(\wt{\kappa}_{k',\iota'}^{-1})^*(\widetilde{\chi}_{k',\iota'})
	\widetilde{\chi}'_{k,\iota} = \left((\wt{\kappa}_{k,\iota})\circ (\wt{\kappa}_{k',\iota'}^{-1}) \right)^*(\wt{\kappa}_{k,\iota}^{-1})^*(\widetilde{\chi}_{k',\iota'}
	\widetilde{\chi}'_{k,\iota}).
\end{equation*}
The map $\left((\wt{\kappa}_{k,\iota})\circ (\wt{\kappa}_{k',\iota'}^{-1}) \right)^*$ 
is a diffeomorphism on the support of $\widetilde{\chi}_{k',\iota'}
\widetilde{\chi}'_{k,\iota}$, and its derivatives as well as the derivatives 
of its inverse are all bounded independently of $\iota, \iota'$. We deduce that
\begin{equation}\label{eq:WKBnn2.6 }
	\begin{aligned}
\|\widetilde{\chi}'_{k,\iota}
\red{\psi_t\mathcal{R}_{\delta,t}}\psi_{n-1,t}
\widetilde{\chi}_{k,\iota}f \|_{C^L}
&\leq O(1) \max_{(k',\iota')\sim (k,\iota)}
\|(\wt{\kappa}_{k,\iota}^{-1})^*\left(
\widetilde{\chi}_{k',\iota'} \widetilde{\chi}'_{k,\iota}\red{\psi_t\mathcal{R}_{\delta,t}}
\widetilde{\chi}_{k,\iota}\psi_{n-1,t}
f\right)\|_{C^L(\R^d)}\\
&= O(1)
\|(\wt{\kappa}_{k,\iota}^{-1})^*\left(
	\widetilde{\chi}'_{k,\iota}\red{\psi_t\mathcal{R}_{\delta,t}}
\widetilde{\chi}_{k,\iota}\psi_{n-1,t}
f\right)\|_{C^L(\R^d)}.
\end{aligned}
\end{equation}
\corM{
Expanding the term $\red{\psi_t\mathcal{R}_{\delta,t}}$ as in \eqref{eq:WKB2.0}, 
we get from \eqref{eq:WKBn2} that 
\begin{equation}\label{eq:WKBnn2.7}
\begin{aligned}
&(\wt{\kappa}_{k,\iota}^{-1})^*\left(\widetilde{\chi}'_{k,\iota}\red{\psi_t\mathcal{R}_{\delta,t}}
	\widetilde{\chi}_{k,\iota}\psi_{n-1,t} f \right)\\
& = (\wt{\kappa}_{k,\iota}^{-1})^*\left(\widetilde{\chi}'_{k,\iota} \wt{Q}^{n-1}_{t,\omega}
	\widetilde{\chi}_{k,\iota}'\widetilde{\chi}_{k,\iota} f-\left(\wit{q}_\omega(x,d_x\wt{\phi}_{t,\delta}(x))
		+\left(\frac{h}{i}H^1_{\wt{\phi}_{t,\delta},\wit{q}_\omega}
		+\frac{h}{2i}\mathrm{div}_{\wit{g}} (H^1_{\wt{\phi}_{t,\delta},\wit{q}_\omega})
		\right)\right)\widetilde{\chi}_{k,\iota}\psi_{n-1,t} f
	 \right),
\end{aligned}
\end{equation}
where we used as well that $\psi_t \succ \psi_{n-1,t}$, 
$\widetilde{\chi}'_{k,\iota}\succ \widetilde{\chi}_{k,\iota}$ and the fact that 
the last three terms in \eqref{eq:WKBnn2.7} are local operators. 
Next, write $\psi_{n-1,t,k,\iota} := (\wt{\kappa}_{k,\iota}^{-1})^*
\left(\widetilde{\chi}_{k,\iota}\psi_{n-1,t}\right)$. Using 
\eqref{eq:LocalSymbola} we get that \eqref{eq:WKBnn2.7} is equal to 
%, 
%
\begin{equation}\label{eq:WKBnn2.7.0}
\begin{aligned}
	(\chi_{k}'\circ \kappa_{k}^{-1})
	\Op_h(\wt{q}_{\kappa_k,\iota,\omega,n-1}^t)
	(\wt{\kappa}_{k,\iota}^{-1})^*(\widetilde{\chi}_{k,\iota}f)
	-
	\left(
	(\wt{\kappa}_{k,\iota}^{-1})^*(\wit{q}_\omega(x,d_x\wit{\phi}_{t,\delta}(x)))
	+
	\mathcal{L}_{\kappa_k,\iota,t}
	\right)
	\psi_{n-1,t,k,\iota}  (\wt{\kappa}_{k,\iota}^{-1})^*f,
\end{aligned}
\end{equation}
where $\wt{q}_{\kappa_k,\iota,\omega,n-1}^t\in S^{-\infty}_{\beta+\varepsilon}(T^*\R^d)$ and}
\begin{equation*}
	\mathcal{L}_{\kappa_k,\iota,t}:=(\wt{\kappa}_{k,\iota}^{-1})^*\left(
		\frac{h}{i}H^1_{\wit{\phi}_{t,\delta},\wit{q}_\omega}
		+\frac{h}{2i}\mathrm{div}_{\wit{g}} (H^1_{\wit{\phi}_{t,\delta},\wit{q}_\omega})
	\right)(\wt{\kappa}_{k,\iota})^*.
\end{equation*}
3. We first estimate 
\corM{
\begin{equation*}
\begin{aligned}
	I_{k,\iota}&:=
	\|\mathcal{L}_{\kappa_k,\iota,t}
	\psi_{n-1,t,k,\iota}  (\wt{\kappa}_{k,\iota}^{-1})^*\widetilde{\chi}'_{k,\iota}f\|_{C^L(\R^d)}. \\
	% &= \| \mathcal{L}_{\kappa_k,\iota,t}
	% \psi_{n-1,t,k,\iota}  (\wt{\kappa}_{k,\iota}^{-1})^*\widetilde{\chi}'_{k,\iota}f\|_{C^L(\R^d)}.
	\end{aligned}
\end{equation*}
}
% since $(\wt{\kappa}_{k,\iota}^{-1})^*(\psi_t \widetilde{\chi}_{k,\iota}') \equiv 1$ on the support of $\psi_{n-1,t,k,\iota}$.
%
%
Using the observation after \eqref{eq:WKBnn2.5}, it readily follows that
\begin{equation}\label{eq:WKBnn2.713}
	\|
	(\wt{\kappa}_{k,\iota}^{-1})^*
	\widetilde{\chi}_{k,\iota}'
	f\|_{C^L(\R^d)} \leq O_{k}(1)\|f\|_{C^L},
\end{equation}
uniformly in $\iota$. 
\par
\corM{
Next, recall that $\wit{q}_\omega=\wh{\pi}^*q_\omega$ and write 
$q_{\omega,\kappa_k}=(\wh{\kappa}_k^{-1})^*q_\omega$ for $q_\omega$ 
expressed in the local coordinate chart $\kappa_k$.} A direct computation shows 
that 
\begin{equation}\label{eq:WKBnn2.713.0}
	(\red{\wh{\wt{\kappa}}_{k,\iota}^{-1}})^* \wit{q}_\omega(x,d_x\wit{\phi}_{t,\delta})
=q_{\omega,\kappa_k}(x,d_x(\red{\wt{\kappa}_{k,\iota}^{-1}})^* \wit{\phi}_{t,\delta})
\end{equation}
 Hence, 
\begin{equation*}
\begin{split}
\mathcal{L}_{\kappa_k,\iota,t}=\bigg(
	\frac{h}{i}
	(\nabla_\xi q_{\omega,\kappa_k})(x,d_x(\wt{\kappa}_{k,\iota}^{-1})^{*}\wit{\phi}_{t,\delta})
	\cdot \nabla_x
	+\frac{h}{2i}
	\mathrm{div}_{g} ((\nabla_\xi q_{\omega,\kappa_k})(x,d_x(\wt{\kappa}_{k,\iota}^{-1})^{*}\wit{\phi}_{t,\delta})
	)
\bigg).
\end{split}
\end{equation*}
%
% The symbol estimates \eqref{eq:SymbolClass} (with $x=y$ and $\eta=0$) 
% and 
\corM{
Since $q_\omega \in S^{-\infty}_\beta(T^*X)$, get by applying  
Lemma \ref{lem:CLCompo}, Proposition \ref{lem:ControlDeriv}, Remark \ref{rem:Cut-off_est1} 
and the product rule that for all $L\in\N$}
\begin{equation}\label{eq:WKBnn2.9a}
\begin{split}
	&\| \psi_{n-1,t,k,\iota}
	q_{\omega,\kappa_k}(\cdot,d_x(\kappa_{k,\iota}^{-1})^* \wit{\phi}_{t,\delta})\|_{C^L} 
	\leq O_L(h^{-L(\beta+\varepsilon)}),\\
	&\| \psi_{n-1,t,k,\iota}
	(\nabla_{\xi_j} q_{\omega,\kappa_k})(\cdot,d_x(\wt{\kappa}_{k,\iota}^{-1})^{*}\wit{\phi}_{t,\delta})
	\|_{C^L} 
	\leq O_L(h^{-\beta-L(\beta+\varepsilon)}),\\
	&\| 
	(\nabla_\xi q_{\omega,\kappa_k})(\cdot,d_x(\wt{\kappa}_{k,\iota}^{-1})^{*}\wit{\phi}_{t,\delta})
	\cdot \nabla_x\psi_{n-1,t,k,\iota}\|_{C^L} 
	\leq O_L(h^{-(\beta+\varepsilon) -L(\beta+\varepsilon)}),\\
	&\|	\mathrm{div}_{g}((\nabla_\xi q_{\omega,\kappa_k})(\cdot,d_x(\wt{\kappa}_{k,\iota}^{-1})^{*}\wit{\phi}_{t,\delta})
		)\psi_{n-1,t,k,\iota}\|_{C^L} 
	\leq O_L(h^{-2(\beta+\varepsilon) -L(\beta+\varepsilon)}),
\end{split}
\end{equation}
uniformly in $t\in [0,\mathfrak{d}|\log h|]$ \corM{and in $k,\iota$}. Hence, combining 
\eqref{eq:WKBnn2.713} and \eqref{eq:WKBnn2.9a} with the product rule and assumption 
\eqref{eq:WKBnn11.0} yields that
\begin{equation}\label{eq:WKBnn2.9b}
	I_{k,\iota} %= O(1)\|\mathcal{L}_{\kappa_k,\iota,t}
%	\psi_{n-1,t,k,\iota}
%	(\wt{\kappa}_{k,\iota}^{-1})^*
%	\widetilde{\chi}_{k,\iota}'
%	f\|_{C^L(\R^d)}
	\leq 
	O_L(h^{1-2(\beta+\varepsilon) -L(\beta+\varepsilon)})
\end{equation}
uniformly in $t\in [0,\mathfrak{d}|\log h|]$, uniformly in $k$ and in $\iota$. 
\\
\par
4. In view of \eqref{eq:WKBnn2.6 }, \eqref{eq:WKBnn2.7.0} and \eqref{eq:WKBnn2.713.0} 
it remains to estimate 
\corM{
\begin{equation*}
\left\|(\chi_{k}'\circ \kappa_{k}^{-1})
	\bigg(
	\Op_h(\wt{q}_{\kappa_k,\iota,\omega,n-1}^t)
	-
	q_{\omega,\kappa_k}(x,d_x(\wt{\kappa}_{k,\iota}^{-1})^* \wit{\phi}_{t,\delta})
	(\psi_{n-1,t}\circ\wt{\kappa}_{k,\iota}^{-1})
	\bigg)
	(\wt{\kappa}_{k,\iota}^{-1})^*(\widetilde{\chi}_{k,\iota}f)\right\|_{C^L}.
\end{equation*}	
}
\corM{
By \eqref{eq:LocalSymbola} we know that 
 $\wt{q}_{\kappa_k,\iota,\omega,n-1}^t = (\wt{q}_{\kappa,\iota,\omega,n}^t)^0
+O(h^{1-2(\beta+ \varepsilon)}) \in S^{-\infty}_{\beta+\varepsilon}(T^*\R^d)$
\eqref{eq:LocalSymbolb}, \eqref{eq:WKBnn2.7}, \eqref{eq:WKBnn2.7.0} we know that 
the leading part of $\wt{q}_{\kappa_k,\iota,\omega,n-1}^t$ satisfies 
\begin{equation}\label{eq:WKBnn2.7.1}
	(\wt{q}_{\kappa_k,\iota,\omega,n-1}^t)^0 = 
	((\widehat{\kappa}_k^{-1})^* q_\omega)(x,\xi+d_x(\widetilde{\kappa}_{k,\iota}^{-1})^* \corM{\wt{\phi}}_{t,\delta})
	(\psi_{n-1,t}\circ \wt{\kappa}_{k,\iota}^{-1})(x), 
	\text{ on } T^*O,
\end{equation}
where $O\subset \R^d$ is a small open neighbourhood of $\supp \chi_k'\circ\kappa_k^{-1}$.}
Put $v:= (\wt{\kappa}_{k,\iota}^{-1})^*(\widetilde{\chi}_{k,\iota}'f)$ and 
notice that the $C^L$ norms of $v$ satisfy the same estimates as in \eqref{eq:WKBnn2.713}. 
Thus, it suffices to prove that
\begin{equation}\label{eq:WKBnn2.611}
	\|
	\Op_h(r_{k,\iota,t})		
	v\|_{C^L(\supp \chi_k'\circ\kappa_{k}^{-1})}
	\leq C_{n,L} h^{1-(L+2)(\beta+\varepsilon)},
\end{equation}
where 
\corM{
\begin{equation}\label{eq:WKBnn2.611.0}
		r_{\kappa_k,\iota,t}(x,\xi):=\wt{q}_{\kappa_k,\iota,\omega,n-1}^t(x,\xi)
	- 	q_{\omega,\kappa_k}(x,d_x(\wt{\kappa}_{k,\iota}^{-1})^* \wit{\phi}_{t,\delta})
	(\psi_{n-1,t}\circ\wt{\kappa}_{k,\iota}^{-1})(x) \in S^0_{\beta+\varepsilon}(T^*O).
\end{equation}
}
Taylor expanding the \corM{leading part in \eqref{eq:WKBnn2.7.1},} we find that
\begin{equation}\label{eq:WKBnn2.11}
	\begin{split}
	q_{\omega,\kappa_k}(x,d_x(\red{\wit{\phi}_{t, \delta}}\circ\kappa_{k,\iota}^{-1})+\xi)
	=
	&q_{\omega,\kappa_k}(x,d_x(\red{\wit{\phi}_{t, \delta}}\circ\kappa_{k,\iota}^{-1})(x)) 
	+  e_{t,\kappa,\iota}(x,\xi)\xi
	\end{split}
\end{equation}
with 
\begin{equation*}
	e_{t,\kappa_k,\iota}(x,\xi):= 
	\int_0^1\nabla_\xi
	q_{\omega,\kappa_k}(x,
	d_x(\red{\wit{\phi}_{t, \delta}}\circ\kappa_{k,\iota}^{-1})(x)+\tau \xi) 
	d\tau.
\end{equation*}
\corM{
Combining  \eqref{eq:WKBnn2.7.1}, \eqref{eq:WKBnn2.611.0}
and \eqref{eq:WKBnn2.11}, we find that}
\begin{equation}
\begin{split}
	r_{\kappa_k,\iota,t}(x,\xi;h)&=
	%\psi_{n-1,t,k,\iota}(x)
	(\psi_{n-1,t}\circ\wt{\kappa}_{k,\iota}^{-1})(x)e_{t,\kappa,\iota}(x,\xi)\xi
	+h^{1-2(\beta+\varepsilon)}s_{\kappa_k,\iota,t}(x,\xi)\\
	&\corM{:=
	e'_{t,\kappa_k,\iota}(x,\xi)\xi
	+h^{1-2(\beta+\varepsilon)}s_{\kappa_k,\iota,t}(x,\xi)}\\
\end{split}
\end{equation}
\corM{
for some symbol $s_{\kappa_k,\iota,t}\in S^{-\infty}_{\beta+\varepsilon}(T^*O)$ 
with symbol estimates uniform in $k$ and $\iota$. 
Similar to the symbol estimates \eqref{eq:SymbolClass} we see by Remark \ref{rem:Cut-off_est1} 
that $e_{t,\kappa_k,\iota}'\in h^{-\beta}S^{0}_{\beta+\varepsilon}(T^*O)$ 
with symbol estimates uniform in $k$ and $\iota$.} 
Integration by parts shows that
\begin{equation*}
\begin{split}
	&\Op_h(r_{\kappa_k,\iota,t})v(x)\\
	&=
	\frac{1}{(2h\pi)^d} \iint_{\R^{2d}} 
	\e^{\frac{i}{h} \xi \cdot (x-y)} 
	\left( 
	\corM{e_{t,\kappa_k,\iota}'(x,\xi)}\xi
	+  h^{1-2(\beta+\varepsilon)}s_{k,\iota,t}(x,\xi)\right) v(y) d y d\xi \\
	&=
	\frac{1}{(2h\pi)^d} \iint_{\R^{2d}} 
		\e^{\frac{i}{h} \xi \cdot (x-y)} 
	\left( 
	\corM{e_{t,\kappa_k,\iota}'(x,\xi)}\left(\frac{1+\xi\cdot hD_y}{1+\xi^2}\right)^{d+1}hD_y
		v(y)  
	+  h^{1-2(\beta+\varepsilon)}s_{k,\iota,t}(x,\xi)v(y) \right) d y d\xi.
\end{split}
\end{equation*}
The integral in the last line is absolutely convergent since 
$\supp v \subset \supp \chi_k\circ\kappa_k^{-1}$ and since 
\begin{equation*}
	\left(\frac{1+\xi\cdot hD_y}{1+\xi^2}\right)^{d+1}hD_y
	v(y) 
	= O_L(h^{1-(\beta+\varepsilon)})
	\left(\frac{1}{1+|\xi|}\right)^{d+1},
\end{equation*}
by \eqref{eq:WKBnn11.0} and the paragraph above \eqref{eq:WKBnn2.611}. 
When we apply derivatives $\partial_x^\alpha$ to 
$\Op_h(r_{\kappa_k,\iota,t})v(x)$ they either fall onto the symbols 
$e_{t,\kappa_k,\iota}'$ and $s_{\kappa_k,\iota,t}$, or onto the exponential. In the 
latter case we use that $\partial_x	\e^{\frac{i}{h} \xi \cdot (x-y)}  = 
-\partial_y	\e^{\frac{i}{h} \xi \cdot (x-y)} $ and integration by parts 
which makes the derivative fall onto $v$. Using \eqref{eq:WKBnn2.713}, \eqref{eq:WKBnn11.0} we 
find that for $|\alpha|\leq L$ 
\begin{equation*}
	\partial^\alpha_y\left(\frac{1+\xi\cdot hD_y}{1+\xi^2}\right)^{d+1}hD_y
	v(y) 
	= O_L(h^{1-(L+1)(\beta+\varepsilon)})
	\left(\frac{1}{1+|\xi|}\right)^{d+1}.
\end{equation*}
\corM{
Using also that $s_{\kappa_k,\iota,t}\in S^{-\infty}_{\beta+\varepsilon}(T^*O)$ we conclude that}
\begin{align*}
	\|\Op_h(r_{\kappa_k,\iota})
	v \|_{C^L(\supp \chi_k'\circ\kappa_{k}^{-1})} 
	&= O_{k,L}(1)
		\corM{h^{1-(L+2)(\beta+\varepsilon)}}.
\end{align*}
In particular, the constant in the estimate only depends on $k$ and $L$. 
Combining \eqref{eq:WKBnn2.611}, \eqref{eq:WKBnn2.9b}, \eqref{eq:WKBnn2.6 }, 
\eqref{eq:WKBnn2.2}, \eqref{eq:WKBnn2.4.a}
\begin{equation}\label{eq:WKBnn2end}
		\|\red{\psi_t\mathcal{R}_{\delta,t}}\psi_{n-1,t}f\|_{C^L} 
		=O_{L}(1)
		\corM{h^{1-(L+2)(\beta+\varepsilon)}},
\end{equation}
which ends the proof of the proposition. 
\end{proof}
\MI{We may now give estimates on each of the terms appearing in the WKB expansion.}
\begin{prop}\label{Porp:SobolevEst}
Let $0< \varepsilon_{1} < \MI{\frac{\varepsilon_0}{2}}$%\frac{1}{2}-\beta$
 and let $N\in \N$. 
There exist $\mathfrak{d}>0$ \MI{such that, for all $L\in \N$, the 
following holds. There exists} $h_0 \MI{= h_0(\varepsilon_1,\corM{N}, L)}>0$, 
such that for all 
$h\in (0,h_0]$, all $t \in [0,\mathfrak{d}|\log h|]$, all 
$n\in \{0,..., N\}$, the solution $b_n$ given in 
\eqref{eq:transEq4} and \eqref{eq:transEq5} satisfies
\begin{equation}\label{eq:SobolevEst}
\begin{split}
&\|b_n\|_{C^L} 
	= \MI{O(h^{-L(\beta+ \varepsilon_{1})- n(1- \varepsilon_1) })}.
\end{split}
\end{equation} 
\end{prop}
\begin{proof}
\MI{Let us set, for every $n\in \N$, $\wit{\varepsilon}_n := \frac{\varepsilon_1}{10 \times \tilde{ 2^n}}$.}
\MI{First of all, w}e will show by induction \MI{over $n$ that, for every $n\in \N$,  for all $0<\varepsilon \leq \wit{\varepsilon}_n$, there exists $\gd(\varepsilon)>0$ such that for all  $L\in \N$, we have for all $h>0$ smaller than some $h_0(n,\varepsilon, L)$ and} 
all $t\in [0, \gd(\varepsilon) |\log h|]$
\begin{equation}\label{eq:NormeCLInduc}
\|b_n\|_{C^L} \leq \red{C_L} h^{-L(\beta+ \red{2^{n+1}} \varepsilon)} A_n
\end{equation}
for all 
%$L\in \N$ and all 
%$n$ such that $\beta+ \red{2^{n+1}}\varepsilon < \frac{1}{2}$, 
where $A_n\geq 1$ is a quantity that depends on $h$ and $n$, but not on $L$, 
\red{and $C_L>0$ is an $L$-dependent constant}.
\par 
%
%Let us prove the claim by induction over $n$. 
When $n=0$, it is a consequence 
of Proposition \ref{lem:lem_P1} and equation \eqref{eq:transEq4} 
\red{with $A_0=1$}. Suppose 
that \eqref{eq:NormeCLInduc} holds up to index $n-1$, for all $L\in \N$. 
Applying \corM{\eqref{eq:transEq5},} Proposition \ref{lem:lem_P1} and recalling that $t=O(|\log h|)$, we get
\begin{equation}\label{eq:Induction1}
\begin{split}
		\|b_n\|_{C^L} 
		&\leq \|T^t_{\delta,\Lambda_0} a_n\|_{C^L}
		+\frac{1}{2}\int_0^t 
		\|
		T^{t-s}_{\delta,\Lambda_{s}} \Delta b_{n-1}(s,\cdot\,;\delta)
		-2h^{-2}\delta T^{t-s}_{\delta,\Lambda_{s}} \red{\psi_{n,s}} \mathcal{R}_{\delta,t} b_{n-1}
		(s,\cdot\,;\delta)
		\|_{C^L} ds\\
		&\leq O_L(h^{-(\beta+ 2\varepsilon)L}) \MI{\|a_n\|_{C^L}}\\
		&+O(|\log h|)\max_{s\in[0,t]} \left( \|T^{t-s}_{\delta,\Lambda_{s}} \Delta b_{n-1}(s,\cdot\,;\delta)\|_{C^L}  
		+ \|h^{-2} \delta T^{t-s}_{\delta,\Lambda_{s}}\red{\psi_{n,s}} \mathcal{R}_{\delta,t} b_{n-1}
		(s,\cdot\,;\delta)\|_{C^L}\right).
\end{split}
\end{equation}
Now,  by assumption,  we have for any $k\in \N$, 
$\|\Delta b_{n-1} \|_{C^k} \leq C_{\MI{k}} h^{-(k+2)(\beta+ 2^{\red{n}}\varepsilon)} \red{A_{n-1}}$, 
so we may apply Proposition \ref{lem:lem_P1} with $f= \red{A_{n-1}^{-1}} 
h^{2(\beta + 2^{\red{n}} \varepsilon)} \Delta b_{n-1}$ and $2^{\red{n}} \varepsilon$ 
playing the role of $\varepsilon$ to obtain
\begin{equation*}
	\|T^{t-s}_{\delta,\Lambda_{s}} \Delta 
	b_{n-1}(s,\cdot\,;\delta)\|_{C^L} 
	\leq O_L(1) A_{n-1} h^{-2(\beta + 2^{\red{n}} \varepsilon)}  
		h^{- L (\beta + 2^{\red{n+1}} \varepsilon)}.
\end{equation*}
To deal with the last term in \corM{\eqref{eq:Induction1}}, we first recall that, by the induction 
hypothesis,  we have $\|b_{n-1}\|_{C^k} \leq C A_{n-1} 
h^{-k( \beta + 2^{\red{n}} \varepsilon)}$. We may thus apply Proposition 
\ref{lem:lem_P2} with $f= A_{n-1}^{-1} b_{n-1}$ and $2^{\red{n}}\varepsilon$ 
playing the role of $\varepsilon$ to deduce that, for any $L\in \N$,
$$\|\red{\psi_{n,s}}\mathcal{R}_{\delta,t} b_{n-1}
		(s,\cdot\,;\delta)\|_{C^L} \leq C_{L} A_{n-1} \corM{h^{1-2(\beta+2^{\red{n}} 
		\varepsilon) -L(\beta+2^n\varepsilon)}}.$$
\red{Here we also used the fact that $\psi_{s}\succ\psi_{n,s}$ \eqref{eq:CutOfffun31}, 
Remark \ref{rem:Cut-off_est1} and the product rule.} Next, 
apply again Proposition \ref{lem:lem_P1} with 
$f= A_{n-1}^{-1} h^{-1+2(\beta+2^{\red{n}}\varepsilon) + 2^{\red{n}} \varepsilon}\red{\psi_{n,s}}\mathcal{R}_{\delta,t} b_{n-1}(s,\cdot\,;\delta)$ 
and $2^{\red{n}} \varepsilon$ playing the role of $\varepsilon$ to obtain
\begin{equation*}
	\|T^{t-s}_{\delta,\Lambda_{s}}\red{\psi_{n,s}}\mathcal{R}_{\delta,t} b_{n-1}
		(s,\cdot\,;\delta)\|_{C^L} 
	\leq O_L(1) 
	A_{n-1} h^{1-2(\beta+2^{\red{n}} 
		\varepsilon) - 2^{\red{n}} \varepsilon}  
		h^{- L (\beta + 2^{\red{n+1}} \varepsilon)}.
\end{equation*}
All in all, we have obtained that
\begin{align*}
\|b_n\|_{C^L} \leq O_L(1) h^{-(\beta+ 2^{\red{n+1}}\varepsilon)L}
 \left( \MI{\|a_n\|_{C^L}}+ |\log h| A_{n-1} h^{-2(\beta + 2^{\red{n}} \varepsilon)} 
 + |\log h| h^{-2} \delta A_{n-1} h^{1-2(\beta+  2^{\corM{n}}\varepsilon)}\right).
\end{align*}
Now, \red{thanks to \eqref{eq:CondBetaDelta}}, the \red{third} term is the largest, 
provided \MI{$h$ is} small enough \corM{depending on $n$ and $L$}. Therefore,  \MI{taking}
\begin{equation}\label{eq:ConditionAn}
A_n \MI{\geq} |\log h| h^{-2} \delta A_{n-1} h^{1-\corM{2(\beta+ 2^{\red{n}}\varepsilon)}},
\end{equation}
we deduce (\ref{eq:NormeCLInduc}) at step \red{$n$}. In particular, using 
\eqref{eq:CondBeta} \MI{and the definition of $\wit{\varepsilon}_n$}, 
we see that we may take \MI{$A_n = h^{-n \left(1 - \varepsilon_1\right)}$}, 
\corM{provided that $h>0$ is small enough so that 
$$
	|\log h| h^{\frac{3 \varepsilon_0}{20}} \leq 1.
$$
}
%
%$A_n = 
%h^{-n(1+ 3\cdot 2^{\red{n}}\varepsilon)}$ \MI{in (\ref{eq:ConditionAn})}.  In particular, this quantity is smaller 
%than $h^{-n(2\beta +\varepsilon_{1})}$, provided $\varepsilon$ 
%and thus $\gd$ are chose small enough. \red{Note that this condition 
%depends on $N$ but is independent of $L$.}
The statement then follows from (\ref{eq:NormeCLInduc}), \MI{by taking $\varepsilon= \wit{\varepsilon}_N$, and taking $\gd$ accordingly}.
\end{proof}

We may deduce an estimate on the error term in the WKB method, 
appearing in (\ref{eq:WKB_Ansatz.0}).
\begin{corollary}\label{cor:SobolevEstRest}
\red{Let $\psi_t$ be as in \eqref{eq:CutOfffun2}, \eqref{eq:CutOfffun31} and} 
write 
\begin{equation*}
	R_N:=
	-\frac{ih^{N+1}}{2}\e^{\frac{i}{h} \wt{\phi}_{t,\delta}(x)}
	\Delta_{\widetilde{g}} b_{N-1}
	+\delta\sum_{n=1}^{N-1}h^n(1-\psi_{n,t})\wt{Q}_\omega \e^{\frac{i}{h} \wt{\phi}_{t,\delta}(x)} b_{n-1}
	+\delta h^{N-1} \e^{\frac{i}{h} \wt{\phi}_{t,\delta}(x)} 
	\mathcal{R}_{\delta,t} b_{N-1}.
\end{equation*}
For every $L, M\in \N$, there exists $N(M, L)\in \N$ such that, 
for all $\gd>0$ and $h_0>0$ small enough, we have that 
for all $h\in ]0,h_0]$
\begin{equation}\label{eq:BorneReste}
\|\corM{\psi_t} R_N\|_{H^L} = O(h^M),
\end{equation}
\corM{uniformly in $t\in [0,\mathfrak{d}|\log h|]$.} 
Here we use the non-semiclassical Sobolev norm with $H^L=H^L_1$. 
\end{corollary}
\begin{proof}
1. Using Proposition \ref{lem:ControlDeriv} \corM{and Lemma \ref{lem:CLCompo}}
shows that for all $L\in\N$
\begin{equation}\label{eq:expCLnorm}
	\|\e^{\frac{i}{h} \wt{\phi}_{t,\delta}}\|_{C^L(\wt{\mathcal{O}}_{t,\delta})} 
	\leq O_L( h^{- L}),
\end{equation}
uniformly in $t \in [0,\mathfrak{d}|\log h|]$. 
\par
2. Using a partition of unity of lifted cut-off functions as in 
the discussion \corM{before} \eqref{eq:sa8.3}, we get 
\begin{equation*}
\begin{split}
\|\sum_{n=1}^{N-1}h^n \psi_t(1-\psi_{n,t})\wt{Q}_\omega \e^{\frac{i}{h} \wt{\phi}_{t,\delta}} 
b_{n-1}\|_{\red{H^L_h}} 
\leq 
\sum_{n=1}^{N-1}h^n \sum_{k,\iota,k',\iota'}
\|\psi_t(1-\psi_{n,t})\wt{\chi}_{k,\iota}\wt{Q}_\omega\wt{\chi}_{k',\iota'}
\e^{\frac{i}{h} \wt{\phi}_{t,\delta}} b_{n-1}\|_{\red{H^L_h}},
\end{split}
\end{equation*}
where the sum \corM{over $k,k'$ is finite and over $\iota,\iota'$ has at most}  
$O(h^{-C\mathfrak{d}})$ many terms. 
Recall from Proposition \ref{lem:TransportEqSol} that 
$\supp b_{n-1}(t,\cdot;\delta) \subset \supp \psi_{n-1,t}$, so 
the supports of $b_{n-1}(t,\cdot;\delta)$ and $(1 -\psi_{n,t})$ 
are disjoint. Using Lemma \ref{Lem:DistCutOff}, we see that, for 
$0<\varepsilon \leq \varepsilon_0$, if $t\leq \gd |\log h|$ for 
$\gd$ and $h$ small enough, we have
\begin{equation*}
	\mathrm{dist}(\supp (1-\psi_{n,t})\psi_t, \supp b_{n-1}(t,\cdot;\delta)) 
	\geq h^{2\varepsilon}/C. 
\end{equation*}
Using \eqref{eq:liftePseudo3} and Remark \ref{rem:Cut-off_est1} we get that 
\begin{equation*}
\begin{split}
	\|\psi_t(1-\psi_{n,t})\wt{\chi}_{k,\iota}\wt{Q}_\omega\wt{\chi}_{k',\iota'}
	\e^{\frac{i}{h} \wt{\phi}_{t,\delta}} b_{n-1}\|_{\red{H^L_h}}
	&\leq 
	O_{k,k',L}(h^\infty)\|\e^{\frac{i}{h} \wt{\phi}_{t,\delta}} b_{n-1}\|_{\red{H^0_h}}\\
	&=
	O_{k,k',L}(h^\infty),
\end{split}
\end{equation*}
\corM{where in the last line we used} \eqref{eq:SobolevEst} and the fact 
that the support of $b_{n-1}$ has a diameter bounded polynomially with 
respect to $h^{-1}$. All in all
\begin{equation}\label{eq:sobnew1}
	\begin{split}
	\|\sum_{n=1}^{N-1}h^n \psi_t(1-\psi_{n,t})\wt{Q}_\omega 
		\e^{\frac{i}{h} \wt{\phi}_{t,\delta}} b_{n-1}\|_{\red{H^L_h}} 
	= O(h^\infty).
	\end{split}
\end{equation}
\red{
Since $\|u\|_{H^L}\leq \mathcal{O}(h^{-L})\|u\|_{H^L_h}$, $u\in H^L_h$, we 
conclude that the estimate \eqref{eq:sobnew1} holds as well for the $H^L$ norm.
}
\\
\par
3.  \MI{Let $\varepsilon_1 < \frac{\varepsilon_0}{2}$.}
To deal with the first term in $R_N$, we apply \eqref{eq:expCLnorm} 
along with Proposition \ref{Porp:SobolevEst}, which gives us
\begin{equation*}
\begin{split}
	\big{\|}-\frac{ih^{N+1}}{2}\e^{\frac{i}{h} \wt{\phi}_{t,\delta}}
	\Delta_{\widetilde{g}} b_{N-1}\big{\|}_{C^L} 
	&= O(h^{N+1})\|\e^{\frac{i}{h} \wt{\phi}_{t,\delta}} b_{N-1}\|_{C^{L+2}}\\
	& = \corM{O(h^{(N-1) \varepsilon_1 - L})}	
%	&= \MI{O(h^{N \varepsilon_1 +2 - 2\beta(L+2)}).}
%\\ &= O( h^{1+2\beta + \varepsilon_0 - (L+2) + N(1-2\beta - \varepsilon_0}).
\end{split}
\end{equation*}
In particular,  for every $L\in \N$, this term can be made smaller than 
any power of $h$ by taking $N$ large enough \MI{(and thus $\gd$ small enough)}. \red{By Proposition \ref{lem:TransportEqSol} 
we know that $\supp b_{N-1}\Subset \mathcal{O}''_{t,\delta}\subset  \mathcal{O}_{t,\delta}$, 
so by \eqref{eq:ConjPseudo1.01} we have that for every $L, M\in \N$, there exists $N(M, L)\in \N$ such that, 
for all $\gd>0$ and $h>0$ small enough
\begin{equation*}
	\big{\|}-\frac{ih^{N+1}}{2}\e^{\frac{i}{h} \wt{\phi}_{t,\delta}}
	\Delta_{\widetilde{g}} b_{N-1}\big{\|}_{H^L} 
	= O(h^{M}).
\end{equation*}
}
4. To deal with the last term in $R_N$, we note that Propositions 
\ref{lem:lem_P2} and \ref{Porp:SobolevEst} imply that, for any $k,N\in \N$,
\begin{equation*}
	\|\red{\psi_t}\mathcal{R}_{\delta,t} b_{N-1}\|_{C^k} 
	\leq 
	C_{n,k} h^{1-(k+2)(\beta+\varepsilon_{\MI{1}})-\varepsilon_{\MI{1}}} 
			h^{-(N-1)\MI{(1-\varepsilon_1)}},
\end{equation*}
so that
\begin{equation*}
\begin{split}
	\big{\|}\delta h^{N-1} \red{\psi_t} \e^{\frac{i}{h} \wt{\phi}_{t,\delta}(x)} 
	\mathcal{R}_{\delta,t} b_{N-1} \big{\|}_{C^L}
	&= O (\delta h^{(N-1)\MI{\varepsilon_1}+(1-2\beta - 3 \varepsilon_1) - L}).
\end{split}
\end{equation*}
In particular,  for every $L\in \N$, this term can also be made smaller 
than any power of $h$ by taking $N$ large enough \MI{(and $\gd$ small enough)}. \red{Since $\supp \psi_t\Subset \mathcal{O}''_{t,\delta}$, 
we conclude as in step 3, which ends the proof.}
\end{proof}
\section{Local expansions for the propagated Lagrangian states on $X$}\label{sec:Proj}
\subsection{To the universal cover and back again}
As in the previous sections, let $(X,g)$ be a compact 
Riemannian manifold of negative sectional curvature. 
Let $\mO=\mO_0 \subset X$ be an open, simply connected 
relatively compact set, and let $\phi_0:\mO\to \R$ be a 
monochromatic smooth function as in \eqref{eq:LagState_monochrom} 
such that the hypotheses of Proposition \ref{Prop:OnTheCover2} are satisfied. 
Consider the Lagrangian state 
\begin{equation}\label{eq:UVC1}
	f_h(x) =  a(x ;h)\e^{\frac{i}{h} \phi_0(x)},
	\quad \red{x\in \mO\subset X},
\end{equation}
where $a(\cdot\,;h)$ is a smooth compactly supported map on $\mathcal{O}$ 
with $\|a(\cdot;h)\|_{C^L} = O_L(1)$ and 
such that $a(x;h)\sim a_0(x) +ha_1(x) + \dots $ where $a_n$ are smooth 
compactly supported maps on $\mathcal{O}$. We assume that there exists 
a compact $h$-independent set $K\Subset \mathcal{O}$ such that 
\begin{equation}\label{eq:UVC2}
	\supp a(\cdot\,;h) \subset K
\end{equation}
for all $h\in ]0,h_0]$. 
\\
\par 
The aim of this section is to construct a WKB approximation of the 
propagated state 
\begin{equation}\label{eq:UVC3}
	u(t,x;h) = \e^{\red{-}i\frac{t}{h}P_h^\delta}f_h(x)
\end{equation}
on $X$. In other words, we want to construct an approximate solution to 
\begin{equation}\label{eq:UVC4}
	\begin{cases}
		ih\partial_t u = P^\delta_h u \\
		u|_{t=0} = f_h.
	\end{cases}
\end{equation}
Before we proceed, let us give a brief outline of the general strategy 
to achieve this goal. First we lift all 
quantities to the universal cover $\wt{X}$ of $X$. There, we will use 
Proposition \ref{prop:WKB} to obtain an approximate solution $u_\delta$ 
to the lifted equation 
\begin{equation}\label{eq:UVC5}
	\begin{cases}
		ih\partial_t \wt{u} = \wt{P}^\delta_h \wt{u}  \\
		\wt{u} |_{t=0} = \wt{f}_h.
	\end{cases}
\end{equation}
on $\wt{X}$. Then, using \eqref{eq:PullBackMapping2}, we ``project'' 
the approximate solution $u_\delta$ back down to the base manifold $X$, 
and we show that this yields the desired approximate solution to 
\eqref{eq:UVC4}. 
\\
\par
\textbf{Lifting to the universal cover $\wt{X}$.} 
Recall from section \ref{sec:notation} that we denote by $(\widetilde{X},\widetilde{g})$ 
the universal cover of $(X,g)$, which is thus a simply connected (non-compact) manifold 
of negative curvature, and we denote by $\widetilde{\pi} : \widetilde{X} \longrightarrow X$ the 
covering map. 
\par
Let $\wt{\mO}_0,\wt{K}\subset \wt{X}$ be lifts of $\mO_0,K$, respectively. 
That is to say,   \MI{$\wt{K}\subset \widetilde{\mO}_0$ are bounded subsets of $\wt{X}$}
%a connected relatively compact open subset of $\wt{X}$ 
such that 
$\wt{\pi}(\widetilde{\mO}_0) = \mO_0$ and $\wt{\pi}(\wt{K})=K$. 
\red{Note that there are many possible lifts of these sets and we simply choose one.} 
We lift $\phi_0$ to a smooth function $\widetilde{\phi}_0$ on $\wt{\mO}_0$ 
such that $\widetilde{\phi}_0 = \wt{\pi}^*\phi_0$ on $\wt{\mO_0}$. Similarly, 
we lift $a$ and $a_n$ to smooth compactly supported functions $\wt{a}$ and $\wt{a}_n$ on 
$\widetilde{\mO}_0$ with support in $\wt{K}$ such that $\widetilde{a} = \wt{\pi}^*a$,  
$\widetilde{a}_n = \wt{\pi}^*a_n$ on $\widetilde{\mO}_0$. 
\par
Recall from Section \ref{subsec:LiftPseudo} how we lift the pseudo-differential 
operator $Q_\omega$ to a pseudo-differential operator $\widetilde{Q}_\omega$ on 
the universal cover $\widetilde{X}$. When $Q_\omega$ is a multiplication operator 
on $X$ by the function $q_\omega$, then $\widetilde{Q}_\omega$ is simply the multiplication 
operator on $\wt{X}$ by the lifted function $\wt{q}_\omega:=\wt{\pi}^*q_\omega$. Furthermore, 
recall from \eqref{eq:liftedSchroedinger} the definition of the lifted Schrödinger operator 
$\widetilde{P}_h^\delta$. 
\par 
The above lifts yield a monochromatic Lagrangian state 
\begin{equation*}
	\wt{f}_h(\widetilde{x}) =  
	\wt{a}(\widetilde{x} ;h)\e^{\frac{i}{h} \wt{\phi}_0(\widetilde{x})}
	\text{ on } \wt{X},
\end{equation*}
satisfying the assumptions of Proposition \ref{prop:WKB}, so, given $N\in\N$   
there exists a function  
\begin{equation}\label{eq:liftpropstate}
	u_\delta(t,\widetilde{x};\red{h}):=
	b_N(t,\red{\wt{x}};\delta,h)\e^{\frac{i}{h} 
	\widetilde{\phi}_{t,\delta}(\widetilde{x})}
	:=
	\e^{\frac{i}{h} 
	\widetilde{\phi}_{t,\delta}(\widetilde{x})} 
	\sum_{n=0}^{N-1}h^n b_n(t,\widetilde{x};\delta),
\end{equation}
%
%\eqref{eq:WKB_Ansatz}, i.e.
with $b_n$ given in \eqref{eq:transEq4}, \eqref{eq:transEq5}, 
and phase $\widetilde{\phi}_{t,\delta}$ given in \eqref{eq:eik0.0}, 
\eqref{eq:Eikonal}, so that 
\begin{equation*}
\begin{dcases}
	&ih\partial_t u_\delta 
	= 
	\widetilde{P}_h^\delta u_\delta + \wt{R}_N \red{\text{ on }\mathcal{O}''_{t,\delta}}\\
	& u_\delta(0,\widetilde{x};\red{h}) = \e^{\frac{i}{h} \widetilde{\phi}_{0}(\widetilde{x})} 
	\sum_{n=0}^{N-1}h^n \widetilde{a}_n(\widetilde{x}).
\end{dcases}
\end{equation*}
By Corollary \ref{cor:SobolevEstRest}, \red{applied to $R_N = \psi_t\wt{R}_N$}, 
we see that for every $L, M\in \N$, there exists $N(M, L)\in \N$ such that, 
for $\gd>0$, $h_0>0$ small enough, we have for all 
$h\in]0,h_0]$
\begin{equation}\label{eq:RemEstimate1}
\|\red{\psi_t} \wt{R}_N\|_{H^L} = O(h^M),
\end{equation}
\corM{uniformly in $0\leq t \leq \gd |\log h|$.} 
We know from Proposition \ref{lem:TransportEqSol}, that 
$\supp b_N(t,\cdot;\delta,h) \subset \supp \psi_{N-1,t}$, 
see also \eqref{eq:CutOfffun31}, \eqref{eq:CutOfffun2}. 
In particular, when $\wt{Q}_\omega$ is a multiplication operator then we 
have $\supp b_N(t,\cdot;\delta,h) \subset \supp \psi_{0,t}$.
\corM{By} \eqref{eq:CutOfffun31}, we see that $b_N(t,\cdot;\delta,h)$ is compactly 
supported in \corM{$\wt{\mO}_{t,\delta}''\subset\wt{\mO}_{t,\delta}$}. \corM{Moreover, by \eqref{eq:SetInclusion}} 
$\wt{\phi}_0$ and $\wt{\phi}_{t,\delta}$ are well-defined in 
\corM{$\wt{\mO}_{t,\delta}''$}. Furthermore, we know from \eqref{eq:ConjPseudo1.01}, 
that for any $0<\varepsilon <\varepsilon_0$ there exist $\mathfrak{d},h_0>0$ such 
that 
\begin{equation}\label{eq:BoundedSupp}
\mathrm{vol}_{\wt{g}}(\wt{\mO}_{t,\delta})= 
	O( h^{-\varepsilon})
\end{equation}
uniformly in $h\in]0,h_0]$ and in $t\in [0,\mathfrak{d}|\log h|]$. 
\\
\par
\textbf{Back to the base manifold $X$.} 
The aim now is to show that ``projecting'' the WKB state $u_\delta$ back 
down to the base manifold $X$ yields a good approximation of the solution $u(t,x;h)$ of \eqref{eq:UVC3}.
Recall \eqref{eq:PullBackMapping2}, \eqref{eq:PullBackMapping4} and define 
\begin{equation}\label{eq:ProjectedLagrangianState}
	\widehat{u}_{h,N}(t,x) 
	:= (^t\wt{\pi}^* u_\delta(t,\cdot\,;h))(x)
	=\sum_{\wt{\pi}(\wt{x})=x} 
	u_\delta(t,\wt{x};h)
	=\sum_{\wt{\pi}(\wt{x})=x} 
	b_N(t,\widetilde{x};\delta,h)
	\e^{\frac{i}{h} \widetilde{\phi}_{t,\delta}(\widetilde{x})}, 
	~~ x\in X.
\end{equation}
Since $a(x;h)\sim a_0(x) +ha_1(x) + \dots $, it follows that for any 
$N\in\N$ there exists a smooth compactly supported function $r_N(x;h)$ with 
$\supp r_N(\cdot;h)\subset K$ and $\|r(\cdot;h)\|_{C^L} = O_L(1)$, such that 
\begin{equation*}
	a(x;h) = \sum_{n=0}^{N-1}h^n a_n(x) + h^N r(x;h).
\end{equation*}
Note that 
\begin{equation}\label{eq:ProjectedLagrangianState.1}
	\widehat{u}_{h,N}(0,x) = f_h(x)-h^N r(x;h)\e^{\frac{i}{h} \phi_0(x)}
\end{equation}
with $f_h$ as in \eqref{eq:UVC1}. 
%
% \par
% Recall the quantities introduced in the paragraph after \eqref{eq:sa8.2}.  
% By \eqref{eq:ProjectedLagrangianState} we find that 
% %
% \red{
% \begin{equation}\label{eq:ProjectedLagrangianState4}
% 	\widehat{u}_{h,N}
% 	=\sum_{k\in K,\iota \in I_k} (\wit{\pi}_{k,\iota }^{-1})^*
% 	\left(\wt{\chi}_{k,\iota}^2
% 	b_N
% 	\e^{\frac{i}{h} \widetilde{\phi}_{t,\delta}}\right),
% \end{equation}
% }
% %
% where the sum is finite since $b_N$ has compact support. More precisely 
% $K$ is finite independently of $h$, and $|I_k| = O(h^{-\varepsilon})$ 
% for all $h\in ]0,h_0]$ and $t\in [0,\mathfrak{d}|\log h|]$, provided $\gd$ 
% is smaller than some $\gd(\varepsilon)$. The 
% estimate on the cardinality of $I_k$ is a direct consequence of 
% \eqref{eq:BoundedSupp}. 
% %
% \par
%
\\
\par
Consider first the case when $Q_\omega$ is as in the \ref{eq:Pot}, 
so $\wt{Q}_\omega$ is the operator multiplication with $\wt{q}_\omega 
= \wt{\pi}^*q_\omega$, see \eqref{eq:liftedSchroedinger} and Proposition 
\ref{prop:WKB}. Since differential operators are local operators, we deduce from 
\eqref{eq:PullBackMapping5} that 
\red{
\begin{equation}\label{eq:ProjectedLagrangianState4.1}
	(i h \partial_t  + \frac{h^2}{2}\Delta_g 
	-\delta Q_\omega)\widehat{u}_{h,N}(t,x)
	=\sum_{\wt{\pi}(\wt{x})=x} 
	(i h \partial_t + \frac{h^2}{2}\Delta_{\wt{g}} -
	\delta \wt{Q}_\omega)
	u_\delta(t,\wt{x}).
	% = 
	% \sum_{k\in K,\iota \in I_k}
	% (\wit{\pi}_{k,\iota}^{-1})^*	\left(\wt{\chi}_{k,\iota}
	% (i h \partial_t + \frac{h^2}{2}\Delta_{\wt{g}} -
	% \delta \wt{Q}_\omega)
	% b_N
	% \e^{\frac{i}{h} \widetilde{\phi}_{t,\delta}}\right).
\end{equation}
}
When $Q_\omega$ is as in the \ref{eq:Pseudo},  \MI{understanding} its action on $\wh{u}_{h,N}$ 
requires a little bit more work. 
% Let $\chi_k',\chi_k''\in C^\infty_c(U_k;[0,1])$ 
% be so that $\chi_k'$ is equal to $1$ on a small neighbourhood of $\supp \chi_k$ 
% and so that $\chi_k''$ is equal to $1$ on a small neighbourhood of 
% $\supp \chi_k'$. We then define $\wt{\chi}_{k,\iota}'$ and 
% $\wt{\chi}_{k,\iota}''$ similarly to $\wt{\chi}_{k,\iota}$.  
%
%
% By the pseudolocality of $Q_\omega$, cf. Section \ref{sec:SemClass}, we have that 
% for every $L\in\N$
\corM{
By Definition \ref{def:Quantization} we get 
\begin{equation}\label{eq:CTC}
 	Q_\omega \widehat{u}_{h,N}(t,x) 
	=
 	\sum_k \vartheta_k'\kappa_k^* 
	\Op_h(q_{\omega,k}) 
	(\kappa_k^{-1})^* \vartheta_k' \widehat{u}_{h,N}(t,x) .
 	%+ O_{L}(h^\infty)\|\widehat{u}_{h,N}(t,\cdot) \|_{H^L_h(X)} ,
\end{equation}
with $q_{\omega,k}:=(\vartheta_k q_{\omega})\circ\widehat{\kappa}_k^{-1} \in S^{-\infty}_\beta(T^*\R^d)$. 
}
%
% uniformly in $t\in [0,\mathfrak{d}|\log h|]$. 
\red{
By \eqref{eq:ProjectedLagrangianState}, we get 
\begin{equation}\label{eq:projeLagState1}
	\corM{\vartheta_k'}\widehat{u}_{h,N}
	= 
	\sum_{\iota\in I_k}(\wit{\pi}_{k,\iota}^{-1})^* \wt{\vartheta}_{k,\iota}'
	u_\delta(t,\wt{x};h).
\end{equation}
Here $|I_k| = O(h^{-\varepsilon})$ for all $h\in ]0,h_0]$ and $t\in [0,\mathfrak{d}|\log h|]$, 
provided $\gd$ is smaller than some $\gd(\varepsilon)$. This estimate on the cardinality of 
$I_k$ is a direct consequence of \eqref{eq:BoundedSupp} and the fact that $u_\delta, b_N$ 
have compact support in $\mathcal{O}''_{t,\delta}$, see \eqref{eq:liftpropstate} and 
the paragraph after \eqref{eq:RemEstimate1}. Then, 
}
\red{
\begin{equation}\label{eq:projeLagState2}
\begin{split}
	Q_\omega \widehat{u}_{h,N}(x) 
	% &= \sum_{k\in K} \left(\vartheta_k Q_\omega \vartheta_k\widehat{u}_{h,N}\right)\!(x)\\
	% %
	&=
	\sum_{k\in K,\iota\in I_k} \left(\corM{\vartheta_k'}\kappa_k^*
	\Op(q_{\omega,k})(\kappa_k^{-1})^*
	(\wit{\pi}_{k,\iota}^{-1})^* 
	\corM{\wt{\vartheta}_{k,\iota}'}u_\delta\right)\!(x)\\
	&=
	\sum_{k\in K,\iota\in I_k} \left((\wit{\pi}_{k,\iota}^{-1})^* 
	\corM{\wt{\vartheta}_{k,\iota}'}\wit{\kappa}_{k,\iota}^*
	\Op(q_{\omega,k})(\wit{\kappa}_{k,\iota}^{-1})^*
	\corM{\wt{\vartheta}_{k,\iota}'}u_\delta\right)\!(x)\\
	&\stackrel{\text{Def. } \ref{def:QuantizationLift}}{=}
	\sum_{\wt{\pi}(\wt{x})=x} (\wt{Q}_\omega u_\delta) (\wt{x}).
\end{split}
\end{equation}
}
\red{
where in the last equality we used that $\wt{Q}_\omega$ is properly 
supported.}
\par
\red{
Next, let $\psi_t$ be as in \eqref{eq:CutOfffun2} and recall 
\eqref{eq:CutOfffun31}. Our aim is to control $^t\wt{\pi}^*(1-\psi_t)\wt{Q}_\omega u_\delta$, 
see also \eqref{eq:PullBackMapping2}. 
}
\red{
Recall \eqref{eq:sa6}, \eqref{eq:sa8.3} and the discussion above \eqref{eq:sa8.3}. 
Then, for $k\in K$ (a finite $h$-independent index set), 
\begin{equation}\label{eq:CO_control1}
\begin{split}
	(\kappa_k^{-1})^*\chi_k \,^t\wt{\pi}^*((1-\psi_t)\wt{Q}_\omega u_\delta)
	&=\sum_{\iota\in I_k} (\wit{\kappa}_{k,\iota}^{-1})^* (\wit{\chi}_{k,\iota}(1-\psi_t)\wt{Q}_\omega u_\delta)%\\
	%&=\sum_{k'\in K, \iota'\in I_{k'},\iota\in I_k} (\wit{\pi}_{k,\iota}^{-1})^* (\wit{\chi}_{k,\iota}(1-\psi_t)\wt{Q}_\omega \wit{\chi}_{k',\iota'}u_\delta),
\end{split}
\end{equation}
where $|I_k|=\mathcal{O}(h^{-\varepsilon})$ by the same argument as in the paragraph above 
\eqref{eq:projeLagState2} and the fact that $\wt{Q}_\omega$ is properly supported. Keeping 
in mind that $|I_k|=O(h^{-\varepsilon})$, the H\"older inequality and \eqref{eq:sa8.3} imply that 
\begin{equation}\label{eq:CO_control2}
	\|\,^t\wt{\pi}^*((1-\psi_t)\wt{Q}_\omega u_\delta)\|_{H^L_h(X)}
	\leq O(h^{-\varepsilon}) \|(1-\psi_t)\wt{Q}_\omega u_\delta\|_{H^L_h(\wt{X})}.
\end{equation}
Let $\{\chi_k\}_{k\in K}\subset C^\infty_c(X;[0,1)$ be a partition of unity of $X$ 
subordinate to an open covering of $X$ by chart domains as in the paragraph above 
\eqref{eq:sa8.3}. In particular recall that this gives rise to a locally finite 
partition of unity on $\wt{X}$ by the lifted charts $\wt{\chi}_{k,\iota}\in C^\infty_c(\wt{X};[0,1])$. 
Next, recall \eqref{eq:CutOfffun31} and in particular that 
$\psi_{N,t}\succ \psi_{N-1,t}$. As discussed in the paragraph before  
\eqref{eq:BoundedSupp} we have that $\supp b_N(t,\cdot;h,\delta)
\subset \supp \psi_{N-1,t}$. So we can see each $\wit{\chi}_{k,\iota}(1-\psi_t)$ and $\psi_{N,t}\wt{\chi}_{k',\iota'}$ 
as lifted cut-off functions with compact support \corM{contained in compact $h$-independent sets}. Furthermore, by 
Lemma \ref{Lem:DistCutOff} the distance between their supports is bounded from below by $h^{2\varepsilon}$, 
for $0<\varepsilon <\varepsilon_0$, uniformly in $t\in [0,\mathfrak{d}|\log h|]$, and 
their derivatives are bounded by the right hand side of \eqref{eq:SetInclusion_Cut-off2}. 
Since $\beta+\varepsilon_0<1/2$, we may then use \eqref{eq:liftePseudo3}, 
while keeping in mind that $|I_k|=O(h^{-\varepsilon})$, and deduce that 
\begin{equation}\label{eq:projeLagState2.5}
\begin{split}
	\|(1-\psi_t)\wt{Q}_\omega u_\delta\|_{H^L_h(\wt{X})}
	&\leq \sum_{k,k'\in K, \iota,\iota }
	\|\chi_{k,\iota}(1-\psi_t)\wt{Q}_\omega \chi_{k',\iota'} \psi_{N,t} u_\delta\|_{H^L_h(\wt{X})}
	\\
	&= O_L(h^\infty)\|u_\delta\|_{H^{\corM{0}}_h(\wt{X})}\\
	&=\corM{O_{L}(h^\infty)},
\end{split}
\end{equation}
uniformly in $t\in [0,\mathfrak{d}|\log h|]$. Here, in the second line 
we also used the fact that $K$ is a finite 
$h$-independent set of indices and that for each $k,k'$ the sum over 
$\iota,\iota'$ has at most $O(h^{-\varepsilon})$ many non-zero terms. 
Indeed, this last assertion follows by the same 
argument as in the paragraph after \eqref{eq:projeLagState1}. \corM{In the 
last line of \eqref{eq:projeLagState2.5} we also used the fact that 
$\|u_\delta\|_{H^{\corM{0}}_h(\wt{X})}= \mO(h^{-\varepsilon})$ 
uniformly in $t\in [0,\mathfrak{d}|\log h|]$. This is a direct consequence 
of Proposition \ref{Porp:SobolevEst}, Proposition \ref{lem:TransportEqSol}, 
\eqref{eq:ConjPseudo1.01} and \eqref{eq:WKB_Ansatz}.}
Combining \eqref{eq:projeLagState2} and \eqref{eq:projeLagState2.5} gives 
% \eqref{eq:CTC},  
% \eqref{eq:projeLagState2.3}, \eqref{eq:projeLagState2.4} and 
%
%
\begin{equation}\label{eq:ProjectedLagrangianState4.2}
\begin{split}
	Q_\omega \widehat{u}_{h,N}(t,x) 
	=
	\sum_{\wit{\pi}(\widetilde{x}) = x}
	\psi_t\wit{Q}_\omega u_\delta 
 	+ \corM{O_{H^{L}_h(\wt{X})}(h^\infty)}
\end{split}
\end{equation}
Since $\psi_t\equiv 1$ on the support of $u_\delta$, 
see \eqref{eq:CutOfffun31} and the paragraph after 
\eqref{eq:RemEstimate1}, it follows from \eqref{eq:PullBackMapping5} and 
\eqref{eq:ProjectedLagrangianState4.2} that 
\begin{equation}\label{eq:ProjectedLagrangianState4.2b}
\begin{split}
	\left(i h \partial_t + \frac{h^2}{2}\Delta_g -
	\delta Q_\omega \right)\widehat{u}_{h,N}(x)
&= 
\sum_{\wit{\pi}(\widetilde{x}) = x}
	\psi_t\left(i h \partial_t + \frac{h^2}{2}\Delta_{\widetilde{g}}
		  - \delta\widetilde{Q}_\omega \right) 
	u_\delta
	+ \corM{O_{H^{L}_h(\wt{X})}(h^\infty)}
\\
&=: \widehat{R}_{N,t} .
\end{split}
\end{equation}
}
For any $L\in \N$, we can use \eqref{eq:sa7}, to see that 
$\|g\|_{H^{L}_h(\wt{X})} = O(1)h^{-L}\|g\|_{H^{L}(\wt{X})}$ for all 
$g\in H^{L}(\wt{X})$. Hence, we may replace the semiclassical Sobolev norm 
$\|\cdot\|_{H^{L}_h(\wt{X})}$ \corM{in the remainder estimate} by $\|\cdot\|_{H^{L}(\wt{X})}$. 
Next, notice that each term in the sum on the right hand side of 
\eqref{eq:ProjectedLagrangianState4.2b} satisfies \eqref{eq:RemEstimate1} 
in any $H^L$ norm. The number of terms is $O(h^{-\varepsilon})$ provided 
$t\in [0, \mathfrak{d} |\log h|]$ for some $0< \mathfrak{d} \leq \gd (\varepsilon)$. 
So, we deduce from (\ref{eq:BorneReste}) that, for any $M,L\in \N$, we may find 
$N\in \N$ such that for $\gd>0$, $h_0>0$ small enough, we have that 
for all $h\in]0,h_0]$
\begin{equation}\label{eq:SobolevEstRemainder}
\|\widehat{R}_{N,t}\|_{H^L(X)} = O(h^{M}),
\end{equation}
uniformly in $t\in [0, \leq \mathfrak{d} |\log h|]$. Since $P_h^\delta$ is 
the infinitesimal generator of the unitary group $U_h^\delta(t)$, it follows 
that these two operators commute. Using additionally that 
$\delta Q_\omega = O(\delta):L^2(X)\to L^2(X)$, we find that for 
$f\in C^\infty(X)$
\begin{equation}\label{eq:Step1SobNormEst}
\begin{split}
	\| \Delta U_h^\delta(t) f\|_{L^2(X)}
	&= h^{-2 }\| h^2\Delta U_h^\delta(t) f\|_{L^2(X)}\\
	&\leq  h^{-2 }\left(
		\| P_h^\delta U_h^\delta(t) f\|_{L^2(X)}
		+
		\| \delta Q_\omega U_h^\delta(t) f\|_{L^2(X)}
		\right)
	\\
	& \leq 
	h^{-2 }\left(
		\| h^2\Delta f\|_{L^2(X)}
		+
		O(\delta)\|  f\|_{L^2(X)}
		\right).
\end{split}
\end{equation}
Since $Q_\omega \in \Psi^{-\infty}_{h,\beta}$, we have that 
$\delta (h^2\Delta)^n Q_\omega = O(\delta):L^2(X)\to L^2(X)$ 
for any $n\in\N$. Using this, we may iterate the argument 
in \eqref{eq:Step1SobNormEst} and get 
\begin{equation}\label{eq:Step2SobNormEst}
 \| \Delta^{n} U_h^\delta(t) f\|_{L^2(X)}
 \leq O(h^{-2n})\sum_{k=0}^{n}
  \| h^{2k}\Delta^{k} f\|_{L^2(X)}.
\end{equation}
In view of \eqref{eq:ProjectedLagrangianState.1}, 
Duhamel's principle implies 
\begin{equation*}
	\widehat{u}_{h,N}(t) 
	= U_h^\delta(t)\left( f_h-h^N r(\cdot;h)\e^{\frac{i}{h} \phi_0}\right)
+ \frac{1}{ih}\int_0^t U_h^\delta(t-s) \widehat{R}_{N,s}ds,
\end{equation*}
so that, for any $M\in \N$, we have
\begin{equation}\label{eq:WKB2.20}
\begin{split}
	&\|U_h^\delta(t)f_h -\widehat{u}_{h,N}\|_{H^{2n}(X)} 
	\leq 
	\frac{1}{h}\int_0^t \left\|U_h^\delta(t-s)
	\widehat{R}_{N,s} \right\|_{H^{2n}(X)} d s
	+h^N\|U_h^\delta(t)r(\cdot;h)\e^{\frac{i}{h} \phi_0}\|_{H^{2n}(X)} 
	\\
	&\stackrel{\eqref{eq:Step2SobNormEst}}{\leq}  
	O(|\log h| h^{-1-2n})\sum_{k=0}^{n}
	\| h^{2k}\Delta^{k} \widehat{R}_{N,s}\|_{L^2(X)}
	+O(h^{N-2n})\sum_{k=0}^{n}\|h^{2k}\Delta^{k}r(\cdot;h)\e^{\frac{i}{h} \phi_0}\|_{H^{2n}(X)}\\
&= O (h^M)
\end{split}
\end{equation}
thanks to (\ref{eq:SobolevEstRemainder}), up to taking $N$ larger and 
therefore $\gd>0$ and $h_0>0$ smaller. In the last line we also used \corM{that} all 
derivatives of $\phi_0$ and $r(\cdot;h)$ are bounded \corM{uniformly in} $h$. 
\par
Now, since $M$ and $n$ are arbitrary, we deduce from the 
Sobolev inequalities \eqref{eq:sa8.2} that the remainder 
can be made arbitrarily small in any $C^L$ norm, by taking $N$ large enough 
and $\gd>0$ and $h_0>0$ smaller. In other words, for every 
$M,L\in \N$, we may find a $N\in \N$ such that for $\gd>0$, $h_0>0$ small 
enough, we have that for all $h\in]0,h_0]$
\begin{equation}\label{eq:CTC2}
U_h^\delta(t)f_h\red{(x)}
	=\sum_{\wt{\pi}(\wt{x})=x} 
	b_N(t,\widetilde{x};\delta,h)
	\e^{\frac{i}{h} \widetilde{\phi}_{t,\delta}(\widetilde{x})} + O_{C^L}(h^M),
\end{equation}
uniformly in $t\in [0, \leq \mathfrak{d} |\log h|]$.
\par
For each $x\in X$ and $t\geq 0$ we set 
\begin{equation}\label{eq:DefA}
 A_{x,t}:= \{ \wit{x}\in \wit{\mathcal{O}}_{t,\MI{0}}
 			\text{ with } \wit{\pi}(\wit{x}) =x \}.
\end{equation}
Since the manifold $X$ is compact, there exists 
$c>0$ such that $\left[\pi(\widetilde{x}) = \pi(\widetilde{x}')\right] 
\Rightarrow \big[(x=x')$ or $\mathrm{dist} (x,x') >c\big]$. 
Therefore, we have that the cardinality of $A_{x,t}$ is bounded by 
\begin{equation}\label{eq:DefA2}
	|A_{x,t}| \leq C \e^{Ct}.
\end{equation}
Thus,  \MI{recalling that all the $b_N (t, \cdot ; \delta, h)$ are supported in $\wt{\cO}''_{\delta,t} \subset \wit{\mathcal{O}}_{t,0}$, we get}
\begin{equation}\label{eq:ProjectedLagrangianState2}
	\widehat{u}_{h,N}(t,x) 
	= 
	\sum_{\widetilde{x} \in A_{x,t}} 
	b_N(t,\widetilde{x};\delta,h)
	\e^{\frac{i}{h} \widetilde{\phi}_{t,\delta}(\widetilde{x})},
\end{equation}
and \eqref{eq:CTC2} becomes 
\begin{equation*}\label{eq:CTC3} 
	U_h^\delta(t)f_h = \sum_{\widetilde{x} \in A_{x,t}} 
		b_N(t,\widetilde{x};\delta,h)
		\e^{\frac{i}{h} \widetilde{\phi}_{t,\delta}(\widetilde{x})} 
		+ O_{C^L}(h^M).
\end{equation*}
\subsection{Rescaling to the scale of the wavelength}
Let $L\in \N$. We shall write 
\begin{equation*}
	\psi_{h,t,\delta}:= U_h^\delta (t) f_h.
\end{equation*}
Let $\mathcal{U}\subset X$ be an open subset on which there exists 
a family of vector fields $(V_1,\dots,V_d)$ forming an orthonormal 
frame of the tangent bundle $T\cU$. Given 
$x\in \mathcal{U}$, we write 
$\exp_x(\boldsymbol{y}):=\exp_x(\sum_{j=1}^d y_jV_j(x))$, $\boldsymbol{y}\in \R^d$, 
and if 
$\widetilde{x}\in \widetilde{X}$ is a lift of $x$, we denote by 
$\widetilde{\exp}_{\widetilde{x}}(\boldsymbol{y})$ a lift of 
$\exp_x(\boldsymbol{y})$ depending continuously on $\boldsymbol{y}$ 
and such that $\widetilde{\exp}_{\widetilde{x}}(\boldsymbol{0})  
= \widetilde{x}$.
\par
If $x\in \mathcal{U}$, then for any compact set \red{$\mathfrak{K}\subset \R^d$} 
and any $L,M\in \N$, there exists a $N\in \N$ such 
that for $\gd>0$, $h_0>0$ small enough, we have that for all 
$h\in]0,h_0]$ and all $t\in [0,  \mathfrak{d} |\log h|]$ 
\begin{align*}
\psi_{h, t, x, \delta}(\boldsymbol{y}) 
&:= \psi_{h,t, \delta} 
	\left(\exp_{x}(h \boldsymbol{y}) \right)\\
&= \sum_{\widetilde{x} \in A_{x,t}} 
	b_N(t,\widetilde{\exp}_{\widetilde{x}}(h \boldsymbol{y}) ;\delta,h)
	\e^{\frac{i}{h} \widetilde{\phi}_{t,\delta} 
	\left( \widetilde{\exp}_{\wt{x}}(h \boldsymbol{y}) \right)}  
	+ O_{C^L(\mathfrak{K})}(h^M).
\end{align*}
Here we also assumed tacitly that $h>0$ is small enough (depending on 
$\mathfrak{K}$) so that $\widetilde{\exp}_{\widetilde{x}}(h \boldsymbol{y}) 
\in \wt{\mathcal{O}}_{t,\delta}$. \red{Moreover, the remainder term 
is uniform in $x\in X$, $\delta$ as in Hypothesis \ref{Hyp:BetaDelta}, 
and $t\in [0,\mathfrak{d}\log h|]$.} From now on we work with $L,M,N\in\N$, 
$\gd>0$ and $h_0>0$ as above.
\\ 
\\
\textbf{First order approximation of the amplitudes}
Recall that 
\begin{equation*}
	b_N(t,\widetilde{\exp}_{\widetilde{x}}(h \boldsymbol{y}) ;\delta,h)
	= \sum_{n=0}^{N-1}h^n b_n\left(t,\widetilde{\exp}_{\wt{x}}(h \boldsymbol{y}) ;
	\delta\right).
\end{equation*}
Now,  thanks to Proposition \ref{Porp:SobolevEst}, for any $n\geq 1$, the 
$C^L(\mathfrak{K})$ norm of $\boldsymbol{y} \mapsto h^n b_n\left(t,\widetilde{\exp}_x
(h \boldsymbol{y}) ;\delta\right)$ is $O(h^{\gamma})$, for some $\gamma>0$ 
independent of $n$.  
This comes from the fact that $\beta < \frac{1}{2}$, and that, when we 
differentiate $\boldsymbol{y} \mapsto  b_n\left(t,\widetilde{\exp}_x
(h \boldsymbol{y})\right)$ $L$ times, we gain a factor $h^L$ which balances 
the $h^{-(\beta +\varepsilon_0)L}$ in (\ref{eq:SobolevEst}).
\par
We thus have
\begin{equation*}
	\psi_{h, t, x, \delta}(\boldsymbol{y}) 
	= \sum_{\widetilde{x} \in  A_{x,t}} 
	b_0\left(t,   \widetilde{\exp}_{\wt{x}}(h \boldsymbol{y}) ;\delta\right)
	\e^{\frac{i}{h} \widetilde{\phi}_{t,\delta} 
	\left( \widetilde{\exp}_{\widetilde{x}}(h \boldsymbol{y}) \right)} 
	+ O_{C^L(\fK)}(h^\gamma)
\end{equation*}
for some $\gamma>0$.
\par 
Now, Proposition \ref{Porp:SobolevEst} along with a Taylor expansion also 
implies that, upon potentially further decreasing $\gd>0$ and $h_0>0$,
\begin{equation*}
	\|b_0\left(t, \widetilde{\exp}_{\wt{x}}(h \boldsymbol{y}) ;\delta\right) 
	- b_0\left(t,  \widetilde{x} ;\delta\right)\|_{C^L(\mathfrak{K})} = O(h^\gamma)
\end{equation*}
for some (possibly smaller) $\gamma>0$ and uniformly in 
$t\in [0,  \mathfrak{d} |\log h|]$.
Combining this with Proposition \ref{Prop:LeadingAmpl}, we obtain 
\begin{equation*}
	\|b_0\left(t,   \widetilde{\exp}_{\wt{x}}(h \boldsymbol{y}) ;\delta\right) 
	- b_0\left(t,  \widetilde{x} ; 0\right)\|_{C^L(\mathfrak{K})} = O(h^\gamma),
\end{equation*}
uniformly in $t\in [0,  \mathfrak{d} |\log h|]$.
Summing up, we have shown in this paragraph that for any compact set 
$\mathfrak{K}\Subset \R^d$ and any $L,M\in \N$ there exists a $N\in \N$ such 
that for $\gd>0$, $h_0>0$ small enough, there exists a $\gamma>0$ such that 
for all $h\in]0,h_0]$ and all $t\in [0,  \mathfrak{d} |\log h|]$
\begin{equation}
\psi_{h, t, x, \delta}(\boldsymbol{y}) 
= \sum_{\widetilde{x} \in A_{x,t}} 
	 b_0\left(t,  \widetilde{x} ; 0\right)
	\e^{\frac{i}{h} \widetilde{\phi}_{t,\delta} 
	\left( \widetilde{\exp}_{\widetilde{x}}(h \boldsymbol{y}) \right)}  
	+ O_{C^L(\fK)}(h^\gamma), \quad \boldsymbol{y}\in \mathfrak{K}, 
\end{equation}
In particular, in this first-order expansion, the amplitude does not depend 
on $\boldsymbol{y}$.
\\ 
\\
\textbf{First order approximation of the phases.}
Let $\mathfrak{K}\Subset \R^d$ be a compact set as above. 
For each $\widetilde{x} \in A_{x,t}$,  let us write
\begin{equation*}
	\xi_{t, \widetilde{x}, \delta} 
	:= \frac{\partial}{\partial \boldsymbol{y}}\Big{|}_{\boldsymbol{y}=0} 
	 \widetilde{\phi}_{t, \delta} 
	 (\widetilde{\exp}_{\widetilde{x}}(\boldsymbol{y})) \in \R^d.
\end{equation*} 
By \eqref{eq:BorneC1} we know that there exists a constant 
$C>0$ such that for all $h\in ]0,h_0]$, all $t\in[0,\gd |\log h|]$ and 
all $\widetilde{x} \in A_{x,t}$ 
\begin{equation}\label{estim:xi}
	|\xi_{t, \widetilde{x}, \delta} | \leq C.
\end{equation} 
By Taylor expansion,
\begin{equation*}
	\widetilde{\phi}_{t,\delta} 
	\left( \widetilde{\exp}_{\widetilde{x}}(h \boldsymbol{y})\right)
	 = \widetilde{\phi}_{t,\delta} (\widetilde{x}) 
	 	+ h \xi_{t, \widetilde{x}, \delta} \cdot \boldsymbol{y} 
		+ r_{\widetilde{x}, h}(\boldsymbol{y}).
\end{equation*}
It follows from \eqref{eq:BorneCL} that for $\gd,h_0>0$ small enough
\begin{equation*}
	\|r_{\widetilde{x}, h}\|_{C^L(\mathfrak{K})} = O(h^{1+\gamma}) 
\end{equation*}
for some $\gamma>0$ and uniformly in $t\in [0,  \mathfrak{d} |\log h|]$. 
Therefore, we may write
\begin{equation*}
\e^{\frac{i}{h} \widetilde{\phi}_{t,\delta} 
\left( \widetilde{\exp}_{\widetilde{x}}(h \boldsymbol{y}) \right)} 
= \e^{\frac{i}{h} \widetilde{\phi}_{t,\delta} (\widetilde{x}) +
	i  \xi_{t, \widetilde{x}, \delta} \cdot \boldsymbol{y}} 
	+ O_{C^L(\fK)}(h^\gamma)
\end{equation*}
for some $\gamma>0$ and $\boldsymbol{y}\in \mathfrak{K}$. 
\par
%
%Now, recall from (\ref{eq:PhaseApprochee}) that for 
%$\gd,h_0>0$ small enough 
%
%\begin{equation}\label{eq:PhaseApprochee2}
%\widetilde{\phi}_{t,\delta}(\widetilde{x}) 
%	-  \widetilde{\phi}_{t,0}(\widetilde{x})
%= \delta \theta_t (\widetilde{x}) 
%+ O \left( \delta^2 h^{-2\beta-2\varepsilon} \right),
%\end{equation}
%
%uniformly in $t\in [0,  \mathfrak{d} |\log h|]$, 
%with 
%
%\begin{equation}\label{eq:RappelTheta}
%\theta_t (\widetilde{x}) :=  -\int_0^t  q_\omega 
%	\left( \widetilde{\Phi}_{0}^{-s} (\widetilde{x}, 
%	d_{\widetilde{x}} \widetilde{\phi}_{t,0}) \right)
%	ds.
%\end{equation}
%
Furthermore, Lemma \ref{lem:BoundDiffFlows} implies that
\begin{equation*}
	|\xi_{t, \widetilde{x}, \delta} - \xi_{t, \widetilde{x}, 0}| = O(h^\gamma)
\end{equation*}
for some $\gamma>0$.
\par 
All in all, we have shown the following result:
\begin{prop}
For any $L\in \N$ and any compact set $\fK\subset \R^d$, there 
exist $\gd>0$, $h_0>0$ and $\gamma>0$ such that, for all 
$h\in ]0,h_0]$, all $0\leq t\leq \gd |\log h|$ and all $\by\in \fK$
\begin{equation*}
\psi_{h, t, x, \delta}(\boldsymbol{y}) 
= \sum_{\widetilde{x} \in A_{x,t}}
	 b_0\left(t,  \widetilde{x} ; 0\right)  
	 \e^{\frac{i}{h} \widetilde{\phi}_{t,\delta}( \widetilde{x})}
	 \e^{i  \xi_{t, \widetilde{x}, 0} \cdot \boldsymbol{y}} 
	 + O_{C^L(\fK)}(h^\gamma).
\end{equation*}
\red{Here, the remainder term 
is uniform in $x\in X$, $\delta$ as in Hypothesis \ref{Hyp:BetaDelta}, 
and $t\in [0,\mathfrak{d}\log h|]$.}
\end{prop}
This expression deserves several comments:
\begin{itemize}
\item Up to a $O_{C^L}(h^\gamma)$ remainder, the function 
$\psi_{h, t, x, \delta}$ can be written as a sum of plane waves, 
with \red{$O(h^{-c \gd})$} terms. In particular, the dependence on the 
variable $\by$ is only in the exponentials $\e^{i  \xi_{t, \widetilde{x}, 0} 
\cdot \boldsymbol{y}}$.
\item The amplitudes $ b_0\left(t,  \widetilde{x} ; 0\right)$ and the 
directions of propagation $\xi_{t, \widetilde{x}, 0}$ do not depend on 
$\delta$, and hence do not depend on the random parameter $\omega$. The 
only quantities which depend on the random parameter $\omega$ are the 
phases $\widetilde{\phi}_{t,\delta}( \widetilde{x})$. 
\end{itemize}
In the sequel,  we will often need to have $\widetilde{x}$ vary in an 
open set, and in particular in a ball of small radius; we will then think 
of it as a random variable, being picked uniformly in the open set. So, 
let $x_0\in X$, and let $\vartheta>0$. Thanks to Propositions 
\ref{lem:ControlDeriv} and \ref{Porp:SobolevEst}, we know that, there 
exists $\gamma(\vartheta)>0$ such that, for all $x\in B(x_0, h^\vartheta)$, 
we have for all lifts $\widetilde{x}$,  $\widetilde{x}_0$ of $x, x_0$ with 
$\widetilde{x} \in B(\widetilde{x}_0, h^{\vartheta})$:
\begin{align*}
| b_0\left(t,  \widetilde{x} ; 0\right) -  b_0\left(t,  \widetilde{x_0} ; 0\right)| &= O(h^{\gamma})\\
| \xi_{t, \widetilde{x}, 0} -  \xi_{t, \widetilde{x_0}, 0}| &= O(h^{\gamma}).
\end{align*}
Therefore, 
we  have the following result:
\begin{corollary}\label{cor:LocalNearX0}
For any $L\in \N$, any compact $\fK\subset \R^d$ and any $\vartheta >0$, 
there exists $\gd>0$ and $\gamma>0$ such that, for all $x_0\in X$, all 
$h\in ]0,h_0]$, all $x\in B(x_0, h^\vartheta)$ and 
all $0\leq t\leq \gd |\log h|$, and for all $\by\in \fK$
\begin{equation}\label{eq:TaylorFinal}
\psi_{h, t, x, \delta}(\boldsymbol{y}) 
= \sum_{\widetilde{x} \in A_{x,t}} 
	 b_0\left(t,  \widetilde{x_0} ; 0\right)  
	 \e^{\frac{i}{h} \widetilde{\phi}_{t,\delta}( \widetilde{x}) }  
	 \e^{i  \xi_{t, \widetilde{x_0},0} \cdot \boldsymbol{y}} 
	 + O_{C^L(\fK)}(h^\gamma).
\end{equation}
\red{Here, the remainder term 
is uniform in $x\in X$, $\delta$ as in Hypothesis \ref{Hyp:BetaDelta}, 
and $t\in [0,\mathfrak{d}\log h|]$.}
\end{corollary}
Note that, in the expression above, the phases 
$ \widetilde{\phi}_{t,\delta}( \widetilde{x})$ are the only quantities depending 
on the random parameter $\omega$ and on the quasi-random parameter 
$\widetilde{x}$. This remark will be essential in the next two sections.

\section{Independence of the phases}\label{Sec:Indep}
From Corollary \ref{cor:LocalNearX0} we see that, up to a small error, 
the propagated Lagrangian state is given by the superposition of random functions 
where the randomness comes from the phases  $\widetilde{\phi}_{t,\delta}( \widetilde{x})$, 
given by the expression \eqref{eq:PhiIntegrale2}.  To derive Theorem 
\ref{th:MartinEtMaximeSontDesBeauxGosses} from this expression we will use a 
Central Limit Theorem. However, the principal obstacle is that the family 
of random variables $(\widetilde{\phi}_{t,\delta}( \widetilde{x}))_{\widetilde{x}\in A_{x,t}}$ 
is in general not independent. For instance, in the \ref{eq:Pot}, two lifts 
$\wit{x}, \wit{x}'\in A_{x,t}$ may correspond to trajectories which cross 
(or come close to each other) on $X$, making (part of) the phases 
$\widetilde{\phi}_{t,\delta}( \widetilde{x})$,  $\widetilde{\phi}_{t,\delta}( \widetilde{x}')$ 
dependent. This issue is remedied in Section \ref{sec:DepTim} by removing the 
``exceptional times'' at which trajectories cross, and thus defining new 
phases $ \widetilde{\phi}^0_{t,\delta}( \widetilde{x})$.
\par
Actually, even in the \ref{eq:Pseudo}, two lifts $\wit{x}, \wit{x}'\in A_{x,t}$ 
may correspond to trajectories which come close to each other 
\emph{in the phase space $S^*X$}, making the random variables 
$\widetilde{\phi}_{t,\delta}( \widetilde{x})$ and $\widetilde{\phi}_{t,\delta}( \widetilde{x}')$ 
dependent: this phenomenon is illustrated in Figure \ref{Fig:Dependence}. 
However, we will show in Section \ref{sec:DepPoints} that this happens only 
for a small set of points $x\in X$, which we may thus neglect when proving 
Theorem \ref{th:MartinEtMaximeSontDesBeauxGosses}. 
\begin{figure}
\begin{center}
\begin{tikzpicture}[scale=1.5]
\draw[smooth] (0,1) to[out=30,in=150] (2,1) to[out=-30,in=210] (3,1) to[out=30,in=150] (5,1) to[out=-30,in=30] (5,-1) to[out=210,in=-30] (3,-1) to[out=150,in=30] (2,-1) to[out=210,in=-30] (0,-1) to[out=150,in=-150] (0,1);
\draw[smooth] (0.4,0.1) .. controls (0.8,-0.25) and (1.2,-0.25) .. (1.6,0.1);
\draw[smooth] (0.5,0) .. controls (0.8,0.2) and (1.2,0.2) .. (1.5,0);
\draw[smooth] (3.4,0.1) .. controls (3.8,-0.25) and (4.2,-0.25) .. (4.6,0.1);
\draw[smooth] (3.5,0) .. controls (3.8,0.2) and (4.2,0.2) .. (4.5,0);
\draw[smooth] (2.5,-0.85) .. controls (2.4,-0.8) and (2.4,0.8) .. (2.5,0.85);
\draw[dashed] (2.5,-0.85) .. controls (2.6,-0.8) and (2.6,0.8) .. (2.5,0.85);
\draw[smooth] (2.12,0) .. controls (2.11,0.1) and (2.1,0.85) .. (2.2,0.9);
\draw[dashed] (2.2,0.9) .. controls (2.35,0.85) and (2.2,-0.85) .. (2.15,-0.92);
\draw[smooth] (2.15,-0.92) .. controls (1.9,-0.5) and (1.8,0.9) .. (2,1);
\draw[dashed] (2,1) .. controls (2.1,0.7) and (1.9,-1) .. (1.8,-1.1);
\draw[smooth] (1.8,-1.1) .. controls (1.6,-1) and (1.3,-0.1) .. (1.4,-0.05);
\draw[dashed] (1.4,-0.05) .. controls (1.6,-0.15) and (1.6,-1) .. (1.5,-1.23);
\draw[smooth] (1.5,-1.23) .. controls (1.35,-1.1) and (1.5,-0.1) .. (1.8,-0);
\draw  [red] (2.42,0) node{$\bullet$};
\draw [red, thick] [->] (2.42,0) -- (2.42,0.4) ;
\draw [red] (2.6, 0.2) node{$\rho_0$};
\draw [blue](2.12,  -0.2) node{$\rho$};
\draw  [blue, thick] [->] (2.12,0) -- (2.11,0.4) ;
\draw [blue] (2.12,0) node{$\bullet$};
\draw [magenta](1.92,0) node{$\bullet$};
\draw  [magenta, thick] [->] (1.92,0) -- (1.88,0.4) ;
\draw [magenta] (1.75,  0.15) node{$\rho'$};
\draw  (1.48,  -0.6) node{$\bullet$};
\draw (1.3, -0.6) node{$x$};
\end{tikzpicture}
\end{center}
\caption{An example where the phases $\widetilde{\phi}_{t,\delta}( \widetilde{x})$
and $\widetilde{\phi}_{t,\delta}( \widetilde{x}')$ are not independent. Here, 
$\Lambda_0$ is a piece of the unstable manifold of 
the periodic point $\rho_0$ to which $\rho$ and $\rho'$ 
both belong. There are two lifts of $x$, $\wit{x}$ 
and $\wit{x}'$ such that $\rho_{\wit{x}}$ and 
$\rho_{\wit{x}'}$ originate respectively from 
the points $\rho$ and $\rho'$; since $\rho$ and $\rho'$ 
belong to the same geodesic, the phases $\widetilde{\phi}_{t,\delta}( \widetilde{x})$
and $\widetilde{\phi}_{t,\delta}( \widetilde{x}')$ are dependent, since the integrals 
in \eqref{eq:PhiIntegrale2} defining them contain a part 
which is equal.}\label{Fig:Dependence}
\end{figure}
The aim of this section is thus to replace 
(\ref{eq:TaylorFinal}) by an expansion where, around 
most points of $X$, the different terms in the sum are 
independent. 

This is the content of the following proposition.
The proposition also states that most close points yield independent phases. 
To make this second point more precise, we need to introduce the following 
sets for every $x\in X$, $\varepsilon>0$ and $t>0$:
\begin{equation}\label{eq:defVt}
\mathcal{V}_{t, \varepsilon}(x)
:=  \{ y\in X ; \exists s\in [-t, t], \exists \wit{x}\in A_{x,t}  
	\text{ with }
	\mathrm{dist}_{X} \left( \pi_X\Phi^{s} (\rho_{\wit{x}}  ) , y\right) 
< h^{\beta- \varepsilon}\}.
\end{equation}
This is thus a set of points in $X$ which may be approached by trajectories 
passing through $x$. In \eqref{eq:defVt} we used the following notation:  
For every $\wit{x}\in A_{x,t}$, see \eqref{eq:DefA}, we define
\begin{equation*}
	\rho_{\wit{x}} 
	:= \wih{\pi} (\wit{x}, d_{\wit{x}} \wit{\phi}_{t,0} )\in T^*X.
\end{equation*}
%
%so that (\ref{eq:RappelTheta}) rewrites $
%\theta_t (\widetilde{x}) :=  -\int_0^t  q_\omega 
%	\left( \Phi^{-s} (\rho_{\wit{x}}) \right)
%	ds.$
%
\begin{prop}\label{PropSection8}
Let $\varepsilon>0$, let $k\in \N$ and let $\mathfrak{K}\Subset \R^d$ be 
a compact set. There exists $\varepsilon_1>0$, $\gamma>0$, 
$C_0, c_0>0$,  $\gd(\varepsilon)>0$, $h_0>0$, such that the 
following holds for all $h\in (0, h_0]$. We may write 
\begin{equation*}
	X = \left(\bigsqcup_{i\in \wih{I}_h} \wih{X}_i  \right) \sqcup X^0_h,
\end{equation*}
where each $\wih{X}_i$ has a diameter $\leq C_0 h^{\beta - \varepsilon_1}$ 
and has volume $\geq c_0 h^{d(\beta - \varepsilon_1)}$, and where 
$\Vol(X_h^0) = o_{h\to 0}(1)$.
\par 
Furthermore, for every $x\in \wih{X}_i$, we have
\begin{equation}\label{eq:ExpansionpIndep}
\psi_{h, t, x, \delta}(\boldsymbol{y}) = \sum_{\widetilde{x} \in A_{x,t}} 
	 b_0\left(t,  \widetilde{x} ; 0\right)  \e^{\frac{i}{h} 
	 \widetilde{\phi}^0_{t,\delta}( \widetilde{x})} 
	  \e^{i  \xi_{t, \widetilde{x},0} \cdot \boldsymbol{y}} 
	  + O_{C^k(\mathfrak{K})}(h^\gamma), 
	\quad \boldsymbol{y}\in \mathfrak{K},
\end{equation}
uniformly in $0\leq t\leq \gd |\log h|$. Here 
$\widetilde{\phi}^0_{t,\delta}( \widetilde{x})$
is defined in \eqref{eq:DefThetaZero} below and 
\begin{itemize}
\item  for each fixed $x\in X\backslash X_h^0$, 
the phases $(\widetilde{\phi}^0_{t,\delta}( \widetilde{x}) )_{\{\wit{x} 
\in A_{x,t} \}}$ 
are \red{statistically} independent.  
\item Let $x_1,..., x_p \in \red{\wih{X}}_i$ be such that, for all $1\leq j,j' \leq p$, 
we have $x_j \notin \mathcal{V}_{t, \varepsilon}(x_{j'})$. Then the phases 
$(\widetilde{\phi}^0_{t,\delta}( \widetilde{x}_j))_{\{j =1,..., p ; ~ 
\wit{x}_j \in A_{x_j,t}\}}$ are independent. 
\end{itemize}
\end{prop}
\begin{rem}\label{rem:RemSec8}
Thanks to Corollary \ref{cor:LocalNearX0}, it is possible, in 
(\ref{eq:ExpansionpIndep}), to replace for every $x\in \wih{X}_i$ 
the amplitude $ b_0\left(t,  \widetilde{x} ; 0\right)$ by 
$ b_0\left(t,  \widetilde{x}_i ; 0\right)$,  where $x_i$ is a fixed 
point in $\wih{X}_i$, and where $\wit{x}_i$ corresponds to the lift 
of $x_i$ in the same 
sheet as $\wit{x}$, i.e., the point $\wit{x}_i\in \wit{X}$ 
such that $\mathrm{dist}_{\wit{X}} (\wit{x}, 
\wit{x}_i) = \mathrm{dist}_X(x, x_i)$. (This point is unique 
as soon as $h$ is small enough so that the diameter 
of $\wh{X}_i$ is smaller than the injectivity radius $r_I$.)
\end{rem}
Proving Proposition \ref{PropSection8} will take the rest of Section 
\ref{Sec:Indep}.
\red{Note that, although the expression \eqref{eq:PhiIntegrale2} for the phases involve the perturbed trajectories $\zeta_\delta^{s,t}$,  most of the arguments in the proof of Proposition \ref{PropSection8} will only concern the unperturbed trajectories given by the flow $\Phi^t = \Phi^t_0$.}
\subsection{Removing exceptional points}\label{sec:DepPoints}
The first step in the proof of Proposition \ref{PropSection8} is to 
remove exceptional points, which will correspond to the set $X_h^0$. 
\par
\MI{To this end},  \red{recall that $r_I$ denotes the injectivity radius of $X$.  For $0<T_0<T$},  we define the \emph{bad} set
\begin{equation}\label{eq:BadSet}
	\begin{split}
		X_{T_0,T, \gamma} 
		:= \bigg\{ x\in X \text{ such that } &\exists t_1\in [T_0, T], 
		\exists t_2\in [r_I, T] \\
		&\text{ such that }  
		\exists  \wit{x} \in A_{x, t_1} \text{ with } \mathrm{dist}_X 
		(\pi_X(\Phi^{t_2} (\rho_{\wit{x}})), x) < h^\gamma \bigg\}.
	\end{split}
\end{equation}
\red{In the sequel, $T_0$ will be independent of $h$, while $T$ will be bounded by $\gd |\log h|$.}
Roughly speaking, the bad set contains all points $x\in X$ which are attained 
by some trajectory of the time evolution of $\Lambda_0$ \red{at some time $t_1\in [T_0,T]$}, and which are closely 
approached by the same trajectory at a future time \red{$t_1+ t_2$, with $t_2\in [r_I, T]$}.
\begin{lem}\label{lem:MostPointsIndep}
Let $0< \gamma < \gamma'$. There exists $T_1>0$, $\gd>0$, $h_0>0$ such that, 
if $0<h\leq h_0$, $T_1\leq T_0 \leq t \leq \gd |\log h|$, and 
$x\in X \setminus X_{T_0,t, \gamma}$, the following holds. 
For all $\wit{x} \neq  \wit{x}'\in A_{x,t}$ and all 
$0 \leq s_1, s_2 \leq t$, we have 
\begin{equation*}
\mathrm{dist}_{T^*X} (\Phi^{-s_1}(\rho_{\wit{x}}), 
 \Phi^{-s_2}(\rho_{\wit{x}'})) \geq h^{\gamma'}.
\end{equation*}
\end{lem}

\begin{rem}\label{rem:DeltaToZero}
Let $\gamma'<\beta$. Upon potentially shrinking $\gd>0$ and $h_0>0$, we deduce 
from \eqref{eq:PetitCoupDeStressDeLaFin} and \eqref{eq:CondBeta} that if 
$\mathrm{dist}_{T^*X} (\Phi^{-s_1}(\rho_{\wit{x}}), \Phi^{-s_2}(\rho_{\wit{x}'})) \geq h^{\gamma'}$ 
for some $0 \leq s_1, s_2 \leq \red{t \leq } \gd |\log h|$, then we also have
\begin{equation*}
	\mathrm{dist}_{T^*X} \left(\zeta_\delta^{s_1,t}(x), \zeta_\delta^{s_2,t}(x)\right) 
	\geq h^{\gamma'}/2.
\end{equation*}
\par
In particular,  equation \eqref{eq:PhiIntegrale2} \red{and Lemma \ref{lem:MostPointsIndep} imply} 
that, in the \ref{eq:Pseudo}, the phases 
$\widetilde{\phi}^0_{t,\delta}( \widetilde{x})$ and 
$\widetilde{\phi}^0_{t,\delta}( \widetilde{x}')$ are 
independent if $\wit{x}, \wit{x'}$ satisfy the assumptions 
of Lemma \ref{lem:MostPointsIndep}.
\end{rem}
\begin{proof}[Proof of Lemma \ref{lem:MostPointsIndep}]
\red{Let $t_0$ be as in \MI{Lemma} \ref{Lem:DernierLemmeAvantLesVacances}, and set $T_1:= t_0$. }
Let $0<\gamma< \gamma'$. We will make an argument by contradiction. 
Suppose that the result is false, so that for any $\gd>0$ and $h>0$ 
arbitrarily small, we may find \red{$T_0< t < \gd |\log h|$,  $x\in X$,  $\wit{x}\neq \wit{x}'\in A_{x,t}$ and} $s_1, s_2 \in [0,t]$ 
such that $\mathrm{dist}_{T^*X} (\Phi^{-s_1}(\rho_{\wit{x}}), 
\Phi^{-s_2}(\rho_{\wit{x}'})) < h^{\gamma'}$. 
\par 
Let us take $0<\varepsilon< \gamma$ such that $\gamma + \varepsilon < \gamma'$.
Without loss of generality, we may suppose that $s_2 \geq s_1$. 
Thanks to Lemma \ref{Lem:Gron},  we have $\mathrm{dist}_{T^*X} 
(\Phi^{s_2 - s_1}(\rho_{\wit{x}}), 
\rho_{\wit{x}'}) < h^{\gamma'- \varepsilon}$, provided 
$\gd$ is small enough. \red{We thus have $\mathrm{dist}_X 
		(\pi_X(\Phi^{s_2-s_1} (\rho_{\wit{x}})), x) < h^\gamma $.  We will now show that $s_2-s_1\geq r_I$,  and this will give us a contradiction with the fact that $x\notin X_{T_0, t, \gamma}$.}
\par 
\red{To this end, we first} claim that, for $h$ and $\gd$ small enough,
\begin{equation}\label{eq:DistanceLifts} 
 \mathrm{dist}_{T^*X}(\rho_{\wit{x}},  \rho_{\wit{x}'})> h^\varepsilon.
\end{equation}  
Indeed, if $\mathrm{dist}_{T^*X}(\rho_{\wit{x}}, 
\rho_{\wit{x}'})\leq h^\varepsilon$, then taking $\gd$ \red{and $h_0$} small enough and
using Lemma \ref{Lem:Gron}, we would have 
$\mathrm{dist}_X(\Phi^{-s}(\rho_{\wit{x}},) \Phi^{-s}(\rho_{\wit{x}'}))< \red{h^{\varepsilon/2}}$ 
for all $s\in [0,t]$. By Lemma \ref{Lem:DernierLemmeAvantLesVacances} \red{and the definition we took for $T_1$,} 
there would exist \red{$\tau \in (-\varepsilon, \varepsilon)$} such that for all $\tau'$ small enough, 
$ \mathrm{dist}^2_X ( \Phi^{-t +t_0}(\rho_{\wit{x}}), 
\Phi^{-t +t_0+ \tau'}(\rho_{\wit{x}'})) \red{\geq} 4 \mathrm{dist}^2_X ( \Phi^{-t}
(\rho_{\wit{x}}), \Phi^{-t + \tau}(\rho_{\wit{x}'})).$ Let us write 
$\Phi^{-t + \tau}(\rho_{\wit{x}'})= (x_0, \xi_0)$,  and set 
$\rho'= (x_0,  \lambda \xi_0)$, with $\lambda =\frac{t}{t-\tau_0}$, so that 
$\pi_X(\Phi^t(\rho')) = x$. The map $f: [0, t] \ni s 
\mapsto \mathrm{dist}_X^2\left(\Phi^s(\rho'), \Phi^s( \Phi^{-t}(\rho_{\wit{x}}))
\right)$ would then be non-negative, convex (by Corollary \ref{cor:Convex}), 
and satisfy $f(t)=0$, and $f(t_0)\red{\geq} 4 f(0)$.  \red{Since $t\geq t_0$,  this would imply that $f\equiv 0$,  and thus that $\rho'= \Phi^{-t}(\rho_{\wit{x}})$. We would deduce that
 $\rho_{\wit{x}} = \rho_{\wit{x}'}$, which is a contradiction. Thus,  (\ref{eq:DistanceLifts}) must hold.}

\par 
Now,  (\ref{eq:DistanceLifts}) implies in particular that we cannot have 
$\mathrm{dist}_{T^*X} (\Phi^{s}(\rho_{\wit{x}}), \rho_{\wit{x}'}) < h^\gamma$ 
for $s< \frac{h^\varepsilon}{2}$. On the other hand, thanks to Remark \ref{Rem:LeavingBall},
we cannot have $\mathrm{dist}_{T^*X} (\Phi^{s}(\rho_{\wit{x}}),  
\rho_{\wit{x}'}) < h^\gamma$ either for $  \frac{h^\varepsilon}{{2}} \leq  s< r_I$.  Therefore, 
we must have $s_2-s_1 \geq  r_I$, which contradicts the assumption that 
$x\in X_{T_0,t,\gamma}$.
\end{proof}
Next, we show that the bad set \eqref{eq:BadSet} has small measure. 
\begin{prop}\label{Prop:ExceptionalPoints}
For any $\gamma>0$, there exists $T_1>0$, $\gd>0$, $r>0$ and $h_0>0$ such 
that,  for all $h\leq h_0$,  if $T_1\leq T_0\leq T\leq \gd |\log h|$, then
\begin{equation*}
	\mathrm{Vol}(X_{T_0,T, \gamma}) < h^r.
\end{equation*}
\end{prop}
\begin{proof}
\textbf{\red{Step 1: Dividing $X_{T_0, T, \gamma}$ into smaller sets.}}
\MI{Let $t_0$ be as in Lemma \ref{Lem:DernierLemmeAvantLesVacances}, and set $T_1 := \min(r_I,  t_0)$.}
Let $\varepsilon_1,T, \gamma >0$, let $\rho_0\in S^*X$ and 
$r_I\leq T_1, T_2\leq T$.  We define
\begin{align*}
X_{T, \gamma, \rho_0, T_1, T_2, \varepsilon_1} 
	:= \{& x\in X \text{ such that } \exists t_1\in [T_1- h^{\varepsilon_1}, 
	T_1+ h^{\varepsilon_1}],  \exists t_2\in [T_2- h^{\varepsilon_1}, 
	T_2+ h^{\varepsilon_1}],   \exists  \wit{x} \in A_{x, t_1}\\
&\text{ with } \mathrm{dist}_{T^*X}(\rho_{\wit{x}}, \rho_0)< h^{\varepsilon_1} 
\text{ and }  \mathrm{dist}_X (\Phi^{t_2} (\rho_{\wit{x}}), x) < h^\gamma\}.
\end{align*}

\red{This is thus the set of points $x$ such that there is a trajectory coming from $\Lambda_0$, reaching $x$ at time $t_1$ at some point $\rho_{\wit{x}}$,  and coming back $h^\gamma$-close to $x$ at some time $t_2$. Here, $t_1$ and $t_2$ are $h^{\varepsilon_1}$-close to $T_1$ and $T_2$ respectively, and $\rho_{\wit{x}}$ is $h^{\varepsilon_1}$-close to $\rho_0$.}

\red{Covering $S^*X$ by balls of radius $h^{\varepsilon_1}$ and covering $[T_0, T]$ and $[r_I, T]$ with intervals of length $2 h^{\varepsilon_1}$, we thus see that $X_{T_0T, \gamma}$ is included in a union of sets $(X_{T, \gamma, \rho^j_0, T^j_1, T^j_2, 
\varepsilon_1})_{j\in \mathfrak{J}_h}$, with $\mathfrak{J}_h$ a set of cardinal  $O(h^{-c \varepsilon_1})$, for some $c>0$.  Therefore,  to prove the result, it is enough to show that there exists $r'>0$  independent of $\varepsilon_1$ such that, for every $\rho_0, T_1, T_2$, we have 
\begin{equation}\label{eq:BoundSmallVol}
\mathrm{Vol} \left( X_{T, \gamma, \rho_0, T_1, T_2, \varepsilon_1} \right) = O(h^{r'}).
\end{equation}
In order to show (\ref{eq:BoundSmallVol}), we will study how the trajectories of points $X_{T, \gamma, \rho_0, T_1, T_2, \varepsilon_1}$ diverge from each other, as depicted in Figure \ref{Fig:Eloignement}.}
%the union of 
%at most $O(h^{-c \varepsilon_1})$, for some $c>0$, sets $X_{T, \gamma, \rho_0, T_1, T_2, 
%\varepsilon_1}$, so that it suffices to show that, for $\varepsilon_1$ 
%small enough, each $X_{T, \gamma, \rho_0, T_1, T_2, \varepsilon_1} $ 
%has volume $O(h^r)$ with $r$ independent of $\varepsilon_1$.
%

\definecolor{qqzzqq}{rgb}{0,0.6,0}
\definecolor{xdxdff}{rgb}{0.49019607843137253,0.49019607843137253,1}
\definecolor{qqqqff}{rgb}{0,0,1}
\definecolor{ffqqqq}{rgb}{1,0,0}
\begin{figure}
\begin{center}
\begin{tikzpicture}[line cap=round,line join=round,>=triangle 45,x=1cm,y=1cm,scale=0.6]
\clip(-5.92117993409683,-6.688877728315953) rectangle (7.453121684126218,8.25022613850898);
\draw [shift={(0,-13.5)},line width=2pt,color=qqqqff]  plot[domain=1.0808390005411683:2.060753653048625,variable=\t]({1*8.5*cos(\t r)+0*8.5*sin(\t r)},{0*8.5*cos(\t r)+1*8.5*sin(\t r)});
\draw [->,line width=1pt,color=ffqqqq] (0,-5) -- (0,-4);
\draw [->,line width=1pt,color=qqzzqq] (0.9931506043228757,-5.058219863255552) -- (0.84,-4.1);
\draw [line width=2pt,color=ffqqqq] (0,-5)-- (0,7);
\draw [shift={(37.925167373842484,-2.0708612289737482)},line width=2pt,color=qqzzqq]  plot[domain=2.89426880950201:3.222304974439406,variable=\t]({1*37.05264058436181*cos(\t r)+0*37.05264058436181*sin(\t r)},{0*37.05264058436181*cos(\t r)+1*37.05264058436181*sin(\t r)});
\draw [color=ffqqqq](-1.1142968695604993,1.0579283161620414) node[anchor=north west] {$\rho_{\widetilde{x}}$};
\draw [color=qqzzqq](1.4831357378884928,1.237328035086795) node[anchor=north west] {$\rho_{\widetilde{x}'}$};
\draw [color=ffqqqq](-1.436186960977967,-5.106898388706763) node[anchor=north west] {$\rho= \Phi^{-t_1}(\rho_{\widetilde{x}})$};
\draw [color=qqzzqq](1.597299195386064,-3.8674094215902834) node[anchor=north west] {$\Phi^{\tau}(\rho')= \Phi^{-t'_1+ \tau}(\rho_{\widetilde{x}})$};
\draw [color=ffqqqq](-2.4318957452595173,7.255373151744437) node[anchor=north west] {$\Phi^{t_2}(\rho_{\widetilde{x}})$};
\draw [color=qqzzqq](2.2496618096578986,7.108591563533275) node[anchor=north west] {$\Phi^{t'_2}(\rho_{\widetilde{x}'})$};
\draw [->,line width=1pt] (0.9586687622096514,6.9916172230675455) -- (0,7);
\draw [->,line width=1pt] (0.9586687622096514,6.9916172230675455) -- (2,7);
\draw [color=qqqqff](-3.3769657384366747,-4.42191764372134) node[anchor=north west] {$\Lambda_0$};
\draw (0.1741192936263741,8.089281092800248) node[anchor=north west] {$\leq 2 h^\gamma$};
\begin{scriptsize}
\draw [fill=ffqqqq] (0,-5) circle (2.5pt);
\draw [fill=qqzzqq] (0.9931506043228757,-5.058219863255552) circle (2.5pt);
\draw [fill=ffqqqq] (0,7) circle (2.5pt);
\draw [fill=qqzzqq] (1,1) circle (2.5pt);
\draw [fill=qqzzqq] (2,7) circle (2.5pt);
\draw [fill=ffqqqq] (0,1) circle (2.5pt);
\end{scriptsize}
\end{tikzpicture}
\end{center}
\caption{\red{Illustration of Step 2 of the proof of Proposition \ref{Prop:ExceptionalPoints}.  By Corollary \ref{cor:Convex} and Lemma \ref{Lem:DernierLemmeAvantLesVacances},  points of $\Lambda_0$ have trajectories that are ultimately diverging from each other (as long as they remain closer than $r_I$). Therefore,  the fact that they are at a distance $\leq 2 h^\gamma$ after times $t_1+t_2$ (resp. $t_1'+t_2'$) imply that they must also be at a distance $\leq 2 h^\gamma$ after times $t_1$ (resp. $t_1'$).}}\label{Fig:Eloignement}
\end{figure}

\red{\textbf{Step 2: Comparing the trajectories of points of $X_{T, \gamma, \rho_0, T_1, T_2, \varepsilon_1}$.}}
\red{Let $x, x'\in X_{T, \gamma, \rho_0, T_1, T_2, \varepsilon_1}$, let 
$t_1, t_2, t_1', t_2'$ be the associated times and let $\rho_{\wit{x}}, \rho_{\wit{x}'}$ 
be the associated points in $S^*X$. By assumption, we have $\mathrm{dist}_{T^*X} (\rho_{\wit{x}}, \rho_{\wit{x}'})< 2 
h^{\varepsilon_1}$. In particular,  thanks to Lemma \ref{Lem:Gron} if $\gd$ is smaller than some $\gd(\varepsilon_1)$, we have 
\begin{equation}\label{eq:alwaysclose}
\forall t \in [-\gd |\log h|, \gd |\log h|],~~~~
\mathrm{dist}_{X} (\Phi^t(\rho_{\wit{x}}),  \Phi^t(\rho_{\wit{x}'}))< r_I
\end{equation}}
\par 
\red{%By assumption, we have $\mathrm{dist}_{T^*X} (\rho_{\wit{x}}, \rho_{\wit{x}'})< 2 h^{\varepsilon_1}$.  
Write $\rho:=\Phi^{-t_1}(\rho_x)\in \Lambda_0$, 
and $\rho':= \Phi^{- t_1'} (\rho_{\wit{x}})\in \Lambda_0$.
Thanks to Lemma \ref{Lem:DernierLemmeAvantLesVacances}, we know that there exists %$t_0>0$,  
$\varepsilon_2>0$ and a (small) 
$\tau\in \R$ such that for all $\lambda\in (1-\varepsilon_2, 1+ \varepsilon_2)$ and all $\tau'\in (-\varepsilon_2, \varepsilon_2)$, we have 
\begin{equation}\label{eq:Ecartement}
\mathrm{dist}_X(\Phi^{t_0}(\rho), \Phi^{\lambda t_0+ \tau'}(\rho')) \geq
2 \mathrm{dist}_X\left(\rho, \Phi^{\tau}(\rho')\right).
\end{equation}}

\red{The parameters $\lambda$ and $\tau'$ will be chosen later.  %Let us write $\rho'':= \Phi^{\tau}(\rho')$.  
Equation (\ref{eq:Ecartement}) implies that there exists $t_3\in [0,t_0]$ such that 
\begin{equation*}
	\frac{d}{dt}\Big{|}_{t=t_3} 
	\mathrm{dist}_X \left(\Phi^t(\rho), \Phi^{\lambda t + \tau'}(\rho')\right) 
	\geq 0.
\end{equation*}}
\MI{Since we assumed that $t_1 \geq T_1 \geq t_0$,} \red{we may apply Corollary \ref{cor:Convex} \MI{along with (\ref{eq:alwaysclose})} to see that
\begin{equation}\label{eq:LaConvexiteCestCool}
	\mathrm{dist}_X \left(\Phi^{t_1+ t_2}(\rho), 
	\Phi^{\lambda (t_1+t_2) + \tau' }(\rho')\right) - \mathrm{dist}_X \left(\Phi^{t_1}(\rho), \Phi^{\lambda t_1 + \tau'}(\rho') \right) 
	\geq 0.
\end{equation}}

\red{Now,  since  $|t_j-t_j'|< 2 h^{\varepsilon_1}$ for $j=1,2$,  we see that for $h$ small enough, there exist $\lambda, \tau'$ with $|\lambda-1 | < \varepsilon_2$, $|\tau'| < \varepsilon_2$ such that 
\begin{align*}
\begin{cases} t_1' &=  \lambda t_1 + \tau'\\
t_2' &= \lambda t_2.
\end{cases}
\end{align*}
Therefore,  equation (\ref{eq:LaConvexiteCestCool}) gives
\begin{equation*}
\begin{aligned}
0&\leq 
	\mathrm{dist}_X \left(\Phi^{t_1+t_2}(\rho), 
	\Phi^{t_1'+t'_2 }(\rho')\right) - \mathrm{dist}_X \left(\Phi^{t_1}(\rho), \Phi^{t'_1}(\rho') \right)\\
	&= \mathrm{dist}_X \left(\Phi^{t_2}(\rho_{\wit{x}}), 
	\Phi^{t'_2 }(\rho_{\wit{x}'})\right) - \mathrm{dist}_X(x,x')\\
	& \leq 2 h^\gamma -  \mathrm{dist}_X(x,x').
\end{aligned}
\end{equation*}
We deduce that, $x$ being fixed, $x'$ must belong to a set of volume $O(h^{d\gamma})$ so that (\ref{eq:BoundSmallVol}) holds. The result follows.}
\end{proof}

\subsection{Removing dependent times in the \ref{eq:Pot}}\label{sec:DepTim}
Given $x,y \in X$,  $t\geq 0$ and  $\wit{x} \in A_{x,t}$, $\wit{y} \in A_{y,t}$, 
let us write 
\begin{align*}
\gI_{\wit{x},\wit{y},\varepsilon, t} &:= \left\{ s\in [0, t] 
\text{ such that } \exists s'\in [0, t] \text{ with }  
\mathrm{dist}_X \left( \Phi^{-s} (\rho_{\wit{x}}), 
	\Phi^{-s'}(\rho_{\wit{y}}) \right) < h^{\beta-\varepsilon}  \right\}.
\end{align*}
This is thus the set of times $s$ at which $\Phi^s(\rho_{\wt{x}})$ 
is approached by $\Phi^{s'}(\rho_{\wt{y}})$ for some time $s'$. 

\begin{lem}\label{Lem:LePtiLemmeDeLaFin}
For all $\varepsilon>0$, there exists $\gd>0$ such that, for all  $0\leq t\leq \gd |\log h|$, for all $x\in X$, for all $y\in X\setminus \mathcal{V}_{t, 2 \varepsilon}(x)$, we have 
$$\forall \wit{x}\in A_{x,t}, \forall \wit{y} \in A_{y,t},~~\gI_{\wit{x},\wit{y},\varepsilon, t} = \emptyset.$$
\end{lem}
\begin{proof}
Let $\varepsilon>0$, $x\in X, y\in X \setminus \mathcal{V}_{x, 2\varepsilon}(x)$, and let $\wit{x}\in A_{x,t},  \wit{y} \in A_{y,t}$.  Suppose for contradiction that there exists $s\in \gI_{\wit{x},\wit{y},\varepsilon, t}$, so that there exists $s'\in [0,t]$ such that $\mathrm{dist}_X \left( \Phi^{-s} (\rho_{\wit{x}}), 
	\Phi^{-s'}(\rho_{\wit{y}}) \right) < h^{\beta-\varepsilon}$.  Using (\ref{eq:ExpRate}), we deduce that $\mathrm{dist}_X \left( \Phi^{-s+s'} (\rho_{\wit{x}}), 
	\rho_{\wit{y}} \right) < h^{\beta-2\varepsilon}$, which contradicts the fact that $y\notin \mathcal{V}_{t, 2 \varepsilon}(x)$.
\end{proof}
We then set
\begin{equation*}
	\gI_{x, \varepsilon,t} 
	:= \bigcup_{{\wit{x}, \wit{x}' \in A_{x,t}}} \gI_{\wit{x}, \wit{x}',
	\varepsilon,t}.
\end{equation*}
\begin{prop}\label{Prop:TimeInteraction}
Let $\varepsilon >0$. There exist $\gd>0$, 
$h_0>0$ such that for all $h\in]0,h_0]$, for all $T_0\geq r_I$ and 
all $x\in X \setminus X_{T_0, \gd |\log h|, 
\varepsilon}$, and all $t\in [0, \gd |\log h|]$, 
we have,
\begin{equation}\label{eq:BorneUnionI.0}
|\gI_{x,\varepsilon,t}|\leq h^{\beta - 5 \varepsilon}.
\end{equation}
Furthermore,  there exists 
$0< \gamma_0< \beta$ such that,  for any $x_0\in X$,  we have
\begin{equation}\label{eq:BorneUnionI}
\left|\bigcup_{x\in B(x_0,  h^{\gamma_0})\setminus X_{T_0,\gd |\log h|, 
 \varepsilon}} \gI_{x,\varepsilon,t} \right| \leq h^{\beta - 6 \varepsilon}.
\end{equation}
Here $B(x_0, h^{\gamma_0})$ denotes the geodesic ball of radius $h^{\gamma_0}$ 
around $x_0$. 
\end{prop}
\begin{proof}[Proof of Proposition \ref{Prop:TimeInteraction}]
1.  \red{First of all, note that by the triangle inequality,} there exists $t_0>0$ such that, for all 
$\rho, \rho'\in \mathcal{E}_{0, (\frac{1}{2}, 2)}$,  if 
$\mathrm{dist}_X (\rho, \rho') < \frac{r_I}{3}$,  then for all 
$\tau \in (-t_0, t_0)$, we have $\mathrm{dist}_X (\Phi^\tau (\rho),  \Phi^\tau( \rho')) <r_I$.
\par 
Let us write, for $k,k'\in \N_0$, $\wit{x}, \wit{x}'\in A_{x,t}$,
\begin{equation}\label{eq:DefII}
\begin{aligned}
\gI_{\wit{x}, \wit{x}',\varepsilon ,t,k, k'} 
:= \big\{& s\in [0, t]\cap [kt_0, (k+1)t_0)  \text{ such that } 
\exists s'\in [0, t]\cap [k't_0, (k'+1)t_0)\\
& \text{ with }  \mathrm{dist}_X \left( \Phi^{-s}(\rho_{\wit{x}}), 
\Phi^{-s'}(\rho_{\wit{x}'})\right)  < h^{\beta-\varepsilon}  \big\}.
\end{aligned}
\end{equation}
Since the sets $\gI_{x,\varepsilon,t}$ 
are 
unions of $\gI_{\wit{x}, \wit{x}',\varepsilon ,t,k, k'} $ over $k,k'\in \N$, 
with $k,k' \leq O(|\log h|)$ and over $O(h^{-C\gd})$
many $\wit{x}$ and $\wit{x}'$  (by \eqref{eq:DefA2}) , it is sufficient to prove bounds on $|\gI_{\wit{x}, 
\wit{x}',\varepsilon,t,k, k'}|$.
\\
\\
2. 
Let $s\in \gI_{\wit{x}, \wit{x}',\varepsilon ,t,k, k'} $, and let 
$s'\in [0, t]\cap [k't_0, (k'+1)t_0)$ be an associated time for $\wit{x}'$.
Let us write
\begin{equation*}
	\rho := \Phi^{-s}(\rho_{\wit{x}}) 
	~~~~ \rho' := \Phi^{-s'}(\rho_{\wit{x}'}),
\end{equation*}
so that
\begin{equation}\label{eq:DistRho}
\mathrm{dist}_X (\rho, \rho')< h^{\beta-\varepsilon}.
\end{equation}
We claim that, for $\gd>0$ small enough,  we have 
\begin{equation}\label{eq:InfDist}
\mathrm{dist}_{T^*X} (\rho, \rho') > h^{2 \varepsilon}.
\end{equation}
To prove (\ref{eq:InfDist}),   we argue by contradiction. Suppose that 
$\mathrm{dist}_{T^*X} (\rho, \rho') \leq h^{2\varepsilon}$.  
Since $\wit{x}\neq \wit{x}'$, we must have $|s-s'| \geq r_I$ (see the end 
of the proof of Lemma \ref{lem:MostPointsIndep}). Suppose that $s\geq s'$. 
Then, thanks to \eqref{eq:ExpRate}, we would have 
\begin{align*}
\mathrm{dist}_X (\Phi^{s-s'} (\rho_{\wit{x}'}), x) 
= \mathrm{dist}_X (\Phi^{s-s'} (\rho_{\wit{x}'}), \Phi^{s}(\rho)) 
&\leq h^{-C\gd}  \mathrm{dist}_{T^*X} (\Phi^{-s'} (\rho_{\wit{x}'}), \rho) \\
&=  h^{-C\gd} \mathrm{dist}_{T^*X} (\rho', \rho) 
\leq h^{2 \varepsilon -C\gd},
\end{align*}
which contradicts the fact that $x\notin X_{T_0,\gd |\log h|, 
\varepsilon}$, provided $\gd$ is chosen small enough. 
We reach a contradiction in the same way if we suppose that $s'\geq s$.
\\
\par 
3. Equations (\ref{eq:DistRho}) and (\ref{eq:InfDist}) imply that 
$\rho$ and $\rho'$ are close when projected on the base manifold $X$, 
but at a much larger distance in $T^*X$.  
Working in a local chart, we see that this implies that there exists 
$\tau\in [0,  h^{\beta - 4 \varepsilon}]$ such that (provided $h$ 
is small enough), we have
\begin{equation}\label{eq:CaSecarte}
\mathrm{dist}_X
(\Phi^\tau(\rho), \Phi^\tau(\rho')) > h^{\beta - \varepsilon}.
\end{equation}
Now, we know that,  for all $\tau \in [0, t_0]$, we have 
$\mathrm{dist}_X (\Phi^\tau(\rho) \Phi^\tau(\rho')) < r_I$,  so that 
$ [0, t_0] \ni \tau \mapsto \mathrm{dist}^2_X (\Phi^\tau(\rho) 
\Phi^\tau(\rho'))$ is convex thanks to Corollary \ref{cor:Convex}. 
We deduce from (\ref{eq:CaSecarte}) that for all $t\in [h^{\beta - 
4 \varepsilon}, t_0]$, we have $\mathrm{dist}_X (\Phi^\tau(\rho), 
\Phi^\tau(\rho')) > h^{\beta - \varepsilon}$, so that $[s+h^{\beta - 
4 \varepsilon}, t_0]\cap \gI_{\wit{x}, \wit{x}',\varepsilon ,t,k, k'} 
 = \emptyset$. Since this holds for all $s\in \gI_{\wit{x}, \wit{x}',
\varepsilon ,t,k, k'}$,  the first part of the statement follows readily.
\par
4. For the second part, we  note that Lemma \ref{Lem:Gron} implies that 
there exists $c>0$ such that
\begin{equation*}
	\left( \mathrm{dist}_X(x,y) < h^{\beta} \right) \Longrightarrow 
\gI_{\wit{x}, \wit{x}',  \varepsilon ,t,k, k'} \subset 
\gI_{\wit{y}_{\wit{x}}, \wit{y}'_{\wit{x}'},  \varepsilon + 
c \gd  ,t,k, k'},
\end{equation*}
where $\wit{y}_{\wit{x}}$ is a lift of $y$ such that 
$\mathrm{dist}_{\wit{X}}(\wit{y}_{\wit{x}}, \wit{x}) = 
\mathrm{dist}_X(y,x)$, and similarly for $\wit{y}'_{\wit{x}'}$.
The set defined in (\ref{eq:BorneUnionI}) is thus contained in
\begin{equation*}
	\bigcup_{x_j}  \gI_{x_j,\varepsilon + c\gd ,t}  ,
\end{equation*}
where $B(x_0, h^{\gamma_0})\subset \bigcup_{x_j}B(x_j, h^\beta)$,
%where  the $x_j$ are at a distance $< h^{\beta} $ from each other, 
with $x_j\in B(x_0, h^{\gamma_0}) \setminus X_{T_0,\gd |\log h|, \varepsilon}$.
In particular,  if $\gamma_0$ is taken slightly smaller than $\beta$, 
the number of sets in the union is a $O(h^{-\varepsilon_1})$ for an 
$\varepsilon_1$ arbitrarily small.
\par 
Now, each $|\gI_{x_j, \varepsilon +  c \gd ,t}|$ can be bounded just as in 
the first part of the proof, and the result follows by taking $\varepsilon_1$ 
small enough.
\end{proof}
\subsection{Proof of Proposition \ref{PropSection8}}
In the sequel, we take $\varepsilon_1>0$ which will be fixed below, and write 
\begin{equation*}
	X = \bigsqcup_{i\in I_h} X_i,
\end{equation*}
where each $X_i$ has a diameter $\leq C h^{\beta - \varepsilon_1}$ and 
has volume $\geq c h^{d(\beta - \varepsilon_1)}$. Such a partition can 
be obtained by using finitely many local charts, and by using cubes of 
size $h^{\beta - \varepsilon_1}$ in each chart.
\par
Let $\varepsilon>0$ (which we will also fix below), take $T_0$ large 
enough so that Lemma \ref{lem:MostPointsIndep} and Proposition 
\ref{Prop:ExceptionalPoints} apply,  set 
\begin{equation*}
	\wih{X}_i:= X_i \setminus X_{T_0,\gd |\log h|,  \varepsilon},
\end{equation*}
and write 
\begin{equation*}
	\wih{I}_h:= \left\{i\in I_h \text{ such that } \mathrm{Vol}(\wih{X}_i)
	\geq \frac{c}{2} h^{d(\beta-\varepsilon_1)} \right\}.
\end{equation*}
Note that $I_h \setminus \wih{I}_h = o_{h\to 0}(|I_h|)$, so that if we write 
$X_h^0 = X^0_{h,\varepsilon}:= X \setminus \bigsqcup_{i\in \wih{I}_h} \wih{X}_i$, 
we have, for any $\varepsilon>0$,
\begin{equation*}
	\mathrm{Vol}\left(X^0_{h,\varepsilon} \right) = o_{h\to 0}(1).
\end{equation*}
For every $i\in \wih{I}_h$, we will define sets $\gI_{i,t} \subset \R$.
\begin{itemize}
\item In the \ref{eq:Pseudo}, we may take any $\varepsilon, \varepsilon_1>0$, 
and we set $$\gI_{i,t}:= \emptyset.$$
\item In the \ref{eq:Pot}, we take $\varepsilon>0$ small enough so that
\begin{equation}\label{eq:CondEpsilon,laFin}
\delta h^{\beta - 6 \varepsilon} = O(h^{1 + \varepsilon}),
\end{equation}
which is possible thanks to (\ref{eq:CondBetaDelta2}). 
For each $i\in \wih{I}_h$ we then define
$$\gI_{i,t} := \bigcup_{x\in \wih{X}_i} \gI_{x,\varepsilon,t}.$$ 
Proposition \ref{Prop:TimeInteraction} gives us the existence of 
a $\gamma_0(\varepsilon)>0$. Taking $\varepsilon_1 =\beta- \gamma_0>0$, 
we deduce from the proposition that
\begin{equation}\label{eq:IlPleutAPeyresq}
|\gI_{i,t}| \leq h^{\beta - 6 \varepsilon}.
\end{equation}
\end{itemize}

We then set, for each $\wit{x}\in \wit{\mathcal{O}}_t$ with 
$\wit{\pi} (\wit{x})\in  \wih{X}_i$,
\begin{equation}\label{eq:DefThetaZero}
\begin{aligned}
	\wit{\phi}_{t,\delta}^0(\wit{x}) 
	&:= %\phi_0(y_0^{-t}(x))
	\phi_{t,0}(x) - \delta
		\int_{[0,t] \setminus \gI_{i,t}} q_\omega\left(\zeta_\delta^{s,t}(x)\right)d s\\
	\wit{\phi}_{t,\delta}^1(\wit{x}) 
	&:= - \delta
		\int_{ \gI_{i,t}} q_\omega\left(\zeta_\delta^{s,t}(x)\right)  d s
\end{aligned}
\end{equation}
so that 
\begin{equation*}
\wit{\phi}_{t,\delta}(\wit{x})  
= \wit{\phi}_{t,\delta}^0(\wit{x}) + \wit{\phi}_{t,\delta}^1(\wit{x}) .
\end{equation*}
In the \ref{eq:Pseudo}, we have $ \wit{\phi}_{t,\delta}^1(\wit{x}) \equiv 0$, while 
in the \ref{eq:Pot}, equations (\ref{eq:IlPleutAPeyresq}) and 
(\ref{eq:CondEpsilon,laFin}) imply that 
\begin{equation}\label{eq:SmallTimes1}
|\wit{\phi}_{t,\delta}^1(\wit{x})  |\leq C \delta h^{\beta - 6 \varepsilon} 
= O (h^{1+\varepsilon}).
\end{equation}

\red{Equation (\ref{eq:SmallTimes1}), along with equation (\ref{eq:TaylorFinal}) imply that there exist $\varepsilon, \gamma'>0$ such that
\begin{equation}\label{eq:TaylorFinal2}
\psi_{h, t, x, \delta}(\boldsymbol{y}) 
= \sum_{\widetilde{x} \in A_{x,t}} 
	 b_0\left(t,  \widetilde{x_0} ; 0\right)  
	\left[ \e^{\frac{i}{h} \widetilde{\phi}^0_{t,\delta}( \widetilde{x}) }  
	 \e^{i  \xi_{t, \widetilde{x_0},0} \cdot \boldsymbol{y}} + O(h^\varepsilon) \right]
	 + O_{C^L(\fK)}(h^{\gamma'}).
\end{equation}
Since the number of terms in the sum is $O(h^{-\varepsilon/2})$ provided $\gd$ is small enough, this gives us (\ref{eq:ExpansionpIndep}), by setting $\gamma = \min (\gamma', \varepsilon/2)$.}
%
%This proves equation (\ref{eq:ExpansionpIndep}). 
The first point of Proposition 
\ref{PropSection8} then follows from the definition of $\mathfrak{I}_{i,t}$ 
along with Remark \ref{rem:DeltaToZero} and the first point of Hypothesis \ref{HypPot}.
The second point of Proposition \ref{PropSection8} (with an arbitrarily 
small $\varepsilon>0$, provided $\gd$ is taken small enough) follows from 
Lemma \ref{Lem:LePtiLemmeDeLaFin} and Remark \ref{rem:DeltaToZero}.
\section{Proof of the main results}\label{sec:ProofMainTheorems}
The aim of this section is to prove Theorems \ref{th:MartinEtMaximeSontDesSuperBeauxGosses} and 
\ref{th:MartinEtMaximeSontDesBeauxGosses}. 
The starting point is formula \eqref{eq:ExpansionpIndep}. From now on, we 
fix $\mathcal{U}\subset X$  an open set, and $V$ an orthonormal 
frame on $\mathcal{U}$.
\subsection{Varying $\omega$ with $x$ fixed}\label{Sec:OmegaRandom}
Recall from the previous section that we decomposed
\begin{equation*}
	X = \left(\bigsqcup_{i\in \wih{I}_h} \wih{X}_i  \right) \sqcup X^0_h,
\end{equation*}
with $\mathrm{Vol}(X_h^0) = o_{h\to 0}(1)$.
\par 
Notice that, \red{by regularity of the Lebesgue measure}, up to taking $X_h^0$ slightly larger (but still of negligible 
volume), we can arrange so that $\wih{X}_i$ is either inside $\mathcal{U}$ 
or in its complement.
%\red{: up to a set of measure 0, we may write $\cU$ as the disjoint union of open sets on which we can define coordinate charts; in these charts,  we thus have open sets $U_j \subset \R^d$. By regularity of the Lebesgue measure,  we can find sets  $S_h$,  $S'_h$ which are unions of disjoint open cubes of diameter $h^{\beta - \varepsilon_1}$, with $S_h \subset U_j \subset S'_h$ and $\Vol(S_h' \setminus S_h) =o_{h\to 0}(1)$. We then take $\left(\bigsqcup_{i\in \wih{I}_h} \wih{X}_i  \right) = S_h$}. 
%
\par 
For each $\ell\in \wih{I}_h$, each $x\in \wh{X}_\ell$ (as in Remark \ref{rem:RemSec8}), 
and all $0\leq t\leq o_{h\to 0}( |\log h|)$, we write 
\begin{equation}\label{eq:EffectivFunction}
\psi^0_{h, t, x, \omega}(\boldsymbol{y}) 
:= \sum_{\widetilde{x} \in A_{x,t}} 
	 b_0\left(t,  \widetilde{x}_\ell ; 0\right)  
	 \e^{\frac{i}{h} \wit{\phi}^0_{t,\delta} (\wit{x}) } 
	  \e^{i  \xi_{t, \widetilde{x}_\ell,0} \cdot \boldsymbol{y}},
\end{equation}
so that \eqref{eq:ExpansionpIndep}, refined as described in Remark 
\ref{rem:RemSec8}, now reads as 
\begin{equation}\label{eq:EffectivFunction2}
	\psi_{h, t, x, \delta}(\boldsymbol{y}) 
	= \psi^0_{h, t, x, \omega}(\boldsymbol{y}) 
	  + R_h, \quad \text{with} \quad R_h=O_{C^k(\mathfrak{K})}(h^\gamma),
\end{equation}
\red{the bounds on $R_h$ being uniform in $\omega$.}
We deduce that, for any $\chi\in C_c^\infty(\R^d)$, and any continuous bounded 
functional $F: C^\infty(\R^d)\to \R$,
we have $F(\chi \psi_{h, t_h, x, \omega}) = F(\chi\psi^0_{h, t_h, x, \omega}) 
+ o_{h\to 0}(1)$.  Now,  recalling (\ref{eq:DistanceFunctions}), both 
$\mathrm{d}( \chi \psi_{h, t_h, x, \omega},  \psi_{h, t_h, x, \omega})$ and 
$\mathrm{d}( \chi \psi^0_{h, t_h, x, \omega},  \psi^0_{h, t_h, x, \omega})$ 
can be made smaller than any $\varepsilon>0$, by taking $\chi$ equal to $1$ 
in a large set depending on $\varepsilon$, but not on $h$. Using the continuity 
of $F$, we deduce that
\begin{equation}\label{eq:EffectivFunction3}
F(\psi_{h, t_h, x, \omega}) = F(\psi^0_{h, t_h, x, \omega}) + o_{h\to 0}(1),
\end{equation}
\red{uniformly in $\omega$.}
For the rest of the current \MI{sub}section we fix $x\in \wh{X}_\ell$, 
and the randomness which we shall consider comes solely from the random 
perturbation $q_\omega$. To emphasize the fact that in this section  
we compute expectations with respect to the random variable $\omega$, 
we will denote such expectations by $\E_\omega$.
% and we will see the phases $\Theta^0_{\omega,t} (\wit{x})$ as random variables 
% with respect to $\omega$, i.e. the random parameter coming from 
% the random perturbation $q_\omega$. 
Furthermore, we see $\R^d \ni \boldsymbol{y}\to\psi^0_{h, t, x, \omega}(\boldsymbol{y})$ 
as a random smooth function on $\R^d$.
\par
The aim of this subsection is to prove the following proposition. 
Recall that $C^L(\R^d)$ is equipped with the topology of convergence of 
derivatives of order $\leq L$ on all compact subsets of $\R^d$.  \red{Recall also that the amplitude of the Lagrangian state we consider (see \eqref{eq:trans1}) is written as $a= a_0 +O(h)$.}
\begin{prop}\label{Prop:RandOmega}
Let $t_h > 0$ be such that $\lim_{h\to 0} t_h  = +\infty$ and $t_h= o(|\log h|)$.
Let $\ell\in \wih{I}_h$ and let $x\in \wih{X}_{\ell}$. Then, for every 
bounded continuous map $F : C^\infty(\R^d) \longrightarrow \R$, we have
\begin{equation*}
	\E_\omega [F(\psi_{h,t_h,x,\omega})] \underset{h\to 0}{\longrightarrow} 
	\E_{\mathrm{BGF}_{\lambda_a}}[F].
\end{equation*}
with $\lambda_a= \frac{\|a_0\|^2}{\mathrm{Vol}(X)}$.
\end{prop}
This result immediately implies Theorem \ref{th:MartinEtMaximeSontDesSuperBeauxGosses}.  
In order to prove Proposition \ref{Prop:RandOmega}, we will need the following \MI{proposition}.
\begin{prop}\label{lem:MoyenneNulle}
Let $t_h > 0$ be such that $t_h= o(|\log h|)$ and $\liminf_{h\to 0} t_h >0$. 
Then for all $x\in X\setminus X_h^0$ and all $\wit{x}\in A_{x,t_h}$, we have
\begin{equation}\label{eq:ZeroAverage}
\left|\E_\omega \left[  \e^{\frac{i}{h} \wit{\phi}^0_{t_h,\delta} (\wit{x}) }  \right] \right|  
= O(\min\{h \delta^{-1} h^{-\beta/2},\delta h^{-2\beta}\}).
\end{equation}
\end{prop}
Note that, when $\delta = h^\alpha$ as in Remark \ref{rem:specialdelta}, then 
the right-hand side of (\ref{eq:ZeroAverage}) is $O(h^{\Gamma})$ with 
$\Gamma = \min (1- \alpha - \frac{\beta}{2},  \alpha - 2\beta)$.
\begin{rem}
Following the exact same steps of the proof below, we can show that for all $x\in X$ 
and all $\wit{x}\in A_{x,t_h}$
\begin{equation*}
	\left|\E_\omega \left[  \e^{\frac{i}{h} \wit{\phi}_{t_h,\delta} (\wit{x}) }   \right] \right|  
	=  O(\min\{h \delta^{-1} h^{-\beta/2},\delta h^{-2\beta}\}).
\end{equation*}
\end{rem}

\begin{proof}[Proof of Proposition \ref{lem:MoyenneNulle}]
\red{Let $x\in X \setminus X_h^0$,  so that we have $x\in \widehat{X}_k$ for some $k\in \widehat{I}_h$.  We consider a lift  $\wit{x}\in \wit{X}$, as in the previous sections.}
Recall \eqref{eq:PhiIntegrale2}, \eqref{eq:PhaseApprochee}, and write 
\begin{equation*}
	\wit{\phi}^0_{t_h,\delta} (\wit{x}) 
	=  \wit{\phi}_{t,0}(\wit{x}) + \Theta^0_{\omega,t_h} (\wit{x}) 
	+  \Theta^1_{\omega,t_h} (\wit{x}),
\end{equation*}
with 
\begin{equation}\label{eq:DefTheta0}
	\Theta^0_{\omega,t_h} (\wit{x}) 
	:= -\delta \int_{[0,t_h]\setminus  \gI_{k,t_h}} q_\omega  
		\left( \Phi_0^{-s}(\rho_{\wit{x}}) \right) ds,
\end{equation}
and 
\begin{equation}\label{eq:DefTheta1}
	\Theta^1_{\omega,t_h} (\wit{x}) 
	:= -\Theta^0_{\omega,t_h} (\wit{x})  - \delta
	\red{\int_{[0,t_h]\setminus  \gI_{k,t_h}}}  q_\omega\left(\zeta_\delta^{s,t}(x)\right) ds.
\end{equation}
Next, we need the following  
\begin{lem}\label{lem:LipschitzReg}
	For any $j\in J_h$, $\omega_j \mapsto \Theta^1_{\red{\omega}, t_h}$ is 
	differentiable almost everywhere, and for any $\varepsilon>0$, 
	there exists $\gd>0$ such that for all $t\leq \gd |\log h|$, we have
	\begin{equation}\label{eq:BornerTheta1}
	\|\partial_{\omega_j} \Theta^1_{\red{\omega}, t_{\red{h}}}\|_{L^\infty} 
		\leq \delta^2 h^{-\beta - \varepsilon},
	\end{equation}
\end{lem}
We will prove this result further below and continue for now with the 
proof of Proposition \ref{lem:MoyenneNulle}. 
\par 
Let us denote by $m(\omega_j)$ the common density of the variables $\omega_j$, 
which we suppose $C^2$ and supported in some bounded set $[-M,M]$, $M>0$, 
as in Hypothesis \ref{Hyp:Prob}.
% 
%Up to rescaling the functions $q_j$, we may suppose that $M=\frac{1}{2}$.
We shall write 
\begin{equation*}
	\theta_j 
	:=  - \delta\int_{[0,t_{\red{h}}] \setminus \mathcal{I}_{k, t_h} } 
		(\partial_{\omega_j} q_{\omega}  (\Phi_0^{-s} (x, d_x \phi_{t,0}))d s \red{=  - \delta\int_{[0,t_{\red{h}}] \setminus \mathcal{I}_{k, t_h} } 
		(q_j  (\Phi_0^{-s} (x, d_x \phi_{t,0}))d s},
\end{equation*}
so that $ \frac{1}{h} \Theta^0_{\omega,t_h} (\wit{x}) = \sum_{j\in J_h} \theta_j \omega_j$.
\par 
Since $\wit{\phi}_{t,0}(\wit{x})$ is independent of $\omega$, we find that 
\begin{align*}
\left|\E\left[ \e^{\frac{i}{h}\wit{\phi}^0_{t_h,\delta} (\wit{x})}\right]\right| &=: |I|\\
&= \left| \int_{[-M,M]^{|J_h|}} e^{\frac{i}{h} \Theta^1_{\red{\omega}, t_h}(\wit{x})} 
	\prod_{j\in J_h} e^{\frac{i}{h} \theta_j \cdot \omega_j}  m(\omega_j) d\omega_j  \right|\\
& = \left|\int_{[-M,M]^{|J_h|}} e^{\frac{i}{h} \Theta^1_{\red{\omega}, t_h}(\wit{x})}    
	 e^{\frac{i}{h} \vec{\theta} \cdot \omega}  \boldsymbol{m}(\omega) d\omega\right|,
\end{align*}
where $\boldsymbol{m}(\omega) := \prod_{j\in J_h}  m(\omega_j)$, where $d\omega$ 
denotes the product measure on $[-M,M]^{|J_h|}$, and where $\vec{\theta}$ denotes 
the vector in $\R^{J_h}$ whose entries are the $\theta_j$. 
\par 
\textbf{Step 1: A lower bound on most of the $\theta_j$.} Thanks to 
(\ref{eq:PotentialEverywhere}) and \eqref{eq:IlPleutAPeyresq}, there exists 
$c_0>0$ (independent of $t_h$ and of $h$), such that
$$\sum_{j\in J_h} |\theta_j| \geq c_0 \delta t_h.$$

We claim that there exists constants $c_1, c_2, \epsilon>0$, independent of $t_h$ and $h$, and a set $J'_h\subset J_h$
 such that $ \red{c_1 h^{\red{-}\beta}} \leq |J_h'| \leq c_2 t_h h^{\red{-}\beta}$,  and
\begin{equation}\label{eq:LowerTheta}
\forall j \in J_h', ~~ |\theta_j|\geq  \epsilon \delta h^\beta.
\end{equation}

% t_h \geq c_0'>0,$$
%since $t_h$ is bounded from below.
To prove (\ref{eq:LowerTheta}), we write $J_{h,\epsilon,1} := \{ j \in J_h ; 0< |\theta_j| \leq \epsilon \delta h^\beta \}$ 
and $J_{h,\epsilon,2} := \{ j \in J_h ; |\theta_j| \geq \epsilon \delta h^\beta \}$. Thanks to \MI{ the first four points of}
Hypothesis \ref{HypPot},  there exists 
$C>0$ such that 
\MI{\begin{equation}\label{eq:BoundNumberofSupports}
|J_{h,\epsilon,1} \cup J_{h,\epsilon,2}| \leq C t_h h^{- \beta}.
\end{equation} 
Indeed,  for every $j\in J_h$, there exists $\rho_j\in X$ (resp $\rho_j\in T^*X$) such that $B(x_j,  c_0 h^\beta)\subset \mathrm{supp}(q_j) \subset B(\rho_j, c_1 h^\beta)$, for some $c_0< c_1$ independent of $h$ and $j$.  The fact that each point belongs to at most $O(1)$ of the $B(x_j,  c_0 h^\beta)$ then implies that, for any $c_2>0$, any point belongs to at most $O(1)$ of the $B(x_j,  c_2 h^\beta)$. From this, we may deduce that each geodesic segment of length $h^\beta$ belongs to at most $O(1)$ of the $B(x_j,  c_1 h^\beta)$, and hence of the $\mathrm{supp}(q_j)$. Equation (\ref{eq:BoundNumberofSupports}) follows.}

Furthermore, up to taking $C$ larger in (\ref{eq:BoundNumberofSupports}),  we have
\begin{equation}\label{eq:UpperTheta}
\forall j \in J_h,  |\theta_j|\leq C t_h \delta h^\beta.
\end{equation}
\red{This follows from Remark \ref{Rem:LeavingBall} and from the fact that the support of $\red{q_{j}}$ has 
diameter $O(h^{\beta})$.}

%Indeed, the support of $\red{q_{j}}$ has 
%diameter $O(h^{\beta})$. Hence, if $s$ is such that 
%$\Phi_0^{-s} (x, d_x \phi_{t,0})$ belongs to the 
%support of $\red{q_{j}}$, then Corollary 
%\ref{cor:Convex} (applied to  $\Phi_0^{-s} (x, d_x \phi_{t,0})$ 
%and $\left(\pi_X \left( \Phi_0^{-s} (x, d_x \phi_{t,0})\right), 0\right)$) 
%implies that $\Phi_0^{-s'} (x, d_x \phi_{t,0})$ does not 
%belong to the support of $\red{q_{j}}$ for 
%all $s'\in [s + c h^{-\beta} , \red{s+}c]$ for some $c>0$ 
%independent of $h$, $t_h$. Equation (\ref{eq:UpperTheta}) follows.

We thus get
\begin{align*}
c_0 \delta t_h \leq \sum_{j\in J_{h,\epsilon, 1}} |\theta_j| + \sum_{j\in J_{h,\epsilon, 2}} |\theta_j| \leq C \epsilon\delta t_h + C \delta t_h h^{\red{\beta}} |J_{h,\epsilon, 2}|.
\end{align*}
Taking $\epsilon$ small enough,  the first term \red{can be made} smaller than $\frac{c_0}{2}\delta t_h$. We deduce that there exists $\epsilon>0$ such that
$$|J_{h, \varepsilon,2}| \geq \frac{c_0}{\red{2 \epsilon}} h^{\red{-}\beta} .$$
Taking $J_h' = J_{h, \epsilon,2}$, equation (\ref{eq:LowerTheta}) follows.

\textbf{Step 2: Integrating by parts.}
Let us denote by $\vec{\theta}'$ the vector whose entries 
are $\theta_j$ if $j\in J_h'$, and $0$ otherwise. In particular, we have
\begin{equation*}
\|\vec{\theta}'\| \geq \epsilon  \delta h^{\beta/2},
\end{equation*}
where $\|\vec{\theta}'\|$ denotes the $\ell^2$ norm of $\vec{\theta}'$. 
In particular, thanks to \eqref{eq:CondBetaDelta}, this norm is much 
larger than $\MI{h}$. 

We perform an integration by parts to deduce that
\begin{align*}
I &= - \frac{i h}{\|\vec{\theta}'\|^2} \int_{[-M,M]^{|J_h|}} \vec{\theta}' 
	\cdot \nabla_\omega \left( e^{\frac{i}{h} \Theta^1_{\red{\omega}, t_h}(\wit{x})} 
	\boldsymbol{m}(\omega) \right)  e^{\frac{i}{h} \vec{\theta} \cdot \omega}  d\omega \\
&=   - \frac{i h}{\|\vec{\theta}'\|^2} \int_{[-M,M]^{|J_h|}} \vec{\theta}' 
	\cdot \nabla_\omega \left( e^{\frac{i}{h} \Theta^1_{\red{\omega}, t_h}(\wit{x})} \right) 
	 \boldsymbol{m}(\omega)  e^{\frac{i}{h} \vec{\theta} \cdot \omega}  d\omega \\
& \phantom{=}- \frac{i h}{\|\vec{\theta}'\|^2} \int_{[-M,M]^{|J_h|}} \vec{\theta}' 
	\cdot \nabla_\omega \left(  \boldsymbol{m}(\omega) \right)  e^{\frac{i}{h} \Theta^1_{\red{\omega}, t_h}(\wit{x})}  
		e^{\frac{i}{h} \vec{\theta} \cdot \omega}  d\omega
\\
&=: I_1+I_2.
\end{align*}

Using the Cauchy-Schwarz inequality along with (\ref{eq:BornerTheta1}), we see that 
we have, almost everywhere
\begin{align*}
\left|\vec{\theta}' \cdot \nabla_\omega e^{\frac{i}{h} \Theta^1_{\red{\omega}, t_h}(\wit{x})} \right| 
&\leq \|\vec{\theta}'\| \left( \sum_{j\in J_h'} \left|\partial_{\omega_j}e^{\frac{i}{h} \Theta^1_{\red{\omega}, t_h}(\wit{x})} \right|^2 \right)^{1/2}\\
&\leq C\|\vec{\theta}'\|  h^{-1} \delta^2 h^{-\beta - \varepsilon} h^{-\beta/2} \MI{\sqrt{t_h}},
\end{align*}
from which we deduce
\begin{align*}
|I_1| &\leq C h \|\vec{\theta}'\|^{-1} h^{-1}  \delta^2 h^{-\beta - \varepsilon} h^{-\beta/2} \MI{\sqrt{t_h}}\\
&\leq C \delta h^{-2\beta},
\end{align*}
which is small thanks to (\ref{eq:CondBeta}). 

\textbf{Step 3: An independence argument to deal with $I_2$.}
To deal with $I_2$, we note that formally 
\begin{align*}
\boldsymbol{m}_1(\omega) &:= \vec{\theta}' \cdot \partial \left(  \boldsymbol{m}(\omega) \right)\\
&= \sum_{j\in J'_h} \theta_j  m'(\omega_j) \prod_{i\neq j} m(\omega_i)\\
&= \boldsymbol{m}(\omega)  \sum_{j\in J'_h} \theta_j  \frac{m'(\omega_j)}{m(\omega_j)}.
\end{align*}

Let us write $X_j = \theta_j  \frac{m'(\omega_j)}{m(\omega_j)}$, 
so that $\boldsymbol{m}_1= \boldsymbol{m} \sum_j X_j$.

The family $(X_j)_{j\in J_h'}$ may be seen as a family of independent centred 
almost surely finite random variables on the space $[-M, M]^{J_h'}$ equipped with the 
measure $\boldsymbol{m}(\omega) d\omega$.  Let us show that it is square-integrable. 
To this end, we start by noting that $\frac{(m')^2}{m}$ is bounded 
by $4\|\MI{m}''\|_{\infty}$. Indeed by Taylor expansion we have that 
$0\leq m(x_0 +x) \leq \|m''\|_\infty x^2 + x m'(x_0) + m(x_0)$. The 
discriminant of this quadratic equation must be positive, so  
$|m'(x_0)|^2 \leq 4 \|m''\|_{\infty} m(x_0)$, for any $x_0\in \R$.

We thus have
\begin{align*}
\int_{[-M,M]^{J_h'}} \left|\sum_{j\in J_h'}  X_j \right|^2 \boldsymbol{m}(\omega) d\omega 
=  \int_{[-M,M]^{J_h'}} \sum_{j\in J_h'} \left| X_j \right|^2 \boldsymbol{m}(\omega) d\omega 
\leq C \| \vec{\theta}'\|^2.
\end{align*} 

We therefore deduce from Hölder's inequality that
\begin{align*}
|I_2|
&\leq \frac{h}{\|\vec{\theta}'\|^2} \left( \int_{[-M,M]^{|J_h|}} 
	\left|e^{\frac{i}{h} \Theta^1_{\red{\omega}, t_h}(\wit{x})}  e^{\frac{i}{h} \vec{\theta} 
	\cdot \omega}\right|^2  \boldsymbol{m}(\omega) d\omega\right)^{1/2} \times 
	\left( \int_{[-M,M]^{J_h'}} \left|\sum_{j\in J_h'}  X_j \right|^2 
		\boldsymbol{m}(\omega) d\omega \right)^{1/2}\\
&\leq C \frac{h}{\|\vec{\theta}'\|} \\ 
&\leq C h \delta^{-1} h^{-\beta/2}. 
\end{align*}
The result follows.
\end{proof}
\begin{proof}[Proof of Lemma \ref{lem:LipschitzReg}] \red{In the course of the proof, we will write $t$ instead of $t_h$ to lighten notation.}
By definition,  we have in the sense of distributions
\begin{equation}\label{eq:DecompoPhase}
\partial_{\omega_j} \Theta^1_{\red{\omega}, t}(\wit{x}) = \red{-}  \theta_j - \delta	\red{\int_{[0,t]\setminus  \gI_{k,t}}}  (\partial_{\omega_j} q_\omega) (\zeta_\delta^{s,t}(x)) d s  - \delta	\red{\int_{[0,t]\setminus  \gI_{k,t}}}  ( \nabla q_\omega) \cdot  \partial_{\omega_j} {\zeta_\delta^{s,t}(x)}  d s
\end{equation}

The first two terms in (\ref{eq:DecompoPhase})  give
\begin{align*}
\delta	\red{\int_{[0,t]\setminus  \gI_{k,t}}}  \left[ (\partial_{\omega_j} q_\omega)  (\Phi_0^{-s} (x, d_x \phi_{s,0}))  - (\partial_{\omega_j} q_\omega)  ({\zeta_\delta^{s,t}(x)}) \right]d s, 
\end{align*}
which is a smooth function of $\omega_j$.

Since the trajectories we compare are at a distance 
$O(\delta h^{-\beta - \varepsilon})$ from each 
other (thanks to (\ref{eq:PetitCoupDeStressDeLaFin})), 
the integrand is of the order of $\delta \times O(\delta h^{-2\beta - \varepsilon})$, 
but the integration takes place only on \red{unions of} time intervals 
of \red{total length} $O(t h^{\beta})$ \red{thanks to Remark \ref{Rem:LeavingBall}}, so we get 
$O( \delta^2 h^{-\beta - 2 \varepsilon})$. 
Hence, the first two terms in (\ref{eq:DecompoPhase}) give 
$O( \delta^2 h^{-\beta - 2 \varepsilon})$.

To bound the last term in (\ref{eq:DecompoPhase}), let 
$s\in [0,t]$, and $\lambda, \lambda'\in [-M,M]$. We claim that 
\begin{equation}\label{eq:DistZeta}
\mathrm{dist}_{T^*X} \left(\zeta_{\delta,  \omega_j = \lambda}^{s,t}(x), \zeta_{\delta,  \omega_j = \lambda'}^{s,t}(x)\right) = O(\delta  |\lambda - \lambda'| h^{- \varepsilon}).
\end{equation}
\red{In the previous equation,  we consider that, for the trajectories we compare,  all the $\omega_i$ for $i\neq j$ take the same value, and only $\omega_j$ differs.}
Equation (\ref{eq:DistZeta}) implies that $\omega_j \mapsto \zeta_\delta^{s,t}$ is a Lipschitz function (hence differentiable almost everywhere by Rademacher's theorem) with Lipschitz constant $ O(\delta  h^{-\varepsilon})$.
We thus deduce that the last term in (\ref{eq:DecompoPhase}) is a  $O(\delta^2 h^{-\beta - \varepsilon})$.

To prove (\ref{eq:DistZeta}), we will first proof that, for any $0<\lambda_1<\lambda_2$ and any $\rho\in \mathcal{E}_{0, (\lambda_1, \lambda_2)}$, we have for all  $s\in \R$ with $|s| \leq t \leq \gd |\log h|$ and $\gd$ small enough
 \begin{equation}\label{eq:UnPtiGronwallEtOnEstBons}
 \mathrm{dist}_{T^*X}\left( \Phi^{s}_{\delta, \omega_j = \lambda} (\rho), \Phi^{s}_{\delta, \omega_j = \lambda'}(\rho) \right)= O( \delta |\lambda - \lambda'| h^{-\varepsilon}).
 \end{equation}
Indeed,  \red{Remark \ref{Rem:LeavingBall}} implies that
the set of $s\in [-t,t]$ such that  $\Phi^s_0(\rho)$ belongs to 
$\mathrm{supp}(\red{q_j})$ is included in a 
union of $C t$ intervals  of length $\red{\leq} C h^{\red{\beta}}$, with $C>0$ 
independent of $h$ and $t$.  Since, by Lemma \ref{Lem:Gron} 
and condition \eqref{eq:CondBetaDelta},  $\Phi^s_\delta(\rho)$ 
remains at a distance $o(h^\beta)$  from $\Phi^s_0(\rho)$, we 
deduce that the set of $s\in [-t,t]$ such that both $\Phi^s_0(\rho)$ 
and $\Phi^s_\delta(\rho)$ belong to $\mathrm{supp}(\red{q_j})$ 
is included in a union of $C' t$ intervals  of length $\red{\leq } C' h^{\beta}$, 
with $C'>0$ independent of $h$ and $t$.  Here, $C'$ and the intervals 
can be chosen independent of the value of the random parameter $\omega$.
 
Let us denote these intervals by $[T_j,  T_j + C'h^{\beta}]$, 
with $j\leq C't$, and where $T_j < T_{j+1}$. Thanks to (\ref{eq:Gronwall2}), 
we have \red{ for all $s\in [T_j,  T_j + C'h^{\beta}]$} 
\begin{equation}\label{eq:DistDifferentOmega}
\begin{aligned}
\mathrm{dist}_{T^*X}\left( \Phi^s_{\delta, \omega_j = \lambda} (\rho), \Phi^s_{\delta, \omega_j = \lambda'}  (\rho) \right) &\leq C e^{Ch^{\beta}} \mathrm{dist}_{T^*X}\left( \Phi^{T_j}_{\delta, \omega_j = \lambda} (\rho), \Phi^{T_j}_{\delta, \omega_j = \lambda'}  (\rho) \right) + C \delta  |\lambda - \lambda'| h^{-\beta} \left( e^{C h^{\beta}} -1 \right)\\
&\leq C'' \mathrm{dist}_{T^*X}\left( \Phi^{T_j}_{\delta, \omega_j = \lambda} (\rho), \Phi^{T_j}_{\delta, \omega_j = \lambda'}  (\rho) \right) + C'' \delta |\lambda - \lambda'|.
\end{aligned}
\end{equation} 
On the other hand, (\ref{eq:Gronwall}) implies that, for all $s\in [T_j + C' h^\beta,   T_{j+1}]$, we have 
\begin{equation}\label{eq:DistSameOmega}
\begin{aligned}
\mathrm{dist}_{T^*X}\left( \Phi^s_{\delta, \omega_j = \lambda} (\rho), \Phi^s_{\delta, \omega_j = \lambda'}  (\rho) \right) &\leq C e^{C (s- T_j)}  \mathrm{dist}_{T^*X}\left( \Phi^{T_j + C' h^\beta}_{\delta, \omega_j = \lambda} (\rho), \Phi^{T_j + C' h^\beta}_{\delta, \omega_j = \lambda'}  (\rho) \right).
\end{aligned}
\end{equation}
Applying  (\ref{eq:DistDifferentOmega}) and (\ref{eq:DistSameOmega}) a number of times which remains $O(t)$, we obtain $$
 \mathrm{dist}_{T^*X}\left( \Phi^s_{\delta, \omega_j = \lambda} \left( \rho \right), \Phi^s_{\delta, \omega_j = \lambda'} \left( \rho \right)\right) \leq C \delta |\lambda - \lambda'| e^{Ct},$$
 from which (\ref{eq:UnPtiGronwallEtOnEstBons}) follows by taking $\gd$ small enough.
 
To deduce (\ref{eq:DistZeta}) from (\ref{eq:UnPtiGronwallEtOnEstBons}), we argue similarly to the proof of Lemma \ref{lem:BoundDiffFlows}. We write  $y = y_{\delta, \omega_j = \lambda}^{-s, -(t-s)}(x)$ 
 and $y' =  y_{\delta,\omega_j = \lambda'}^{-s, -(t-s)}(x)$. 
 Thanks to (\ref{eq:UnPtiGronwallEtOnEstBons}) and Lemma \ref{Lem:Gron}, we have 
 \begin{align*}
 &\mathrm{dist}_X\left( \Phi^{s}_{\delta, \omega_j = \lambda'}  \left(\Phi^{t-s}_0 (y, d_y \phi_0)\right),   \Phi^{s}_{\delta, \omega_j = \lambda'}  \left(\Phi^{t-s}_0 (y', d_{y'} \phi_0)\right))\right)\\
 &=
 \mathrm{dist}_X\left( \Phi^{s}_{\delta, \omega_j = \lambda'}  \left(\Phi^{t-s}_0 (y, d_y \phi_0)\right), x)\right) =O( \delta |\lambda - \lambda'| h^{-\varepsilon}).
 \end{align*}
 
Now,  since $\Phi^{s}_{\delta, \omega_j = \lambda'}  \left(\Phi^{t-s}_0 (y, d_y \phi_0)\right)$ 
and $\Phi^{s}_{\delta, \omega_j = \lambda'}  \left(\Phi^{t-s}_0 (y', d_{y'} \phi_0)\right)$ 
both belong to $\Phi^{s}_{\delta, \omega_j = \lambda'}  \left(\Phi^{t-s}_0 (\Lambda_0)\right)$, 
equation (\ref{eq:C2Perturb}) implies that they must be at a distance $O( \delta |\lambda - \lambda'| h^{-\varepsilon})$ 
in $T^*X$.  Using Lemma \ref{Lem:Gron}, we deduce that $\mathrm{dist}_X(y,y') = O( \delta |\lambda - \lambda'| h^{-2\varepsilon})$, 
and then (\ref{eq:DistZeta}), by taking $\varepsilon$ smaller.
\end{proof}

We may now proceed with the proof of Proposition \ref{Prop:RandOmega}. 
\begin{proof}[Proof of Proposition \ref{Prop:RandOmega}]

Recall that thanks to (\ref{eq:EffectivFunction3}), it is sufficient 
to prove the proposition with $\psi_{h,t_h,x,\omega}$ replaced by 
$\psi^0_{h,t_h,x,\omega}$.
We will prove the result in two steps: first we will show that
$\psi_{h,t_h,x,\omega}^0$ converges to the $\mathrm{BGF}$ in \textit{finite dimensional 
distributions}. That is to say that for any $n\in\N$ and all 
$\by_1,\dots, \by_k\in \R^d$ the random vector $(\psi^0_{h, t_h, x, \omega}(\by_1),\dots,
\psi^0_{h, t,_h x, \omega}(\by_k))$ converges in law to 
the random vector 
$(\mathfrak{f}(\boldsymbol{y}_1),\dots,\mathfrak{f}(\boldsymbol{y}_k))$ 
where $\mathfrak{f}$ is a random function following the law of the 
Berry Gaussian field $\mathrm{BGF}_{\lambda_a}$, as in Definition \ref{DefBerry}. 
In other words, we will show that 
\begin{equation}\label{eq:fidiconvergence}
	(\psi^0_{h, t_h, x, \omega}(\by_1),\dots,
	\psi^0_{h, t,_h x, \omega}(\by_k))
	\overset{d}{\longrightarrow}
	(\mathfrak{f}(\boldsymbol{y}_1),\dots,\mathfrak{f}(\boldsymbol{y}_k)), 
	\quad h\to 0. 
\end{equation}
Secondly, we will show that the 
sequence of random functions $\psi^{0}_{h,t_h,x,\omega}\in C^\infty(\R^d)$ 
is \emph{tight} which, by Prokhorov's theorem \cite[Theorem 14.3]{Kal97}, 
is equivalent to relative compactness in distribution. Hence any subsequence 
of the sequence of random functions $\psi_{h,t_h,x,\omega}^0$ has a further 
subsequence which converges in distribution. Its limit must coincide with the 
limit of the convergence in finite dimensional distribution \eqref{eq:fidiconvergence}. 
So we may conclude the statement of the proposition that 
$\psi_{h,t_h,x,\omega}^0$ converges to the $\mathrm{BGF}_{\lambda_a}$ in distribution. 
\\
\par 
1. To prove \eqref{eq:fidiconvergence} we wish to apply a multivariate 
Lindeberg Central Limit Theorem to the sum over $\wit{x}\in A_{x,t}$ of the 
random vectors 
\begin{equation*}
\begin{split}
\eta_{\wit{x}}(t_h,h) 
&= (\eta_{\wit{x}}^1(t_h,h),..., \eta_{\wit{x}}^k(t_h,h))\\
&:=\left(b_0\left(t_h,  \widetilde{x} ; 0\right)  
 \e^{\frac{i}{h} \wit{\phi}^0_{t_h,\delta} (\wit{x})}
  \e^{i  \xi_{t_h, \widetilde{x},0} \cdot \boldsymbol{y}_1} ,\dots,b_0
  \left(t_h,  \widetilde{x} ; 0\right) \e^{\frac{i}{h} \wit{\phi}^0_{t_h,\delta} (\wit{x})} 
   \e^{i  \xi_{t, \widetilde{x},0} \cdot \boldsymbol{y}_k} \right).
\end{split}
\end{equation*}
By construction, the family of random variables $(\eta_{\wit{x}}(t_h,h))_{\wit{x}\in A_{x,t}}$ 
is independent (see Proposition \ref{PropSection8}). Thanks to (\ref{eq:ZeroAverage}), we have 
for each $\wit{x}  \in A_{x,t}$
\begin{equation*}
\E[\eta^\ell_{\wit{x}}(t_h,h)] = O(h^{\Gamma}) ~~\forall \ell \in \{1,..., k\}
\end{equation*}
\begin{equation}\label{eq:Cov1}
\E[\eta^\ell_{\wit{x}}(t_h,h)\overline{\eta_{\wit{x}}^{\ell'}(t_h,h)}]
= 	| b_0\left(t_h,  \widetilde{x} ; 0\right)|^2  \e^{i  \xi_{t_h, \widetilde{x},0} 
\cdot (\boldsymbol{y}_\ell-\boldsymbol{y}_{\ell'})} ~~\forall \ell, \ell' \in \{1,..., k\}.
\end{equation}
Let us denote by $M^{h,t_h}= (m^{h,t_h}_{\ell, \ell'})_{\ell, \ell'}$ the sum of 
the covariance matrices of the vectors $\eta_{\wit{x}}(t_h,h)$. Equation (\ref{eq:Cov1}) 
implies that it depends on $h$ only through $t_h$, as
\begin{equation}\label{eq:Cov1.0}
	m^{h,t_h}_{\ell, \ell'}
	= \sum_{\wit{x} \in A_{x,t_h}}   | b_0\left(t_h,  \widetilde{x} ; 0\right)|^2 
	\e^{i  \xi_{t_h, \widetilde{x},0} \cdot (\boldsymbol{y}_\ell-\boldsymbol{y}_{\ell'})}.
\end{equation}
Now, applying Lemma 3.9\footnote{\red{This lemma essentially relies on the mixing property of the geodesic flow on a manifold of negative curvature.  While it could be proven by introducing suitable distributions in anisotropic Sobolev spaces in the spirit of \cite{FS}, the authors of \cite{IngRiv} prove it in a more pedestrian way by using Egorov's theorem.}} in \cite{IngRiv} we get that \red{for all $x\in X$}
\begin{equation}\label{eq:Cov1.1}
 m^{h,t_h}_{\ell, \ell'} \underset{h \to 0}{\longrightarrow} m_{\ell, \ell'}
 = \frac{\|a_0\|_{L^2}^2}{\Vol(X)} \int_{\mathbb{S}^d} 
 \e^{i\xi \cdot (\boldsymbol{y}_\ell-\boldsymbol{y}_\ell')} d\xi.
\end{equation}
But the matrix $(m_{\ell, \ell'})_{\ell, \ell'}$ is the covariance 
matrix of the random vector 
$(\mathfrak{f}(\boldsymbol{y}_1),\dots,\mathfrak{f}(\boldsymbol{y}_k))$ 
where $\mathfrak{f}$ is a random function following the law of the 
Berry Gaussian field, as in Definition \ref{DefBerry}, with normalization 
constant $\lambda_a= \frac{\|a_0\|^2}{\mathrm{Vol}(X)}$. In particular, 
the matrix $M^{h,t_h}$ is invertible for small enough $h$. 
\par
Lastly, by \MI{Lemma} \ref{cor:DecayAmplitude}, 
\begin{equation*}
	\sup\limits_{\wit{x}\in A_{x,t}}  | b_0\left(t,  \widetilde{x} ; 0\right)| 
\xrightarrow[t\rightarrow 0]{}0.
\end{equation*}
So all assumptions of the multivariate Lindeberg Central Limit Theorem 
hold\footnote{Thanks to the Cramér-Wold Theorem \cite[Theorem 29.4]{Bili}, 
the multivariate Lindeberg Central Limit Theorem follows from the usual Lindeberg 
Central Limit Theorem \cite[Chapter 27]{Bili}}. Thus, as $h\rightarrow 0$, the 
vector $\sum_{\wit{x}\in A_{x,t_h}} \eta_{\wit{x}}(t_h,h)$ converges in law to a 
Gaussian random vector %$(\zeta_1,\dots,\zeta_k)$ 
with covariance $(m_{\ell, \ell'})_{\ell, \ell'}$ and expectation $0$. This 
concludes the proof of \eqref{eq:fidiconvergence}.
\\
\par
2. It remains to prove \emph{tightness}. The sequence of random smooth functions 
$\psi^0_{h,t_h,x,\omega}\in C^\infty(\R^d)$ is said to be tight if 
\begin{equation}\label{eq:tightness}
	\sup_{\mathcal{K}} \liminf_{h\to 0} \prob[\psi^{0}_{h,t_h,x,\omega}\in \mathcal{K}] = 1, 
\end{equation}
where the supremum is taken over all compact subset 
$\mathcal{K}\Subset C^\infty(\R^d)$. 
\par 
Let $\boldsymbol{a} = (a_{k,\ell})_{k,\ell \in \N^2}$ be a sequence of positive 
real numbers depending on two parameters and consider the set 
\begin{equation}\label{eq:CompSet}
\mathcal{K}(\boldsymbol{a}):= \{f\in C^\infty(\R^d) ~|~\forall \red{k}, \ell \in \N,  
\|f\|_{C^\ell (\overline{B(0,\red{k})})} \leq a_{k,\ell}\}.
\end{equation}
It follows from the Arzela-Ascoli theorem that $\mathcal{K}(\boldsymbol{a})$ 
is a compact subset of $C^\infty(\R^d)$ for the topology of convergence of 
all derivatives over all compact sets. In view of \eqref{estim:xi}, we deduce from 
\eqref{eq:Cov1}, \eqref{eq:Cov1.0} and \eqref{eq:Cov1.1}, that, 
for any $k,n\in \N$
\begin{equation*}
	\erw_\omega [\|\psi_{h,t_h,x,\omega}^0\|^2_{H^k(\overline{B(0,n)}})] =O_{k,n}(1)
\end{equation*}
uniformly in $h\in ]0,h_0]$. By the Sobolev embeddings, we then conclude that, for any $\ell,n\in \N$
\begin{equation*}
	\erw_\omega [\|\psi^{0}_{h,t_h,x,\omega}\|_{C^{\ell}(\overline{B(0,n)}})] \leq C_{\ell,n}
\end{equation*}
for some constant $C_{\ell,n}>0$, uniformly in $h\in ]0,h_0]$. By the Markov 
inequality we find that for any $\varepsilon >0$
\begin{equation*}
	\prob_\omega[\|\psi^{0}_{h,t_h,x,\omega}\|_{C^{\ell}(\overline{B(0,n)})}
	>C_{\ell,n} 2^{\ell+n}\varepsilon^{-1}] 
	\leq \varepsilon 2^{-(\ell+n)}.
\end{equation*}
Now for $\varepsilon>0$ put $a_{\ell,n}:= C_{\MI{\ell,n}} 2^{\ell+n}\varepsilon^{-1}$ for 
$\ell,n \in\N$. Then, 
\begin{equation*}
\begin{split}
	\sup_{\mathcal{K}} \liminf_{h\to 0} \prob[\psi^{0}_{h,t_h,x,\omega}\in \mathcal{K}] 
	& \geq \liminf_{h\to 0} \prob[\psi^{0}_{h,t_h,x,\omega}\in \mathcal{K}(\boldsymbol{a})] \\
	&= 1 - \limsup_{h\to 0} \prob[\psi^{0}_{h,t_h,x,\omega}\in \mathcal{K}(\boldsymbol{a})^c]\\
	&\geq 1 - \limsup_{h\to 0} \sum_{\ell,n \in \N}\prob_\omega[\|\psi_{h,t_h,x,\omega}\|_{C^{\ell}(\overline{B(0,n)})}
	> C_{\ell,n} 2^{\ell+n}\varepsilon^{-1}] \\
	&\geq 1 - \varepsilon. 
\end{split}
\end{equation*}
Since $\varepsilon>0$ can be chosen arbitrarily, we conclude \eqref{eq:tightness}. 
This concludes the proof of Proposition \ref{Prop:RandOmega}. 
\end{proof}
\subsection{Picking $x$ at random}
Our aim in this section is to prove Theorem \ref{th:MartinEtMaximeSontDesBeauxGosses}. 
We thus fix  a family of times $t_h$ with  $t_h\to \infty$, $t_h=o(|\log h|)$.
 
Recall that we want to study the statistics of the random function 
$\psi_{h, t_h, x, \omega}(\boldsymbol{y})$ when $x$ is chosen uniformly 
at random in $\mathcal{U}\subset X$. To this end, we partitioned $X$ as
\begin{equation*}
	X = \left( \bigsqcup_{i\in \wih{I}_h} \wih{X}_i\right) \sqcup X_h^0,
\end{equation*}
where each $\wih{X}_i$ has a diameter $\leq C h^{\beta - \varepsilon_1}$ 
and has volume $\geq c h^{d(\beta - \varepsilon_1)}$ and is contained either 
inside $\mathcal{U}$ or in its complement, and where 
$ \Vol(X_h^0) = o_{h\to 0}(1)$. Let us denote by $\wih{I}_h'$ the set of 
indices $i\in \wih{I}_h$ such that $\wih{X}_i \subset \mathcal{U}$.
If $F$ is a continuous bounded functional over $C^\infty(\R^d)$, we find by 
\eqref{eq:EffectivFunction3} that 
\begin{equation}\label{eq:Decoupage1}
	\frac{1}{\Vol(\mathcal{U})}\int_{\mathcal{U}} F(\psi_{h, t_h, x, \omega}) dx 
	= 
	\frac{1}{\Vol(\mathcal{U})}\int_{\mathcal{U}} F(\psi^0_{h, t_h, x, \omega}) dx 
	+
	o_{h\to 0}(1).
\end{equation}
Further, since $F$ is bounded and $ \Vol(X_h^0) = o_{h\to 0}(1)$, we find 
that $\omega$-almost surely  
\begin{equation}\label{eq:Decoupage}
 \frac{1}{\Vol(\mathcal{U})}\int_{{\mathcal{U}}} F(\psi_{h, t_h, x, \omega}) dx 
 = \sum_{i\in \wih{I}'_h} \left( \frac{\Vol(\wih{X}_i)}{\Vol({\mathcal{U}})} 
 \times \frac{1}{\Vol(\wih{X}_i)} 
 \int_{\wih{X}_i} F(\psi^0_{h, t_h, x, \omega}) dx \right) +
  o_{h\to 0}(1).
 \end{equation}
Thanks to Proposition \ref{Prop:RandOmega} and to the 
dominated convergence theorem, we know that 
\begin{equation}\label{eq:Aprox1}
	\E_\omega \left[  \frac{1}{\Vol(\mathcal{U})}\int_{\mathcal{U}} 
	F(\psi_{h, t_h, x, \omega}) dx \right]
	 =  \E_{\mathrm{BGF}_{\lambda_a}}[F] + o_{h\to 0}(1).
\end{equation}
So $\omega$-almost surely  
\begin{equation}\label{eq:Decoupage2}
\begin{split}
	\bigg|\frac{1}{\Vol(\mathcal{U})}\int_{\mathcal{U}} F(\psi_{h, t_h, x, \omega}) dx 
	&- \E_{\mathrm{BGF}_{\lambda_a}}[F] \bigg| \\
	&= 
	\left|\frac{1}{\Vol(\mathcal{U})}\int_{\mathcal{U}} F(\psi_{h, t_h, x, \omega}) dx 
	- \E_\omega  \frac{1}{\Vol(\mathcal{U})}\int_{\mathcal{U}}
	F(\psi_{h, t_h, x, \omega}) dx  \right| 
	+
	o_{h\to 0}(1).
\end{split}
\end{equation}
By \eqref{eq:Decoupage} and the dominated convergence theorem we find that 
$\omega$-almost surely  
\begin{equation}\label{eq:Decoupage3}
	\begin{split}
		&\bigg|\frac{1}{\Vol(\mathcal{U})}\int_{\mathcal{U}} F(\psi_{h, t_h, x, \omega}) dx 
		- \E_{\mathrm{BGF}_{\lambda_a}}[F] \bigg| \\
		&\leq 
		\sum_{i\in {\wih{I}'_h}}\frac{\Vol({\wih{X}_i)}}{\Vol(\mathcal{U})} 
		\left|\frac{1}{\Vol(\wih{X}_i)} 
		\int_{ {\wih{X}_i}} F(\psi^0_{h, t_h, x, \omega}) dx 
		- \E_\omega  \frac{1}{\Vol( {\wih{X}_i})} 
		\int_{ {\wih{X}_i}}
		F(\psi^0_{h, t_h, x, \omega}) dx  \right| \\
		&+
		o_{h\to 0}(1).
	\end{split}
\end{equation}
%
% %
We will show below that there exists an $\varepsilon_2>0$ such that 
for all $i\in  {\wih{I}'_h}$
\begin{equation}\label{eq:PointwiseLimit_n1}
	\prob_\omega\left[\left|\frac{1}{\Vol( {\wih{X}_i)}} 
	\int_{ {\wih{X}_i}} F(\psi^0_{h, t_h, x, \omega}) dx 
	- \E_\omega  \frac{1}{\Vol( {\wih{X}_i)}} 
	\int_{ {\wih{X}_i}}
	F(\psi^0_{h, t_h, x, \omega}) dx  \right|\geq h^{\varepsilon_2} \right] = O(h^\infty).
\end{equation}
Assuming \eqref{eq:PointwiseLimit_n1}, and using the fact that 
$| {\wih{I}'_h}|$ only grows polynomially in $h^{-1}$, as well as 
\begin{equation*}
	 \sum_{i\in  {\wih{I}_h'}} 
	 \frac{\Vol( {\wih{X}_i)}}{\Vol(\mathcal{U})}
	=1 + o_{h\to 0}(1),
\end{equation*}
we conclude from \eqref{eq:Decoupage3} that every $\eta>0$  
\begin{equation}\label{eq:PointwiseLimit_n2}
	\prob_\omega\left[\left|\frac{1}{\Vol(\mathcal{U})} \int_{\mathcal{U}}  
	F(\psi_{h, t_h, x, \omega}) dx -
	\E_{\mathrm{BGF}_{\lambda_a}}[F]\right|\geq \eta \right] = O(h^\infty),
\end{equation}
for $h>0$ small enough. 
This proves Theorem \ref{th:MartinEtMaximeSontDesBeauxGosses}, provided 
that we have shown \eqref{eq:PointwiseLimit_n1} which will be our task for 
the rest of this paper. %
\\
\\
\textbf{Picking $x$ at random locally}

Our aim is now to prove \eqref{eq:PointwiseLimit_n1}. Let 
$i\in \wih{I}'_h$ and let $x_0\in \wih{X}_i$. We may partition 
$\wih{X}_i$ further, as 
\begin{equation*}
	\wih{X}_i = \bigsqcup_{j\in J_{h,i}} X_{i,j},
\end{equation*}
with $|J_{h,i}|$ of the order of $h^{- d \frac{\varepsilon_1}{2}}$, 
each $X_{j,i}$ having diameter at most $c h^{\beta - \frac{\varepsilon_1}{2}}$.
\par 
Let $F$ be a continuous bounded functional over $C^\infty(\R^d)$. We may write
\begin{equation*} 
\frac{1}{\mathrm{Vol} (\wih{X}_i)} \int_{\wih{X}_i} F(\psi^{0}_{h, t_h, x, \omega}) dx 
= \frac{1}{\mathrm{Vol} (\wih{X}_i)}  
	\sum_{j\in J_{h,i}} \int_{X_{i,j}} F(\psi^{0}_{h, t_h, x, \omega}) dx.
\end{equation*}
Let us define 
\begin{equation*}
	Z := \frac{1}{\mathrm{Vol} (\wih{X}_i)} \int_{\wih{X}_i} 
F(\psi^{0}_{h, t_h, x, \omega})d x, \quad \text{and} \quad 
Z_j := \frac{1}{\mathrm{Vol} (\wih{X}_{i})} 
\int_{X_{i,j}} F(\psi^{0}_{h, t_h, x, \omega}) d x,
\end{equation*}
which are random variable depending on $\omega$. We have 
\begin{align*}
\E_\omega(Z_j) = \frac{1}{\mathrm{Vol} (\wih{X}_{i})} \int_{X_{i,j}} 
 \E_\omega\left[ F(\psi^0_{h, t_h, x, \omega})\right] dx.
\end{align*}

% }
Next, we want to show that $Z_j$ is independent from most $Z_{j'}$. 

\begin{lem}
There exists \red{$0<\gamma< \frac{d \varepsilon_1}{2}$} such that the following holds. 
For each $j\in J_{h,i}$, there exists $\mathcal{C}_j\subset J_{h,i}$ 
with 
\begin{equation*}
	\mathrm{Card}\left(\mathcal{C}_j\right) \red{= O\left(  h^{-d\frac{\varepsilon_1}{2} + \gamma} \right)}
\end{equation*}
such that $Z_j$ is independent from $\{ Z_{j'}\}_{j'\in J_{h ,i}\setminus \mathcal{C}_j}$.
\end{lem}

\begin{figure}
\begin{center}
\definecolor{qqqqff}{rgb}{0,0,1}
\definecolor{qqzzff}{rgb}{0,0.6,1}
\definecolor{ffqqqq}{rgb}{1,0,0}
\definecolor{ffzzqq}{rgb}{0.1,1,0.1}
\begin{tikzpicture}[line cap=round,line join=round,>=triangle 45,x=1cm,y=1cm]
\clip(-7.608489309033779,-7.866899689780639) rectangle (9.577524390060917,7.8127701791145);
\fill[line width=3pt,color=ffzzqq, fill opacity =0.0] (-4,-4) -- (4,-4) -- (4,4) -- (-4,4) -- cycle;
\fill[line width=1pt,color=ffqqqq,fill=ffqqqq,fill opacity=0.11] (-4,-4) -- (-3,-4) -- (-3,-3) -- (-4,-3) -- cycle;
\fill[line width=1pt,color=ffqqqq,fill=ffqqqq,fill opacity=0.11] (-3,-4) -- (-2,-4) -- (-2,-3) -- (-3,-3) -- cycle;
\fill[line width=1pt,color=ffqqqq,fill=ffqqqq,fill opacity=0.11] (-2,-4) -- (-1,-4) -- (-1,-3) -- (-2,-3) -- cycle;
\fill[line width=1pt,color=ffqqqq,fill=ffqqqq,fill opacity=0.11] (-1,-4) -- (0,-4) -- (0,-3) -- (-1,-3) -- cycle;
\fill[line width=1pt,color=ffqqqq,fill=ffqqqq,fill opacity=0.11] (-4,-3) -- (-3,-3) -- (-3,-2) -- (-4,-2) -- cycle;
\fill[line width=1pt,color=ffqqqq,fill=ffqqqq,fill opacity=0.11] (-4,-2) -- (-3,-2) -- (-3,-1) -- (-4,-1) -- cycle;
\fill[line width=1pt,color=ffqqqq,fill=ffqqqq,fill opacity=0.11] (-4,-1) -- (-3,-1) -- (-3,0) -- (-4,0) -- cycle;
\fill[line width=1pt,color=ffqqqq,fill=ffqqqq,fill opacity=0.11] (-4,0) -- (-3,0) -- (-3,1) -- (-4,1) -- cycle;
\fill[line width=1pt,color=ffqqqq,fill=ffqqqq,fill opacity=0.11] (-4,1) -- (-3,1) -- (-3,2) -- (-4,2) -- cycle;
\fill[line width=2pt,color=ffqqqq,fill=ffqqqq,fill opacity=0.11] (-4,2) -- (-3,2) -- (-3,3) -- (-4,3) -- cycle;
\fill[line width=2pt,color=ffqqqq,fill=ffqqqq,fill opacity=0.11] (-4,3) -- (-3,3) -- (-3,4) -- (-4,4) -- cycle;
\fill[line width=2pt,color=ffqqqq,fill=ffqqqq,fill opacity=0.11] (-3,3) -- (-2,3) -- (-2,4) -- (-3,4) -- cycle;
\fill[line width=2pt,color=ffqqqq,fill=ffqqqq,fill opacity=0.11] (-2,3) -- (-1,3) -- (-1,4) -- (-2,4) -- cycle;
\fill[line width=2pt,color=ffqqqq,fill=ffqqqq,fill opacity=0.11] (0,-4) -- (1,-4) -- (1,-3) -- (0,-3) -- cycle;
\fill[line width=2pt,color=ffqqqq,fill=ffqqqq,fill opacity=0.11] (1,-4) -- (2,-4) -- (2,-3) -- (1,-3) -- cycle;
\fill[line width=2pt,color=ffqqqq,fill=ffqqqq,fill opacity=0.11] (2,-4) -- (3,-4) -- (3,-3) -- (2,-3) -- cycle;
\fill[line width=2pt,color=ffqqqq,fill=ffqqqq,fill opacity=0.11] (3,-4) -- (4,-4) -- (4,-3) -- (3,-3) -- cycle;
\fill[line width=2pt,color=ffqqqq,fill=ffqqqq,fill opacity=0.11] (-3,-3) -- (-2,-3) -- (-2,-2) -- (-3,-2) -- cycle;
\fill[line width=2pt,color=ffqqqq,fill=ffqqqq,fill opacity=0.11] (-2,-3) -- (-1,-3) -- (-1,-2) -- (-2,-2) -- cycle;
\fill[line width=2pt,color=ffqqqq,fill=ffqqqq,fill opacity=0.11] (-1,-3) -- (0,-3) -- (0,-2) -- (-1,-2) -- cycle;
\fill[line width=2pt,color=ffqqqq,fill=ffqqqq,fill opacity=0.11] (-3,-2) -- (-2,-2) -- (-2,-1) -- (-3,-1) -- cycle;
\fill[line width=2pt,color=ffqqqq,fill=ffqqqq,fill opacity=0.11] (-3,-1) -- (-2,-1) -- (-2,0) -- (-3,0) -- cycle;
\fill[line width=2pt,color=ffqqqq,fill=ffqqqq,fill opacity=0.11] (-3,0) -- (-2,0) -- (-2,1) -- (-3,1) -- cycle;
\fill[line width=2pt,color=ffqqqq,fill=ffqqqq,fill opacity=0.11] (-3,1) -- (-2,1) -- (-2,2) -- (-3,2) -- cycle;
\fill[line width=2pt,color=ffqqqq,fill=ffqqqq,fill opacity=0.11] (-3,2) -- (-2,2) -- (-2,3) -- (-3,3) -- cycle;
\fill[line width=2pt,color=ffqqqq,fill=ffqqqq,fill opacity=0.11] (0,-3) -- (1,-3) -- (1,-2) -- (0,-2) -- cycle;
\fill[line width=2pt,color=ffqqqq,fill=ffqqqq,fill opacity=0.11] (1,-3) -- (2,-3) -- (2,-2) -- (1,-2) -- cycle;
\fill[line width=2pt,color=ffqqqq,fill=ffqqqq,fill opacity=0.11] (2,-3) -- (3,-3) -- (3,-2) -- (2,-2) -- cycle;
\fill[line width=2pt,color=ffqqqq,fill=ffqqqq,fill opacity=0.11] (3,-3) -- (4,-3) -- (4,-2) -- (3,-2) -- cycle;
\fill[line width=2pt,color=ffqqqq,fill=ffqqqq,fill opacity=0.11] (-2,-2) -- (-1,-2) -- (-1,-1) -- (-2,-1) -- cycle;
\fill[line width=2pt,color=ffqqqq,fill=ffqqqq,fill opacity=0.11] (-1,-2) -- (0,-2) -- (0,-1) -- (-1,-1) -- cycle;
\fill[line width=2pt,color=ffqqqq,fill=ffqqqq,fill opacity=0.11] (0,-2) -- (1,-2) -- (1,-1) -- (0,-1) -- cycle;
\fill[line width=2pt,color=ffqqqq,fill=ffqqqq,fill opacity=0.11] (1,-2) -- (2,-2) -- (2,-1) -- (1,-1) -- cycle;
\fill[line width=2pt,color=ffqqqq,fill=ffqqqq,fill opacity=0.11] (2,-2) -- (3,-2) -- (3,-1) -- (2,-1) -- cycle;
\fill[line width=2pt,color=ffqqqq,fill=ffqqqq,fill opacity=0.11] (3,-2) -- (4,-2) -- (4,-1) -- (3,-1) -- cycle;
\fill[line width=2pt,color=ffqqqq,fill=ffqqqq,fill opacity=0.11] (-2,-1) -- (-1,-1) -- (-1,0) -- (-2,0) -- cycle;
\fill[line width=2pt,color=ffqqqq,fill=ffqqqq,fill opacity=0.11] (-1,-1) -- (0,-1) -- (0,0) -- (-1,0) -- cycle;
\fill[line width=2pt,color=ffqqqq,fill=ffqqqq,fill opacity=0.11] (0,-1) -- (1,-1) -- (1,0) -- (0,0) -- cycle;
\fill[line width=2pt,color=ffqqqq,fill=ffqqqq,fill opacity=0.11] (1,-1) -- (2,-1) -- (2,0) -- (1,0) -- cycle;
\fill[line width=2pt,color=ffqqqq,fill=ffqqqq,fill opacity=0.11] (2,-1) -- (3,-1) -- (3,0) -- (2,0) -- cycle;
\fill[line width=2pt,color=ffqqqq,fill=ffqqqq,fill opacity=0.11] (3,-1) -- (4,-1) -- (4,0) -- (3,0) -- cycle;
\fill[line width=2pt,color=ffqqqq,fill=ffqqqq,fill opacity=0.11] (-2,0) -- (-1,0) -- (-1,1) -- (-2,1) -- cycle;
\fill[line width=2pt,color=ffqqqq,fill=ffqqqq,fill opacity=0.11] (-1,0) -- (0,0) -- (0,1) -- (-1,1) -- cycle;
\fill[line width=2pt,color=ffqqqq,fill=ffqqqq,fill opacity=0.11] (0,0) -- (1,0) -- (1,1) -- (0,1) -- cycle;
\fill[line width=2pt,color=ffqqqq,fill=ffqqqq,fill opacity=0.11] (-2,1) -- (-1,1) -- (-1,2) -- (-2,2) -- cycle;
\fill[line width=2pt,color=ffqqqq,fill=ffqqqq,fill opacity=0.11] (-1,1) -- (0,1) -- (0,2) -- (-1,2) -- cycle;
\fill[line width=2pt,color=ffqqqq,fill=ffqqqq,fill opacity=0.11] (0,1) -- (1,1) -- (1,2) -- (0,2) -- cycle;
\fill[line width=2pt,color=ffqqqq,fill=ffqqqq,fill opacity=0.11] (1,1) -- (2,1) -- (2,2) -- (1,2) -- cycle;
\fill[line width=2pt,color=ffqqqq,fill=ffqqqq,fill opacity=0.11] (2,1) -- (3,1) -- (3,2) -- (2,2) -- cycle;
\fill[line width=2pt,color=ffqqqq,fill=ffqqqq,fill opacity=0.11] (3,1) -- (4,1) -- (4,2) -- (3,2) -- cycle;
\fill[line width=2pt,color=ffqqqq,fill=ffqqqq,fill opacity=0.11] (1,0) -- (2,0) -- (2,1) -- (1,1) -- cycle;
\fill[line width=2pt,color=ffqqqq,fill=ffqqqq,fill opacity=0.11] (2,0) -- (3,0) -- (3,1) -- (2,1) -- cycle;
\fill[line width=2pt,color=ffqqqq,fill=ffqqqq,fill opacity=0.11] (3,0) -- (4,0) -- (4,1) -- (3,1) -- cycle;
\fill[line width=2pt,color=ffqqqq,fill=ffqqqq,fill opacity=0.11] (3,2) -- (4,2) -- (4,3) -- (3,3) -- cycle;
\fill[line width=2pt,color=ffqqqq,fill=ffqqqq,fill opacity=0.11] (-2,2) -- (-1,2) -- (-1,3) -- (-2,3) -- cycle;
\fill[line width=2pt,color=ffqqqq,fill=ffqqqq,fill opacity=0.11] (-1,2) -- (0,2) -- (0,3) -- (-1,3) -- cycle;
\fill[line width=2pt,color=ffqqqq,fill=ffqqqq,fill opacity=0.11] (-1,3) -- (0,3) -- (0,4) -- (-1,4) -- cycle;
\fill[line width=2pt,color=ffqqqq,fill=ffqqqq,fill opacity=0.11] (0,2) -- (1,2) -- (1,3) -- (0,3) -- cycle;
\fill[line width=2pt,color=ffqqqq,fill=ffqqqq,fill opacity=0.11] (0,3) -- (1,3) -- (1,4) -- (0,4) -- cycle;
\fill[line width=2pt,color=ffqqqq,fill=ffqqqq,fill opacity=0.11] (1,2) -- (2,2) -- (2,3) -- (1,3) -- cycle;
\fill[line width=2pt,color=ffqqqq,fill=ffqqqq,fill opacity=0.11] (1,3) -- (2,3) -- (2,4) -- (1,4) -- cycle;
\fill[line width=2pt,color=ffqqqq,fill=ffqqqq,fill opacity=0.11] (2,2) -- (3,2) -- (3,3) -- (2,3) -- cycle;
\fill[line width=2pt,color=ffqqqq,fill=ffqqqq,fill opacity=0.11] (2,3) -- (3,3) -- (3,4) -- (2,4) -- cycle;
\fill[line width=2pt,color=ffqqqq,fill=ffqqqq,fill opacity=0.11] (3,3) -- (4,3) -- (4,4) -- (3,4) -- cycle;
\fill[line width=2pt,color=qqzzff,fill=qqzzff,fill opacity=0.1] (-9,-7) -- (-9,-5) -- (9,-1) -- (9,-3) -- cycle;
\fill[line width=2pt,color=qqzzff,fill=qqzzff,fill opacity=0.14] (-9,0) -- (-9,-2) -- (9,0) -- (9,2) -- cycle;
\fill[line width=2pt,color=qqzzff,fill=qqzzff,fill opacity=0.16] (-9,4) -- (-9,2) -- (9,3) -- (9,5) -- cycle;
\draw [line width=3pt,color=ffzzqq] (-4,-4)-- (4,-4);
\draw [line width=3pt,color=ffzzqq] (4,-4)-- (4,4);
\draw [line width=3pt,color=ffzzqq] (4,4)-- (-4,4);
\draw [line width=3pt,color=ffzzqq] (-4,4)-- (-4,-4);
\draw [line width=1pt,color=ffqqqq] (-4,-4)-- (-3,-4);
\draw [line width=1pt,color=ffqqqq] (-3,-4)-- (-3,-3);
\draw [line width=1pt,color=ffqqqq] (-3,-3)-- (-4,-3);
\draw [line width=1pt,color=ffqqqq] (-4,-3)-- (-4,-4);
\draw [line width=1pt,color=ffqqqq] (-3,-4)-- (-2,-4);
\draw [line width=1pt,color=ffqqqq] (-2,-4)-- (-2,-3);
\draw [line width=1pt,color=ffqqqq] (-2,-3)-- (-3,-3);
\draw [line width=1pt,color=ffqqqq] (-3,-3)-- (-3,-4);
\draw [line width=1pt,color=ffqqqq] (-2,-4)-- (-1,-4);
\draw [line width=1pt,color=ffqqqq] (-1,-4)-- (-1,-3);
\draw [line width=1pt,color=ffqqqq] (-1,-3)-- (-2,-3);
\draw [line width=1pt,color=ffqqqq] (-2,-3)-- (-2,-4);
\draw [line width=1pt,color=ffqqqq] (-1,-4)-- (0,-4);
\draw [line width=1pt,color=ffqqqq] (0,-4)-- (0,-3);
\draw [line width=1pt,color=ffqqqq] (0,-3)-- (-1,-3);
\draw [line width=1pt,color=ffqqqq] (-1,-3)-- (-1,-4);
\draw [line width=1pt,color=ffqqqq] (-4,-3)-- (-3,-3);
\draw [line width=1pt,color=ffqqqq] (-3,-3)-- (-3,-2);
\draw [line width=1pt,color=ffqqqq] (-3,-2)-- (-4,-2);
\draw [line width=1pt,color=ffqqqq] (-4,-2)-- (-4,-3);
\draw [line width=1pt,color=ffqqqq] (-4,-2)-- (-3,-2);
\draw [line width=1pt,color=ffqqqq] (-3,-2)-- (-3,-1);
\draw [line width=1pt,color=ffqqqq] (-3,-1)-- (-4,-1);
\draw [line width=1pt,color=ffqqqq] (-4,-1)-- (-4,-2);
\draw [line width=1pt,color=ffqqqq] (-4,-1)-- (-3,-1);
\draw [line width=1pt,color=ffqqqq] (-3,-1)-- (-3,0);
\draw [line width=1pt,color=ffqqqq] (-3,0)-- (-4,0);
\draw[line width=1pt,color=ffqqqq](-4,0)-- (-4,-1);
\draw[line width=1pt,color=ffqqqq](-4,0)-- (-3,0);
\draw[line width=1pt,color=ffqqqq](-3,0)-- (-3,1);
\draw[line width=1pt,color=ffqqqq](-3,1)-- (-4,1);
\draw[line width=1pt,color=ffqqqq](-4,1)-- (-4,0);
\draw[line width=1pt,color=ffqqqq](-4,1)-- (-3,1);
\draw[line width=1pt,color=ffqqqq](-3,1)-- (-3,2);
\draw[line width=1pt,color=ffqqqq](-3,2)-- (-4,2);
\draw[line width=1pt,color=ffqqqq](-4,2)-- (-4,1);
\draw[line width=1pt,color=ffqqqq](-4,2)-- (-3,2);
\draw[line width=1pt,color=ffqqqq](-3,2)-- (-3,3);
\draw[line width=1pt,color=ffqqqq](-3,3)-- (-4,3);
\draw[line width=1pt,color=ffqqqq](-4,3)-- (-4,2);
\draw[line width=1pt,color=ffqqqq](-4,3)-- (-3,3);
\draw[line width=1pt,color=ffqqqq](-3,3)-- (-3,4);
\draw[line width=1pt,color=ffqqqq](-3,4)-- (-4,4);
\draw[line width=1pt,color=ffqqqq](-4,4)-- (-4,3);
\draw[line width=1pt,color=ffqqqq](-3,3)-- (-2,3);
\draw[line width=1pt,color=ffqqqq](-2,3)-- (-2,4);
\draw[line width=1pt,color=ffqqqq](-2,4)-- (-3,4);
\draw[line width=1pt,color=ffqqqq](-3,4)-- (-3,3);
\draw[line width=1pt,color=ffqqqq](-2,3)-- (-1,3);
\draw[line width=1pt,color=ffqqqq](-1,3)-- (-1,4);
\draw[line width=1pt,color=ffqqqq](-1,4)-- (-2,4);
\draw[line width=1pt,color=ffqqqq](-2,4)-- (-2,3);
\draw[line width=1pt,color=ffqqqq](0,-4)-- (1,-4);
\draw[line width=1pt,color=ffqqqq](1,-4)-- (1,-3);
\draw[line width=1pt,color=ffqqqq](1,-3)-- (0,-3);
\draw[line width=1pt,color=ffqqqq](0,-3)-- (0,-4);
\draw[line width=1pt,color=ffqqqq](1,-4)-- (2,-4);
\draw[line width=1pt,color=ffqqqq](2,-4)-- (2,-3);
\draw[line width=1pt,color=ffqqqq](2,-3)-- (1,-3);
\draw[line width=1pt,color=ffqqqq](1,-3)-- (1,-4);
\draw[line width=1pt,color=ffqqqq](2,-4)-- (3,-4);
\draw[line width=1pt,color=ffqqqq](3,-4)-- (3,-3);
\draw[line width=1pt,color=ffqqqq](3,-3)-- (2,-3);
\draw[line width=1pt,color=ffqqqq](2,-3)-- (2,-4);
\draw[line width=1pt,color=ffqqqq](3,-4)-- (4,-4);
\draw[line width=1pt,color=ffqqqq](4,-4)-- (4,-3);
\draw[line width=1pt,color=ffqqqq](4,-3)-- (3,-3);
\draw[line width=1pt,color=ffqqqq](3,-3)-- (3,-4);
\draw[line width=1pt,color=ffqqqq](-3,-3)-- (-2,-3);
\draw[line width=1pt,color=ffqqqq](-2,-3)-- (-2,-2);
\draw[line width=1pt,color=ffqqqq](-2,-2)-- (-3,-2);
\draw[line width=1pt,color=ffqqqq](-3,-2)-- (-3,-3);
\draw[line width=1pt,color=ffqqqq](-2,-3)-- (-1,-3);
\draw[line width=1pt,color=ffqqqq](-1,-3)-- (-1,-2);
\draw[line width=1pt,color=ffqqqq](-1,-2)-- (-2,-2);
\draw[line width=1pt,color=ffqqqq](-2,-2)-- (-2,-3);
\draw[line width=1pt,color=ffqqqq](-1,-3)-- (0,-3);
\draw[line width=1pt,color=ffqqqq](0,-3)-- (0,-2);
\draw[line width=1pt,color=ffqqqq](0,-2)-- (-1,-2);
\draw[line width=1pt,color=ffqqqq](-1,-2)-- (-1,-3);
\draw[line width=1pt,color=ffqqqq](-3,-2)-- (-2,-2);
\draw[line width=1pt,color=ffqqqq](-2,-2)-- (-2,-1);
\draw[line width=1pt,color=ffqqqq](-2,-1)-- (-3,-1);
\draw[line width=1pt,color=ffqqqq](-3,-1)-- (-3,-2);
\draw[line width=1pt,color=ffqqqq](-3,-1)-- (-2,-1);
\draw[line width=1pt,color=ffqqqq](-2,-1)-- (-2,0);
\draw[line width=1pt,color=ffqqqq](-2,0)-- (-3,0);
\draw[line width=1pt,color=ffqqqq](-3,0)-- (-3,-1);
\draw[line width=1pt,color=ffqqqq](-3,0)-- (-2,0);
\draw[line width=1pt,color=ffqqqq](-2,0)-- (-2,1);
\draw[line width=1pt,color=ffqqqq](-2,1)-- (-3,1);
\draw[line width=1pt,color=ffqqqq](-3,1)-- (-3,0);
\draw[line width=1pt,color=ffqqqq](-3,1)-- (-2,1);
\draw[line width=1pt,color=ffqqqq](-2,1)-- (-2,2);
\draw[line width=1pt,color=ffqqqq](-2,2)-- (-3,2);
\draw[line width=1pt,color=ffqqqq](-3,2)-- (-3,1);
\draw[line width=1pt,color=ffqqqq](-3,2)-- (-2,2);
\draw[line width=1pt,color=ffqqqq](-2,2)-- (-2,3);
\draw[line width=1pt,color=ffqqqq](-2,3)-- (-3,3);
\draw[line width=1pt,color=ffqqqq](-3,3)-- (-3,2);
\draw[line width=1pt,color=ffqqqq](0,-3)-- (1,-3);
\draw[line width=1pt,color=ffqqqq](1,-3)-- (1,-2);
\draw[line width=1pt,color=ffqqqq](1,-2)-- (0,-2);
\draw[line width=1pt,color=ffqqqq](0,-2)-- (0,-3);
\draw[line width=1pt,color=ffqqqq](1,-3)-- (2,-3);
\draw[line width=1pt,color=ffqqqq](2,-3)-- (2,-2);
\draw[line width=1pt,color=ffqqqq](2,-2)-- (1,-2);
\draw[line width=1pt,color=ffqqqq](1,-2)-- (1,-3);
\draw[line width=1pt,color=ffqqqq](2,-3)-- (3,-3);
\draw[line width=1pt,color=ffqqqq](3,-3)-- (3,-2);
\draw[line width=1pt,color=ffqqqq](3,-2)-- (2,-2);
\draw[line width=1pt,color=ffqqqq](2,-2)-- (2,-3);
\draw[line width=1pt,color=ffqqqq](3,-3)-- (4,-3);
\draw[line width=1pt,color=ffqqqq](4,-3)-- (4,-2);
\draw[line width=1pt,color=ffqqqq](4,-2)-- (3,-2);
\draw[line width=1pt,color=ffqqqq](3,-2)-- (3,-3);
\draw[line width=1pt,color=ffqqqq](-2,-2)-- (-1,-2);
\draw[line width=1pt,color=ffqqqq](-1,-2)-- (-1,-1);
\draw[line width=1pt,color=ffqqqq](-1,-1)-- (-2,-1);
\draw[line width=1pt,color=ffqqqq](-2,-1)-- (-2,-2);
\draw[line width=1pt,color=ffqqqq](-1,-2)-- (0,-2);
\draw[line width=1pt,color=ffqqqq](0,-2)-- (0,-1);
\draw[line width=1pt,color=ffqqqq](0,-1)-- (-1,-1);
\draw[line width=1pt,color=ffqqqq](-1,-1)-- (-1,-2);
\draw[line width=1pt,color=ffqqqq](0,-2)-- (1,-2);
\draw[line width=1pt,color=ffqqqq](1,-2)-- (1,-1);
\draw[line width=1pt,color=ffqqqq](1,-1)-- (0,-1);
\draw[line width=1pt,color=ffqqqq](0,-1)-- (0,-2);
\draw[line width=1pt,color=ffqqqq](1,-2)-- (2,-2);
\draw[line width=1pt,color=ffqqqq](2,-2)-- (2,-1);
\draw[line width=1pt,color=ffqqqq](2,-1)-- (1,-1);
\draw[line width=1pt,color=ffqqqq](1,-1)-- (1,-2);
\draw[line width=1pt,color=ffqqqq](2,-2)-- (3,-2);
\draw[line width=1pt,color=ffqqqq](3,-2)-- (3,-1);
\draw[line width=1pt,color=ffqqqq](3,-1)-- (2,-1);
\draw[line width=1pt,color=ffqqqq](2,-1)-- (2,-2);
\draw[line width=1pt,color=ffqqqq](3,-2)-- (4,-2);
\draw[line width=1pt,color=ffqqqq](4,-2)-- (4,-1);
\draw[line width=1pt,color=ffqqqq](4,-1)-- (3,-1);
\draw[line width=1pt,color=ffqqqq](3,-1)-- (3,-2);
\draw[line width=1pt,color=ffqqqq](-2,-1)-- (-1,-1);
\draw[line width=1pt,color=ffqqqq](-1,-1)-- (-1,0);
\draw[line width=1pt,color=ffqqqq](-1,0)-- (-2,0);
\draw[line width=1pt,color=ffqqqq](-2,0)-- (-2,-1);
\draw[line width=1pt,color=ffqqqq](-1,-1)-- (0,-1);
\draw[line width=1pt,color=ffqqqq](0,-1)-- (0,0);
\draw[line width=1pt,color=ffqqqq](0,0)-- (-1,0);
\draw[line width=1pt,color=ffqqqq](-1,0)-- (-1,-1);
\draw[line width=1pt,color=ffqqqq](0,-1)-- (1,-1);
\draw[line width=1pt,color=ffqqqq](1,-1)-- (1,0);
\draw[line width=1pt,color=ffqqqq](1,0)-- (0,0);
\draw[line width=1pt,color=ffqqqq](0,0)-- (0,-1);
\draw[line width=1pt,color=ffqqqq](1,-1)-- (2,-1);
\draw[line width=1pt,color=ffqqqq](2,-1)-- (2,0);
\draw[line width=1pt,color=ffqqqq](2,0)-- (1,0);
\draw[line width=1pt,color=ffqqqq](1,0)-- (1,-1);
\draw[line width=1pt,color=ffqqqq](2,-1)-- (3,-1);
\draw[line width=1pt,color=ffqqqq](3,-1)-- (3,0);
\draw[line width=1pt,color=ffqqqq](3,0)-- (2,0);
\draw[line width=1pt,color=ffqqqq](2,0)-- (2,-1);
\draw[line width=1pt,color=ffqqqq](3,-1)-- (4,-1);
\draw[line width=1pt,color=ffqqqq](4,-1)-- (4,0);
\draw[line width=1pt,color=ffqqqq](4,0)-- (3,0);
\draw[line width=1pt,color=ffqqqq](3,0)-- (3,-1);
\draw[line width=1pt,color=ffqqqq](-2,0)-- (-1,0);
\draw[line width=1pt,color=ffqqqq](-1,0)-- (-1,1);
\draw[line width=1pt,color=ffqqqq](-1,1)-- (-2,1);
\draw[line width=1pt,color=ffqqqq](-2,1)-- (-2,0);
\draw[line width=1pt,color=ffqqqq](-1,0)-- (0,0);
\draw[line width=1pt,color=ffqqqq](0,0)-- (0,1);
\draw[line width=1pt,color=ffqqqq](0,1)-- (-1,1);
\draw[line width=1pt,color=ffqqqq](-1,1)-- (-1,0);
\draw[line width=1pt,color=ffqqqq](0,0)-- (1,0);
\draw[line width=1pt,color=ffqqqq](1,0)-- (1,1);
\draw[line width=1pt,color=ffqqqq](1,1)-- (0,1);
\draw[line width=1pt,color=ffqqqq](0,1)-- (0,0);
\draw[line width=1pt,color=ffqqqq](-2,1)-- (-1,1);
\draw[line width=1pt,color=ffqqqq](-1,1)-- (-1,2);
\draw[line width=1pt,color=ffqqqq](-1,2)-- (-2,2);
\draw[line width=1pt,color=ffqqqq](-2,2)-- (-2,1);
\draw[line width=1pt,color=ffqqqq](-1,1)-- (0,1);
\draw[line width=1pt,color=ffqqqq](0,1)-- (0,2);
\draw[line width=1pt,color=ffqqqq](0,2)-- (-1,2);
\draw[line width=1pt,color=ffqqqq](-1,2)-- (-1,1);
\draw[line width=1pt,color=ffqqqq](0,1)-- (1,1);
\draw[line width=1pt,color=ffqqqq](1,1)-- (1,2);
\draw[line width=1pt,color=ffqqqq](1,2)-- (0,2);
\draw[line width=1pt,color=ffqqqq](0,2)-- (0,1);
\draw[line width=1pt,color=ffqqqq](1,1)-- (2,1);
\draw[line width=1pt,color=ffqqqq](2,1)-- (2,2);
\draw[line width=1pt,color=ffqqqq](2,2)-- (1,2);
\draw[line width=1pt,color=ffqqqq](1,2)-- (1,1);
\draw[line width=1pt,color=ffqqqq](2,1)-- (3,1);
\draw[line width=1pt,color=ffqqqq](3,1)-- (3,2);
\draw[line width=1pt,color=ffqqqq](3,2)-- (2,2);
\draw[line width=1pt,color=ffqqqq](2,2)-- (2,1);
\draw[line width=1pt,color=ffqqqq](3,1)-- (4,1);
\draw[line width=1pt,color=ffqqqq](4,1)-- (4,2);
\draw[line width=1pt,color=ffqqqq](4,2)-- (3,2);
\draw[line width=1pt,color=ffqqqq](3,2)-- (3,1);
\draw[line width=1pt,color=ffqqqq](1,0)-- (2,0);
\draw[line width=1pt,color=ffqqqq](2,0)-- (2,1);
\draw[line width=1pt,color=ffqqqq](2,1)-- (1,1);
\draw[line width=1pt,color=ffqqqq](1,1)-- (1,0);
\draw[line width=1pt,color=ffqqqq](2,0)-- (3,0);
\draw[line width=1pt,color=ffqqqq](3,0)-- (3,1);
\draw[line width=1pt,color=ffqqqq](3,1)-- (2,1);
\draw[line width=1pt,color=ffqqqq](2,1)-- (2,0);
\draw[line width=1pt,color=ffqqqq](3,0)-- (4,0);
\draw[line width=1pt,color=ffqqqq](4,0)-- (4,1);
\draw[line width=1pt,color=ffqqqq](4,1)-- (3,1);
\draw[line width=1pt,color=ffqqqq](3,1)-- (3,0);
\draw[line width=1pt,color=ffqqqq](3,2)-- (4,2);
\draw[line width=1pt,color=ffqqqq](4,2)-- (4,3);
\draw[line width=1pt,color=ffqqqq](4,3)-- (3,3);
\draw[line width=1pt,color=ffqqqq](3,3)-- (3,2);
\draw[line width=1pt,color=ffqqqq](-2,2)-- (-1,2);
\draw[line width=1pt,color=ffqqqq](-1,2)-- (-1,3);
\draw[line width=1pt,color=ffqqqq](-1,3)-- (-2,3);
\draw[line width=1pt,color=ffqqqq](-2,3)-- (-2,2);
\draw[line width=1pt,color=ffqqqq](-1,2)-- (0,2);
\draw[line width=1pt,color=ffqqqq](0,2)-- (0,3);
\draw[line width=1pt,color=ffqqqq](0,3)-- (-1,3);
\draw[line width=1pt,color=ffqqqq](-1,3)-- (-1,2);
\draw[line width=1pt,color=ffqqqq](-1,3)-- (0,3);
\draw[line width=1pt,color=ffqqqq](0,3)-- (0,4);
\draw[line width=1pt,color=ffqqqq](0,4)-- (-1,4);
\draw[line width=1pt,color=ffqqqq](-1,4)-- (-1,3);
\draw[line width=1pt,color=ffqqqq](0,2)-- (1,2);
\draw[line width=1pt,color=ffqqqq](1,2)-- (1,3);
\draw[line width=1pt,color=ffqqqq](1,3)-- (0,3);
\draw[line width=1pt,color=ffqqqq](0,3)-- (0,2);
\draw[line width=1pt,color=ffqqqq](0,3)-- (1,3);
\draw[line width=1pt,color=ffqqqq](1,3)-- (1,4);
\draw[line width=1pt,color=ffqqqq](1,4)-- (0,4);
\draw[line width=1pt,color=ffqqqq](0,4)-- (0,3);
\draw[line width=1pt,color=ffqqqq](1,2)-- (2,2);
\draw[line width=1pt,color=ffqqqq](2,2)-- (2,3);
\draw[line width=1pt,color=ffqqqq](2,3)-- (1,3);
\draw[line width=1pt,color=ffqqqq](1,3)-- (1,2);
\draw[line width=1pt,color=ffqqqq](1,3)-- (2,3);
\draw[line width=1pt,color=ffqqqq](2,3)-- (2,4);
\draw[line width=1pt,color=ffqqqq](2,4)-- (1,4);
\draw[line width=1pt,color=ffqqqq](1,4)-- (1,3);
\draw[line width=1pt,color=ffqqqq](2,2)-- (3,2);
\draw[line width=1pt,color=ffqqqq](3,2)-- (3,3);
\draw[line width=1pt,color=ffqqqq](3,3)-- (2,3);
\draw[line width=1pt,color=ffqqqq](2,3)-- (2,2);
\draw[line width=1pt,color=ffqqqq](2,3)-- (3,3);
\draw[line width=1pt,color=ffqqqq](3,3)-- (3,4);
\draw[line width=1pt,color=ffqqqq](3,4)-- (2,4);
\draw[line width=1pt,color=ffqqqq](2,4)-- (2,3);
\draw[line width=1pt,color=ffqqqq](3,3)-- (4,3);
\draw[line width=1pt,color=ffqqqq](4,3)-- (4,4);
\draw[line width=1pt,color=ffqqqq](4,4)-- (3,4);
\draw[line width=1pt,color=ffqqqq](3,4)-- (3,3);
\draw [line width=2pt,color=qqqqff] (-9,3)-- (9,4);
\draw [line width=2pt,color=qqqqff] (-9,-1)-- (9,1);
\draw [line width=2pt,color=qqqqff] (-9,-6)-- (9,-2);
\draw [->,line width=1pt] (2.4,-4.8) -- (3,-4.8);
\draw [->,line width=1pt] (2.4,-4.8) -- (2,-4.8);
\draw [->,line width=1pt] (-6.004505832087556,3.3006812054859775) -- (-5.998009791280535,4.15559392147635);
\draw [->,line width=1pt] (-6.004505832087556,3.3006812054859775) -- (-5.998009791280535,2.1651764426523483);
\draw (-7.2805507209089106,4.012675516565678) node[anchor=north west] {$O(h^\eta)$};
\draw [->,line width=1pt] (0,5) -- (4,5);
\draw [->,line width=1pt] (0,5) -- (-4,5);
\draw (-1.617349075285628,6.118133370140026) node[anchor=north west] {$O(h^{\beta- \varepsilon_1})$};
\draw (1.9088648908632837,-5.059622551681511) node[anchor=north west] {$O(h^{\beta- \frac{\varepsilon_1}{2}})$};
\draw [color=ffzzqq](4.459378876087497,2.8315650133410437) node[anchor=north west] {$\widehat{X}_i$};
\draw [color=ffqqqq](3.826029765662692,-4.083922570756814) node[anchor=north west] {$X_{i,j}$};
\draw [color=qqqqff](-7.077845459758943,1.9585702935663143) node[anchor=north west] {$\mathcal{V}_{t,\eta}(\widetilde{x})$};
\end{tikzpicture}
\caption{An illustration of (\ref{eq:NumberOfIntersection})}\label{Fig:Intersec}
\end{center}
\end{figure}
\begin{proof}
\red{Set $\eta := \beta - \frac{3}{4}\varepsilon_1$, so that $h^\eta$ is an intermediate scale between the diameter of the $X_{i,j}$ and that of $\wih{X}_i$.}
For every $x\in X_{i,j}$,  and every $\wit{x}\in A_{x,t}$, 
we define
\begin{equation*}
	\mathcal{V}_{t, \eta}(\wit{x}):=  \left\{ y\in X ; \exists s\in [-t, t]   \text{ such that }
\mathrm{dist}_{X} \left( \Phi^{s}(\rho_{\wit{x}}), y\right)< h^{\eta}\right\}.
\end{equation*}
This set is the $h^{\eta}$-neighbourhood of a geodesic segment 
of length  $2t$, so that it \red{has volume $O(t h^{(d-1) \eta})$.  \MI{Thanks to Remark \ref{Rem:LeavingBall},} we see that the set of $s$ such that $ \Phi^{s}(\rho_{\wit{x}}) \in \wih{X}_i$ \MI{is included in a union of $O(t)$ intervals, each of them having length $O(h^{\beta - \varepsilon_1})$. }  We deduce that $\mathcal{V}_{t, \eta}(\wit{x})\cap \wih{X}_i$ has volume $O(t h^{(d-1) \eta + \beta - \varepsilon_1})$, so that 
the number of different sets $X_{i,j'}$ intersected by the set $\mathcal{V}_{t, \eta}(\wit{x})$ is 
\begin{equation}\label{eq:NumberOfIntersection}
\begin{aligned}
O(t h^{(d-1)( \eta - \beta) + \varepsilon_1 \left ( -1 +d/2\right)})&=   O( t h^{-d \frac{\varepsilon_1}{2}} h^{ \frac{(d-1) \varepsilon_1}{4}})\\
&= O( h^{-d \frac{\varepsilon_1}{2}} h^{ \frac{(d-1) \varepsilon_1}{4} - c \gd})
\end{aligned}
\end{equation}}
 for some $c>0$. \red{This is illustrated in Figure \ref{Fig:Intersec}.}
\par
Now, let $x'\in X_{i,j}$, so that $\mathrm{dist}_X(x,x') < h^{\beta- 
\frac{\varepsilon_1}{2}}$, and let $\wit{x}'_{\wit{x}}$ be a lift 
of $\wit{x}'$ such that $\mathrm{dist}_{\wit{X}} (\wit{x}, \wit{x}') 
= \mathrm{dist}_{X} (x,x')$. Thanks to Lemma \ref{Lem:Gron}, we have
\begin{equation*}
	\mathrm{dist}_{X} \left(\Phi^s(\rho_{\wit{x}}), 
	\Phi^s(\rho_{\wit{x}'}) \right)
	 < h^{\beta - \frac{1}{2}\varepsilon_1 - c'\gd},
\end{equation*}
for some $c'>0$, provided $h$ is small enough. Hence,
\begin{equation*}
	\mathcal{V}_{t, \eta}(\wit{x}'_{\wit{x}}) \subset \mathcal{V}_{t, \eta - \varepsilon_2}(\wit{x}),
\end{equation*}
for any $\varepsilon_2>0$, provided that $\gd$ is small enough \red{(recall that $\eta < \beta - \frac{1}{2} \varepsilon_1$).} 
Therefore, the set $\bigcup_{\wit{x}\in A_{x,t}} \bigcup_{x'\in X_{i,j}} 
\mathcal{V}_{t, \eta}(\wit{x}'_{\wit{x}})$ intersects at most 
$O( h^{-d \frac{\varepsilon_1}{2}} h^{ \frac{(d-1) \varepsilon_1}{4} - c \gd})$ different sets $X_{i,j'}$.  
Taking \red{$\gamma = \frac{(d-1) \varepsilon_1}{4} - c' \gd$} and using \red{the last statement in}
Proposition \ref{PropSection8}, the result follows.
\end{proof}
Let $p\in \N$. Let us estimate
\begin{align*}
\E_\omega \left[ \left| Z- \E[Z] \right|^{2p} \right] 
&= \E_\omega \left[ \left|\sum_{j\in J_{h,i}}
	 \left(Z_j - \E[Z_j]\right) \right|^{2p} \right]\\
&= \sum_{j_1, ..., j_{2p} \in J_{h,i}} \E_\omega 
	\left[  \left(Z_{j_1} - \E_\omega[Z_{j_1}]\right) \cdots 
		 \overline{\left(Z_{j_{2p}} - \E_\omega[Z_{j_{2p}}]\right)}
    \right],
\end{align*}
where the last $p$ factors come with complex conjugates. Let us write 
\begin{equation*}
\mathcal{J}_p := \left\{(j_1,..., j_{2p})\in (J_{h,i})^{2p} 
\text{ such that } \E_\omega
 \left[  \left(Z_{j_1} - \E_\omega[Z_{j_1}]\right) \cdots 
 	\overline{\left(Z_{j_{2p}} - \E_\omega[Z_{j_{2p}}]\right)} 
\right] \neq 0 \right\}.
\end{equation*}
\begin{lem}\label{lem:BoundNumberTerms}
We have
\begin{equation*}
|\mathcal{J}_p| = O\left(h^{p \left(- d \varepsilon_1+ \gamma\right)}\right).
\end{equation*}
\end{lem}
We will prove Lemma \ref{lem:BoundNumberTerms} below. For now we continue with 
the proof of \eqref{eq:PointwiseLimit_n1}. 
\par 
Using the fact that 
\begin{equation*}
	|Z_j| \leq \sup |F| \times \frac{\mathrm{Vol}(X_{i,j}) }{\mathrm{Vol} (\wih{X}_i)} 
\leq \sup |F| \times C h^{\frac{d \varepsilon_1}{2}},
\end{equation*}
we deduce that
\begin{align*}
\E_\omega \left[ \left| Z- \E_\omega[Z] \right|^{2p} \right] 
&\leq |\mathcal{J}_p| (\sup |F| \times C h^{\frac{d \varepsilon_1}{2}})^{2p} \\
&\leq C(p,F) h^{p \gamma}. 
\end{align*}
We may then apply Markov inequality to deduce that
\begin{equation*}
	\mathbb{P} \left[ |Z- \E_\omega[Z]| 
	> h^{\MI{\frac{\gamma}{4}}}\right] 
	\leq C_p h^{p \red{\frac{\gamma}{2}}}.
\end{equation*}

Since this is valid for any $p\in \N$, we have found an $\varepsilon_2>0$ such that
\begin{equation*}
	\mathbb{P} \left[ |Z- \E_\omega[Z]| > h^{\varepsilon_2}\right] = O(h^\infty),
\end{equation*}
proving \eqref{eq:PointwiseLimit_n1}.  
\begin{proof}[Proof of Lemma \ref{lem:BoundNumberTerms}]
Given $k \in \{0,..., \red{2p}\}$ and $(i_1, j'_{i_1}), ..., (i_k, j'_{i_k}) \in \{1,...,2p\}
\times J_{h,i}$, let us write
\begin{equation*}
	\mathcal{J}_{p ; (i_1, j'_{i_1}), ..., (i_k, j'_{i_k})}
	:= \left\{ (j_1, ..., j_{\red{2}p})\in \mathcal{J}_p 
	\text{ with } j_{i_\ell} = j'_{i_{\ell}} 
	\text{ for all } \ell \in \{1,...,k\}\right\}.
\end{equation*}
In particular, $\mathcal{J}_p = \mathcal{J}_{p, \emptyset}$, 
corresponding to $k=0$. We shall write
\begin{equation*}
	\mathfrak{J}_{p,k} 
	:= \max_{(i_1, j'_{i_1}), ..., (i_k, j'_{i_k})} 
	|\mathcal{J}_{p ; (i_1, j'_{i_1}), ..., (i_k, j'_{i_k})}|,
\end{equation*}
\MI{where the indices $i_1,..., i_k$ are all different.}
We have $\mathfrak{J}_{p, 2p} =1$. We will thus do a descending 
recurrence to estimate $\mathfrak{J}_{p,k}$. First of all, 
note that, for any $k \in \{0,..., 2p-2\}$, any 
$(i_1, j'_{i_1}), ..., (i_k, j'_{i_k}) \in \{1,...,2p\} \times J_{h,i}$ 
and any $i_{k+1}\in \{1,..., 2p\} \setminus \{i_1,..., i_k\}$, we have
\begin{equation*}
	|\mathcal{J}_{p ; (i_1, j'_{i_1}), ..., (i_k, j'_{i_k})}| 
	\leq \sum_{j'_{i_{k+1}}}  
	|\mathcal{J}_{p ; (i_1, j'_{i_1}), ..., (i_{k+1}, j'_{i_{k+1}})}|,
\end{equation*}
so that
\begin{equation}\label{eq:RecCombinFacile}
\mathfrak{J}_{p,k} \leq |J_{h,i}| \times \mathfrak{J}_{p,k+1} \red{\leq C h^{-d \frac{\varepsilon_1}{2}} \mathfrak{J}_{p,k+1}}.
\end{equation}
We shall write a more precise recurrence estimate, using 
independence. Indeed, suppose that $(j_1,..., j_{2p}) 
\in \mathcal{J}_{p ; (i_1, j'_{i_1}), ..., (i_k, j'_{i_k})}$ 
for some $k\leq 2p-2$, and let $i_{k+1}\in \{1,..., 2p\} \setminus \{i_1,..., i_k\}$.
\par 
Suppose that $j_{i_{k+1}} \notin \bigcup_{\ell=1}^k \mathcal{C}_{j_{\red{i_\ell}}}$. 
Then, there must exist $i_{k+2} \in \{1,..., 2p\} \setminus \{i_1,..., i_{k+1}\}$ 
with $j_{i_{k+2}} \in \mathcal{C}_{j_{i_{k+1}}}$. Indeed, if this were not 
the case, the variable $Z_{i_{k+1}} - \E_\omega[Z_{i_{k+1}}]$ (or its 
complex conjugate) would be independent from all the other variables 
appearing in 
\begin{equation*}
	\E_\omega \left[  \left(Z_{j_1} - \E_\omega[Z_{j_1}]\right) 
\cdots  \overline{\left(Z_{j_{2p}} - \E_\omega[Z_{j_{2p}}]\right)} \right],
\end{equation*}
so that the expectation would be zero. Therefore, if 
$j_{i_{k+1}} \notin \bigcup_{\ell=1}^k \mathcal{C}_{j_{\red{i_\ell}}}$, we have 
\begin{equation*}
 \mathcal{J}_{p ; (i_1, j'_{i_1}), ..., (i_k, j'_{i_k}), (i_{k+1}, j_{i_{k+1}})} 
 \subset \bigcup_{i_{k+2} \in \{1,..., 2p\} \setminus \{i_1,..., i_{k+1}\}} 
 	\bigcup_{j_{i_{k+2}} \in \mathcal{C}_{j_{i_{k+1}}}} 
	 \mathcal{J}_{p ; (i_1, j'_{i_1}), ..., (i_k, j'_{i_k}), (i_{k+2}, j_{i_{k+2}})},
\end{equation*}
so that
\begin{equation}\label{eq:RecIndep}
\forall j_{i_{k+1}} \notin \bigcup_{\ell=1}^k \mathcal{C}_{j_{\red{i_\ell}}}, ~~  
|\mathcal{J}_{p ; (i_1, j'_{i_1}), ..., (i_k, j'_{i_k}), (i_{k+1}, j_{i_{k+1}})}| 
\leq C h^{-d\frac{\varepsilon_1}{2} + \gamma} \mathfrak{J}_{p, k+2}.
\end{equation}
All in all, we may write
\begin{align*}
|\mathcal{J}_{p ; (i_1, j'_{i_1}), ..., (i_k, j'_{i_k})}|
 &\leq \sum_{j_{i_{k+1}}\in J_{h,i}}  
 	|\mathcal{J}_{p ; (i_1, j'_{i_1}), ..., (i_k, j'_{i_k}), (i_{k+1}, j_{i_{k+1}})}|\\
& \leq \sum_{\underset{j_{i_{k+1}} \in 
	\bigcup_{\ell=1}^k \mathcal{C}_{j_{\red{i_\ell}}}}{j_{i_{k+1}}\in J_{h,i}}}  
	|\mathcal{J}_{p ; (i_1, j'_{i_1}), ..., (i_k, j'_{i_k}), (i_{k+1}, j_{i_{k+1}})}| \\
& \phantom{\leq}	
	+\sum_{\underset{j_{i_{k+1}} \notin \bigcup_{\ell=1}^k 
	\mathcal{C}_{j_{\red{i_\ell}}}}{j_{i_{k+1}}\in J_{h,i}}}   
	|\mathcal{J}_{p ; (i_1, j'_{i_1}), ..., (i_k, j'_{i_k}), (i_{k+1}, j_{i_{k+1}})}|\\
&\stackrel{\eqref{eq:RecIndep}}{\leq}  C h^{-d\frac{\varepsilon_1}{2} + \gamma} \mathfrak{J}_{p, k+1} 
+ \sum_{j_{i_{k+1}}\in J_{h,i}} C h^{d\frac{\varepsilon_1}{2} + \gamma} \mathfrak{J}_{p, k+2}\\
&\leq C h^{-d\frac{\varepsilon_1}{2} + \gamma} |J_{h,i}| \mathfrak{J}_{p, k+2},
\end{align*}
thanks to (\ref{eq:RecCombinFacile}).
\par 
We conclude that
\begin{equation*}
	\mathfrak{J}_{p, k} \leq C h^{-d\frac{\varepsilon_1}{2} + \gamma} | J_{h,i}| 
	\times   \mathfrak{J}_{p, k+2}\leq C  h^{-d\varepsilon_1 + \gamma} 
	\times   \mathfrak{J}_{p, k+2},
\end{equation*}
and the lemma follows by a descending induction.
\end{proof}
We may now conclude with the proof of the second part 
of Corollary \ref{cor:LWL},  \red{saying that $\omega$-almost surely, 
$\mathrm{BGF}_{\lambda_a}$ is the local weak limit of 
$\{\e^{\red{-}\frac{i}{h_j}t_{h_j}P_{h_j}^\delta}f_{h_j}\}_{h_j}$ in the 
frame $V$}.
\begin{proof}[Proof of Corollary \ref{cor:LWL}.]
Let $0<\varepsilon<1$.  Thanks to (\ref{eq:Step2SobNormEst}) along with 
(the proof of) \cite[Lemma 1]{LWL}, there exists a compact set 
$\mathcal{K}_{\varepsilon} \subset C^\infty(\R^d)$ such that for all 
$j\in \N$ and all $\omega$,
\begin{equation*}
	\mathbb{P}_{\mathrm{x}}[
		\e^{\frac{i}{h_j}t_{h_j}P_{h_j}^\delta}f_{h_j}(\exp_{\textsc{x}}(h_j\cdot)))
		\notin \mathcal{K}_{\varepsilon}] < \varepsilon.
\end{equation*}
The space $C_b(\mathcal{K}_\varepsilon)$ is separable, so we may take a 
sequence of functionals $(F_n)$ that is dense in 
$C_b(\mathcal{K}_\varepsilon)$.  We may extend each $F_n$ to $C^\infty(\R^d)$ 
so that $\|F_n\|_{C^\infty(\R^d)} = \|F_n\|_{\mathcal{K}_{\varepsilon}}$.
\par 
By the first part of Corollary \ref{cor:LWL}, $\omega$-almost surely, we have 
for all $n\in \N$
\begin{equation*}
	\erw_{\mathrm{x}}[
		F_n(\e^{\frac{i}{h_j}t_{h_j}P_{h_j}^\delta}f_{h_j}(\exp_{\textsc{x}}(h_j\cdot)))] 
		\to  \E_{\mathrm{BGF}_{\lambda_a}}[F_n], \quad j\to \infty.
\end{equation*}
Now,  let $F\in C_b (C^\infty (\R^d))$, and let $n\in \N$ be such that 
$|F(\psi)-F_{n}(\psi)|< \varepsilon$ for all $\psi \in \mathcal{K}_\varepsilon$. 
In particular,  for all $j\in \N$ and all $\omega$, we have
\begin{equation*}
	\left|\erw_{\mathrm{x}}[
		(F- F_n)(\e^{\frac{i}{h_j}t_{h_j}P_{h_j}^\delta}f_{h_j}(\exp_{\textsc{x}}(h_j\cdot)))] \right| \leq \varepsilon (1+ \| F\| + \|F_n \|)\leq \varepsilon (2+ 2\| F\|),$$
		so that, $\omega$-almost surely,
		$$\limsup_{j\to \infty}  \left|\erw_{\mathrm{x}}[
		F(\e^{\frac{i}{h_j}t_{h_j}P_{h_j}^\delta}f_{h_j}(\exp_{\textsc{x}}(h_j\cdot)))] 
		-  \E_{\mathrm{BGF}_{\lambda_a}}[F] \right| \leq  \varepsilon (2+ 2\| F\|).
\end{equation*}
Since this is true for all $F$ and all $\varepsilon$, the result follows.
\end{proof}
\end{document}